\documentclass[11pt]{amsart}
\usepackage[pdftex,margin=1.1in, marginpar=.8in]{geometry}

\usepackage{amsmath}
\usepackage{amsfonts}
\usepackage{amssymb}
\usepackage{amsthm}

\usepackage[colorlinks=true,urlcolor=blue,linkcolor=blue,citecolor=blue,pagebackref,breaklinks]{hyperref}
\usepackage{cite}

\usepackage{mathrsfs}

\usepackage{bm}
\usepackage{dsfont}
\usepackage{upgreek}
\usepackage{xcolor}
\usepackage{bbold}

\usepackage{fullpage}

\usepackage{graphicx}
\usepackage[all,cmtip]{xy}
\usepackage{enumerate}
\usepackage{blkarray}
\usepackage{array}
\usepackage{todonotes}

\usepackage{leftidx}
\usepackage{mathdots}
\usepackage{arydshln}
\usepackage{stmaryrd}

\newenvironment{psm}
{\left(\begin{smallmatrix}}
{\end{smallmatrix}\right)}

\numberwithin{equation}{subsection}

\newtheorem{thm}{Theorem}[subsection]
\newtheorem{prop}[thm]{Proposition}
\newtheorem{lemma}[thm]{Lemma}

\newtheorem{defn}[thm]{Definition}
\newtheorem{conj}[thm]{Conjecture}

\theoremstyle{remark}
\newtheorem{rmk}[thm]{Remark}

\newcommand{\mr}[1]{\mathrm{#1}}

\newcommand{\bA}{\mathbb{A}}

\newcommand{\bC}{\mathbb{C}}

\newcommand{\bF}{\mathbb{F}}
\newcommand{\bG}{\mathbb{G}}

\newcommand{\bQ}{\mathbb{Q}}
\newcommand{\bR}{\mathbb{R}}

\newcommand{\bT}{\mathbb{T}}
\newcommand{\bU}{\mathbb{U}}
\newcommand{\bV}{\mathbb{V}}
\newcommand{\bW}{\mathbb{W}}

\newcommand{\bZ}{\mathbb{Z}}

\newcommand{\xx}{{\tt x}}
\newcommand{\yy}{{\tt y}}
\newcommand{\ww}{{\tt w}}

\newcommand{\cA}{\mathcal{A}}

\newcommand{\cC}{\mathcal{C}}
\newcommand{\cD}{\mathcal{D}}
\newcommand{\cE}{\mathcal{E}}
\newcommand{\cF}{\mathcal{F}}
\newcommand{\cG}{\mathcal{G}}

\newcommand{\cI}{\mathcal{I}}

\newcommand{\cK}{\mathcal{K}}
\newcommand{\cL}{\mathcal{L}}
\newcommand{\cM}{\mathcal{M}}

\newcommand{\cO}{\mathcal{O}}
\newcommand{\cP}{\mathcal{P}}

\newcommand{\cS}{\mathcal{S}}
\newcommand{\cT}{\mathcal{T}}
\newcommand{\cU}{\mathcal{U}}

\newcommand{\fd}{\mathfrak{d}}
\newcommand{\fe}{\mathfrak{e}}
\newcommand{\ff}{\mathfrak{f}}
\newcommand{\fg}{\mathfrak{g}}

\newcommand{\ffi}{\mathfrak{i}}

\newcommand{\fm}{\mathfrak{m}}

\newcommand{\fp}{\mathfrak{p}}
\newcommand{\fq}{\mathfrak{q}}
\newcommand{\fr}{\mathfrak{r}}
\newcommand{\fs}{\mathfrak{s}}

\newcommand{\fv}{\mathfrak{v}}

\newcommand{\fS}{\mathfrak{S}}

\newcommand{\fX}{\mathfrak{X}}
\newcommand{\fY}{\mathfrak{Y}}

\newcommand{\sA}{\mathscr{A}}

\newcommand{\sC}{\mathscr{C}}

\newcommand{\sG}{\mathscr{G}}

\newcommand{\sL}{\mathscr{L}}

\newcommand{\sS}{\mathscr{S}}
\newcommand{\sT}{\mathscr{T}}

\newcommand{\sY}{\mathscr{Y}}
\newcommand{\sZ}{\mathscr{Z}}

\DeclareMathAlphabet{\mathpzc}{OT1}{pzc}{m}{it}

\newcommand{\pP}{\mathpzc{P}}
\newcommand{\pV}{\mathpzc{V}}

\newcommand{\pzS}{\mathpzc{S}}

\newcommand{\be}{\mathbf{e}}

\newcommand{\ut}{\underline{t}}

\DeclareMathOperator{\GL}{GL}

\DeclareMathOperator{\GU}{GU}

\DeclareMathOperator{\U}{U}

\DeclareMathOperator{\Sp}{Sp}

\newcommand{\sk}{\vspace{0.1in}}

\newcommand{\alg}{\mathrm{alg}}

\newcommand{\can}{\mathrm{can}}

\newcommand{\cond}{\mathrm{cond}}
\newcommand{\cont}{\mathrm{cont}}

\newcommand{\diag}{\mathrm{diag}}

\newcommand{\id}{\mathrm{id}}
\newcommand{\im}{\mathrm{Im}}
\newcommand{\Nm}{\mathrm{Nm}}
\newcommand{\ord}{\mathrm{ord}}

\newcommand{\rmf}{\mathrm{f}}
\newcommand{\sgn}{\mathrm{sgn}}

\newcommand{\Tr}{\mathrm{Tr}}
\newcommand{\triv}{\mathrm{triv}}
\newcommand{\tor}{\mathrm{tor}}
\newcommand{\ur}{\mathrm{ur}}

\newcommand{\val}{\mathrm{val}}
\newcommand{\vol}{\mathrm{vol}}

\DeclareMathOperator{\Gal}{Gal}
\DeclareMathOperator{\Her}{Her}
\DeclareMathOperator{\Hom}{Hom}
\DeclareMathOperator{\Ind}{Ind}
\DeclareMathOperator{\Isom}{Isom}

\DeclareMathOperator{\Lie}{Lie}

\DeclareMathOperator{\Spec}{Spec}

\newcommand{\adic}{\text{-}\mathrm{adic}}

\newcommand{\lhra}{\ensuremath{\lhook\joinrel\longrightarrow}}
\newcommand{\lra}{\longrightarrow}
\newcommand{\ra}{\rightarrow}
\newcommand{\hra}{\hookrightarrow}

\newcommand{\ol}{\overline}
\newcommand{\wh}{\widehat}
\newcommand{\wt}{\widetilde}

\newcommand{\llb}{\llbracket}
\newcommand{\rrb}{\rrbracket}

\newcommand{\bid}{\mathbf{1}}

\newdir{d}{{\subset}}   

\makeatletter
\def\l@section{\@tocline{1}{0pt}{1pc}{}{}}
\def\l@subsection{\@tocline{2}{0pt}{1pc}{4.6em}{}}
\def\l@subsubsection{\@tocline{3}{0pt}{1pc}{7.6em}{}}
\renewcommand{\tocsection}[3]{%
  \indentlabel{\@ifnotempty{#2}{\makebox[2.3em][l]{%
    \ignorespaces#1 #2.\hfill}}}#3}
\renewcommand{\tocsubsection}[3]{%
  \indentlabel{\@ifnotempty{#2}{\hspace*{2.3em}\makebox[2.3em][l]{%
    \ignorespaces#1 #2.\hfill}}}#3}
\renewcommand{\tocsubsubsection}[3]{%
  \indentlabel{\@ifnotempty{#2}{\hspace*{4.6em}\makebox[3em][l]{%
    \ignorespaces#1 #2.\hfill}}}#3}
\makeatother

\newcommand{\clabel}{g}

\usepackage[shortlabels]{enumitem}

\newcommand{\so}{\mathrm{so}}
\newcommand{\Ig}{\mathscr{T}}
\newcommand{\mult}{\mathrm{mult}}
\newcommand{\Lsub}{C}
\newcommand{\Meas}{\cM eas}
\newcommand{\Sieg}{\mr{Sieg}}
\newcommand{\Kling}{\mr{Kling}}
\newcommand{\FJ}{\mathrm{FJ}}

\newcommand{\bfE}{\bm{E}}
\newcommand{\bfh}{\bm{h}}
\newcommand{\bmtheta}{\protect\bm{\theta}}
\newcommand{\sph}{\mathrm{sph}}
\newcommand{\wtp}{{\tt N}}
\newcommand{\ns}{\mathrm{ns}}
\newcommand{\rms}{\mathrm{s}}
\newcommand{\Sel}{\mr{Sel}}

\makeatletter
\newcommand*\bdot{\mathpalette\bdot@{.5}}
\newcommand*\bdot@[2]{\mathbin{\vcenter{\hbox{\scalebox{#2}{$\,\,\m@th#1\bullet\,\,$}}}}}
\makeatother

\usepackage{tikz-cd}
\newcommand{\pzy}{\mathpzc{y}}
\newcommand{\pzR}{\mathpzc{R}}
\newcommand{\pzT}{\mathpzc{T}}

\usepackage{float}

\usepackage[section]{glossaries}

\usepackage{glossary-mcols}  
\glsdisablehyper

\setglossarystyle{mcolindex}  

\makeglossaries

\newglossaryentry{G}
{
  name={\ensuremath{G,G'}},
  description={},
  user1={\ensuremath{G}},
  user2={\ensuremath{G'}},
  sort=G
}

\newglossaryentry{xi}
{
  name={\ensuremath{\xi,\xi_0,\xi_{p\adic}}},
  description={},
  user1={\ensuremath{\xi_0}},
  user2={\ensuremath{\xi}},
  user3={\ensuremath{\xi_{p\adic}}},
  sort=xi
}

\newglossaryentry{eta}
{
  name={\ensuremath{\eta_{\cK/\bQ}}},
  description={},
  sort=eta
}

\newglossaryentry{k0}
{
  name={\ensuremath{k_0}},
  description={},
  sort=k0
}

\newglossaryentry{Disc}
{
  name={\ensuremath{D_{\cK/\bQ}}},
  description={},
  sort=Disc
}

\newglossaryentry{L1}
{
  name={\ensuremath{\mathcal{L}_1,\mathcal{L}_2}},
  description={},
  user1={\ensuremath{\mathcal{L}_1}},
  user2={\ensuremath{\mathcal{L}_2}},
  sort=L1,
}

\newglossaryentry{L3}
{
  name={\ensuremath{\mathcal{L}_3,\mathcal{L}_4}},
  description={},
  user1={\ensuremath{\mathcal{L}_3}},
  user2={\ensuremath{\mathcal{L}_4}},
  sort=L3
}

\newglossaryentry{L5}
{
  name={\ensuremath{\mathcal{L}_5,\mathcal{L}_6}},
  description={},
  user1={\ensuremath{\mathcal{L}_5}},
  user2={\ensuremath{\cL_6}},
  sort=L5
}

\newglossaryentry{iota}
{
  name={\ensuremath{\iota_\infty,\iota_p}},
  description={},
  sort=iota,
  user1={\ensuremath{\iota_\infty}},
  user2={\ensuremath{\iota_p}}
}

\newglossaryentry{rho}
{
  name={\ensuremath{\varrho_\fp,\varrho_{\bar{\fp}}}},
  description={},
  user1={\ensuremath{\varrho_\fp}},
  user2={\ensuremath{\varrho_{\bar{\fp}}}},
  sort=rho
}

\newglossaryentry{ev}
{
  name={\ensuremath{\be_v}},
  description={},
  sort=ev
}

\newglossaryentry{sS}
{
  name={\ensuremath{\sS}},
  description={},
  sort=sS
}

\newglossaryentry{sStor}
{
  name={\ensuremath{\sS^\tor}},
  description={},
  sort=sStor
}

\newglossaryentry{IS}
{
  name={\ensuremath{\cI_{\sS^\tor}}},
  description={},
  sort=IS
}

\newglossaryentry{sSmin}
{
  name={\ensuremath{\sS^{\min}}},
  description={},
  sort=sSmin
}

\newglossaryentry{omega}
{
  name={\ensuremath{\underline{\omega},\omega}},
  description={},
  user1={\ensuremath{\underline{\omega}}},
  user2={\ensuremath{\omega}},
  sort=omega
}

\newglossaryentry{Ha}
{
  name={\ensuremath{\mr{Ha}}},
  description={},
  sort=Hext2
}

\newglossaryentry{E}
{
  name={\ensuremath{E}},
  description={},
  sort=E
}

\newglossaryentry{ttD}
{
  name={\ensuremath{{\tt D},{\tt D}^1,{\tt D}^0,{\tt D}^{-1}}},
  description={},
  user1={\ensuremath{{\tt D}}},
  user2={\ensuremath{{\tt D}^1}},
  user3={\ensuremath{{\tt D}^0}},
  user4={\ensuremath{{\tt D}^{-1}}},
  sort=D
}

\newglossaryentry{PD}
{
  name={\ensuremath{P_{\tt D}}},
  description={},
  sort=PD
}

\newglossaryentry{Kpn}
{
  name={\ensuremath{K^1_{p,n},K^0_{p,n}}},
  description={},
  user1={\ensuremath{K^1_{p,n}}},
  user2={\ensuremath{K^0_{p,n}}},
  sort=Kpn
}

\newglossaryentry{Ign}
{
  name={\ensuremath{\Ig_n,\Ig^\tor_n}},
  description={},
  user1={\ensuremath{\Ig_n}},
  user2={\ensuremath{\Ig^\tor_n}},
  sort=Tn
}

\newglossaryentry{Ignm}
{
  name={\ensuremath{\Ig_{n,m},\Ig^\tor_{n,m}}},
  description={},
  user1={\ensuremath{\Ig_{n,m}}},
  user2={\ensuremath{\Ig^\tor_{n,m}}},
  sort=Tnm
}

\newglossaryentry{Sm}
{
  name={\ensuremath{\sS^\tor_{m}}},
  description={},
  sort=Sm
}

\newglossaryentry{F}
{
  name={\ensuremath{{\tt F^\pm}}},
  description={},
  sort=F
}

\newglossaryentry{delta+-}
{
  name={\ensuremath{\delta^+_3,\delta^-_1}},
  description={},
  user1={\ensuremath{\delta^+_3}},
  user2={\ensuremath{\delta^-_1}},
  sort=delta+-
}

\newglossaryentry{Vnm}
{
  name={\ensuremath{V_{n,m},V^0_{n,m}}},
  description={},
  user1={\ensuremath{V_{n,m}}},
  user2={\ensuremath{V^0_{n,m}}},
  sort=Vnm
}

\newglossaryentry{pV}
{
  name={\ensuremath{\pV,\pV^0}},
  description={},
  user1={\ensuremath{\pV}},
  user2={\ensuremath{\pV^0}},
  sort=VV
}

\newglossaryentry{Tso}
{
  name={\ensuremath{T_{\so}}},
  description={},
  sort=Tso
}

\newglossaryentry{ut}
{
  name={\ensuremath{\ut}},
  description={},
  sort=t
}

\newglossaryentry{Wut}
{
  name={\ensuremath{W_{\ut}}},
  description={},
  sort=Wut
}

\newglossaryentry{omegaut}
{
  name={\ensuremath{\omega_{\ut}}},
  description={},
  sort=omegaut
}

\newglossaryentry{Mt}
{
  name={\ensuremath{M_{(0,0,t^+;t^-)}\left(K^p_fK^1_{p,n},\epsilon^+,\epsilon^-;F\right)}},
  description={},
  sort=Mt
}

\newglossaryentry{Mt0}
{
  name={\ensuremath{M^0_{(0,0,t^+;t^-)}\left(K^p_fK^1_{p,n},\epsilon^+,\epsilon^-;F\right)}},
  description={},
  sort=Mt0
}

\newglossaryentry{Cj}
{
  name={\ensuremath{\sC^{1+}_{2,n,m},\sC^{1+}_{2,n,m},\sC^{1-}_{1,n,m}}},
  description={},
  user1={\ensuremath{\sC^{1+}_{j,n,m}}},
  user2={\ensuremath{\sC^{1-}_{1,n,m}}},
  sort=Cj
}

\newglossaryentry{Up}
{
  name={\ensuremath{U^+_{p,2},U^+_{p,3},U^-_{p,1}}},
  description={},
  user1={\ensuremath{U^+_{p,j}}},
  user2={\ensuremath{U^-_{p,1}}},
  user3={\ensuremath{U^+_{p,2}}},
  user4={\ensuremath{U^+_{p,3}}},
  sort=Up
}

\newglossaryentry{Lambdaso}
{
  name={\ensuremath{\Lambda_\so}},
  description={},
  sort=Lambda
}

\newglossaryentry{eso}
{
  name={\ensuremath{e_\so}},
  description={},
  sort=eso
}

\newglossaryentry{cM}
{
  name={\ensuremath{\cM^0_\so,\cM_\so}},
  user1={\ensuremath{\cM^0_\so}},
  description={},
  sort=M
}

\newglossaryentry{clabel}
{
  name={\ensuremath{({\tt Z},\Phi,\delta)}},
  description={},
  sort=Z
}

\newglossaryentry{clabel1}
{
  name={\ensuremath{({\tt Z}^{(1)},\Phi^{(1)},\delta^{1)})}},
  description={},
  sort=Z
}

\newglossaryentry{clabelg}
{
  name={\ensuremath{({\tt Z}^{(\clabel)},\Phi^{(\clabel)},\delta^{\clabel)})}},
  description={},
  sort=Z
}

\newglossaryentry{P}
{
  name={\ensuremath{P,P'}},
  user1={\ensuremath{P}},
  user2={\ensuremath{P'}},
  description={},
  sort=P
}

\newglossaryentry{C(K)}
{
  name={\ensuremath{C(K^p_f)}},
  description={},
  sort=CK
}

\newglossaryentry{C(K)ord}
{
  name={\ensuremath{C(K^p_fK^1_{p,n})_\ord}},
  description={},
  sort=CKord
}

\newglossaryentry{Xgn}
{
  name={\ensuremath{\fX_{\clabel,n}}},
  description={},
  sort=Xgn
}

\newglossaryentry{PflatD}
{
  name={\ensuremath{P^{\flat}_{\tt D}(\bQ_p)}},
  description={},
  sort=PflatD
}

\newglossaryentry{Vflatnm}
{
  name={\ensuremath{V^\flat_{n,m}}},
  description={},
  sort=Vflat
}

\newglossaryentry{Zg}
{
  name={\ensuremath{\sZ_{\clabel,n},\sY_{\clabel,n}}},
  description={},
  user1={\ensuremath{\sZ_{\clabel,n}}},
  user2={\ensuremath{\sY_{\clabel,n}}},
  sort=Zg
}

\newglossaryentry{Xi}
{
  name={\ensuremath{\Xi^\ord_{\clabel,n},\ol{\Xi}^\ord_{\clabel,n}}},
  user1={\ensuremath{\Xi^\ord_{\clabel,n}}},
  user2={\ensuremath{\ol{\Xi}^\ord_{\clabel,n}}},
  description={},
  sort=Xi
}

\newglossaryentry{Cgn}
{
  name={\ensuremath{\sC^\ord_{\clabel,n}}},
  description={},
  sort=Cgn
}

\newglossaryentry{K'pg}
{
  name={\ensuremath{K^{\prime}_{p,\clabel}}},
  description={},
  sort=Kg
}

\newglossaryentry{K'fgn}
{
  name={\ensuremath{K^{\prime}_{f,\clabel,n}}},
  description={},
  sort=Kfgn
}

\newglossaryentry{Phig}
{
  name={\ensuremath{\Phi_\clabel}},
  description={},
  sort=Phig
}

\newglossaryentry{Cab}
{
  name={\ensuremath{C^{\rmf},C^{\mr{ab}}}},
  user1={\ensuremath{C^{\rmf}}},
  user2={\ensuremath{C^{\mr{ab}}}},
  description={},
  sort=Cab
}

\newglossaryentry{VJ}
{
  name={\ensuremath{V^{J}_{\GU(2)}}},
  description={},
  sort=VJ
}

\newglossaryentry{VJbeta}
{
  name={\ensuremath{V^{J,\beta}_{\GU(2)}}},
  description={},
  sort=VJ
}

\newglossaryentry{FJ}
{
  name={\ensuremath{\mr{FJ}_\beta,\FJ_{\clabel,\beta}}},
  user1={\ensuremath{\mr{FJ}_\beta}},
  user2={\ensuremath{\FJ_{\clabel,\beta}}},
  description={},
  sort=FJ
}

\newglossaryentry{Vflat}
{
  name={\ensuremath{\pV^\flat}},
  description={},
  sort=Vflat
}

\newglossaryentry{sSG'}
{
  name={\ensuremath{\sS_{G',K'_{f,\clabel,n}}}},
  description={},
  sort=sSzG'
}

\newglossaryentry{Gray}
{
  name={\ensuremath{\cG^\natural}},
  description={},
  sort=Gray
}

\newglossaryentry{cT}
{
  name={\ensuremath{\cT}},
  description={},
  sort=T
}

\newglossaryentry{cGmu}
{
  name={\ensuremath{(\cG[N])^\mu,(\cG[N])^\rmf,(\cG[N])^{\mr{ab}}}},
  user1={\ensuremath{(\cG[N])^\mu}},
  user2={\ensuremath{(\cG[N])^{\rmf}}},
  user3={\ensuremath{(\cG[N])^{\mr{ab}}}},
  description={},
  sort=Gmu
}

\newglossaryentry{etapt}
{
  name={\ensuremath{\eta,\bar{\eta}}},
  user1={\ensuremath{\eta}},
  user2={\ensuremath{\bar{\eta}}},
  description={},
  sort=etapt
}

\newglossaryentry{cK}
{
  name={\ensuremath{\cK}},
  description={},
  sort=K1
}

\newglossaryentry{Kinfty}
{
  name={\ensuremath{\cK_\infty}},
  description={},
  sort=K2
}

\newglossaryentry{GammaK}
{
  name={\ensuremath{\Gamma_\cK}},
  description={},
  sort=GammaK
}

\newglossaryentry{Omega}
{
  name={\ensuremath{\Omega_\infty,\Omega_p}},
  user1={\ensuremath{\Omega_\infty}},
  user2={\ensuremath{\Omega_p}},
  description={},
  sort=Omega
}

\newglossaryentry{pi}
{
  name={\ensuremath{\pi}},
  description={},
  sort=pi
}

\newglossaryentry{BCpi}
{
  name={\ensuremath{\mr{BC}(\pi)}},
  description={},
  sort=BCpi
}

\newglossaryentry{L}
{
  name={\ensuremath{L}},
  description={},
  sort=L
}

\newglossaryentry{rhopi}
{
  name={\ensuremath{\rho_\pi}},
  description={},
  sort=rhopi
}

\newglossaryentry{PsiK}
{
  name={\ensuremath{\Psi_\cK}},
  description={},
  sort=PsiK
}

\newglossaryentry{Sel}
{
  name={\ensuremath{\Sel_{\pi,\cK,\xi}}},
  description={},
  sort=Sel
}

\newglossaryentry{Xpi}
{
  name={\ensuremath{X_{\pi,\cK,\xi}}},
  description={},
  sort=Xpi
}

\newglossaryentry{char}
{
  name={\ensuremath{\mr{char}_A(M)}},
  user1={},
  user2={},
  description={},
  sort=char
}

\newglossaryentry{Lpi}
{
  name={\ensuremath{\cL_{\pi,\cK,\xi}}},
  description={},
  sort=Lpi
}

\newglossaryentry{Lpipartial}
{
  name={\ensuremath{\cL^{\Sigma\,\cup\{\ell,\ell'\}}_{\pi,\cK,\xi}}},
  description={},
  sort=Lpi
}

\newglossaryentry{q}
{
  name={\ensuremath{q}},
  description={},
  sort=q
}

\newglossaryentry{D}
{
  name={\ensuremath{D}},
  description={},
  sort=D
}

\newglossaryentry{zeta}
{
  name={\ensuremath{\zeta_0,\zeta}},
  user1={\ensuremath{\zeta_0}},
  user2={\ensuremath{\zeta}},
  description={},
  sort=zeta
}

\newglossaryentry{delta}
{
  name={\ensuremath{\delta}},
  description={},
  sort=delta
}

\newglossaryentry{fs}
{
  name={\ensuremath{\fs}},
  description={},
  sort=s
}

\newglossaryentry{piD}
{
  name={\ensuremath{\pi^D}},
  description={},
  sort=piD
}

\newglossaryentry{piU}
{
  name={\ensuremath{\pi^{\GU(2)}}},
  description={},
  sort=piU
}

\newglossaryentry{f}
{
  name={\ensuremath{f,f^D,f^{\GU(2)}}},
  user1={\ensuremath{f}},
  user2={\ensuremath{f^D}},
  user3={\ensuremath{f^{\GU(2)}}},
  description={},
  sort=
}

\newglossaryentry{PU}
{
  name={\ensuremath{P_{\GU(3,1)}}},
  description={},
  sort=PU
}

\newglossaryentry{tau}
{
  name={\ensuremath{\tau,\tau_o,\tau_{p\adic}}},
  user1={\ensuremath{\tau}},
  user2={\ensuremath{\tau_0}},
  user3={\ensuremath{\tau_{p\adic}}},
  description={},
  sort=
}

\newglossaryentry{EklingF}
{
  name={\ensuremath{E^\Kling(-;F(s,\xi_0\tau_0))}},
  user1={\ensuremath{E^\Kling(g;F(s,\xi_0\tau_0))}},
  description={},
  sort=EKingF
}

\newglossaryentry{imath}
{
  name={\ensuremath{\imath}},
  description={},
  sort=imath
}

\newglossaryentry{c}
{
  name={\ensuremath{c}},
  description={},
  sort=c
}

\newglossaryentry{cS}
{
  name={\ensuremath{\cS}},
  description={},
  sort=S
}

\newglossaryentry{QU}
{
  name={\ensuremath{Q_{\GU(3,3)}}},
  description={},
  sort=QU
}

\newglossaryentry{ESiegf}
{
  name={\ensuremath{E^{\Sieg}(-;f(s,\xi_0\tau_0))}},
  user1={\ensuremath{E^{\Sieg}(g;f(s,\xi_0\tau_0))}},
  description={},
  sort=ESiegf
}

\newglossaryentry{Fvarphi}
{
  name={\ensuremath{F(f(s,\xi_0\tau_0),\varphi)}},
  description={},
  sort=Fphi
}

\newglossaryentry{ell}
{
  name={\ensuremath{\ell,\ell'}},
  user1={\ensuremath{\ell}},
  user2={\ensuremath{\ell'}},
  description={},
  sort=l
}

\newglossaryentry{Sigma}
{
  name={\ensuremath{\Sigma,\Sigma_{\mr{s}},\Sigma_{\mr{ns}}}},
  user1={\ensuremath{\Sigma}},
  user2={\ensuremath{\Sigma_{\mr{s}}}},
  user3={\ensuremath{\Sigma_{\mr{ns}}}},
  description={},
  sort=Sigma
}

\newglossaryentry{chi}
{
  name={\ensuremath{\chi_{\theta},\chi_h}},
  description={},
  sort=chi
}

\newglossaryentry{cv}
{
  name={\ensuremath{c_v}},
  description={},
  sort=cv
}

\newglossaryentry{Kv}
{
  name={\ensuremath{K_v}},
  description={},
  sort=Kv
}

\newglossaryentry{Kpf}
{
  name={\ensuremath{K^p_f}},
  description={},
  sort=Kpf
}

\newglossaryentry{ftwisted}
{
  name={\ensuremath{f^{\GU(2)}_{\chi_h}}},
  description={},
  sort=ftwisted
}

\newglossaryentry{xiQ}
{
  name={\ensuremath{\xi^\bQ_0}},
  description={},
  sort=xiQ
}

\newglossaryentry{tauQ}
{
  name={\ensuremath{\tau^\bQ_0}},
  description={},
  sort=tauQ
}

\newglossaryentry{Upsilonp}
{
  name={\ensuremath{\Upsilon_p}},
  description={},
  sort=Upsilonp
}

\newglossaryentry{MeasYM}
{
  name={\ensuremath{\Meas(Y,M)^\natural}},
  description={},
  sort=MeasYM
}

\newglossaryentry{VGU31}
{
  name={\ensuremath{V_{\GU(3,1)}}},
  description={},
  sort=VGU31
}

\newglossaryentry{VGU33}
{
  name={\ensuremath{V_{\GU(3,3)}}},
  description={},
  sort=VGU33
}

\newglossaryentry{VU2}
{
  name={\ensuremath{V_{\U(2)},V'_{\U(2)}}},
  user1={\ensuremath{V_{\U(2)}}},
  user2={\ensuremath{V'_{\U(2)}}},
  user3={\ensuremath{V_{\U(2),\,\xi}}},
  user4={\ensuremath{V'_{\U(2),\,\xi^{-1}}}},
  description={},
  sort=VU2
}

\newglossaryentry{eord}
{
  name={\ensuremath{e_\ord}},
  description={},
  sort=esord
}

\newglossaryentry{theta}
{
  name={\ensuremath{\bmtheta,\tilde{\bmtheta}_3}},
  user1={\ensuremath{\bmtheta}},
  user2={\ensuremath{\tilde{\bmtheta}_3}},
  description={},
  sort=theta
}

\newglossaryentry{h}
{
  name={\ensuremath{\bfh,\tilde{\bfh}_3}},
  user1={\ensuremath{\bfh}},
  user2={\ensuremath{\tilde{\bfh}_3}},
  description={},
  sort=h
}

\newglossaryentry{h0}
{
  name={\ensuremath{\bfh_0,\bfh'_0}},
  user1={\ensuremath{\bfh_0}},
  user2={\ensuremath{\bfh'_0}},
  description={},
  sort=h0
}

\newglossaryentry{Enorl}
{
  name={\ensuremath{E^\Sieg_{\xi\tau}}},
  description={},
  sort=ESiegxitau
}

\newglossaryentry{d2v}
{
  name={\ensuremath{d_{2,v}\left(s,\xi_0\tau_0\right)}},
  description={},
  sort=d2v
}

\newglossaryentry{d3v}
{
  name={\ensuremath{d_{3,v}(s,\xi_0\tau_0)}},
  description={},
  sort=d3v
}

\newglossaryentry{bESieg}
{
  name={\ensuremath{\bfE^\Sieg}},
  description={},
  sort=ESiegf
}

\newglossaryentry{Kp0}
{
  name={\ensuremath{K_{p,0}}},
  description={},
  sort=Kp0
}

\newglossaryentry{M00}
{
  name={\ensuremath{M_{(0,0)}(K^{\prime }_{f,\clabel};\bZ_p)}},
  description={},
  sort=MU21
}

\newglossaryentry{M00tor}
{
  name={\ensuremath{M_{(0,0)}(K'_{f,\clabel};\bZ/p^m\bZ)}},
  description={},
  sort=MU22
}

\newglossaryentry{MU2}
{
  name={\ensuremath{M_{\GU(2)}\left(K^p_fK_p;R\right)}},
  description={},
  sort=MU23
}

\newglossaryentry{bEKling}
{
  name={\ensuremath{\bfE^\Kling_\varphi}},
  description={},
  sort=EKling
}

\newglossaryentry{Lxi}
{
  name={\ensuremath{\cL^{\Sigma\,\cup\{\ell,\ell'\}}_{\xi,\bQ}}},
  description={},
  sort=Lxi
}

\newglossaryentry{OLur}
{
  name={\ensuremath{\hat{\cO}^\ur_L}},
  description={},
  sort=OLur
}

\newglossaryentry{Ebeta}
{
  name={\ensuremath{\bfE^\Kling_{\varphi,\beta,u}}},
  description={},
  sort=EKlingbeta
}

\newglossaryentry{lambda}
{
  name={\ensuremath{\lambda}},
  description={},
  sort=lambda
}

\newglossaryentry{HV}
{
  name={\ensuremath{H(V)}},
  description={},
  sort=HV
}

\newglossaryentry{HVV}
{
  name={\ensuremath{H(\bV)}},
  description={},
  sort=HVV
}

\newglossaryentry{FJbetav}
{
  name={\ensuremath{\FJ_{\beta,v}\big(-,-;f_v(s,\xi_0\tau_{0})\big)}},
  user1={\ensuremath{\FJ_{\beta,v}\big(g_v,x_v;f_v(s,\xi_0\tau_{0})\big)}},
  description={},
  sort=FJbetav
}

\newglossaryentry{thetaJ1}
{
  name={\ensuremath{\theta^J_1}},
  description={},
  sort=thetaJ1
}

\newglossaryentry{lthetaJ}
{
  name={\ensuremath{l_{\theta^J}}},
  description={},
  sort=lthetaJ
}

\newglossaryentry{UKp}
{
  name={\ensuremath{U_{\cK,p},U_{\cK,\fp},U_{\cK,\bar{\fp}}}},
  user1={\ensuremath{U_{\cK,p}}},
  user2={\ensuremath{U_{\cK,\fp}}},
  user3={\ensuremath{U_{\cK,\bar{\fp}}}},
  description={},
  sort=UKp
}

\newglossaryentry{PN}
{
  name={\ensuremath{\pP_{\wtp}}},
  description={},
  sort=PN
}

\newglossaryentry{tauP}
{
  name={\ensuremath{\tau_{\fp,\pP_\wtp}}},
  description={},
  sort=taupPN
}

\newglossaryentry{PNast}
{
  name={\ensuremath{\pP_{\wtp,\ast}}},
  description={},
  sort=PNast
}

\newglossaryentry{wtp}
{
  name={\ensuremath{{\wtp}}},
  description={},
  sort=N
}

\newglossaryentry{Petersson}
{
  name={\ensuremath{\left<-,-\right>_{p\adic}}},
  user1={\ensuremath{\left<\phi',\phi\right>_{p\adic}}},
  description={},
  sort=zzz1
}

\newglossaryentry{hext}
{
  name={\ensuremath{\breve{\bfh},\breve{\tilde{\bfh}}_3}},
  user1={\ensuremath{\breve{\bfh}}},
  user2={\ensuremath{\breve{\tilde{\bfh}}_3}},
  description={},
  sort=hext
}

\newglossaryentry{thetaext}
{
  name={\ensuremath{\breve{\bmtheta},\breve{\tilde{\bmtheta}}_3}},
  user1={\ensuremath{\breve{\bmtheta}}},
  user2={\ensuremath{\breve{\tilde{\bmtheta}}_3}},
  description={},
  sort=thetaext
}

\newglossaryentry{pVxi}
{
  name={\ensuremath{\pV_{\so,\,\xi},\pV^0_{\so,\,\xi}}},
  user1={\ensuremath{\pV_{\so,\,\xi}}},
  user2={\ensuremath{\pV^0_{\so,\,\xi}}},
  description={},
  sort=Vsoxi
}

\newglossaryentry{Mxi}
{
  name={\ensuremath{\cM_{\so,\,\xi},\cM^0_{\so,\,\xi}}},
  user1={\ensuremath{\cM_{\so,\,\xi}}},
  user2={\ensuremath{\cM^0_{\so,\,\xi}}},
  description={},
  sort=Mxi
}

\newglossaryentry{MxiGamma}
{
  name={\ensuremath{\cM_{\so,\,\xi,\Gamma_\cK},\cM^0_{\so,\,\xi,\Gamma_\cK}}},
  user1={\ensuremath{\cM_{\so,\,\xi,\Gamma_\cK}}},
  user2={\ensuremath{\cM^0_{\so,\,\xi,\Gamma_\cK}}},
  description={},
  sort=MxiGamma
}

\newglossaryentry{bT}
{
  name={\ensuremath{\bT^0_{\so,\,\xi,\Gamma_\cK}}},
  description={},
  sort=Tso
}

\newglossaryentry{IEis}
{
  name={\ensuremath{\cI_{\mr{Eis},\pi,\xi}}},
  description={},
  sort=IEis
}

\newglossaryentry{cEpi}
{
  name={\ensuremath{\cE_{\pi,\xi}}},
  description={},
  sort=Epi
}

\newglossaryentry{Selpartial}
{
  name={\ensuremath{\Sel^{\Sigma\,\cup\{\ell,\ell'\}}_{\pi,\cK,\xi}}},
  description={},
  sort=Selpartial
}

\newglossaryentry{Xpartial}
{
  name={\ensuremath{X^{\Sigma\,\cup\{\ell,\ell'\}}_{\pi,\cK,\xi}}},
  description={},
  sort=Xpartial
}

\newglossaryentry{LpiHida}
{
  name={\ensuremath{\cL^{\mr{Hida}}_{\pi,\cK,\xi}}},
  description={},
  sort=Lpiz
}

\newglossaryentry{diffideal}
{
  name={\ensuremath{\fd_{\cK/\bQ}}},
  description={},
  sort=dKQ
}

\newglossaryentry{Igzmin}
{
  name={\ensuremath{\Ig^{\min}_n}},
  description={},
  sort=Igzmin
}

\newglossaryentry{varrho-1}
{
  name={\ensuremath{\varrho^{-1}_\fp}},
  description={},
  sort=rho-1p
}

\title{Iwasawa--Greenberg main conjecture for non-ordinary modular forms and Eisenstein congruences on GU(3,1)}
\author{Francesc Castella}
\address{F. C.: University of California, Santa Barbara, CA, United States}
\email{\href{mailto:castella@math.ucsb.edu}{castella@math.ucsb.edu}}

\author{Zheng Liu}
\address{Z. L.:University of California, Santa Barbara, CA, United States}
\email{\href{mailto:zliu@math.ucsb.edu}{zliu@math.ucsb.edu}}

\author{Xin Wan}
\address{X. W.: Academy of Mathematics and Systems Science, Chinese Academy of Sciences and University of Chinese Academy of Sciences, Haidian District, Beijing, China}
\email{\href{mailto:xwan@math.ac.cn}{xwan@math.ac.cn}}

\begin{document}

\maketitle

\begin{abstract}
In this paper we prove one side divisibility of the Iwasawa-Greenberg main conjecture for Rankin-Selberg product of a weight two cusp form and an ordinary CM form of higher weight, using congruences between Klingen Eisenstein series and cusp forms on $\mathrm{GU}(3,1)$, generalizing earlier result of the third-named author to allow non-ordinary cusp forms. The main result is a key input in the third author's proof for Kobayashi's $\pm$-main conjecture for supersingular elliptic curves. The new ingredient here is developing a semi-ordinary Hida theory along an appropriate smaller weight space, and a study of the semi-ordinary Eisenstein family. 
\end{abstract}

\tableofcontents 
\numberwithin{equation}{subsection}

\section{Introduction}\label{sec:intro}

Let $p$ be an odd prime number. In this paper, under some assumptions, we prove one divisibility of a two-variable Greenberg type main conjecture for a weight $2$ newform form unramified at $p$. The result is a key ingredient in the third author's proof \cite{Wan-super} of the Iwasawa Main Conjecture for elliptic curves with supersingular reduction at $p$ and $a_p=0$ formulated by Kobayashi \cite{Kobayashi}. 

\vspace{.5em}

Let $\pi$ be an irreducible cuspidal automorphic representation of $\GL_2(\bA_\bQ)$ generated by a newform of weight $2$. Associated to $\pi$ is a continuous two-dimensional $p$-adic Galois representation $\gls{rhopi}$ of $\Gal(\ol{\bQ}/\bQ)$ over $L$, a finite extension of $\bQ_p$, for which we can fix a $\Gal(\ol{\bQ}/\bQ)$-stable $\cO_L$-lattice $T_\pi$. (We use the geometric convention for Galois representations. The determinant of $\rho_\pi$ is $\epsilon^{-1}_{\mr{cyc}}$.) 

Let $\cK$ be an imaginary quadratic field in which $p$ splits as $\fp\bar{\fp}$. Denote by $\cK_\infty$ the maximal abelian pro-$p$ extension of $\cK$ unramified outside $p$. Then the Galois group $\Gal(\cK_\infty/\cK)$ is isomorphic to $\bZ^2_p$ and we denote it by $\Gamma_\cK$. We have the tautological character
\[
   \gls{PsiK}:\Gal(\ol{\bQ}/\cK)\ra\Gamma_\cK\hra \bZ_p\llb \Gamma_\cK\rrb^\times.
\]
Let $\xi:\cK^\times\backslash\bA^\times_\cK\ra\bC^\times$ be an algebraic Hecke character associated to which is a character $\Gal(\ol{\bQ}/\cK)\ra \cO^\times_L$. We denote this Galois character also by $\xi$.

We consider the $\Gal(\ol{\bQ}/\cK)$-module 
\[
    T_{\pi,\cK,\xi}:=T_\pi(\epsilon^2_{\mr{cyc}})|_{\Gal(\ol{\bQ}/\cK)}(\xi^{-1})\otimes\bZ_p\llb\Gamma_\cK\rrb(\Psi^{-1}_\cK).
\]
Define the Selmer group
\begin{equation}\label{eq:Sel}
   \gls{Sel}=\ker\left\{H^1\big(\cK,T_{\pi,\cK,\xi}\otimes_{\cO_L\llb\Gamma_\cK\rrb}\cO_L\llb\Gamma_\cK\rrb^*\big)
   \lra \prod_{\fv\neq \fp}H^1\big(I_\fv,T_{\pi,\cK,\xi}\otimes_{\cO_L\llb\Gamma_\cK\rrb}\cO_L\llb\Gamma_\cK\rrb^*\big)\right\},
\end{equation}
with $\cO_L\llb\Gamma_\cK\rrb^*=\Hom_{\bZ_p}\left(\cO_L\llb\Gamma_\cK\rrb,\bQ_p/\bZ_p\right)$, the Pontryagin dual of $\cO_L\llb\Gamma_\cK\rrb$. (This Selmer group has relaxed condition at $\fp$ and unramified condition at $\bar{\fp}$.) Let
\begin{equation}\label{eq:Xsel}
    \gls{Xpi}:=\Hom_{\bZ_p}\big(\Sel_{\pi,\cK,\xi},\bQ_p/\bZ_p\big)
\end{equation}
which is well-known to be a finitely generated $\cO_L\llb\Gamma_\cK\rrb$-module. We recall the following definition of characteristic ideals.

\begin{defn}
For a Noetherian normal domain $A$ and a finitely generated $A$-module $M$, we define the characteristic ideal of $M$ as
\[
   \gls{char}=\left\{x\in A:\ord_P(x)\geq \mr{length}_{A_P}(M_P) \text{ for all height one prime ideal $P\subset A$}\right\}.
\]
\end{defn}
The Iwasawa--Greenberg main conjectures \cite{Greenberg55} predict that  the characteristic ideal of $X_{\pi,\cK,\xi}$ is generated by the following $p$-adic $L$-function.

\vspace{.5em}
Denote by $\hat{\cO}^\ur_L$ the completion of the maximal unramified extension of $\cO_L$. It follows from the construction in \cite{EisWan} that there is a $p$-adic $L$-function $\gls{Lpi}\in\mr{Frac}\big(\hat{\cO}^\ur_L\llb\Gamma_\cK\rrb\big)$ satisfying the interpolation property: for all algebraic Hecke characters $\tau:\cK^\times\backslash\bA^\times_\cK\ra\bC^\times$ such that its $p$-adic avatar $\tau_{p\adic}$  factors through $\Gamma_\cK$ and $\xi\tau$ has $\infty$-type $\left(k_1,k_2\right)$ with $k_1,k_2\in\bZ$, $k_1\leq 0$, $k_2\geq 2-k_1$, 
\begin{equation}\label{eq:intp} 
\begin{aligned}
   \cL_{\pi,\cK,\xi}(\tau_{p\adic})
   =&\, \left(\frac{\Omega_p}{\Omega_\infty}\right)^{2(k_2-k_1)}\frac{\Gamma(k_2)\Gamma(k_2-1)}{(2\pi i)^{2k_2-1}}
   \cdot \gamma_p\left(\frac{3-(k_1+k_2)}{2},\pi^\vee_{p}\times (\xi_0\tau_0)^{-1}_{\bar{\fp}}\right)\\
    &\times L^{\{\infty,p\}}\left(\frac{k_1+k_2-1}{2},\mr{BC}(\pi)\times\xi_0\tau_0\right),
\end{aligned}
\end{equation}
where $\xi_0\tau_0=\xi\tau|\cdot|^{-\frac{k_1+k_2}{2}}_{\bA_\cK}$ and $\mr{BC}(\pi)$ denotes the unitary automorphic representation of $\GL_2(\bA_\cK)$ obtained as the base change of $\pi$. 

\begin{rmk}
More precisely, by using the doubling method, an imprimitive $p$-adic $L$-function is constructed in \cite{EisWan} as an element in $\hat{\cO}^\ur_L\llb\Gamma_\cK\rrb$, interpolating the $L$-values in (\ref{eq:intp}) with certain local $L$-factors away from $p\infty$ removed. The inverses of these local $L$-factors are easily seen to be $p$-adically interpolated by elements in $\cO_L\llb\Gamma_\cK\rrb$. Multiplying those local factors to the imprimitive $p$-adic $L$-function, one obtains the complete $p$-adic $L$-function $\cL_{\pi,\cK,\xi}$. Note also that the interpolation formulas in \cite{EisWan} are completely computed for $k_1=0$, and one can get the general case by using the results on the computation of the doubling archimedean zeta integrals in \cite{AZIU}.
\end{rmk}

 
We are interested in the following (two-variable) Greenberg type main conjecture. 

\begin{conj}\label{conj}
\[
      \mr{char}_{\cO^{\ur}_L\llb\Gamma_\cK\rrb} \big(X_{\pi,\cK,\xi}\big)= \big( \cL_{\pi,\cK,\xi}\big).
\]
\end{conj}

The main result of this paper is Theorem~\ref{thm:main}, which is a partial result towards this conjecture. Like the previous works \cite{Ur01,Ur04,SU,WanU31} on proving Greenberg type main conjecture for modular forms, the proof uses the congruences between the Klingen Eisenstein series and cuspidal holomorphic automorphic forms. The $L$-values in our case here are the same as those in \cite{WanU31} and we also use Klingen Eisenstein series on $\GU(3,1)$ as in {\it loc.cit}. The main difference is that the modular form is not assumed to be ordinary at $p$, so the standard Hida theory is not applicable. 

\vspace{.5em}

The key idea is to introduce the notion of \emph{semi-ordinary} automorphic forms on $\GU(3,1)$. In \S\S\ref{sec:Hidathy}-\ref{sec:noncusp}, we develop a Hida theory for $p$-adic families of (cuspidal and non-cuspidal) semi-ordinary forms on $\GU(3,1)$ along an appropriate two-dimensional subspace of its three-dimensional weight space. The main results are stated in Theorem~\ref{prop:main}. In \S\ref{sec:construct}, by using the doubling method, we construct a  Klingen Eisenstein family $\bfE^\Kling_\varphi$ and prove its semi-ordinarity. In \S\ref{sec:degFJ}, we study the degenerate Fourier--Jacobi coefficients of $\bfE^\Kling_\varphi$. The analogus computations in \cite{WanU31} assume a sufficient ramification condition (see Definition 6.30 in {\it op.\,cit}) which is not available in our case here, so we need a better way to do the computation at $p$ by using the functional equations for doubling zeta integrals. In \S\ref{sec:nondegFJ}, we study the non-degenerate Fourier--Jacobi coefficients of $\bfE^\Kling_\varphi$. This part is very similar to \cite{WanU31} and we cite many results there, but the presentation is slightly rearranged. For example, the auxiliary data for constructing the Klingen Eisenstein family are chosen at the beginning of the construction (\S\ref{sec:aux}) instead of till the end of the analysis of the non-degenerate Fourier--Jacobi coefficients, and a explanation on the strategy for analyzing the non-degenerate Fourier--Jacobi coefficients is included in \S\ref{sec:strategy}.  In \S\ref{sec:mcpf}, combining the results in \S\ref{sec:degFJ} and \S\ref{sec:nondegFJ}, we deduce a result on the Klingen Eisenstein congruence ideal, and use it as an input for the lattice construction to deduce the results on Selmer groups.

\vspace{1em}
\noindent{\bf Notation.}
We fix a prime $p\geq 3$ and an imaginary quadratic field $\gls{cK}$ in which $p$ splits in $\cK$ as $\fp\bar{\fp}$. Denote by $\gls{Disc}$ the Discriminant of $\cK/\bQ$ and by $\gls{diffideal}$ the different ideal of $\cK/\bQ$. Denote by $\gls{eta}$ the quadratic character of $\bQ^\times\backslash\bA^\times_\bQ$ associated to $\cK/\bQ$. Let $\gls{c}$ be the nontrivial element in $\Gal(\cK/\bQ)$. For $x\in \cK$, denote by $\bar{x}$ its image of under $c$. For a finite place $v$ of $\bQ$, we put $\cK_v=\cK\otimes_{\bQ}\bQ_v$ and $\cO_{\cK,v}=\cO_\cK\otimes_\bZ\bZ_v$. 

Fix embeddings
\begin{equation}\label{eq:Kembd} 
\begin{aligned}
   \glsuseri{iota}:\cK\lhra\bC,
   &&\glsuserii{iota}:\cK\lhra\bQ_p
\end{aligned}
\end{equation} 
such that the valuation of $\bQ_p$ and $\iota_p$ induce the valuation of $\cK$ given by $\fp$.  The embedding $\iota_p:\cK\hra\bQ_p$ induces a homomorphism $\glsuseri{rho}:\cK_p=\cK\otimes_\bQ\bQ_p\ra\bQ_p$. We denote by $\glsuserii{rho}:\cK\otimes_\bQ\bQ_p\ra\bQ_p$ the composition of $\varrho_p$ and the non-trivial element $c\in\Gal(\cK/\bQ)$. Then 
\begin{align*}
   &(\varrho_\fp,\varrho_{\bar{\fp}}):\cK\otimes_\bQ\bQ_p\lra\bQ_p\times\bQ_p,
   &&a\longmapsto \left(\varrho_\fp(a),\varrho_{\bar{\fp}}(a)\right)
\end{align*}  
is an isomorphism. We also fix a totally imaginary element $\gls{delta}\in\cK$ such that $\mr{Nm}(\delta)=\delta\bar{\delta}$ is a $p$-adic unit.


Fix the standard additive character $\be_{\bA_\bQ}=\bigotimes_v\gls{ev}:\bQ\backslash\bA\ra\bC^\times$ with 
\begin{equation}\label{eq:bev}
   \be_v(x)=\left\{\begin{array}{ll} e^{-2\pi i\{x\}_v}, &v\neq\infty,\\ e^{2\pi i x},&v=\infty\end{array}\right.
\end{equation} 
where $\{x\}_v$ is the fractional part of $x$.

\vspace{1em}
\noindent{\bf Acknowledgments.} We would like to thank Eric Urban for pointing out to us the possibility to develop the semi-ordinary Hida theory needed in this paper. We also thank Ming-Lun Hsieh, Kai-Wen Lan and Christopher Skinner for useful communications. During the preparation of this paper, the first author was partially supported by the NSF grant DMS-1946136 and DMS-2101458, the second author was partially supported by the NSF grant DMS-2001527, the third author was partially supported by NSFC grants 11688101, 11621061 and National Key R and D Program of China 2020YFA0712600.

\section{Hida theory for semi-ordinary forms on $\GU(3,1)$}\label{sec:Hidathy}

We define semi-ordinary forms on $\GU(3,1)$ and state the control theorem for semi-ordinary families (Theorem~\ref{prop:main}). The proof of Theorem~\ref{prop:main} is given in the following two sections.

\subsection{Some notation}
Let $L_0$ be a free $\cO_\cK$-module of rank $2$ with basis $\ww_1, \ww_2$, and we equip $L_0\otimes_\bZ\bQ$ with a skew-Hermitian form $\left<\,,\,\right>_{L_0}$ whose matrix with respect to the basis $\ww_1, \ww_2$ is given by a matrix $\glsuseri{zeta}\in\Her_2(\cO_\cK)$ with $\delta\zeta_0$ positive definite. Let $X, Y$ be free $\cO_\cK$-modules of rank $1$ with bases $\xx_1, \yy_1$. Let $X^\vee=\fd^{-1}_{\cK/\bQ}\cdot {\tt x}_1$ and $L=X^\vee\oplus L_0\oplus Y$. Equip $L\otimes_\bZ\bQ$ with the skew-Hermitian form $\left<\,,\,\right>_{L}$ whose matrix with respect to the basis  $\xx_1, \ww_1, \ww_2, \yy_1$ is given by $\glsuserii{zeta}=\begin{pmatrix}&&1\\&\zeta_0\\-1\end{pmatrix}.$

Define the similitude unitary groups $\glsuserii{G}=\GU(2)$ and $\glsuseri{G}=\GU(3,1)$ (over $\bZ)$) as: for all $\bZ$-algebra $R$ 
\begin{equation}\label{eq:ugp}
\begin{aligned}
   \GU(2)(R)&=\{(g,\nu)\in \GL_{\cO_\cK\otimes_{\bZ} R}(L_0\otimes_{\bZ} R)\times R^\times\,:\,\left<gv_1,gv_2\right>_{L_0}=\nu\left<v_1,v_2\right>_{L_0}\},\\
   \GU(3,1)(R)&=\{(g,\nu)\in \GL_{\cO_\cK\otimes_{\bZ} R}(L\otimes_{\bZ} R)\times R^\times\,:\,\left<gv_1,gv_2\right>_{L}=\nu\left<v_1,v_2\right>_{L}\},
\end{aligned}
\end{equation} 
and the unitary groups $\U(2)$ (resp. $\U(3,1)$) as the subgroup of $\GU(2)$ (resp. $\GU(3,1)$)  consisting of elements with $\nu=1$.

\subsection{Shimura variety}\label{subsec:Sh}
We fix an open compact subgroup $K^p_f\subset G(\bA^p_{\bQ,f})$, assumed throughout to be neat. Let $2\pi i\left<\,,\,\right>:L\times L\ra \bZ(1)$ be the alternating pairing $2\pi i\cdot \Tr_{\cK/\bQ}\circ\left<\,,\,\right>_L$ and $h:\bC\ra\mr{End}_{\cO_\cK\otimes_{\bZ}\bR}(L\otimes_{\bZ}\bR)$ be the homomorphism given by
\[
   h(u+iv):\left(\xx_1,\ww_1,\ww_2,\yy_1\right)\mapsto \left(\xx_1,\ww_1,\ww_2,\yy_1\right)
   \begin{pmatrix}
   1\otimes u&&&-1\otimes v\\
   &1\otimes u+\delta\otimes\frac{v}{\sqrt{\delta\bar{\delta}}}\\
   &&1\otimes u+\delta\otimes\frac{v}{\sqrt{\delta\bar{\delta}}}\\
   1\otimes v&&& 1\otimes u
   \end{pmatrix}.
\]
Then the tuple $\left(\cO_\cK,c,L,2\pi i\left<\,,\,\right>,h\right)$ defines a Shimura datum of PEL type with reflex field $\cK$.

Consider the moduli problem sending every locally noetherian connected $\cO_{\cK,(p)}$-scheme $S$ to the set of isomorphism classes of tuples $(A,\lambda,i,\alpha^p)$ with: 
\begin{itemize}[leftmargin=*]
\item $A$ an abelian scheme of relative dimension $4$ over $S$;
\item $\lambda:A\ra A^\vee$ a $\bZ^\times_{(p)}$-polarization;
\item $i:\cO_\cK\hra \mr{End}_S A$ an embedding such that the induced $\cO_\cK$-action on $\Lie A_{/S}$ satisfies the  \emph{determinant condition} defined by $h$, \emph{i.e.}, 
{\begin{align*}
  \det\left(X-i(b)|\Lie A_{/S}\right)&=(X-b)^3(X-\bar{b}) 
\end{align*}}
for all $b\in \cO_\cK$ (where on the right-hand side $b$ is viewed as an element in $\mathscr{O}_S$ under the morphism $\cO_\cK\rightarrow\mathscr{O}_S$);
\item $\alpha^p$ an (integral) $K^p_f$-level structure on $(A, \lambda, i)$ of type $\left(L\otimes\wh{\bZ}^{(p)},\Tr_{\cK/\bQ}\circ\left<\,,\,\right>_L\right)$; that is, a $K^p_f$-orbit of $\cO_\cK$-module isomorphisms $L\otimes\wh{\bZ}^{(p)}\ra T^{(p)}A_{\bar{s}}$, where $T^{(p)}A_{\bar{s}}$ is the prime-to-$p$ Tate module of $A_{\bar{s}}$, together with an isomorphism $\wh{\bZ}^{(p)}(1)\stackrel{\sim}{\ra}\bG_{m,\bar{s}}$ making the following diagram commute: 
\[
   \xymatrix@C+2em{
    (L\otimes\wh{\bZ}^{(p)})\times (L\otimes\wh{\bZ}^{(p)}) \ar[r]^-{\Tr_{\cK/\bQ}\circ\left<\,,\,\right>_L} \ar[d] &\wh{\bZ}^{(p)}(1)\ar[d]\\
    T^{(p)}A_{\bar{s}}\times  T^{(p)}A_{\bar{s}}\ar[r]^-{\lambda\text{-Weil}}& \bG_{m,\ol{s}}.
   }
\]
\end{itemize}
(See \cite[Def.~1.4.1.4]{Thesis}.) 
\sk

Since $K^p_f$ is neat, the above moduli problem is represented by a smooth quasi-projective scheme $\gls{sS}$ over $\cO_{\cK,(p)}$ (see \cite[Thm.~1.4.1.2, Cor.~7.2.3.10]{Thesis}).


Denote by $\gls{sStor}$ the toroidal compactification of $\sS$, which is a proper smooth scheme over $\cO_{\cK,(p)}$ containing $\sS$ as an open dense subscheme with complement being a relative Cartier divisor with normal crossings. (In our special case $\GU(3,1)$ here, there is a unique choice of polyhedral cone decomposition for the toroidal compactification.) We denote by $\gls{IS}$ the ideal sheaf of the boundary of $\sS^\tor$. 
By \cite[Thm.~6.4.1.1]{Thesis}, the universal family $(\sA,\lambda,i,\alpha^p)$ over $\sS$ extends to a degenerating family $(\sG,\lambda,i,\alpha^p)$ over $\sS^\tor$. Moreover, the base change of $\sS$ (resp. $\sS^\tor$) to $\cK$ agrees with the Shimura variety over $\cK$ (resp. its toroidal compactification) representing the moduli problem with full level structure at $p$ (see \cite[(A.4.17), (A.4.18)]{App}).


\subsection{Hasse invariant}

Set $\glsuseri{omega}:=e^*\Omega^1_{\sG/\sS^\tor}$ where $e:\sS^\tor\ra \sG$ is the zero section of the semi-abelian scheme $\sG$ over $\sS^\tor$. Let $\glsuserii{omega}$ be the line bundle $\det\underline{\omega}=\wedge^{\mr{top}}\underline{\omega}$. The minimal compactification of $\sS$ is defined by 
\[
  \gls{sSmin}=\mr{Proj}\Biggl(\bigoplus\limits_{k\geqslant 0}H^0(\sS^\tor,\omega^k)\Biggr).
\]
Let $\pi:\sS^{\tor}\rightarrow\sS^{\min}$ be the canonical projection. The push-forward $\pi_*\omega$ is an ample line bundle, and $\pi^*\pi_*\omega\cong \omega$ (see \cite[Thm.~7.2.4.1]{Thesis}). 

In the following, with a slight abuse of notation, we also denote by $\sS^{\tor}$ and $\sS^{\min}$ their base change to $\bZ_p$ via the map $\cO_{\cK,(p)}\ra \bZ_p$ induced by our fixed embedding $\iota_p$, and let $\sS^\tor_{/\bF_p}$ and $\sS^{\min}_{/\bF_p}$ be their corresponding special fibers. 

Let $\gls{Ha}\in H^0\left(\sS^\tor_{/\bF_p},\omega^{p-1}\right)$ be the Hasse invariant defined as in \cite[\S6.3.1]{Kuga}. In particular, for each geometric point $\bar{s}$ of $\sS^\tor_{/\bF_p}$, the Hasse invariant of the corresponding semi-abelian scheme $\sG_{\bar{s}}$ is nonzero if and only if the abelian part of $\sG_{\bar{s}}$ is ordinary. Because $\pi_*\omega$ is ample, for some $t_E> 0$, there exists an element in $H^0\left(\sS^{\min},(\pi_*\omega)^{t_E(p-1)}\right)$ lifting the $t_E(p-1)$-th power of the push-forward of $\mr{Ha}$; we denote by $\gls{E}$ the pullback under $\pi$ of any such lift, which (because $\pi^*\pi_*\omega\cong \omega$)  
defines an element $E\in H^0\left(\sS^{\tor},\omega^{t_E(p-1)}\right)$. 

\subsection{Some groups}

Before moving on, we need to introduce some more group-theoretic notations. Given a matrix $a\in\GL_4(\cK\otimes_\bQ \bQ_p)$, we define $a^+,a^-\in \GL_4(\bQ_p)$ as
\begin{align*}
 a^+=\varrho_\fp(a),\quad a^-=\varrho_{\bar{\fp}}(a),
\end{align*}
where $\varrho_\fp,\varrho_{\bar{\fp}}:\cK\otimes_\bQ\bQ_p$ are defined as in  {\bf Notation}, and we view an element $g\in G(\bQ_p)$ as a matrix inside $\GL_4(\cK\otimes_\bQ \bQ_p)$ via the basis $(\xx_1,\ww_1,\ww_2,\yy_1)$. Then the corresponding matrices $g^+,g^-\in\GL_4(\bQ_p)$ satisfy
\begin{align*}
   &\ltrans{g}^-\varrho_\fp(\zeta)\,g^+=\nu(g)\cdot \varrho_\fp(\zeta).
\end{align*} 
In the following, we will often write $g\in G(\bQ_p)$ as $(g^+,g^-)$. Note that the map
\begin{equation}\label{eq:Gisom}
\begin{aligned}
   G(\bQ_p)&\ra \GL_4(\bQ_p)\times\bQ^\times_p\\
    g&\mapsto (g^+,\nu(g))
\end{aligned}
\end{equation}
is an isomorphism. 

There is a filtration $\glsuseri{ttD}=\{{\tt D}^i\}_i$  
of $L\otimes_{\bZ}\bZ_p$ given by
\begin{equation}\label{eq:fil-D}
\glsuserii{ttD}=0\;\subset\;\glsuseriii{ttD}=X^+_{p}\oplus L^+_{0,p}\oplus X^-_{p}\;\subset\; \glsuseriv{ttD}=L\otimes_{\bZ}\bZ_p,
\end{equation}
where $X^+_p\subset X^\vee\otimes_{\bZ}\bZ_p=X\otimes_{\bZ}\bZ_p$ (resp. $X^-_p\subset X\otimes_{\bZ}\bZ_p$,  $L^+_{0,p}\subset L_0\otimes_{\bZ}\bZ_p$) is the subspace on which $b\in\cO_\cK$ acts by $\iota_p(b)$ (resp. $\iota_p(\bar{b})$,  $\iota_p(b)$). Note that 
${\tt D}^0$ is isotropic with respect to $\Tr_{\cK/\bQ}\circ\left<\,,\,\right>_L$. For $R=\bZ_p$ or $\bQ_p$, we put
\[
   \gls{PD}(R)=\{g\in G(R)\,:\,g({\tt D}^0)={\tt D}^0\};
\]
\emph{i.e.}, $P_{\tt D}(R)$ is the subgroup of $G(R)$ stabilizing the filtration ${\tt D}$ in $(\ref{eq:fil-D})$. We see that $P_{\tt D}^\pm(R):=\{g^\pm\colon g\in P_{\tt D}(R)\}$ are the subgroups of $\GL_4(R)$ given by
\[
P_{\tt D}^+=\left\{\begin{psm}\ast&\ast&\ast&\ast\\ \ast&\ast&\ast&\ast\\ \ast&\ast&\ast&\ast\\&&&\ast\end{psm}\right\},\quad
P_{\tt D}^-=\left\{\begin{psm}\ast&\ast&\ast&\ast\\ &\ast&\ast&\ast\\ &\ast&\ast&\ast\\&\ast&\ast&\ast\end{psm}\right\}.
\]
Because $P_{\tt D}$ preserves both $X^+_p\oplus L^+_{0,p}$ and $X^-_p$, there is a natural projection $P_{\tt D}(R)\ra \GL_3(R)\times \GL_1(R)$ which in terms of matrices is given by
\begin{align*}
   g=(g^+,g^-)=\left(\begin{psm}a_{11}&a_{12}&a_{13}&a_{14}\\ a_{21}&a_{22}&a_{23}&a_{24}\\ a_{31}&a_{32}&a_{33}&a_{34}\\ &&&a_{44}\end{psm},\begin{psm}b_{11}&b_{12}&b_{13}&b_{14}\\ &b_{22}&b_{23}&b_{24}\\ &b_{32}&b_{33}&b_{34}\\ &b_{42}&a_{43}&b_{44}\end{psm}\right)
   &\mapsto \left(\begin{psm}a_{11}&a_{12}&a_{13}\\a_{21}&a_{22}&a_{23}\\a_{31}&a_{32}&a_{33}\end{psm},b_{11}\right).
\end{align*}

We will consider the $p$-level subgroups given by
\begin{align*}
   \glsuseri{Kpn}&:=\left\{g\in G(\bZ_p)\colon g^+\equiv\begin{psm}\ast&\ast&\ast&\ast\\p\ast&\ast&\ast&\ast\\ &&1&\ast\\ && &1\end{psm}\mod p^n\right\},\\
   \glsuserii{Kpn}&:=\left\{g\in G(\bZ_p)\colon g^+\equiv\begin{psm}\ast&\ast&\ast&\ast\\p\ast&\ast&\ast&\ast\\ &&\ast&\ast\\&& &\ast\end{psm}\mod p^n\right\},
\end{align*}
where $n$ is a positive integer.



\subsection{Igusa towers} 

Let \glsuseri{Ign} be the ordinary locus of level $K^1_{p,n}$ attached to the filtration ${\tt D}$ in $(\ref{eq:fil-D})$ and the PEL type Shimura data introduced in $\S\ref{subsec:Sh}$, constructed in \cite[Theorem~3.4.2.5]{Kuga}. By \cite[Prop.~3.4.6.3]{Kuga}, 
$\Ig_{n}$ is a smooth quasi-projective scheme over $\bZ_p$ representing a moduli problem for tuples $(A,\lambda,i,\alpha^p,\alpha_p)$, where $(A,\lambda,i,\alpha^p)$ is as in $\S\ref{subsec:Sh}$, and $\alpha_p$ is an ordinary $K^1_{p,n}$-level structure of $A$, \emph{i.e.}, a $K^1_{p,n}$-orbit of group scheme embeddings $\mu_{p^n}\otimes{\tt D}^0\hra A[p^n]$ with image isotropic for the $\lambda$-Weil pairing, compatible with the $\cO_\cK$-actions on ${\tt D}^0$ and   $A[p^n]$ through $i$, as described in \cite[\S3.1.1]{HLTT}. 

\begin{rmk} 
In \cite[Chapter~3]{Kuga}, the ordinary locus is defined as a normalization of the naive moduli problem introduced in [\emph{loc.\,cit.}, Def.~3.4.1.1]. In our case, because $p$ is a good prime in the sense of \cite[Def.~1.1.1.6]{Kuga} and $\nu(K^1_{p,n})=\bZ^\times_p$, this construction of the ordinary locus agrees with the moduli problem. (See \cite[B.10]{HLTT} for more details).
\end{rmk}
 
Let \glsuserii{Ign} be the partial toroidal compactification of the ordinary locus $\Ig_{n}$ (\cite[Thm.~5.2.1.1]{Kuga}); it is obtained by gluing to $\Ig_{n}$ the toroidal boundary charts parameterizing degenerating families defined in \cite[Def.~3.4.2.0]{Kuga} (including an extensibility condition on the ordinary level structure). We note that, even though the generic fibre of $\Ig_{n}$ agrees with the Shimura variety of level $K^p_fK^1_{p,n}$ over $\bQ_p$, the generic fibre of $\Ig^\tor_{n}$ is in general just an open subscheme of a toroidal compactification of the Shimura variety of level $K^p_fK^1_{p,n}$ over $\bQ_p$. 


Let \glsuserii{Ignm} (resp. \gls{Sm}) be the base change of $\sT_n^\tor$ (resp. $\sS^\tor_{}$) to $\bZ/p^m\bZ$. By \cite[Lem.~6.3.2.7]{Kuga}, $\sS^\tor_m[1/E]$ agrees with the ordinary locus in \cite[Thm.~5.2.1.1]{Kuga} for full level at $p$, and  by \cite[Cor.~5.2.2.3]{Kuga} the map $ \Ig^\tor_{n,m}\ra \sS^\tor_m[1/E]$ that forgets the ordinary $K^1_{p,n}$-level structure is finite {\'e}tale. 
Concretely, we can describe $\sT^\tor_{n,m}$ as
\begin{equation}\label{eq:tower}
\underline{\Isom}_{\sS^{\tor}_m[1/E]}\left(\mu_{p^n}\otimes {\tt D}^0,\sG[p^n]^\mult\right)/K^1_{p,n}\cap P_{\tt D}(\bZ_p),
\end{equation}
where $P_{\tt D}(\bZ_p)$ acts on an element $\alpha_p\in \underline{\Isom}_{\sS^{\tor}_m[1/E]}\left(\mu_{p^n}\otimes {\tt D}^0,\sG[p^n]^\mult\right)$ by $(g\cdot\alpha_p)(v)=\alpha_p(g^{-1}v)$ for $v\in {\tt D}^0$. The fiber at a geometric point $\bar{s}\in\sS^\tor_{m}$ parameterizes tuples $({\tt F}^\pm,\delta^+_3,\delta^-_1,{\tt F}^+_1[p])$, where $\gls{F}$ are filtrations
\begin{equation}\label{eq:F}
\begin{aligned}
  &{\tt F}^+:0={\tt F}^+_0\subset {\tt F}^+_2\subset{\tt F}^+_3=\sG_{\bar{s}}[p^n]^{\mult+}\\
  &{\tt F}^-:0={\tt F}^-_0\subset{\tt F}^-_1= \sG_{\bar{s}}[p^n]^{\mult -}
\end{aligned}
\end{equation}
compatible with the Weil pairing; $\delta^+_3$ and $\delta^-_1$ are isomorphisms
\begin{align*}
   &\glsuseri{delta+-}: \mu_{p^n}\cong \mr{Gr}^{\tt F^+}_3,
   &&\glsuserii{delta+-}:\mu_{p^n}\cong \mr{Gr}^{\tt F^-}_1;
\end{align*}
and ${\tt F}^+_1[p]$ is a two-step filtration of ${\tt F}^+_2[p]$. 



\subsection{$p$-adic forms forms}\label{subsec:p-adicforms}

Define the space of mod~$p^m$ automorphic forms on $G$ of level $n$ by 
\[
   \glsuseri{Vnm}:=H^0\left(\Ig^\tor_{n,m},\cO_{\sT^\tor_{n,m}}\right).
\]
Letting $\cI_{\Ig^\tor_{n,m}}:=(\Ig^\tor_{n,m}\ra \sS^\tor[1/E])^*\cI_{\sS^\tor}$, we similarly define the space of mod~$p^m$  cuspidal automorphic forms on $G$ of level $n$ by 
\[
   \glsuserii{Vnm}:=H^0\left(\sT^\tor_{n,m},\cI_{\sT^\tor_{n,m}}\right).
\]

Passing to the limit, we obtain corresponding spaces of $p$-adic automorphic forms (with $p$-power torsion coefficients) by 
\[
  \glsuseri{pV}:=\varinjlim\limits_{m}\varinjlim\limits_{n} V_{n,m},\quad
  \glsuserii{pV}:=\varinjlim\limits_{m}\varinjlim\limits_{n} V_{n,m}^0.
\]
Let $\gls{Tso}(\bZ_p)=\bZ_p^\times\times\bZ_p^\times$. The map sending $(a_1,a_2)$ to $g\in G(\bZ_p)$ with $g^+=\begin{psm}1\\&1\\&&a_1\\&&&a^{-1}_2\end{psm}$ and $\nu(g)=1$ identifies $T_\so(\bZ/p^n)=(\bZ/p^n\bZ)^\times\times(\bZ/p^n\bZ)^\times$ with $K^n_{p,0}/K^1_{p,n}$. Hence the group $T_{\so}(\bZ_p)$ naturally acts on $V_{n,m}, V^0_{n,m}, \pV, \pV^0$ making these spaces into $\bZ_p\llb T_{\so}(\bZ_p)\rrb$-modules.




By a \emph{$p$-adic weight} (for a semi-ordinary form) we mean a $\ol{\bQ}_p$-valued character of $T_{\so}(\bZ_p)$, \emph{i.e.}, a pair $(\tau^+,\tau^-)$, where $\tau^\pm:\bZ^\times_p\ra \ol{\bQ}^\times_p$ are continuous characters, and we say that a $p$-adic weight is \emph{arithmetic} if it is of the form $(x,y)\mapsto \epsilon^+(x)x^{t^+}\cdot\epsilon^-(y)y^{t^-}$, where $\epsilon^\pm$ is of finite order and $t^\pm\in\bZ$. If $(\tau^+,\tau^-)$ is arithmetic, we put $\tau_{\rmf}^\pm:=\epsilon^\pm$ and $\tau_{\alg}^\pm=t^\pm$. Given a $p$-adic weight $(\tau^+,\tau^-)$, we denote by $V_{n,m}[\tau^+,\tau^-]$ the subspace of $V_{n,m}\otimes_{\bZ_p}\cO_{\bQ_p(\tau^+,\tau^-)}$ on which $T_{\so}(\bZ_p)$ acts by the inverse of the character $(\tau^+,\tau^-)$. Similarly, we define the eigenspaces $V^0_{n,m}[\tau^+,\tau^-]$, $\pV[\tau^+,\tau^-]$, $\pV^0[\tau^+,\tau^-]$.

\subsection{Classical automorphic forms}

The weights of classical holomorphic automorphic forms on $G$ are indexed by tuples of integers $\gls{ut}=(t^+_1,t^+_2,t^+_3;t^-_1)$ with $t^+_1\geq t^+_2\geq t^+_3$. (When $t^+_3\geq -t^-_1+4$, the archimedean component of the corresponding automorphic representation is isomorphic to a holomorphic discrete series.)

Let $\gls{Wut}$ be the algebraic representation of $\GL_3\times\GL_1$ given by
\begin{equation}\label{eq:W}
W_{\ut}:=W_{(t^+_1,t^+_2,t^+_3)}\boxtimes W_{t^-_4}.\nonumber
\end{equation}
Here, for any algebra $R$, letting $R[\underline{x},\det\underline{x}^{-1}]$ be the polynomial ring in the variables $x_{ij}$ ($1\leq i,j\leq 3$) and $\left(\det(x_{ij})_{1\leq i,j\leq 3}\right)^{-1}$, $W_{(t^+_1,t^+_2,t^+_3)}(R)$ is the $R$-submodule of $R[\underline{x},(\det\underline{x})^{-1}]$ spanned by
\[
x^{a_1}_{11}x^{a_2}_{12}x^{a_3}_{13}\det\begin{pmatrix}x_{11}&x_{12}\\x_{21}&x_{22}\end{pmatrix}^{b_1}\det\begin{pmatrix}x_{11}&x_{13}\\x_{21}&x_{23}\end{pmatrix}^{b_2}\det\begin{pmatrix}x_{12}&x_{13}\\x_{22}&x_{23}\end{pmatrix}^{b_3}
   \det\begin{pmatrix}x_{11}&x_{12}&x_{13}\\ x_{21}&x_{22}&x_{23}\\ x_{31}&x_{32}&x_{33}\end{pmatrix}^{t^+_3},
\]
where  $a_i,b_i\geq 0,\,a_1+a_2+a_3=t^+_1-t^+_2,\,b_1+b_2+b_3=t^+_2-t^+_3$, and $W_{t^-_4}(R)$ is the $R$-submodule of $R[x,x^{-1}]$ spanned by $x^{t^-_4}$. 

The groups $\GL_3$ and $\GL_1$ act on $W_{\ut}$ by right translation. One can check that the left translation of $\begin{psm}a_1\\ \ast&a_2\\ \ast&\ast&a_3\end{psm}$ on $W_{(t^+_1,t^+_2,t^+_3)}$ is by the scalar $a^{t^+_1}_1 a^{t^+_2}_2 a^{t^+_3}_3$; when $R=\bC$, it is the irreducible algebraic representation of $\GL_3(\bC)$ of highest weight $(t^+_1,t^+_2,t^+_3)$. Let $\fe_{\can}:W_{\ut}\ra\bA^1$ be the linear functional defined by the evaluation at $(\bid_3,\bid_1)$.

Let $\underline{\omega}^+$ (resp. $\underline{\omega}^-$) be the subsheaf of $\underline{\omega}$ on which $i(b)$ acts by $b$ (resp. $\bar{b}$) for all $b\in\cO_\cK$. Because $p$ is unramified in $\cK$, $\underline{\omega}^+$ (resp. $\underline{\omega}^-$) is locally free of rank $3$ (resp. rank $1$) and $\underline{\omega}=\underline{\omega}^+\oplus\underline{\omega}^-$. 

Set \begin{align*}
\omega^+_{(t^+_1,t^+_2,t^+_3)}&=\underline{\Isom}_{\sS^\tor}(\cO^{\oplus{3}}_{\sS^\tor},\underline{\omega}^+)\times^{\GL_3} W_{(t^+_1,t^+_2,t^+_3)},\\
\omega^+_{t^-_1}&=\underline{\Isom}_{\sS^\tor}(\cO_{\sS^\tor},\underline{\omega}^+)\times^{\GL_1} W_{t^-_1}\cong (\underline{\omega}^-)^{\otimes t^-_1},
\end{align*}
and put $\gls{omegaut}=\omega^+_{(t^+_1,t^+_2,t^+_3)}\otimes\omega^-_{t^-_1}$.

Letting $F/\bQ_p$ be a finite extension containing the values of $\epsilon^\pm$, and $S^\tor_{K^p_fK^1_{p,n}}$ be the toroidal compactification of the Shimura variety of level $K^p_fK^1_{p,n}$  defined over $\cK$, 
we have 
\[
   \gls{Mt}=
\left(H^0\Big(S^\tor_{K^p_fK^1_{p,n}},\underline{\omega}_{(0,0;t^+;t^-)}\Big)\otimes_{\cK}F\right)[\epsilon^+,\epsilon^-],
\]
which is the space of classical automorphic forms on $G$ of weight $(0,0,t^+;t^-)$, level $K^p_fK^1_{p,n}$, and nebentypus $(\epsilon^+,\epsilon^-)$ for the action of $K^0_{p,n}/K^1_{p,n}$. Here $F$ is viewed as a $\cK$-algebra via $\cK\stackrel{\iota_p}{\ra} \bQ_p\ra F$. Similarly, we have the space of classical cuspidal automorphic forms
\[ 
  \gls{Mt0}
  =\left(H^0\Big(S^\tor_{K^p_fK^1_{p,n}},\underline{\omega}_{(0,0;t^+;t^-)}\otimes\cI_{\sS^\tor}\Big)\otimes_{\cK}F\right)[\epsilon^+,\epsilon^-].
\]

There are classical embeddings
\begin{equation}\label{eq:HtoV}
\begin{aligned}
   H^0\left(\sS^\tor,\underline{\omega}_{(0,0,t^+;t^-)}\right)&\lhra \varprojlim_m\varinjlim_n V_{n,m}[t^+,t^-],\\
   H^0\left(\sS^\tor,\underline{\omega}_{(0,0,t^+;t^-)}\otimes\cI_{\sS^\tor}\right)&\lhra \varprojlim_m\varinjlim_n V^0_{n,m}[t^+,t^-].
\end{aligned}
\end{equation}
induced by the trivialization of $\underline{\omega}$ over the Igusa tower and the canonical functional $\fe_\can$. More precisely, the trivialization of $\underline{\omega}$ arises from the  
the Hodge--Tate map $(\cG_{/S}[p^n]^\mult)^D\otimes_{\bZ}\cO_S\stackrel{\sim}{\ra} e^*\Omega^1_{\cG/S}\otimes_\bZ\bZ/p^n\bZ$, where $\cG^0_{/S}$ is an ordinary semi-abelian variety over $S$, the superscript $^D$ denotes the Cartier dual, and $e:S\ra \cG$ is the zero section.

Similarly, we have embeddings
\begin{equation}\label{eq:MtoV}
\begin{aligned}
   M_{(0,0,t^+;t^-)}\left(K^p_fK^1_{p,n},\epsilon^+,\epsilon^-;F\right)&\lhra \left(\varprojlim_m\varinjlim_n V_{n,m}\otimes_{\bZ_p}F\right)[(t^+,\epsilon^+),(t^-,\epsilon^-)],\\
   M^0_{(0,0,t^+;t^-)}\left(K^p_fK^1_{p,n},\epsilon^+,\epsilon^-;F\right)&\lhra \left(\varprojlim_m\varinjlim_n V^0_{n,m}\otimes_{\bZ_p}F\right)[(t^+,\epsilon^+),(t^-,\epsilon^-)].
\end{aligned}
\end{equation}

\subsection{The $\bU_p$-operator}\label{sec:Up-def}
Given a tuple $(A,\lambda,i,\alpha^p,\alpha_p)$ parameterized by $\Ig_{n,m}$, where $\alpha_p:\mu_{p^n}\otimes {\tt D}^0\ra A[p^n]^\mult$ is an isomorphism up to $K^1_{p,n}\cap P_{\tt D}(\bZ_p)$, the corresponding filtration ${\tt F^\pm}$ of $A[p^n]^{\mult\pm}$ is
\begin{align*}
  {\tt F}^+&\colon\; {\tt F}^+_2=\{e^+_1,e^+_2\}\subset{\tt F}^+_3=\{e^+_1,e^+_2,e^+_3\}=\mr{can}^+\subset {\tt F}^+_4=\{e^+_1,e^+_2,e^+_3,e^+_4\},\\
  {\tt F}^-&\colon\; {\tt F}^-_1=\{e^-_1\}=\mr{\can}^-\subset{\tt F}^-_2=\{e^-_1,e^-_3\}\subset{\tt F}^-_4=\{e^-_1,e^-_2,e^-_3,e^-_4\},
\end{align*}
where $(e^+_1,e^+_2,e^+_3;e^-_1)=\alpha_p(\xx^+_1,\ww^+_1,\ww^+_2;\xx^-_1)$.

Now we define three $\bU_p$-operators $U^+_{p,2},U^+_{p,3},U^-_{p,1}$.

\subsubsection{$\bU_p$-operators on $V_{n,m}$}

For $j=2,3$, let $\glsuseri{Cj}$ denote the solution to the moduli problem classifying tuples $(A,\lambda,i,\alpha^p,\alpha_p,\Lsub)$ with $\Lsub\subset A[p^2]$ a Lagrange subgroup (\emph{i.e.}, maximal isotropic for the $\lambda$-Weil pairing) stable under the $\cO_\cK$-action through $i$ such that $A[p]=\Lsub[p]\oplus {\tt F}^+_j[p]$. With $e^\pm_1,e^\pm_2,e^\pm_3,e^\pm_4$ a basis of $A[p^n]$ as above, such a $\Lsub$ is spanned by $(e^\pm_1,e^\pm_2,e^\pm_3,e^\pm_4)\cdot p^n\gamma_{\Lsub,p}^\pm$, where for $j=2$ 
\begin{equation}\label{eq:gamma+2}
\begin{aligned}
   \gamma^+_{\Lsub,p}&=
   \begin{pmatrix}1&&u_1&\ast\\&1&u_2&\ast\\&&1&\\&&&1\end{pmatrix}\begin{pmatrix}1&&&\\&1&&\\&&\frac{1}{p}&\\&&&\frac{1}{p}\end{pmatrix},\\
   \gamma^-_{\Lsub,p}&=
   \begin{pmatrix}1\\&\iota_p(\bar{\zeta}_0)^{-1}\\&&1\end{pmatrix}\begin{pmatrix}1&\ast&&\ast\\&1\\&-u_2&1&u_1\\&&&1\end{pmatrix}\begin{pmatrix}\frac{1}{p}&&&\\&\frac{1}{p^2}&&\\&&\frac{1}{p}&\\&&&\frac{1}{p^2}\end{pmatrix}\begin{pmatrix}1\\&\iota_p(\bar{\zeta}_0)\\&&1\end{pmatrix}
\end{aligned}
\end{equation}
with $u_1,u_2,\ast\in\bZ_p$, and for $j=3$, 
\begin{equation}\label{eq:gamma+3}
\begin{aligned}
   \gamma^+_{C,p}&=\begin{pmatrix}1&&&\ast\\&1&&\ast\\&&1&\ast\\&&&1\end{pmatrix}\begin{pmatrix}1&&&\\&1&&\\&&1&\\&&&\frac{1}{p}\end{pmatrix},
   &\gamma^-_{C,p}&=
   \begin{pmatrix}1&&&\ast\\&1&&\ast\\&&1&\ast\\&&&1\end{pmatrix}\begin{pmatrix}\frac{1}{p}&&\\&\frac{1}{p^2}\\&&\frac{1}{p^2}\\&&&\frac{1}{p^2}\end{pmatrix}
\end{aligned}
\end{equation}
with $\ast\in\bZ_p$. 

Similarly, let $\glsuserii{Cj}$ be the moduli space classifying tuples $(A,\lambda,i,\alpha^p,\alpha_p,\Lsub)$ with $\Lsub\subset A[p^2]$ a Lagrange subgroup stable under the $\cO_\cK$-action through $i$ such that $A[p]=\Lsub[p]\oplus{\tt F}^-_1[p]$. Then such a $\Lsub$ is spanned by $(e^\pm_1,e^\pm_2,e^\pm_3,e^\pm_4)\cdot p^n\gamma^\pm_{C,p}$ with
\begin{equation}\label{eq:gamma-}
\begin{aligned}
   \gamma^+_{C,p}&=\begin{pmatrix}1&&&\ast\\&1&&\ast\\&&1&\ast\\&&&1\end{pmatrix}\begin{pmatrix}\frac{1}{p}&&&\\&\frac{1}{p}&&\\&&\frac{1}{p}&\\&&&\frac{1}{p^2}\end{pmatrix},
   &\gamma^-_{C,p}&=
   \begin{pmatrix}1&&&\ast\\&1&&\ast\\&&1&\ast\\&&&1\end{pmatrix}\begin{pmatrix}1&&\\&\frac{1}{p}\\&&\frac{1}{p}\\&&&\frac{1}{p}\end{pmatrix}.
\end{aligned}
\end{equation}

For $(\bullet,j)=(+,2),(+,3),(-,1)$, we consider the diagram
\begin{equation}\label{eq:Up-graph}
\xymatrix{
   &\sC^{1\bullet}_{j,n,m}\ar[dl]_-{p_1}\ar[dr]^-{p_2}\\
   \Ig_{n,m}&&\Ig_{n,m}
}
\end{equation}
with $p_1,p_2$ the projections given by
\begin{align*}
  & p_1:(A,\lambda,i,\alpha^p,\alpha_p,C)\mapsto(A,\lambda,i,\alpha^p,\alpha_p),
  && p_2:(A,\lambda,i,\alpha^p,\alpha_p,C)\mapsto(A',\lambda',i',(\alpha^p)',\alpha_p'),
\end{align*}
where $A'=A/\Lsub$, $\lambda'$ is such that $\pi^\vee\circ\lambda\circ\pi=p^2\lambda$ with $\pi:A\ra A'$ the natural projection, $i'$ is the $\cO_{\cK}$-action $i$ on $A$ descended to $A'$, $(\alpha^{p})'=\pi\circ \alpha^p$, and $\alpha_p'$ given by $(e^{+\prime}_1,e^{+\prime}_2,e^{+\prime}_3;e^{-\prime}_1)$ defined as
\[
   (e^{+\prime}_1,e^{+\prime}_2,e^{+\prime}_3;e^{-\prime}_1)=\begin{cases}
   \pi\left((e^+_1,e^+_2,e^+_3;e^-_1)\begin{psm}1&&u_1\\&1&u_2\\&&1&\\&&&1\end{psm}\begin{psm}1\\&1\\&&p^{-1}\\&&&p^{-1}\end{psm}\right), &\bullet=+,\, j=2,\\[9pt]
   \pi\left((e^+_1,e^+_2, e^+_3;e^-_1)\begin{psm}1\\&1\\&&1\\&&&p^{-1}\end{psm}\right),&\bullet=+,\,j=3,\\[9pt] 
   \pi\left(e^+_1,e^+_2, e^+_3;e^-_1)\begin{psm}p^{-1}\\&p^{-1}\\&&p^{-1}\\&&&1\end{psm}\right), &\bullet=-,\,j=1.
   \end{cases}
\]
One can check that if $(e^{+}_1,e^{+}_2,e^{+}_3;e^{-}_1)$ is up to $K^1_{p,n}\cap P_{\tt D}(\bZ_p)$, then $(e^{+\prime}_1,e^{+\prime}_2,e^{+\prime}_3;e^{-\prime}_1)$ is well-defined up to $K^1_{p,n}\cap P_{\tt D}(\bZ_p)$.

Define
\begin{equation}\label{eq:open}
    \glsuseri{Up}\colon\; H^0(\Ig_{n,m},\cO_{\Ig_{n,m}})
   \stackrel{p^*_2}{\lra} H^0(\sC^+_{j,n,m},\cO_{\Ig_{n,m}})
   \xrightarrow{p^{-j}\Tr\,p_1} H^0(\Ig_{n,m},\cO_{\Ig_{n,m}}),
\quad(j=2,3)   
   \end{equation}
\begin{equation}\label{eq:open2}
    \glsuserii{Up}\colon\; H^0(\Ig_{n,m},\cO_{\Ig_{n,m}})
   \stackrel{p^*_2}{\lra} H^0(\sC^-_{1,n,m},\cO_{\Ig_{n,m}})
   \xrightarrow{p^{-3}\Tr\,p_1} H^0(\Ig_{n,m},\cO_{\Ig_{n,m}}).
\end{equation}
The normalization factors $p^{-j}$ and $p^{-3}$ are the inverse of the pure inseparability degree of the corresponding projection $p_1$; they are the optimal normalization to preserve integrality \cite[p.71]{hida-coherent}. By computing the effect of \eqref{eq:open} and \eqref{eq:open2} on Fourier--Jacobi expansions, one can check that the operators preserve $V_{n,m}$ and $V^0_{n,m}$.

\subsubsection{$\bU_p$-operators on $H^0\left(\Ig^{0,\tor}_{n,m},\omega_{\ut}\right)$}\label{sec:UpH}\hfill

Let $\Ig^{0,\tor}_{n,m}$ be the quotient of $\Ig^\tor_{n,m}$ by $K^0_{p,n}$, which can be described as 
\[
\underline{\Isom}_{\sS^{\tor}_m[1/E]}\left(\mu_{p^n}\otimes {\tt D}^0,\sG[p^n]^\mult\right)/K^0_{p,n}.
\]
Compared to $\Ig^\tor_{n,m}$, the level structure at $p$ for $\Ig^{0,\tor}_{n,m}$  forgets $\delta^+_3,\delta^-_1$ and parameterizes $({\tt F}^\pm,{\tt F}^+_1[p])$. We can define $\bU_p$-operators on $H^0\left(\Ig^{0,\tor}_{n,m},\omega_{\ut}\right)$. In order to see that they increase the level, we need to introduce Igusa towers of more general level structure at $p$.

Given $n_1,n_2\leq n_3\leq n_1+n_2$, define the level group 
\[
    K^0_{p,n_1,n_2,n_3}=\left\{g\in G(\bZ_p): g^+=\begin{psm}\ast&\ast&\ast&\ast\\p\ast&\ast&\ast&\ast\\ p^{n_1}\ast&p^{n_1}\ast&\ast&\ast\\p^{n_3}\ast&p^{n_3}\ast&p^{n_2}\ast&\ast\end{psm}\right\}.
\]
Define $\Ig^0_{n_1,n_2,n_3,m}$ as the quotient of $\Ig_{n,m}$, $n\geq n_1,n_2,n_3$, by $K^0_{p,n_1,n_2,n_3}\cap P_{\tt D}(\bZ_p)$, {\it i.e.} the corresponding level structure at $p$ parameterizes isomorphism $\alpha_p:\mu_{p^n}\otimes {\tt D}^0\ra A[p^{\infty}]^\mult$ up to $K^0_{p,n_1,n_2,n_3}\cap P_{\tt D}(\bZ_p)$. 

Like above for $(\bullet,j)=(+,2),(+,3),(-,1)$, let $\sC^{0\bullet}_{j,n_1,n_2,n_3,m}$ be the moduli space parameterizing tuples $(A,\lambda,i,\alpha^p,\alpha_p,\Lsub)$ with $\Lsub\subset A[p^2]$ a Lagrange subgroup stable under the $\cO_\cK$-action through $i$ such that $A[p]=C[p]\oplus {\tt F}^\bullet_j$.

Consider the diagrams
\[\xymatrix{
   &\sC^{0+}_{2,n_1,n_2,n_3,m}\ar[dl]_-{p_1}\ar[dr]^-{p_2}\\
   \Ig^0_{n_1,n_2,n_3,m}&&\Ig^0_{n_1+1,n_2,n_3+1,m}
}
\]

\[\xymatrix{
   &\sC^{0+}_{3,n_1,n_2,n_3,m}\ar[dl]_-{p_1}\ar[dr]^-{p_2}\\
   \Ig^0_{n_1,n_2,n_3,m}&&\Ig^0_{n_1,n_2+1,n_3+1,m}
}
\]
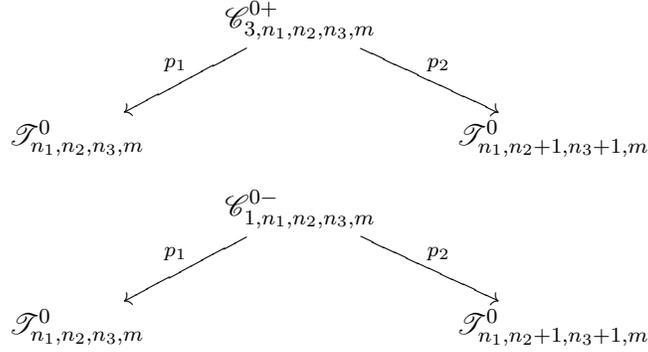
\[\xymatrix{
   &\sC^{0-}_{1,n_1,n_2,n_3,m}\ar[dl]_-{p_1}\ar[dr]^-{p_2}\\
   \Ig^0_{n_1,n_2,n_3,m}&&\Ig^0_{n_1,n_2+1,n_3+1,m}
}
\]
where $p_1,p_2$ are defined in the same way as in the diagram \eqref{eq:Up-graph}. This time for $(e^{+}_1,e^{+}_2,e^{+}_3;,e^{-}_1)$ up to $K^0_{p,n_1,n_2,n_3}\cap P_{\tt D}(\bZ_p)$, the $(e^{+\prime}_1,e^{+\prime}_2,e^{+\prime}_3;,e^{-\prime}_1)$ is well defined up to $K^0_{p,n_1+1,n_2,n_3+1}\cap P_{\tt D}(\bZ_p)$ in case $\sC^{0+}_{2,n,m}$, and well defined up to $K^0_{p,n_1,n_2+1,n_3+1}\cap P_{\tt D}(\bZ_p)$ in case $\sC^{0+}_{3,n,m}$ and case  $\sC^{0-}_{1,n_1,n_2,n_3,m}$.

In order to define the $\bU_p$-operators on the global sections of vector bundles, we also need maps
\begin{align*}
    \pi^*&:H^0\left(\sC^{0-}_{1,n_1+1,n_2,n_3+1,m},p^*_2\omega_{\ut}\right)\lra  H^0\left(\sC^{0-}_{1,n_1,n_2,n_3,m},p^*_1\omega_{\ut}\right)\\
    \pi^*&:H^0\left(\sC^{0-}_{1,n_1,n_2+1,n_3+1,m},p^*_2\omega_{\ut}\right)\lra  H^0\left(\sC^{0-}_{1,n_1,n_2,n_3,m},p^*_1\omega_{\ut}\right).
\end{align*}
Suppose that $\underline{\varepsilon}_A$ (resp. $\underline{\varepsilon}_{A/\Lsub}$) is a basis of $\underline{\omega}_A$ (resp. $\underline{\omega}_{A/\Lsub}$) compatible with $\alpha_p$ (resp. $\alpha'_p$). We have $\underline{\varepsilon}_A=(\pi^*\underline{\omega}_{A/\Lsub})g$ with
\[
   g=\begin{cases}
    h\begin{psm}1\\&1\\&&p^{-1}\\&&&p^{-1}\end{psm}h', &\bullet=+,\, j=2,\\
   h\begin{psm}1\\&1\\&&1\\&&&p^{-1}\end{psm}h', &\bullet=+,\, j=3,\\
   h\begin{psm}p^{-1}\\&p^{-1}\\&&p^{-1}\\&&&1\end{psm}h', &\bullet=-,\, j=1,
   \end{cases}
\]
for some $h,h'=\begin{psm}\ast&p^n\ast&p^n\ast\\\ast&\ast&p\ast\\ \ast&\ast&\ast\\ &&&\ast\end{psm}\in\GL_4(\bZ_p)$. In all the cases, $g^{-1}$ belongs to the semi-group
 \[
    \Delta_+=\left\{\left(h\begin{psm}a_1\\&a_2\\&&a_3\end{psm}h',a_4\right)\,:\,h,h'=\begin{psm}\ast&p\ast&p\ast\\\ast&\ast&p\ast\\\ast&\ast&\ast\end{psm}\in\GL_3(\bZ_p),\,a^{-1}_1a_2,a^{-1}_2a_3\in\bZ_p \right\}.
\]
We make $\Delta_+$ act on $W_{\ut}$ by 
\[
   \left(h\begin{psm}a_1\\&a_2\\&&a_3\end{psm}h',a_4\right)\cdot q(\underline{x},y)=q\left(\begin{psm}p^{-v_p(a_1)}\\&p^{-v_p(a_2)}\\&&p^{-v_p(a_3)}\end{psm}\underline{x}h\begin{psm}a_1\\&a_2\\&&a_3\end{psm}h',a_4p^{-v_p(a_4)y} \right),
\]
and define $\pi^*$ as
\[
    \pi^*\vec{f}(A,\lambda,i,\alpha^p,\alpha_p,\Lsub,\underline{\varepsilon}_A)=\vec{f}(A,\lambda,i,\alpha^p,\alpha_p,\Lsub,\underline{\varepsilon}_A,(\pi^*\underline{\varepsilon}_{A/\Lsub})g)=g^{-1}\cdot \vec{f}(A,\lambda,i,\alpha^p,\alpha_p,\Lsub,\underline{\varepsilon}_{A/\Lsub}).
\]
The $\bU_p$-operators are defined as
{\footnotesize
\begin{equation*}
\begin{aligned}
   \glsuseriii{Up}&:H^0\left(\Ig^0_{n_1+1,n_2,n_3+1,m},\omega_{\ut}\right)
   \stackrel{p^*_2}{\lra} H^0\left(\sC^{0+}_{2,n_1,n_2,n_3,m},p^*_2\omega_{\ut}\right)
   \stackrel{\pi^*}{\lra}  H^0\left(\sC^{0+}_{2,n_1,n_2,n_3,m},p^*_1\omega_{\ut}\right)
   \xrightarrow{p^{-2}\Tr\,p_1} H^0(\Ig^0_{n_1,n_2,n_3,m},\omega_{\ut}),\\
    \glsuseriv{Up}&:H^0\left(\Ig^0_{n_1,n_2+1,n_3+1,m},\omega_{\ut}\right)
   \stackrel{p^*_2}{\lra} H^0\left(\sC^{0+}_{3,n_1,n_2,n_3,m},p^*_2\omega_{\ut}\right)
   \stackrel{\pi^*}{\lra}  H^0\left(\sC^{0+}_{3,n_1,n_2,n_3,m},p^*_1\omega_{\ut}\right)
   \xrightarrow{p^{-3}\Tr\,p_1} H^0(\Ig^0_{n_1,n_2,n_3,m},\omega_{\ut}),\\
   \glsuserii{Up}&:H^0\left(\Ig^0_{n_1,n_2+1,n_3+1,m},\omega_{\ut}\right)
   \stackrel{p^*_2}{\lra} H^0\left(\sC^{0-}_{1,n_1,n_2,n_3,m},p^*_2\omega_{\ut}\right)
   \stackrel{\pi^*}{\lra}  H^0\left(\sC^{0-}_{1,n_1,n_2,n_3,m},p^*_1\omega_{\ut}\right)
   \xrightarrow{p^{-2}\Tr\,p_1} H^0(\Ig^0_{n_1,n_2,n_3,m},\omega_{\ut}).
\end{aligned}
\end{equation*}
}

Similarly as $V_{n,m}$, one can check that these $\bU_p$-operators preserve the spaces $H^0\left(\Ig^{0,\tor}_{n,m},\omega_{\ut}\right)$ and $H^0\left(\Ig^{0,\tor}_{n,m},\omega_{\ut}\otimes\cI_{\Ig^{0,\tor}_{n,m}}\right)$. It is also easy to see that $U^+_{p,2}U^+_{p,3}U^-_{p,1}$ maps $H^0\left(\Ig^{0,\tor}_{n,m},\omega_{\ut}\right)$ (resp. $H^0\left(\Ig^{0,\tor}_{n,m},\omega_{\ut}\otimes\cI_{\Ig^{0,\tor}_{n,m}}\right)$) into $H^0\left(\Ig^{0,\tor}_{n-1,m},\omega_{\ut}\right)$ (resp. $H^0\left(\Ig^{0,\tor}_{n-1,m},\omega_{\ut}\otimes\cI_{\Ig^{0,\tor}_{n-1,m}}\right)$), {\it i.e.} the operator $U^+_{p,2}U^+_{p,3}U^-_{p,1}$ increases the level.

\subsubsection{Adelic}
If we identify $G(\bQ_p)$ with $\GL_4(\bQ_p)\times\GL_1(\bQ_p)$ via \eqref{eq:Gisom}, the $\bU_p$-operators defined above acting on classical automorphic forms on $G$ of weight $(t^+_1,t^+_2.t^+_3;t^-_1)$ corresponds to the following adelic operators (up to the action of the center of $G(\bQ_p)$):
\begin{equation}\label{eq:adeleUp}
\begin{aligned}
   \glsuseriii{Up}:&\quad p^{\left<(t^+_1,t^+_2.t^+_3;-t^-_1)+2\rho_{\mr{c}},(1,1,0;0)\right>}\int_{N(\bZ_p)}R\left(u\begin{psm}p\\&p\\&&1\\&&&1\end{psm},1\right)\,du;
   && \\
   \glsuseriv{Up}:&\quad p^{\left<(t^+_1,t^+_2.t^+_3;-t^-_1)+2\rho_{\mr{c}},(1,1,1;0)\right>}\int_{N(\bZ_p)}R\left(u\begin{psm}p\\&p\\&&p\\&&&1\end{psm},1\right)\,du;
   && \\
   \glsuserii{Up}:&\quad p^{\left<(t^+_1,t^+_2.t^+_3;-t^-_1)+2\rho_{\mr{c}},(0,0,0;-1)\right>}\int_{N(\bZ_p)}R\left(u\begin{psm}1\\&1\\&&1\\&&&p^{-1}\end{psm},1\right)\,du;
   && 
\end{aligned}
\end{equation}
where $N(\bZ_p)=\left\{\begin{psm}1&\ast&\ast&\ast\\&1&\ast&\ast\\&&1&\ast\\&&&1\end{psm}\in\GL_4(\bZ_p)\right\}$, $R(-)$ denotes the right translation, and $2\rho_{\mr{c}}=(2,0,-2;0)$ is the sum of the compact positive roots of $G$.


\subsection{Statement of main theorem}\label{sec:statement-main}

Let $T_{\so}(\bZ_p)^\circ$ be the connected component of $T_\so(\bZ_p)$ containing $\bid_2$, \emph{i.e.}, $T_{\so}(1+p\bZ_p)$, and put $\gls{Lambdaso}=\bZ_p\llb T_\so(\bZ_p)^\circ\rrb$. 

\begin{thm}\label{prop:main}
The following hold:
\begin{enumerate}
\item Let $U_p=U^+_{p,2}U^+_{p,3}U^-_{p,1}$. Then for each $f\in\pV$, $\lim\limits_{r\to\infty}(U_p)^{r!}f$ converges, and we can define the semi-ordinary projector as $\gls{eso}=\lim\limits_{r\to\infty}(U_p)^{r!}$. The $\bZ_p\llb T_{\so}(\bZ_p)\rrb$-modules 
\[
\pV^{0,\ast}_\so=\Hom_{\bZ_p}\left(e_\so\pV^0,\bQ_p/\bZ_p\right),\quad
\pV^{\ast}_\so=\Hom_{\bZ_p}\left(e_\so\pV,\bQ_p/\bZ_p\right)
\]
are both free of finite rank over $\Lambda_\so$.

\item The spaces of $\Lambda_{\so}$-families of tame level $K^p_f$ are defined as
\[
    \glsuseri{cM}=\Hom_{\Lambda_\so}\left(\pV^{0,\ast}_\so,\Lambda_\so\right),
    \quad\cM_\so=\Hom_{\Lambda_\so}\left(\pV^{\ast}_\so,\Lambda_\so\right).
\]
For a given weight $(\tau^+,\tau^-)\in \Hom_\cont\left(T_{\so}(\bZ_p),F^\times\right)$, let $\cP_{\tau^+,\tau^-}$ be the ideal in $\cO_F\llb T_{\so}(\bZ_p)\rrb$ generated by $\{(x,y)-\tau^+(x)\tau^-(y)\colon (x,y)\in T_\so(\bZ_p)\}$. Then for $0\geq t^+\geq -t^-+4$, we have
\begin{equation}\label{eq:specialization}
\begin{aligned}
  \cM^0_\so\otimes_{\bZ_p\llb T_\so(\bZ_p)\rrb}\cO_F\llb T_\so(\bZ_p)\rrb/\cP_{\tau^+,\tau^-}&\stackrel{\sim}{\lra}  \left(\varprojlim_{m}\varinjlim_n e_\so V^0_{n,m}\otimes_{\bZ_p}\cO_F\right)[\tau^+,\tau^-],\\
   \cM_\so\otimes_{\bZ_p\llb T_\so(\bZ_p)\rrb}\cO_F\llb T_\so(\bZ_p)\rrb/\cP_{\tau^+,\tau^-}&\stackrel{\sim}{\lra}  \left(\varprojlim_{m}\varinjlim_n e_\so V_{n,m}\otimes_{\bZ_p}\cO_F\right)[\tau^+,\tau^-].
\end{aligned}
\end{equation}
The semi-ordinary projector preserves the spaces of classical forms, and by combining \eqref{eq:specialization} with \eqref{eq:MtoV}, we have the embeddings
\begin{align*}
   e_\so M^0_{(0,0,t^+;t^-)}\left(K_f^pK^1_{p,n},\epsilon^+,\epsilon^-;F\right)&\lhra \left(\cM^0_\so\otimes_{\bZ_p\llb T_\so(\bZ_p)\rrb}\cO_F\llb T_\so(\bZ_p)\rrb/\cP_{\tau^+,\tau^-}\right)[1/p],\\
   e_\so M_{(0,0,t^+;t^-)}\left(K_f^pK^1_{p,n},\epsilon^+,\epsilon^-;F\right)&\lhra \left(\cM_\so\otimes_{\bZ_p\llb T_\so(\bZ_p)\rrb}\cO_F\llb T_\so(\bZ_p)\rrb/\cP_{\tau^+,\tau^-}\right)[1/p],
\end{align*}
where $t^\pm=\tau_{\alg}^\pm$ and $\epsilon^\pm=\tau_{\rm f}^\pm$.

\item Given $0\geq t^+$, there exists $A\geq -t^++4$ 
such that the above embedding for the cuspidal forms is surjective if $t^-\geq A$.

\item There is the following so-called \emph{fundamental exact sequence} (in the study of Klingen Eisenstein congruences),
\[
   0\ra\cM^0_\so\ra \cM_\so\ra \bigoplus_{\clabel\in 
   (G'_1(\bQ)\ltimes P'_1(\bA^p_{\bQ,f}))\backslash G(\bA^p_{\bQ,f})/K^p_f} \gls{M00}\otimes_{\bZ_p}\bZ_p\llb T_{\so}(\bZ_p)\rrb \ra 0,
\]
where $K^{\prime}_{f,\clabel}\subset G'(\bA^p_{\bQ,f})$ is defined in \eqref{eq:K'}, and $M_{(0,0)}(K^{\prime}_{f,\clabel};\bZ_p)$ denotes the space of classical automorphic forms on $G'$ of weight $(0,0)$ and level $K^{\prime}_{f,\clabel}$.

\end{enumerate}
\end{thm}

\begin{rmk}
In general, one can consider semi-ordinary families whose members have weight $\ut=(t^+_1,t^+_2,t^+_3;t^-_1)$ with  $t^+_1,t^+_2$ fixed and $t^+_3,t^-_1$ varying in the family. We only consider the case $t^+_1=t^+_2=0$ because it suffices for proving Theorem~\ref{thm:main}, and we want to avoid having the main idea obscured by the extra complications of working with vector bundles. For the general case, instead of considering the global sections of the structure sheaf over the Igusa tower, one considers the global sections of a vector bundle $\omega_{\ut}$. For defining $\pV^\flat$ (cf. \S\ref{sec:Cflat}) such that the quotient $\pV/\pV^\flat$ has a nice structure, besides the requirement of vanishing outside the strata labeled by cusp labels in $C(K^p_fK^1_{p,n})^\flat_\ord$, one also requires that the elements are global sections of $\omega^\flat_{\ut}\subset \omega_{\ut}$ with $\omega^\flat_{\ut}$ a subsheaf defined as in \cite[Section 4.1]{HsiehMC}.

\end{rmk}


\section{The proof of Theorem~\ref{prop:main} for the  cuspidal part}
In this section, we prove Theorem~\ref{prop:main} for cuspidal families. The results for cuspidal families will be used to deduce the results for non-cuspidal families in \S\ref{sec:noncusp}
\begin{prop}[Base change property]\label{prop:basechange}
Let $\Ig^{0,\tor}_n[1/E]$ the open subscheme of $\Ig^{0,\tor}_n$ where $E$, our fixed lift of Hasse invariant viewed as a section over $\Ig^{0,\tor}_n$, is nonvanishing. For any classical dominant weight $\ut$, the natural map
\[
H^0\left(\Ig^{0,\tor}_n[1/E],\omega_{\ut}\otimes\cI_{\Ig^{0,\tor}_n}\right)\otimes\bZ/p^m\bZ\hra H^0\left(\Ig^{0,\tor}_{n,m},\omega_{\ut}\otimes\cI_{\Ig^{0,\tor}_{n,m}}\right)
\]
is an isomorphism. 
\end{prop}

\begin{proof}
Over $\Ig^\tor_n$, we have the exact sequence of sheaves 
\[
0\ra \omega_{\ut}\otimes\cI_{\Ig^\tor_n}\stackrel{p^m}{\ra} \omega_{\ut}\otimes \cI_{\Ig^\tor_{n}}\ra \omega_{\ut}\otimes\cI_{\Ig^\tor_{n,m}}\ra 0.
\] 
By \cite[Thm.~8.2.1]{Kuga}, $R^1_{\pi_n,*}(\omega_{\ut}\otimes\cI_{\Ig^\tor_n})=0$ (where $\pi_n$ is the projection $\pi_n:\Ig^\tor_{n}\ra \Ig^{\min}_{n}$), 
so we have the exact sequence of sheaves
\[
0\ra\pi_{n,*}(\omega_{\ut}\otimes\cI_{\Ig^\tor_n})\stackrel{p^{m}}{\ra}\pi_{n,*}(\omega_{\ut}\otimes\cI_{\Ig^\tor_n})\ra \pi_{n,*}(\omega_{\ut}\otimes\cI_{\Ig^\tor_{n,m}})\ra 0
\]
over $\Ig^{\min}_{n}$. Since $\Ig^{\min}_{n}[1/E]$ is affine by definition, taking global sections gives the result.
\end{proof}

\begin{prop}\label{prop:V0-converge}
Let $U_p=U^+_{p,2}U^+_{p,3}U^-_{p,1}$. Let $f\in \pV^0$ and $\vec{f}\in H^0\left(\Ig^{0,\tor}_{n,m},\omega_{\ut}\otimes\cI_{\Ig^{0,\tor}_{n,m}}\right)$. Then both limits  $\lim\limits_{r\ra\infty}(U_p)^{r!}f$ and $\lim
\limits_{r\ra\infty}(U_p)^{r!}\vec{f}$ converge.
\end{prop}

\begin{proof}
Consider $\tilde{\Ig}_n$, the ordinary locus of level 
\[
\tilde{K}^n_{p,1}=\left\{g\subset G(\bZ_p)\mid g^+\equiv \begin{psm}1&\ast&\ast&\ast\\ &1&\ast&\ast\\&&1&\ast\\&&&1\end{psm}\mod p^n\right\}.
\]

Let $\tilde{\Ig}^\tor_n$ be the partial toroidal compactification of $\tilde{\Ig}_n$, and put $\tilde{V}_{n,m}^0=H^0(\tilde{\Ig^\tor_{n,m}},\omega_{\ut}\otimes\tilde{\cI}_{\tilde{\Ig}^\tor_{n,m}})$. The  definition of the $\bU_p$-operators in \S\ref{sec:Up-def} can be naturally extended to $\tilde{V}^0_{n,m}$, so it suffices to show that   $\lim\limits_{r\ra\infty}(U_p)^{r!}f$ converges for every $f\in \tilde{V}^0_{n,m}$. As explained in \cite{hida-coherent}, $\varprojlim\limits_m\varinjlim\limits_n \tilde{V}^0_{n,m}$ is $p$-torsion free and there is an embedding
\[
   \bigoplus_{\ut \text{ dominant}} H^0\left(\sS^\tor,\underline{\omega}_{\ut}\otimes\cI_{\sS^\tor}\right)\otimes\bQ \hra \left(\varprojlim_m\varinjlim_n \tilde{V}^0_{n,m}\right)[1/p].
\] 
The $\bZ_p$-modules $H^0\left(\sS^\tor,\underline{\omega}_{\ut}\otimes\cI_{\sS^\tor}\right)$ are free of finite rank and stable under the action of $U_p$, so the limit $\lim\limits_{r\ra\infty} (U_p)^{r!}$ exists on them. Since
\begin{equation}\label{eq:dense-clas}
   \left(\bigoplus_{\ut \text{ dominant}} H^0\left(\sS^\tor,\underline{\omega}_{\ut}\otimes\cI_{\sS^\tor}\right)\otimes\bQ\right) \cap \varprojlim_m\varinjlim_n \tilde{V}^0_{n,m}
\end{equation}
is dense in $\varprojlim\limits_m\varinjlim\limits_n \tilde{V}^0_{n,m}$ by \cite{hida-coherent}, the convergence of  $\lim\limits_{r\ra\infty}(U_p)^{r!}f$ for every $f\in \tilde{V}^0_{n,m}$ follows.

On the other hand, let $\vec{f}\in H^0\left(\Ig^{0,\tor}_{n,m},\omega_{\ut}\otimes\cI_{\Ig^{0,\tor}_{n,m}}\right)$. By Proposition~\ref{prop:basechange}, $\vec{f}$ lifts to 
\[
\vec{\ff}\in  H^0\left(\Ig^{0,\tor}_{n}[1/E],\omega_{\ut}\otimes\cI_{\Ig^{0,\tor}_n}\right).
\] 
For $l\gg 0$, we have $\vec{\ff} E^l\in H^0\left(\Ig^{0,\tor}_{n},\omega_{\ut}\otimes\cI_{\Ig^{0,\tor}_n}\right)$, which by the Koecher principle can be viewed as an element in the finite dimensional space $M^0_{\ut+lt_E(p-1)}(K^p_fK^0_{p,n};\bQ_p)$. Thus the limit $\lim\limits_{r\ra\infty}(U_p)^{r!}(\vec{\ff} E^l)$ converges.  Since $E\equiv 1\mod p$, we have $\vec{\ff} E^l\equiv \vec{f}\mod p^m$, and the convergence of $\lim\limits_{r\ra\infty}(U_p)^{r!}\vec{f}$ follows.
\end{proof}

Thanks to the convergence of $\lim\limits_{r\ra\infty}(U_p)^{r!}$ on $\pV^0$ and $H^0\left(\Ig^{0,\tor}_{n,m},\omega_{\ut}\otimes\cI_{\Ig^{0,\tor}_{n,m}}\right)$, we can define the semi-ordinary projector on them as $e_\so=\lim\limits_{r\ra\infty}(U_p)^{r!}$.

\begin{prop}\label{prop:VH}
For $n\geq m$ and a dominant weight $(0,0,t^+;t^-)$, the map
\[
e_\so H^0\left(\Ig^{0,\tor}_{n,m},\omega_{(0,0,t^+;t^-)}\otimes\cI_{\Ig^\tor_{n,m}}\right)\ra e_\so V^0_{n,m}[t^+,t^-]
\]
induced by \eqref{eq:HtoV} is an isomorphism.
\end{prop}

\begin{proof}
Given $\vec{f}$, we have
\begin{equation}\label{eq:Up-vecf}
(U^n_p\vec{f})(A,\lambda,i,\alpha^p,\alpha_p,\underline{\varepsilon}_{\alpha_p})=p^{-(2+3+3)n}\sum_C \left(\ltrans{u}^{-1}_C\begin{psm}p^n\\ &p^n\\&&p^{2n}\end{psm},p^{2n}\right)\cdot \vec{f}(A/C,\lambda',i',(\alpha^p)^\prime,\alpha'_p,\underline{\varepsilon}_{\alpha'_p}).
\end{equation}
Here $\Lsub$ runs over the subgroups of $A[p^{4n}]$ which can be spanned as
\begin{align*}
   &(e^+_1,e^+_2,e^+_3,e^+_4)\cdot p^n\begin{psm}1&&u_1&\ast\\&1&u_2&\ast\\&&1&\ast\\&&&1\end{psm}\begin{psm}p^{-n}&&&\\& p^{-n}&&\\&& p^{-2n}&\\&&& p^{-4n}\end{psm},\\
   &(e^-_1,e^-_2,e^-_3,e^-_4)\cdot p^n\begin{psm}1&\ast&\ast&\ast\\&1\\&-u_2&1&-u_1\\&&&1\end{psm}\begin{psm}p^{-2n}&&\\& p^{-5n}\\&& p^{-4n}\\&&& p^{-5n}\end{psm},
\end{align*}
with $e^\pm_1,e^\pm_2,e^\pm_3,e^\pm_4$ a basis of $A[p^{n}]$ compatible with $\alpha_p$. For such a $\Lsub$, $u_\Lsub=\begin{psm}1&&u_1\\&1&u_2\\&&1\end{psm}$. By our definition of the $\Delta_+$-action on $W_{\ut}$, we know that if $w\in W_{(0,0,t^+;t^-)}$ is a vector of non-highest weight, then 
\begin{align*}
   \left(\begin{psm}p^n\\ &p^n\\&&p^{2n}\end{psm},p^{2n}\right)\cdot w\equiv 0\mod p^n.
\end{align*}
Therefore, from \eqref{eq:Up-vecf} we see that $U^n_p\vec{f}$ is actually determined by the projection of its values to the highest weight space.

On one hand, this shows that the map in the statement is injective. On the other hand, 
given $f\in e_\so V^0_{n,m}[t^+,t^-]$, we can  define $\vec{f}$ by the rule 
\begin{equation}\label{eq:fvecf}
   \vec{f}(A,\lambda,i,\alpha^p,\alpha_p,\underline{\varepsilon}_{\alpha_p})=p^{-(2+3+3)n}\sum_C \left(\ltrans{u}^{-1}_C,1\right)\cdot w_{(0,0,t^+;t^-)}f(A/C,\lambda',i',(\alpha^p)^\prime,\alpha'_p),
\end{equation}
where $C$ runs over the same range as above. One can check that $\vec{f}$ satisfies
\[
\vec{f}(A,\lambda,i,\alpha^p,\alpha_p\circ g,\underline{\varepsilon}_{\alpha_p\circ g})=\left(\begin{psm}\ast&p\ast&\\ \ast&\ast\\ u_1&u_2&a_1\end{psm},a_2\right)^{-1}\cdot   \vec{f}(A,\lambda,i,\alpha^p,\alpha_p,\underline{\varepsilon}_{\alpha_p})
\]
for $g^+\equiv\begin{psm}\ast&\ast&u_1&\ast\\ p\ast&\ast&u_2&\ast\\ &&a_1&\ast\\&&&a^{-1}_2\end{psm}\mod p^n$, so \eqref{eq:fvecf}  defines an element in  $H^0\left(\Ig^{0,\tor}_{n,m},\omega_{(0,0,t^+;t^-)}\otimes\cI_{\Ig^\tor_{n,m}}\right)$, and it is easy to check its semi-ordinarity from the semi-ordinarity of $f$. The composition of \eqref{eq:fvecf} and the map in the statement is $U^n_p$. Hence, the map is a bijection.
\end{proof}

\begin{prop}\label{prop:n1}
$e_\so H^0\left(\Ig^{0,\tor}_{n,m},\omega_{\ut}\otimes\cI_{\Ig^\tor_{n,m}}\right)=e_\so H^0\left(\Ig^{0,\tor}_{1,m},\omega_{\ut}\otimes\cI_{\Ig^\tor_{1,m}}\right)$.
\end{prop}
\begin{proof}
This follows immediately from the fact that $U_p$ maps the space $H^0\left(\Ig^{0,\tor}_{n,m},\omega_{\ut}\otimes\cI_{\Ig^\tor_{n,m}}\right)$ into $H^0\left(\Ig^{0,\tor}_{n-1,m},\omega_{\ut}\otimes\cI_{\Ig^\tor_{n,m}}\right)$. (See \S\ref{sec:UpH}.)
\end{proof}

\begin{prop}\label{prop:clas-bound}
For a fixed tame level $K^p_f$ and a fixed integer $B\geq 0$, the dimension of $e_\so M^0_{\ut}(K^p_fK^0_{p,1};\bQ_p)$ is uniformly bounded for all $t^+_1\geq t^+_2\geq t^+_3\geq -t^-_1+4$ with $t^+_1-t^+_2\leq B$.
\end{prop}

\begin{proof}
Suppose that $\Pi$ is an irreducible cuspidal automorphic representation of $G(\bA_\bQ)$ generated by a semi-ordinary form   $\varphi\in e_{\so}M^0_{\ut}\left(K^p_fK^0_{p,n};\bQ_p\right)$ whose (generalized) eigenvalue of $U^+_{p,2}$ (resp. $U^+_{p,3}$, $U^-_{p,1}$) is $\lambda^+_2$ (resp. $\lambda^+_3$, $\lambda^-_1$). The semi-ordinarity condition implies that $\lambda^+_2,\lambda^+_3,\lambda^-_1$ are all $p$-adic units. We view $\Pi_p$ as an irreducible representation of $\GL_4(\bQ_p)$ via \eqref{eq:Gisom}. Let $Q$ be the parabolic subgroup $\left\lbrace\begin{psm}\ast&\ast&\ast&\ast\\\ast&\ast&\ast&\ast\\&&\ast&\ast\\&&&\ast\end{psm}\right\rbrace\subset\GL_4$ with Levi factorization $Q=M_QN_Q$. Let 
\[
   \Pi_{p,N_Q}=\Pi_p/\{\Pi_p(u)v-v\,:\,v\in\Pi_p,\,u\in N_Q(\bQ_p)\},
\]
the Jacquet module of $\Pi_p$ with respect to the parabolic $Q$ acted on by $M_Q(\bQ_p)\cong \GL_2(\bQ_p)\times \GL_1(\bQ_p)\times\GL_1(\bQ_p)$. It follows from Jacquet's Lemma (see \cite[Theorem~4.1.2, Proposition~4.1.4]{Cas95}) that $\Pi^{M_Q(\bQ_p)\cap K^0_{p,n}}_{p,N_Q}$ equals the image under the natural projection of 
\begin{align*}
  &\bigcap_{r\geq 1}\left(U_{p,\mr{loc}}\right)^r\Pi^{K^0_{p,n}}_p,
  &U_{p,\mr{loc}}=\int_{N_Q(\bZ_p)}\Pi_p\left(u\begin{psm}p^2\\&p^2\\&&p^{\phantom{2}}\\&&&p^{-1}\end{psm}\right)\,du.
\end{align*}
The semi-ordinary form $\varphi\in \Pi$ is fixed by $K^0_{p,n}$ and belongs to a generalized eigenspace of $U_{p,\mr{loc}}$ with a nonzero eigenvalue. Thus, we deduce that $\Pi^{M_Q(\bQ_p)\cap K^0_{p,n}}_{p,N_Q}\neq 0$ and $\Pi_p$ is isomorphic to a subquotient of
\[
   \Ind^{\GL_4(\bQ_p)}_{Q(\bQ_p)} \sigma\boxtimes \chi\boxtimes\chi',
\]
where $\sigma$ is an irreducible smooth admissible representation of $\GL_2(\bQ_p)$, and $\chi,\chi'$ are smooth characters of $\bQ^\times_p$. We have $M_Q(\bQ_p)\cap K^0_{p,1}\cong K'_{p}\times (1+p^n\bZ_p)\times (1+p^n\bZ_p)$ with $K'_{p}=\left\lbrace\begin{pmatrix}\ast&\ast\\ p\ast&\ast\end{pmatrix}\right\rbrace\subset\GL_2(\bZ_p)$. If $\sigma$ is supercuspidal, $\Pi^{M_Q(\bQ_p)\cap K^0_{p,n}}_{p,N_Q}\neq 0$ implies that $\sigma^{K'_{p}}\neq 0$ (\cite[Theorem~6.3.5]{Cas95}). However, a supercuspidal representation of $\GL_2(\bQ_p)$ does not contain any nonzero vector fixed by $K'_{p}$. Therefore, we see that $\sigma$ is not supercuspidal, and $\Pi_p$ is isomorphic to a subquotient of a principal series. 
Denote by $\alpha_1,\alpha_2,\alpha_3,\alpha_4$ the evaluations at $p$ of the corresponding characters of $\bQ^\times_p$. Then, up to reordering, we have
\begin{align}
  \label{eq:lambda2}\lambda^+_{2}&=p^{t^+_1+\frac{1}{2}}\alpha_1\cdot p^{t^+_2-\frac{1}{2}}\alpha_2,\\
  \lambda^+_{3}&=p^{t^+_1+\frac{1}{2}}\alpha_1\cdot p^{t^+_2-\frac{1}{2}}\alpha_2\cdot p^{t^+_3-\frac{3}{2}}\alpha_3,\\
  \lambda^-_{1}&=p^{t^-_1-\frac{3}{2}}\alpha^{-1}_4.
\end{align}

Now, in addition to $U^+_{p,2},U^+_{p,3},U^-_{p,1}$, we also consider the action of the operator
\begin{equation}\label{eq:Up+1}
U^+_{p,1}=p^{\left<(t^+_1,t^+_2,t^+_3;-t^-_1)+2\rho_{\mr{c}},(1,0,0;0)\right>}\int_{N(\bZ_p)}R\left(u\begin{psm}p\\&1\\&&1\\&&&1\end{psm},1\right)\,du
\end{equation}
on $M^0_{\ut}\left(K^p_fK^0_{p,n};\bQ_p\right)$. The operator $U^+_{p,1}$ has a geometric interpretation analogous to \eqref{eq:open}, and the above normalization makes all its eigenvalues $p$-integral. 
If $\alpha^{-1}_1\alpha_2\neq p^{\pm 1}$, then $p^{t^+_1+\frac{1}{2}}\alpha_1$ and $p^{t^+_1+\frac{1}{2}}\alpha_2$ are both eigenvalues for the action of $U^+_{p,1}$ on the holomorphic forms in $\Pi$ of level $K^p_fK^0_{p,n}$, so 
\begin{equation}\label{eq:v1v2}
\begin{aligned}
   v_\fp(\alpha_1)+t^+_1+\frac{1}{2}&\geq 0,
   &v_\fp(\alpha_2)+t^+_1+\frac{1}{2}&\geq 0.
\end{aligned}
\end{equation}
By using that \eqref{eq:lambda2} is a $p$-adic unit and our condition $0\leq t^+_1-t^+_2\leq B$,
\begin{equation}\label{eq:v1+v2}
\begin{aligned}
    (v_\fp(\alpha_1)+t^+_1+\frac{1}{2})+(v_\fp(\alpha_2)+t^+_1+\frac{1}{2})&=t^+_1-t^+_2+1
    \leq B+1.
\end{aligned}
\end{equation}
Combining \eqref{eq:v1v2} and \eqref{eq:v1+v2}, we get
\begin{equation}\label{eq:v1}
   0\leq v_\fp(\alpha_j+t^+_1+\frac{1}{2})\leq B+1, \quad j=1,2.
\end{equation}
If $\alpha^{-1}_1\alpha_2=p^{\pm 1}$, then \eqref{eq:lambda2} being a $p$-adic unit implies that
\[
\frac{t^+_1-t^+_2}{2}\leq v_\fp(\alpha_j)\leq \frac{t^+_1-t^+_2}{2}+1,\quad j=1,2.
\]
Combining this with our condition $0\leq t^+_1-t^+_2\leq B$, we get
\begin{equation}\label{eq:v1'}
   0\leq v_\fp(\alpha_j)\leq \frac{B}{2}+1,\quad j=1,2.
\end{equation}
Therefore, all the semi-ordinary forms in $e_\so M^0_{\ut}(K^p_fK^0_{p,1};\bQ_p)$, $t^+_1\geq t^+_2\geq t^+_3\geq -t^-_1+4$, $t^+_1-t^+_2\leq B$, have slope $\leq B$ for the $\bU_p$-operator $U^+_{p,1}U^+_{p,2}U^+_{p,3}U^-_{p,1}$.

Recall that the theory of Coleman families for unitray groups (developed in \cite{brasca} as a generalization of \cite{AIP-siegel}) 
shows that for every point in the weight space $\Hom_\cont((\bZ^\times_p)^4,\bC^\times_p)$, there exists a neighborhood $\cU$ of that point and a projective $\cA_\cU$-module of finite rank interpolating all the cuspidal overconvergent forms of weights in $\cU$ and $U^+_{p,1}U^+_{p,2}U^+_{p,3}U^-_{p,1}$-slope $\leq B$ and a fixed tame level. Since all the algebraic weights $\ut$ are contained in a compact subset of $\Hom_\cont((\bZ^\times_p)^4,\bC_p)$, when $\ut$ varies among all the algebraic weights, there is a uniform bound on the dimension of the space of cuspidal overconvergent forms of weight $\ut$ and slopes $\leq B$ and tame level $K^p_f$.  The Proposition follows.
\end{proof}


\begin{prop}\label{prop:bound-clas}\leavevmode
\begin{enumerate}
\item $\dim_{\bF_p}\,e_\so H^0\left(\Ig^{0,\tor}_{1,1},\omega_{\ut}\otimes\cI_{\Ig^{0,\tor}_{1,1}}\right)<\infty$.
\item (Classicity) There is a canonical embedding 
\[
e_\so M^0_{(0,0,t^+;t^-)}\left(K_f^pK^1_{p,0};\bQ_p\right) \lhra \left(\varprojlim_m\varinjlim_n e_{\so}V^0_{n,m}[t^+,t^-]\right)\otimes\bQ_p,
\]
and given $0\geq t^+$, there exists $A(t^+)\geq -t^++4$ such that if $t^-\geq A(t^+)$ then the embedding is an isomorphism.
\end{enumerate}

\end{prop}
\begin{proof}
(1) Suppose that $\ol{f}_1,\dots,\ol{f}_d\in e_\so H^0\left(\Ig^{0,\tor}_{1,1},\omega_{\ut}\otimes\cI_{\Ig^{0,\tor}_{1,1}}\right)$ are linearly independent. By Proposition~\ref{prop:basechange}, they lift to $f_1,\dots,f_d\in H^0\left(\Ig^{0,\tor}_n[1/E],\omega_{\ut}\otimes\cI_{\Ig^{0,\tor}_n}\right)$. Recall that $E$ has scalar weight $t_E(p-1)$ 
and let $\ut+lk_E(p-1)=(t^+_1+lt_E(p-1),t^+_2+lt_E(p-1),t^+_3+lt_E(p-1);t^-_1+lt_E(p-1))$. For $l\gg 0$, we have 
\[
f_1E^l,\dots,f_dE^l\in H^0\left(\Ig^{0,\tor}_n,\omega_{\ut+lt_E(p-1)}\otimes\cI_{\Ig^{0,\tor}_n}\right).
\]
Because $E\equiv 1\mod p$, we have $\ol{e_\so(f_jE^l)}=\ol{f}_j$, and therefore $e_\so(f_1E^l),\dots,e_\so(f_dE^l)$ 
are linearly independent. Thus $d$ can be at most the bound in Proposition~\ref{prop:clas-bound}.

(2) We have $M^0_{(0,0,t^+;t^-)}\left(K_f^pK^0_{p,1};\bQ_p\right)=H^0\left(\Ig^{0,\tor}_1,\omega_{(0,0,t^+;t^-)}\otimes\cI_{\Ig^{0,\tor}_1}\right)\otimes_{\iota_p}\bQ_p$. (As mentioned above, the base change of $\Ig^{0,\tor}_1$ to $\cK$ is an open subscheme of $S^\tor_{K^p_fK^0_{p,1}}$; here we use the Koecher  principle to get the equality of global sections.) The map \eqref{eq:HtoV} induces
\[
    H^0\left(\Ig^{0,\tor}_1,\omega_{(0,0,t^+;t^-)}\otimes\cI_{\Ig^{0,\tor}_1}\right)\ra \left(\varprojlim_m\varinjlim_n e_{\so}V^0_{n,m}[t^+,t^-]\right),
\]
and this gives the canonical embedding in the statement.

Put $V^0_\so[t^+,t^-]=\varprojlim\limits_m\varinjlim\limits_n e_{\so}V^0_{n,m}[t^+,t^-]$. Note that from Propositions~\ref{prop:VH} and \ref{prop:n1}, for $n\geq m$, we have
\[ 
e_\so V^0_{n,m}[t^+,t^-]\simeq e_\so H^0\left(\sT^{0,\tor}_{1,m},\omega_{(0,0,t^+;t^-)}\otimes \cI_{\Ig^{0,\tor}_{1,m}}\right).
\] 
It follows that $V_\so^0[t^+,t^-]$ is $p$-torsion free, and together with Proposition~\ref{prop:basechange} we get
\begin{align*}
   \dim_{\bF_p}V^0_\so[t^+,t^-]/pV^0_\so[t^+,t^-]
   &=\dim_{\bF_p} e_\so H^0\left(\sT^{0,\tor}_{1,1},\omega_{(0,0,t^+;t^-)}\otimes \cI_{\Ig^{0,\tor}_{1,1}}\right)\\
   &=\dim_{\bQ_p}  e_\so H^0\left(\sT^{0,\tor}_{1}[1/E],\omega_{(0,0,t^+;t^-)}\otimes \cI_{\Ig^{0,\tor}_{1}}\right)\otimes\bQ_p.
\end{align*}
Denote this dimension by $d$. The above shows that $V^0_\so[t^+,t^-]$ is a free $\bZ_p$-module of rank $d$. Thus it suffices to show that there exists $A(t^+)\geq -t^+$ such that $\dim_{\bQ_p} e_\so M^0_{(0,0,t^+;t^-)}\left(K_f^pK^0_{p,1};\bQ_p\right) \geq d$ for $t^-\geq A(t^+)$. 
Pick an unramified Hecke character $\chi$ of $\cK^\times\backslash\bA^\times_\cK$ of $\infty$-type $(-t_E(p-1),0)$. Twisting $E$ by the character obtained by composing $\chi$ with $\det:G(\bQ)\backslash G(\bA_\bQ)\ra \cK^\times\backslash\bA^\times_\cK$, we get a holomorphic form $E'$ of weight $(0,0,0;2t_E(p-1))$ with the same vanishing locus as $E$. Multiplying by $(E')^l$ and applying $e_\so$ gives an injection 
\[
   e_\so M^0_{(0,0,t^+;t^-)}\left(K_f^pK^0_{p,1};\bQ_p\right)\lhra e_\so M^0_{(0,0,t^+;t^-+2t_E(p-1))}\left(K_f^pK^0_{p,1};\bQ_p\right),
\]
so
\[
\dim_{\bQ_p}e_\so M^0_{(0,0,t^+;t^-)}\left(K_f^pK^0_{p,1};\bQ_p\right)\leq  \dim_{\bQ_p}e_\so M^0_{(0,0,t^+;t^-+2t_E(p-1))}\left(K_f^pK^0_{p,1};\bQ_p\right).
\]
By the definition of $d$, we thus see that for each $0\leq j\leq 2t_E(p-1)-1$ there exists $l_j(t^+)\geq 0$ such that
\[
    \dim_{\bQ_p}e_\so M^0_{(0,0,t^+;-t^++4+j+2t_E(p-1)l_j(t^+))}\left(K_f^pK^0_{p,1};\bQ_p\right)\geq d.
\]
Therefore, $A(t^+)=-t^++4+2t_E(p-1) (\max_j\{l_j(t^+)\}+1)$ has the required property.
\end{proof}

\begin{thm}[Vertical Control Theorem]\leavevmode
\label{thm:cuspidal}
$\pV^{0,\ast}_\so$ is a free $\Lambda_\so$-module of finite rank.
\end{thm}
\begin{proof}
Let $\fm$ be a maximal ideal of $\bZ_p\llb T_\so(\bZ_p)\rrb$. (We know that $\fm\cap \Lambda_\so=(p,T^+,T^-)$, where we identify $\Lambda_\so$ with $\bZ_p\llb T^+,T^-\rrb$ by identifying $(1+p,1)\in T_\so(\bZ_p)$ (resp. $(1,1+p)\in T_\so(\bZ_p)$ with $T^+$ (resp. $T^-$).) We show that $\pV^{0,\ast}_{\so,\fm}$, the localization of $\pV^{0,\ast}_\so$ at $\fm$, is a free  $\Lambda_\so$-module of finite rank. 

First, we consider the quotient $\pV^{0,\ast}_{\so,\fm}\otimes_{\Lambda_\so}\Lambda_\so/(p,T^+,T^-)$, which by definition, equals $(e_\so\pV[\fm])^*$. Take $0\geq t^+\geq -t^-+4$ with $\cP_{t^+,t^-}\subset \fm$. Then 
\begin{align*}
 e_\so\pV[\fm]
   &\stackrel{\phantom{\text{Prop~\ref{prop:VH}}}}{=} \text{$p$-torsion of }\varinjlim_m \varinjlim_n e_\so V^0_{n,m}[t^+,t^-]\\
   &\stackrel{\text{(Prop~\ref{prop:VH})}}{=}\text{$p$-torsion of }\varinjlim_m \varinjlim_n e_\so H^0\left(\Ig^{0,\tor}_{n,m},\omega_{(0,0,t^+;t^-)}\otimes\cI_{\Ig^{0,\tor}_{n,m}}\right)\\
   &\stackrel{\text{(Prop~\ref{prop:n1})}}{=}\text{$p$-torsion of }\varinjlim_m e_\so H^0\left(\Ig^{0,\tor}_{1,m},\omega_{(0,0,t^+;t^-)}\otimes\cI_{\Ig^{0,\tor}_{1,m}}\right)\\
   &\stackrel{\phantom{\text{Prop~\ref{prop:VH}}}}{=} e_\so \left[\text{$p$-torsion of }\varinjlim_m H^0\left(\Ig^{0,\tor}_{1,m},\omega_{(0,0,t^+;t^-)}\otimes\cI_{\Ig^{0,\tor}_{1,m}}\right)\right].
\end{align*}   
Also, Proposition~\ref{prop:basechange} implies
\begin{align*}
   \varinjlim_m H^0\left(\Ig^{0,\tor}_{1,m},\omega_{(0,0,t^+;t^-)}\otimes\cI_{\Ig^{0,\tor}_{1,m}}\right) &=\varinjlim_m  H^0\left(\Ig^{0,\tor}_{1}[1/E],\omega_{(0,0,t^+;t^-)}\otimes\cI_{\Ig^{0,\tor}_{1}}\right)\otimes\bZ/p^m\bZ.
\end{align*}
It follows that 
\begin{align*}
   \text{$p$-torsion of }\varinjlim_m H^0\left(\Ig^{0,\tor}_{1,m},\omega_{(0,0,t^+;t^-)}\otimes\cI_{\Ig^{0,\tor}_{1,m}}\right)&= H^0\left(\Ig^{0,\tor}_{1}[1/E],\omega_{(0,0,t^+;t^-)}\otimes\cI_{\Ig^{0,\tor}_{1}}\right)\otimes\bZ/p\bZ\\
   &= H^0\left(\Ig^{0,\tor}_{1,1},\omega_{(0,0,t^+;t^-)}\otimes\cI_{\Ig^{0,\tor}_{1,1}}\right),
\end{align*}
where Proposition~\ref{prop:basechange} is used again for the second equality. Thus,
\[
e_\so\pV[\fm]=e_\so H^0\left(\Ig^{0,\tor}_{1,1},\omega_{(0,0,t^+;t^-)}\otimes\cI_{\Ig^{0,\tor}_{1,1}}\right).
\]
By Proposition~\ref{prop:bound-clas}, we know this is finite dimensional over $\bF_p$. Hence, $\pV^{0,\ast}_{\so,\fm}\otimes_{\Lambda_\so}\Lambda_\so/(p,T^+,T^-)=(e_\so\pV[\fm])^*$ is finite dimensional over $\bF_p$. Let
\[
d=\dim_{\bF_p} \pV^{0,\ast}_{\so,\fm}\otimes_{\Lambda_\so}\Lambda_\so/(p,T^+,T^-).
\]
By Nakayama's Lemma, there exist $F_1,\dots,F_d\in \pV^{0,\ast}_{\so,\fm}$ such that 
\[  
\pV^{0,\ast}_{\so,\fm}=\Lambda_\so F_1+\cdots +\Lambda_\so F_d.
\]  

Next, we show that $F_1,\dots,F_d$ are $\Lambda_\so$-linearly independent. Suppose that $a_1,\dots,a_d\in \Lambda_\so$ are such that $a_1F_1+\cdots+a_dF_d=0$. Given any $0\geq t^+\geq -t^-+4$ with $\cP_{t^+,t^-}\subset \fm$, we put 
\[
\cP^\circ_{t^+,t^-}=\Lambda_\so\cap \cP_{t^+,t^-}.
\] 
In view of Proposition~\ref{prop:basechange}, the $\bZ_p$-module $\varinjlim
\limits_m\varinjlim\limits_n H^0\left(\Ig^{0,\tor}_{n,m},\omega_{\ut}\otimes\cI_{\Ig^{0,\tor}_{n,m}}\right)$ is $p$-divisible, and by Proposition~\ref{prop:VH} so is $e_\so \pV^0[t^+,t^-]$.
Therefore,
\[
\pV^{0,\ast}_{\so,\fm}\otimes_{\Lambda_\so}\Lambda_\so/\cP^\circ_{t^+,t^-}=\left( e_\so \pV^0[t^+,t^-]\right)^*
\] 
is $p$-torsion free. On the other hand,
\[
  \dim_{\bF_p} \left(\pV^{0,\ast}_{\so,\fm}\otimes_{\Lambda_\so}\Lambda_\so/\cP^\circ_{t^+,t^-}\right)\otimes_{\bZ}\bZ/p\bZ=\dim_{\bF_p}\pV^{0,\ast}_{\so,\fm}\otimes_{\Lambda_\so}\Lambda_\so/(p,T^+,T^-)=d.
\]
It follows that the $\bZ_p$-module $\pV^{0,\ast}_{\so,\fm}\otimes_{\Lambda_\so}\Lambda_\so/\cP^\circ_{t^+,t^-}$ is free of rank $d$, and therefore $a_1,\dots,a_d\in \cP^\circ_{t^+,t^-}$. Then from $\bigcap\limits_{0\geq t^+\geq -t^-+4}\cP^\circ_{t^+,t^-}=0$ we conclude that $a_1=\cdots=a_d=0$, and hence $\pV^{0,\ast}_{\so,\fm}$ is a free $\Lambda_\so$-module of rank $d$.
\end{proof}

This completes the proof of part (1) of Theorem~\ref{prop:main} for cuspidal forms. For part (2), because $\pV^{0,\ast}_\so$ is free over $\Lambda_\so$, letting $V^0_{m,\so}[\tau^+,\tau^-]=e_\so \varinjlim\limits_n V^0_{n,m}\otimes_{\bZ_p}\cO_F[\tau^+,\tau^-]$ we have 
\begin{equation}\label{eq:MV}
\begin{aligned}
   \cM^0_\so\otimes_{\bZ_p\llb T_\so(\bZ_p)\rrb} \cO_F\llb T_\so(\bZ_p)\rrb /\cP_{\tau^+,\tau^-}
   &\cong \Hom_{\bZ_p}\left(\pV^{0,\ast}_\so\otimes_{\bZ_p} \cO_F /\cP_{\tau^+,\tau^-},\bZ_p\right)\\
   &=\Hom_{\bZ_p}\left((e_\so\pV^0[\tau^+,\tau^-])^*,\bZ_p\right)\\
   &=\Hom_{\bZ_p}\left(\left(\varinjlim_m V^0_{m,\so}[\tau^+,\tau^-]\right)^*,\bZ_p\right).
\end{aligned}
\end{equation}
Since the left hand side is a free $\bZ_p$-module (because $\cM^0_\so$ is a free $\Lambda_\so$-module), we see that  
$\varinjlim\limits_m  V^0_{m,\so}[\tau^+,\tau^-])$ is $p$-divisible, and $V^0_{m+1,\so}[\tau^+,\tau^-]\ra V^0_{m,\so}[\tau^+,\tau^-]$ is surjective, so
\[
   \Hom_{\bZ_p}\left(\varinjlim_m V^0_{m,\so}[\tau^+,\tau^-],\bQ_p/\bZ_p\right)\cong   \Hom_{\bZ_p}\left(\varprojlim_m V^0_{m,\so}[\tau^+,\tau^-],\bZ_p\right)
\]
and hence
\begin{align*}
   \eqref{eq:MV}=\Hom_{\bZ_p}\left(\Hom_{\bZ_p}\left(\varprojlim_m V^0_{m,\so}[\tau^+,\tau^-],\bZ_p\right),\bZ_p\right)&=\varprojlim_m V^0_{m,\so}[\tau^+,\tau^-]\\
   &=\left(\varprojlim_m \varinjlim_n e_\so V^0_{n,m}\otimes_{\bZ_p}\cO_F\right)[\tau^+,\tau^-],
\end{align*}
concluding the proof of part (2) of Theorem~\ref{prop:main} for cuspidal forms. Part (3)  follows from Proposition~\ref{prop:bound-clas}.

\section{The proof of Theorem~\ref{prop:main} for the non-cuspidal part}\label{sec:noncusp}

We apply the approach in \cite{LRtz} to prove the vertical control theorem for semi-ordinary forms on $\GU(3,1)$ by analyzing the quotient $\pV/\pV^0$ and using the vertical control theorem for cuspidal semi-ordinary forms on $\GU(3,1)$. When studying $\pV/\pV^0$, we introduce an auxiliary space $\pV^\flat$. One difference from {\it loc.cit} is that $q$-expansions are used there to reduce proving some properties for the $\bU_p$-action on $\pV,\pV^\flat$ to matrix computations, but $q$-expansions are not available in the case of $\GU(3,1)$. The analogue of those properties in our case are proved in \S\ref{sec:three-prop} by working with semi-abelian schemes over the boundary of the partial toroidal compactification. 

\subsection{Cusp labels} Following \cite{Thesis}, a \emph{cusp label} is a $K_f$-orbit of triples $\gls{clabel}$, with:
\begin{itemize}[leftmargin=2.5em]
\item{} ${\tt Z}:0\subset{\tt Z}_{-2}\subset{\tt Z}_{-1}={\tt Z}_{-2}^\bot\subset L\otimes\hat{\bZ}$ a {\it fully symplectic admissible} filtration.  
\item{} $\Phi=({\tt X},{\tt Y},\phi,\varphi_{-2},\varphi_0)$ a {\it torus argument}, where 
$\phi:{\tt Y}\ra{\tt X}$ is an $\cO_\cK$-linear embedding of locally free $\cO_{\cK}$-modules with finite cokernel, and where, writing $\mr{Gr}^{\tt Z}_{-i}={\tt Z}_{-i}/{\tt Z}_{-i-1}$,  $\varphi_{-2}:\mr{Gr}^{\tt Z}_{-2}\stackrel{\sim}{\ra}\Hom\left({\tt X}\otimes\hat{\bZ},\hat{\bZ}(1)\right)$ and $\varphi_0:\mr{Gr}^{\tt Z}_0\stackrel{\sim}{\ra}{\tt Y}\otimes\hat{\bZ}$ are isomorphisms such that $\left<v,w\right>=\varphi_0(v)\Big(\phi\big(\varphi_{-2}(w)\big)\Big)$. 
\item $\delta:\mr{Gr}^{\tt Z}=\mr{Gr}^{\tt Z}_{-2}\oplus\mr{Gr}^{\tt Z}_{-1}\oplus\mr{Gr}^{\tt Z}_{0}\stackrel{\sim}{\ra} L\otimes\hat{\bZ}$ is an $\cO_\cK\otimes\hat{\bZ}$-equivariant splitting.
\end{itemize}

Following \cite[Def.~3.2.3.1]{Kuga}, an \emph{ordinary cusp label} of tame level $K_f^p$ is a $K^p_fP_{\tt D}(\bZ_p)$-orbit of triples $({\tt Z},\Phi,\delta)$ as above compatible with ${\tt D}$, in the sense that 
\begin{align*}
   {\tt Z}_{-2}\otimes_{\hat{\bZ}}\bZ_p\subset {\tt D}\subset {\tt Z}_{-1}\otimes_{\hat{\bZ}}\bZ_p.
\end{align*}

There is a unique cusp label for which ${\tt Z}_{-2}=0$, and in our case, all other cusp labels have ${\tt Z}_{-2}$ with rank $1$. As explained in \cite[B.2, B.3, B.11]{HLTT}, the latter are parametrized by a certain double coset space. To recall this, we have the filtration 
\begin{equation}\label{eq:L-fil}
   X^\vee\subset X^\vee\oplus L_0\subset L=X^\vee\oplus L_0\oplus Y.
\end{equation}
Let $\glsuseri{P}\subset G$ be the parabolic subgroup preserving this filtration, and $\glsuserii{P}\subset P$ be the kernel of the natural projection $P\ra\GL(Y)$. Define the triple $\gls{clabel1}$ as
\begin{itemize}[leftmargin=2.5em]
\item{} ${\tt Z}^{(1)}$: ${\tt Z}_{-2}=X^\vee\otimes\hat{\bZ}$, ${\tt Z}_{-1}=(X^\vee\oplus L_0)\otimes\hat{\bZ}$.
\item{} $\Phi^{(1)}=({\tt X}^{(1)},{\tt Y}^{(1)},\phi^{(1)},\varphi^{(1)}_{-2},\varphi^{(1)}_0)$ with ${\tt X}^{(1)}=\Hom_{\cO_\cK}\left(X^\vee,\cO_\cK(1)\right)$, ${\tt Y}^{(1)}=Y$, $\phi^{(1)}:Y\ra\Hom_{\cO_\cK}\left(\fd^{-1}_{\cK/\bQ}\cdot{\tt x}_1,\cO_\cK(1)\right)$ induced by our fixed basis of ${\tt x}_1$ and ${\tt y}_1$ of $X$ and $Y$ and the pairing $2\pi i\cdot \Tr_{\cK/\bQ}\circ\left<\,,\,\right>_L$, $\varphi^{(1)}_{-2}=2\pi\sqrt{-1}\cdot\id$ and $\varphi^{(1)}_{0}=\id$.
\item{} $\delta^{(1)}=\id$.
\end{itemize}
Denote by  ${\tt V}\colon 0\subset {\tt V}_{-2}\subset {\tt V}_{-1}\subset L\otimes\bQ$ the filtration obtained as \eqref{eq:L-fil} tensored with $\bQ$. For every $\clabel\in G(\bA_{\bQ,f})$ define $\gls{clabelg}$ by
\begin{itemize}[leftmargin=2.5em]
\item{} ${\tt Z}^{(\clabel)}_j=\left(\clabel^{-1}({\tt V}_j\otimes\bA_{\bQ,f})\right)\cap\left(L\otimes\hat{\bZ}\right)$.
\item{} $\Phi^{(\clabel)}=({\tt X}^{(\clabel)},{\tt Y}^{(\clabel)},\phi^{(\clabel)},\varphi^{(\clabel)}_{-2},\varphi^{(\clabel)}_0)$ with $({\tt X}^{(\clabel)},{\tt Y}^{(\clabel)},\phi^{(\clabel)})=({\tt X}^{(1)},{\tt Y}^{(1)},\phi^{(1)})$, $\varphi^{(\clabel)}_{-2}$ (resp. $\varphi^{(\clabel)}_{0}$) is the composition of $\varphi^{(1)}_{-2}$ (resp. $\varphi^{(1)}_{0}$) with $\clabel:\mr{Gr}^{{\tt Z}^{(\clabel)}}_{-2}\ra \mr{Gr}^{\tt Z}_{-2}$ (resp.  $\clabel:\mr{Gr}^{{\tt Z}^{(\clabel)}}_{0}\ra \mr{Gr}^{\tt Z}_{0}$).
\item{} $\delta^{(\clabel)}$ is $\mr{Gr}^{{\tt Z}^{(\clabel)}}\stackrel{\mr{Gr}(\clabel)}{\ra}\mr{Gr}^{\tt Z^{(1)}}\stackrel{\delta^{(1)}}{\ra} L\otimes\hat{\bZ}\stackrel{\clabel^{-1}}{\ra}L\otimes\hat{\bZ}$.
\end{itemize}
Then the map that sends $\clabel$ to (the cusp label represented by) $({\tt Z}^{(\clabel)},\Phi^{(\clabel)},\delta^{(\clabel)})$ sets up a bijection between the set of cusp labels with ${\tt Z}_{-2}$ rank $1$ and the double coset space
\begin{align*}
   \gls{C(K)}&=\GL(X\otimes\bQ)\ltimes P'(\bA^p_{\bQ,f})\backslash G(\bA^p_{\bQ,f})/K^p_f.
\end{align*}
(See \cite[Prop.~A.5.9]{App}.)

Similarly, the ordinary cusp labels with ${\tt Z}_{-2}$ of rank $1$ are parameterized by
\begin{equation}\label{eq:ordC_1}
  \gls{C(K)ord}=\im\left( G(\bA^p_{\bQ,f}) P_{\tt D}(\bQ_p)\lra\GL(X\otimes\bQ)\ltimes P'(\bA_{\bQ,f})\backslash G(\bA_{\bQ,f})/K^p_fK^1_{p,n}\right),
\end{equation}
which is in a natural bijection with
\begin{equation}\label{eq:ordC-2} 
\begin{aligned}
    &\left.\left(\GL(X_{(p)})\ltimes (P'(\bA^p_{\bQ,f})\times P'(\bZ_p))\right)\middle\backslash G(\bA^p_{\bQ,f})\times P_{\tt D}(\bZ_p)\right/K^p_f(K^1_{p,n}\cap P_{\tt D}(\bZ_p))\\
    \cong &
    \,\,C(K^p_f)\times \left(U_{\cK,K^p_f}\ltimes P'(\bZ_p)\backslash P_{\tt D}(\bZ_p)/K^1_{p,n}\cap P_{\tt D}(\bZ_p)\right),
\end{aligned}
\end{equation}
where $U_{\cK,K^p_f}=\left\{a\in\cO^\times_\cK:\begin{psm}a\\&\bid_2\\ &&\ltrans{\bar{a}}^{-1}\end{psm}\in K^p_f \right\}$.
%
%
By abuse of notation, we shall denote by $\clabel$ both elements in $G(\bA^p_{\bQ,f})P_{\tt D}(\bQ_p)$ and in its quotient $C(K^p_fK^1_{p,n})_\ord$. The context should eliminate any possible confusions.

\subsection{The formal completion along boundary strata}\label{sec:bd-strata} 
Let $\glsuseri{Zg}$ be the stratum associated to $\clabel\in C(K^p_fK^1_{p,n})_\ord$ in the partial toroidal compactification $\Ig^\tor_n$. Then we have
\[ 
   \Ig^{\tor}_n=\sT_{n}\sqcup\bigsqcup\limits_{\clabel\in C(K^p_fK^1_{p,n})_\ord} \sZ_{\clabel,n}.
\]
(The choice of a polyhedral cone decomposition is unique in our special case $\GU(3,1)$ because with $\Big({\bf S}_{\Phi^{(\clabel)}_{K^p_fK^1_{p,n}}}\Big)$ defined as below and ${\bf P}_{\Phi^{(\clabel)}_{K^p_fK^1_{p,n}}}$ the subset of $\Hom_{\bZ}\Big({\bf S}_{\Phi^{(\clabel)}_{K^p_fK^1_{p,n}}},\bZ\Big)\otimes\bR$ consisting of positive semi-definite Hermitian forms, we have  ${\bf P}_{\Phi^{(\clabel)}_{K^p_fK^1_{p,n}}}=\bR_{\geq 0}$.) Let $\gls{Igzmin}$ be the partial minimal compactification of $\Ig_n$ (constructed in \cite[Theorem~6.2.1.1]{Kuga}). Denoting by $\glsuserii{Zg}$ the stratum associated to $\clabel\in C(K^p_fK^1_{p,n})_\ord$ in $\Ig^{\min}_n$, we similarly have
\[
   \Ig^{\min}_n=\sT_{n}\sqcup\bigsqcup\limits_{\clabel\in C(K^p_fK^1_{p,n})_\ord} \sY_{\clabel,n}.
\] 
The stratum $\sY_{\clabel,n}$ can be identified with the $0$-dimensional Shimura variety $\gls{sSG'}$, where
\begin{align}
   \label{eq:K'} \gls{K'fgn}&= \im \left(P(\bA_{\bQ,f})\cap \clabel K_f \clabel^{-1}\ra G'(\bA^p_{\bQ,f})\right).
\end{align} 

We have the following diagram:
\begin{equation}\label{eq:Xi}
\xymatrix{
  \glsuseri{Xi}\, \ar@{^{(}->}[r] \ar[d]&    \glsuserii{Xi} \ar[ld]\\
\sC^\ord_{\clabel,n}  \ar[d]_-{h}\\
\sS_{G',K'_{f,\clabel,n}},
}
\end{equation}
where $\gls{Cgn}\ra \sS_{G',K'_{f,\clabel,n}}$ is a torsor of an abelian scheme quasi-isogenous to $\underline{\Hom}_{\cO_\cK}({\tt X}^{(\clabel)},\cA)$ with $\cA$ the universal abelian scheme over $\sS_{G',K'_{f,\clabel,n}}$; $\Xi^\ord_{\clabel,n}\ra \sC^\ord_{\clabel,n}$ is a torsor of the torus with character group ${\bf S}_{\Phi^{(\clabel)}_{K^p_fK^1_{p,n}}}$, which is a finite index subgroup of the free quotient of
\[
   \left.\left(\frac{1}{N} {\tt Y}^{(\clabel)}\otimes_{\bZ} {\tt X}^{(\clabel)}\right)\middle/
   \begin{pmatrix}
     y\otimes\phi^{(\clabel)}(y')-y'\otimes\phi^{(\clabel)}(y)\\
     (b\frac{1}{N}y)\otimes x-(\frac{1}{N}y\otimes (b^cx)
   \end{pmatrix}_{y,y'\in {\tt Y}^{(\clabel)},x\in {\tt X}^{(\clabel)},b\in\cO_\cK}\right.
\]
where $N$ is sufficiently large with respect to the level $K^p_fK^1_{p,n}$; and  $\Xi^\ord_{\clabel,n}\hra \ol{\Xi}^\ord_{\clabel,n}$ is the torus embedding with respect to the unique cone decomposition.

Denote by $\gls{Xgn}$ the formal completion of $\Ig^\tor_n$ along the stratum $\sZ_{\clabel,n}$ . Then $\fX_{\clabel,n}$ is canonically isomorphic to the quotient by $\Gamma_\clabel$ of the formal completion of $\ol{\Xi}^\ord_{\clabel,n}$ along the boundary $\ol{\Xi}^\ord_{\clabel,n}-\Xi^\ord_{\clabel,n}$. Here $\Gamma_\clabel$ is the (finite, in our case) subgroup of $\GL({\tt X}^{(\clabel)})\times \GL({\tt Y}^{(\clabel)})$ preserving $\varphi^{(\clabel)}_{-2}$ and $\varphi^{(\clabel)}_{0}$ (as orbits).


\subsection{The subset $C(K^p_fK^1_{p,n})^\flat_\ord$ of $C(K^p_fK^1_{p,n})_\ord$ and the subspace $\pV^\flat$ of $\pV$}\label{sec:Cflat}

By a generic point, denoted by \glsuseri{etapt}, of the formal completion $\fX_{\clabel,n}$, we mean a generic point of $\Spec(\Gamma(U,\cO_U))$ with $U$ an affine open subscheme of $\fX_{\clabel,n}$. Let $(\cG,\lambda,i,\alpha^p,\alpha_p)$ be the pullback to $\Spec(\Gamma(U,\cO_U))$ of the family of semi-abelian varieties over $\Ig^\tor_n$, which is a degenerating family obtained by Mumford's construction. Let
\[ 
   0\lra\gls{cT}\lra \gls{Gray}\lra\cA\ra 0
\]
be the Raynaud extension of $\cG$. Then, for any positive integer $N$, $\cT[N]$ (resp. $\cG^\natural[N]$) is canonically isomorphic to a subgroup $\glsuseri{cGmu}\subset \cG[N]$ (resp. $\glsuserii{cGmu} \subset \cG[N]$), and $(\cG[N])^\mu\subset (\cG[N])^\rmf$. The quotient $\glsuseriii{cGmu}=(\cG[N])^\rmf/(\cG[N])^\mu$ is canonically isomorphic to $\cA[N]$. (See also \cite[\S3.4.2]{Kuga}.)

Let $\glsuserii{etapt}$ be a geometric point over $\eta$. The ordinary level structure $\alpha_{p,\ol{\eta}}$ of $\cG_{\ol{\eta}}$ gives rise to a filtration 
\[
{\tt F}^+_{\bar{\eta}}\colon 0\subset {\tt F}^+_{2,\bar{\eta}}\subset {\tt F}^+_{3,\ol{\eta}}\subset \cG_{\bar{\eta}}[p^n]^{\mult +}
\] 
(refer to text around \eqref{eq:F}). We define a subset $C(K^p_fK^1_{p,n})^\flat_\ord\subset C(K^p_fK^1_{p,n})_\ord$ in terms of the relative positions of ${\tt F}^+_{2,\ol{\eta}}$ and $(\cG_{\bar{\eta}}[p])^\mu$.

\begin{defn}
Define the subset $C(K^p_fK^1_{p,n})^\flat_\ord$ of $C(K^p_fK^1_{p,n})_\ord$ consisting of $\clabel$ such that 
\[  
  {\tt F}^+_{2,\bar{\eta}}\cap (\cG_{\bar{\eta}}[p])^\mu=0
\] 
for a geometric point $\bar{\eta}$ over a generic point $\eta$ of $\fX_{\clabel,n}$. 
\end{defn}

One can check that the preceding definition does not depend on the choice of $\bar{\eta}$.

\begin{prop}
We have
\begin{equation}\label{eq:flatC-1}
    C(K^p_fK^1_{p,n})^\flat_\ord= \im\left( G(\bA^p_{\bQ,f}) P^\flat_{\tt D}(\bQ_p)\lra\GL(X\otimes\bQ)\ltimes P'(\bA_{\bQ,f})\backslash G(\bA_{\bQ,f})/K^p_fK^1_{p,n}\right),
\end{equation}
which is in a natural bijection with
\begin{equation}\label{eq:flatC-2}
   C(K^p_f)\times \left(U_{\cK,K^p_f}\ltimes P'(\bZ_p)\backslash P^\flat_{\tt D}(\bZ_p)/K^1_{p,n}\cap P_{\tt D}(\bZ_p)\right),
\end{equation}
where
\begin{align*}
   \gls{PflatD}&=\left\{g\in P_{\tt D}(\bQ_p)\,\middle|\, g^+=\begin{psm}a_{11}&a_{12}&a_{13}&a_{14}\\ a_{21}&a_{22}&a_{23}&a_{24}\\ a_{31}&a_{32}&a_{33}&a_{34}\\ &&&a_{44}\end{psm} \text{ with }\begin{psm}a_{21}&a_{22}\\a_{31}&a_{32}\end{psm}^{-1}\begin{psm}a_{23}\\a_{33}\end{psm}\in \bZ^2_p \right\}
\end{align*}
and
\begin{align*}
   P^{\flat}_{\tt D}(\bZ_p)&=\left\{\clabel_p\in P_{\tt D}(\bZ_p)\,\middle|\, \clabel^+_p=\begin{psm}a_{11}&a_{12}&a_{13}&a_{14}\\ a_{21}&a_{22}&a_{23}&a_{24}\\ a_{31}&a_{32}&a_{33}&a_{34}\\ &&&a_{44}\end{psm} \text{ with }\begin{psm}a_{21}&a_{22}\\a_{31}&a_{32}\end{psm}\in \GL_2(\bZ_p) \right\}\\
  &=P'(\bZ_p) \begin{psm}&&1\\1\\&1\\&&&1\end{psm}\left(K^0_{p,n} \cap P_{\tt D}(\bZ_p)\right).
\end{align*}
\end{prop}

\begin{proof}
Given $\clabel_p\in P_{\tt D}(\bZ_p)$ with
\begin{align*}
   (\clabel^{-1}_p)^+&=\begin{psm}c_{11}&c_{12}&c_{13}&\ast\\c_{21}&c_{22}&c_{23}&\ast\\ c_{31}&c_{32}&c_{33}&\ast\\&&&\ast\end{psm}, 
   &(\clabel^{-1}_p)^-&=\begin{psm}d_{11}&\ast&\ast&\ast\\ &\ast&\ast&\ast\\&\ast&\ast&\ast\\&\ast&\ast&\ast\end{psm},
\end{align*}
the corresponding ${\tt F}_{2,\ol{\eta}}$ is spanned by the image of $x^+_1,w^+_1$ under $\alpha_{p,\bar{\eta}}$, and $\cT_{\bar{\eta}}[p]$ is spanned by the image under $\alpha_{p,\ol{\eta}}$ of 
\begin{equation}\label{eq:Tspan}
   ({\tt x}^+_1,{\tt w}^+_1,{\tt w}^+_2;{\tt x}^-_1)\begin{pmatrix}c_{11}&c_{12}&c_{13}\\c_{21}&c_{22}&c_{23}\\ c_{31}&c_{32}&c_{33}\\&&&d_{11}\end{pmatrix}\begin{pmatrix}1&1\\0&0\\0&0\\0&1\end{pmatrix},
\end{equation}
Therefore, the condition ${\tt F}^+_{2,\ol{\eta}}\cap (\cG_{\bar{\eta}}[p])^\mu=0$ is equivalent to the condition that $c_{11}\xx^+_1+c_{21}\ww^+_1+c_{31}\ww^+_2$ does not belong to $\mr{Span}_{\bF_p}\{\xx^+_1,\ww^+_1\}$, which is equivalent to $c_{31}\in \bZ^\times_p$. The condition $c_{31}\in\bZ^\times_p$ for $(\clabel^{-1}_p)^+$ is equivalent to the condition $\begin{psm}a_{21}&a_{22}\\a_{31}&a_{32}\end{psm}\in \GL_2(\bZ_p)$ for $\clabel^+_p$. This shows that $C(K^p_fK^1_{p,n})^\flat_\ord$ equals \eqref{eq:flatC-2}. It is also easy to see that there is a natural bijection between \eqref{eq:flatC-1} and \eqref{eq:flatC-2}.
\end{proof}


\begin{defn} Let
\[
   \gls{Vflatnm}=\left\{f\in V_{n,m}\;\colon\; f\vert_{\sZ_{\clabel,n}}= 0\quad \text{for all $\clabel\in C(K^p_fK^1_{p,n})_\ord- C(K^p_fK^1_{p,n})^\flat_\ord$}\right\},
\]
{\it i.e.} the space of forms in $V_{n,m}$ vanishing along all the boundary strata of $\Ig_n^{\tor}$ indexed by ordinary cusp labels outside  $C(K^p_fK^1_{p,n})^\flat_\ord$, and
\[
   \gls{Vflat}=\varinjlim_m\varinjlim_n  V^\flat_{n,m}.
\] 
\end{defn}

\subsection{The exact sequence for $V^\flat_{n,m}$} 
When $\clabel\in C(K^p_fK^1_{p,n})^\flat_\ord$, the level $K'_{f,\clabel,n}$ is independent of $n$, and we denote it by $K'_{f,\clabel}$ with
\begin{equation}\label{eq:K''}
\begin{aligned}
   K^{\prime p}_{f,g}&= \im \left(P(\bA^p_{\bQ,f})\cap \clabel K^p_f\clabel^{-1}\ra G'(\bA^p_{\bQ,f})\right),\\
   \gls{K'pg}&=\left\{h\in G'(\bZ_p):h^+\equiv \begin{pmatrix}\ast &\ast\\ &\ast\end{pmatrix}\mod p\right\}.
\end{aligned} 
\end{equation}
Let $\gls{M00tor}$ be the space of classical automorphic forms on $G'$ valued in $\bZ/p^m\bZ$ of weight $(0,0)$ and level $K^{\prime }_{f,\clabel}$ (which are functions valued in $\bZ/p^m\bZ$ on the finite set $G'(\bQ)\backslash G'(\bA_{\bQ,f})/K'_{f,\clabel}$), and  $M_{(0,0)}(K'_{f,\clabel};\bQ_p/\bZ_p)=\varinjlim\limits_mM_{(0,0)}(K'_{f,\clabel};\bZ/p^m\bZ)$.

\begin{prop}\label{prop:exV}
We have the exact sequences
\begin{equation}\label{eq:exVnm}
   0\lra V^0_{n,m} \lra V^\flat_{n,m}\stackrel{\oplus\Phi_{\clabel}}{\lra} \bigoplus_{\clabel\in C(K^p_f)}  M_{(0,0)}(K'_{f,\clabel};\bZ/p^m\bZ)\otimes\bZ_p\llb T_\so(\bZ_p)/U_{\cK,K^p_f}\rrb \lra 0
\end{equation}
and
\begin{equation}\label{eq:exV}
   0\lra \pV^0 \lra \pV^\flat\stackrel{\oplus\Phi_{\clabel}}{\lra} \bigoplus_{\clabel\in C(K^p_f)}  M_{(0,0)}(K'_{f,\clabel};\bQ_p/\bZ_p)\otimes\bZ_p\llb T_\so(\bZ_p)/U_{\cK,K^p_f}\rrb \lra 0,
\end{equation}
where $\gls{Phig}$ is obtained by restriction to the stratum $\sZ_{\clabel,n}$ (whose global sections are the same as the global sections of $\sY_{\clabel,n}$, which can be identified with the $0$-dimensional Shimura variety $\sS_{G',K'_{f,\clabel}}$), and is sometimes called a Siegel operator.
\end{prop}

\begin{proof}
Let $\pi_n:\Ig^\tor_{n}\ra \Ig^{\min}_{n}$ be the natural map. By \cite[Lem.~8.2.2.10]{Kuga}, we have
\begin{align*}
\left(\pi_{n,*}\cO_{\Ig^\tor_{n,m}}/\pi_{n,*}\cI_{\Ig^\tor_{n,m}}\right)^\wedge_{\sY_{\clabel,n}}\simeq  h_*\cO_{\sC^\ord_{\clabel,n}}\otimes\bZ/p^m\bZ =  \cO_{\sS_{G',K'_{f,\clabel}}}\otimes\bZ/p^m\bZ.
\end{align*}
(Here $h$ is the map in the diagram \eqref{eq:Xi}.) Taking global sections over the affine scheme $\Ig^{\min}_{n,m}$,  we get
\begin{equation}\label{eq:V/V0}
  V_{n,m}/V^0_{n,m}=\bigoplus_{\clabel\in C(K^p_fK^1_{p,n})_\ord} M_{(0,0)}(K'_{f,\clabel,n};\bZ/p^m\bZ).
\end{equation}
By definition, an element in $V_{n,m}$ belongs to $V^\flat_{n,m}$ if and only if it vanishes along the strata associated to the cusp labels outside $C(K^p_fK^1_{p,n})^\flat_\ord$, so 
\begin{equation}\label{eq:Vquot}
   V^\flat_{n,m}/V^0_{n,m}=\bigoplus_{\clabel\in C(K^p_fK^1_{p,n})^\flat_\ord} M_{(0,0)}(K'_{f,\clabel};\bZ/p^m\bZ).
\end{equation}

The action of $T_\so(\bZ_p)$ on $V_{n,m}$ preserves $V^0_{n,m}$, so it descends to $V_{n,m}/V^0_{n,m}$. This action permutes the direct summands of the right-hand side of \eqref{eq:V/V0}. More precisely, the corresponding action of $(a_1,a_2)\in T_\so(\bZ_p)$ on the index set
\[  
    C(K^p_fK^1_{p,n})_\ord=C(K^p_f)\times\left(U_{\cK,K^p_f}\ltimes P'(\bZ_p)\backslash P_{\tt D}(\bZ_p)/K^1_p\cap P_{\tt D}(\bZ_p)\right)
\]
sends $\clabel=\clabel^p\clabel_p$ to $\clabel^p\clabel_p t_{a_1,a_2}$ where $t_{a_1,a_2}\in G(\bZ_p)$ with $t^+_{a_1,a_2}=\begin{psm}1\\&1\\&&a_1\\&&&a^{-1}_2\end{psm}$. It is easy to see that $T_\so(\bZ_p)$-action on $C(K^p_fK^1_{p,n})_\ord$ preserves $C(K^p_fK^1_{p,n})^\flat_\ord$, so the $T_\so(\bZ_p)$ on $V_{n,m}$ preserves $V^\flat_{n,m}$ and induces an action on $V^\flat_{n,m}/V^0_{n,m}$. 

For given $\clabel\in C(K^p_f)$, set
\begin{equation}
\begin{aligned}
   C(K^p_f K^1_{p,n})_{\ord,\clabel}&=\{\text{elements in $C(K^p_f K^1_{p,n})_\ord$ whose projection to $C(K^p_f)$ is $\clabel$}\},\\
   C(K^p_f K^1_{p,n})^\flat_{\ord,\clabel}&=\{\text{elements in $C(K^p_f K^1_{p,n})^\flat_{\ord}$ whose projection to $C(K^p_f)$ is $\clabel$}\}.
\end{aligned}
\end{equation}
The $\bZ_p\llb T_\so(\bZ_p)\rrb$-module structures of $V_{n,m}/V^0_{n,m}$ (resp. $V^\flat_{n,m}/V^0_{n,m}$) are determined by the  $\bZ_p\llb T_\so(\bZ_p)\rrb$-actions on  $C(K^p_fK^1_{p,n})_{\ord,\clabel}$ (resp.  $C(K^p_fK^1_{p,n})^\flat_{\ord,\clabel}$). The $T_\so(\bZ_p)$-action on $C(K^p_fK^1_{p,n})_{\ord,\clabel}$ has too many orbits, and the number of orbits grows with $n$, so $V_{n,m}/V^0_{n,m}$ is not very nice as a $\bZ_p\llb T_\so(\bZ_p)\rrb$-module. In contrast, the $T_\so(\bZ_p)$-action on $C(K^p_fK^1_{p,n})^\flat_{\ord,\clabel}$ is transitive, factoring through $\bZ_p\llb T_\so(\bZ_p/p^n\bZ)/U_{\cK,K^p_f}\rrb$, and the action of this quotient is free. Therefore, by fixing an element in $P^\flat_{\tt D}(\bZ_p)$, which we will always choose to be $\begin{psm}&&1\\1\\&1\\&&&1\end{psm}$, we get the isomorphism
\begin{equation}\label{eq:w-isom}
   \text{RHS of (\ref{eq:Vquot})}\cong M_{(0,0)}(K'_{f,\clabel};\bZ/p^m\bZ)\otimes\bZ_p\llb T_\so(\bZ_p)/U_{\cK,K^p_f}\rrb,
\end{equation}
from which the exact sequence \eqref{eq:exVnm} follows. The exact sequence \eqref{eq:exV} follows from \eqref{eq:exVnm} by taking the direct limit.
\end{proof}

\subsection{Three propositions on the $\bU_p$-action on $\pV$ and $\pV^\flat$.}\label{sec:three-prop}
We introduce some setting for the proof of the following propositions. Given $\clabel\in C(K^p_fK^1_{p,n})_\ord$, recall that $\eta$ denotes a generic point of $\fX_{\clabel,n,m}$, and $\bar{\eta}$ denotes a geometric point over $\eta$. Let $(\cG_{\bar{\eta}},i_{\bar{\eta}},\lambda_{\bar{\eta}},\alpha^p_{\bar{\eta}},\alpha_{p,\bar{\eta}})$ be the pullback to $\bar{\eta}$ of the family of semi-abelian schemes over $\Ig^\tor_{n,m}$. Recall that certain Lagrange subgroups  $\Lsub\subset\cG_{\bar{\eta}}[p^2]$ are used in the construction of $\sC^{1\bullet}ar_{j,n,m}$ for defining $U^\bullet_{p,j}$ in \S\ref{sec:Up-def}. (Here $(\bullet,j)=(+,2),\,(+,3)$ or $(-,1)$.) Let 
\begin{equation}\label{eq:tuple}
(\cG'_{\bar{\eta}}=\cG_{\bar{\eta}}/C,i'_{\bar{\eta}},\lambda'_{\bar{\eta}},(\alpha^p_{\bar{\eta}})',\alpha'_{p,\bar{\eta}})
\end{equation} 
be the tuple defined as in the construction of $\sC^{\bullet}_{j,n,m}$. If $(e^\pm_1,e^\pm_2,e^\pm_3,e^\pm_4)$ is a basis of $\cG_{\bar{\eta}}[p^n]$ compatible with $\alpha_{p,\bar{\eta}}$, then $\Lsub$ is spanned by $(e^\pm_1,e^\pm_2,e^\pm_3,e^\pm_4)\cdot p^n \gamma^\pm_{\Lsub,p}$ with $\gamma_{\Lsub,p}$ given as in equations  \eqref{eq:gamma+2}, \eqref{eq:gamma+3}, and \eqref{eq:gamma-}. From Mumford's construction of the degenerating family over $\fX_{\clabel,n,m}$, we see that the tuple (\ref{eq:tuple}),
corresponds to a geometric point $\bar{\eta}'$ over a generic point $\eta'$ of $\fX_{\clabel',n,m}$, where $\clabel'=\clabel\cdot \gamma_{\Lsub,p}$.

\begin{prop}\label{prop:closeUp}
The space $\pV^\flat$ is preserved by the $\bU_p$-operators.
\end{prop}

\begin{proof}
We shall show that $\clabel'=\clabel\cdot\gamma_{\Lsub,p}\in C(K^p_fK^1_{p,n})^\flat_\ord$ only if $\clabel\in C(K^p_fK^1_{p,n})^\flat_\ord$.   By the definition of the $\bU_p$-operators, this will imply that if $f\in \pV$ vanishes along the strata outside  $C(K^p_fK^1_{p,n})^\flat_\ord$, then so does $U^\bullet_{p,j}(f)$. We use the description \eqref{eq:flatC-1} for $C(K^p_fK^1_{p,n})^\flat_\ord$. Consider the case $(\bullet,j)=(+,2)$. 
Write $\clabel^+_p=\begin{psm}a_{11}&a_{12}&a_{13}&a_{14}\\ a_{21}&a_{22}&a_{23}&a_{24}\\ a_{31}&a_{32}&a_{33}&a_{34}\\ &&&a_{44}\end{psm}$ and and $\clabel^{\prime+}_p=\clabel^+_p \cdot \gamma^+_{C,p}= \begin{psm}a'_{11}&a'_{12}&a'_{13}&a'_{14}\\ a'_{21}&a'_{22}&a'_{23}&a'_{24}\\ a'_{31}&a'_{32}&a'_{33}&a'_{34}\\ &&&a'_{44}\end{psm}$ with $\gamma^+_{C,p}$ as in (\ref{eq:gamma+2}). We have
\begin{equation}\label{eq:g'}
\begin{pmatrix}a'_{11}&a'_{12}&a'_{13}&a'_{14}\\ a'_{21}&a'_{22}&a'_{23}&a'_{24}\\ a'_{31}&a'_{32}&a'_{33}&a'_{34}\\ &&&a'_{44}\end{pmatrix}
= \begin{pmatrix}a_{11}&a_{12}&\frac{a_{13}+u_1a_{11}+u_2a_{12}}{p}&\ast\\ a_{21}&a_{22}&\frac{a_{23}+u_1a_{21}+u_2a_{22}}{p}&\ast\\ a_{31}&a_{32}&\frac{a_{33}+u_1a_{31}+u_2a_{32}}{p}&\ast\\ &&&\frac{a_{44}}{p}\end{pmatrix}
\end{equation}  
and
\begin{equation}\label{eq:a'}
   \begin{pmatrix}a'_{21}&a'_{22}\\a'_{31}&a'_{32}\end{pmatrix}^{-1}\begin{pmatrix}a'_{23}\\a'_{33}\end{pmatrix} =p^{-1}\left[\begin{pmatrix}a_{21}&a_{22}\\a_{31}&a_{32}\end{pmatrix}^{-1}\begin{pmatrix}a_{23}\\a_{33}\end{pmatrix}+\begin{pmatrix}u_1\\u_2\end{pmatrix}\right].
\end{equation}   
Since $u_1,u_2\in\bZ_p$, the right hand side can be integral only if $\begin{psm}a_{21}&a_{22}\\a_{31}&a_{32}\end{psm}^{-1}\begin{psm}a_{23}\\a_{33}\end{psm}\in \bZ^2_p$. Therefore, $\clabel'=\clabel\cdot\gamma_{\Lsub,p}\in  C(K^p_fK^1_{p,n})^\flat_\ord$ only if $\clabel \in C(K^p_fK^1_{p,n})^\flat_\ord$. The argument for the other two cases is similar. \end{proof}

\begin{prop}\label{prop:Up-equiv}
There exists a positive integer $N$ such that the exact sequences in Proposition~\ref{prop:exV} are $(U^{\bullet}_{p,j})^N$-equivariant for $(\bullet,j)=(+,2),(+,3),(-,1)$ in the sense that for all $\clabel\in C(K^p_fK^1_{p,n})^\flat_\ord$ and $f\in\pV^\flat$, we have 
\[
   \Phi_{\clabel}\left((U^\bullet_{p,j})^N f\right)= \Phi_{\clabel}(f).
\] 
\end{prop}
\begin{proof}
Take $\clabel\in C(K^p_fK^1_{p,n})^\flat_\ord$ and we use the setting described at the beginning of this subsection. In order to look at the image of $U^\bullet_{p,j}f$ under the Siegel operator $\Phi_\clabel$, we consider the abelian quotients $\cA_{\bar{\eta}}$ and $\cA_{\bar{\eta}'}$ of the Raynaud extensions
\begin{align*}
   &0\lra\cT_{\bar{\eta}}\lra \cG^\natural_{\bar{\eta}}\lra\cA_{\bar{\eta}}\lra 0,
   &0\lra\cT_{\bar{\eta}'}\lra \cG^\natural_{\bar{\eta}'}\lra\cA_{\bar{\eta}'}\lra 0.
\end{align*}
Let 
\begin{equation}\label{eq:Cs}
\begin{aligned}
  \glsuseri{Cab}&=\big(C+(\cG_{\bar{\eta}}[p^n])^\mu\big)\cap (\cG_{\bar{\eta}}[p^n])^\rmf,\\
  \glsuserii{Cab}&=C^{\rmf}/(\cG_{\bar{\eta}}[p^n])^\mu.
\end{aligned}
\end{equation}
(See the beginning of \S\ref{sec:Cflat} for the definition of $(\cG_{\bar{\eta}}[p^n])^\mu$ and $(\cG_{\bar{\eta}}[p^n])^\rmf$.) Then we have
\[
    \cA_{\bar{\eta}'}=\cA_{\bar{\eta}}/C^{\mr{ab}}.
\]
We use the description \eqref{eq:ordC-2} and \eqref{eq:flatC-2} for $C(K^p_fK^1_{p,n})_\ord$ and $C(K^p_fK^1_{p,n})^\flat_\ord$. Write $\clabel=\clabel^p\clabel_p$ with $\clabel_p\in P_{\tt D}(\bZ_p)$ and 
\begin{equation}\label{eq:cij}
\begin{aligned}
   (\clabel^{-1}_p)^+&=\begin{psm}c_{11}&c_{12}&c_{13}&\ast\\c_{21}&c_{22}&c_{23}&\ast\\ c_{31}&c_{32}&c_{33}&\ast\\&&&\ast\end{psm}, 
   &(\clabel^{-1}_p)^-&=\begin{psm}d_{11}&\ast&\ast&\ast\\ &\ast&\ast&\ast\\&\ast&\ast&\ast\\&\ast&\ast&\ast\end{psm}.
\end{aligned}
\end{equation} 
The condition $\clabel\in C(K^p_fK^1_{p,n})^\flat_\ord$ implies that  $c^{-1}_{31}c_{11},c^{-1}_{31}c_{21}\in \bZ_p$. Thus we see that $(\cG_{\ol{\eta}}[p^n])^\mu$ is spanned by 
\begin{equation}\label{eq:Gmu}
  c^{-1}_{31}c_{11} e^+_1+c^{-1}_{31}c_{21}e^+_2+e^+_3,\quad e^-_1
\end{equation}
(\emph{cf.} \eqref{eq:Tspan}), and $(\cG_{\ol{\eta}}[p^n])^\rmf$ is spanned by
\begin{equation}\label{eq:Gmu-basis}
   e^+_1,\quad e^+_2,\quad e^+_3,\quad e^-_1,\quad 
   (e^-_2,e^-_3,e^-_4)\begin{pmatrix}-\iota_p(\bar{\zeta}_0)\\ &1\end{pmatrix}\begin{pmatrix}0&1\\c^{-1}_{31}c_{11}&c^{-1}_{31}c_{21}\\1 &0\end{pmatrix}.
\end{equation}

We first consider the operator $U^+_{p,3}$. By a direct computation, we see that for $\clabel\in C(K^p_fK^1_{p,n})^\flat_\ord$, the $p$-part of $\clabel'\in C(K^p_fK^1_{p,n})_\ord$ is
\begin{equation}\label{eq:g+p}
   \clabel^{\prime+}_p=\clabel^+_p\cdot\gamma^+_{\Lsub,p}=\begin{psm}1&&&\ast\\ &1&&\ast\\ &&1&\ast\\&&&p^{-1} \end{psm}\clabel^+_p,
\end{equation}
where $\ast\in \bQ_p$. This shows that the cusp label $\clabel'$ is independent of $\Lsub$ and belongs to $C(K^p_fK^1_{p,n})^\flat_\ord$. For given $\Lsub$ spanned by $(e^\pm_1,e^\pm_2,e^\pm_3,e^\pm_4)\cdot p^{n-2}\gamma^\pm_{\Lsub,p}$ with $\gamma_{\Lsub,p}^\pm$ as in (\ref{eq:gamma+3}),
from the definition \eqref{eq:Cs} and the basis \eqref{eq:Gmu-basis} of $(\cG_{\ol{\eta}}[p^n])^\rmf$, we see that  $C^\rmf$ is spanned by
\[
   c^{-1}_{31}c_{11} e^+_1+c^{-1}_{31}c_{21}e^+_2+e^+_3,\quad e^-_1, \quad
   (e^-_2,e^-_3,e^-_4)\cdot p^{n-2}\begin{pmatrix}-\iota_p(\bar{\zeta}_0)\\ &1\end{pmatrix}\begin{pmatrix}0&1\\c^{-1}_{31}c_{11}&c^{-1}_{31}c_{21}\\1 &0\end{pmatrix},
\]
its quotient $\Lsub^{\mr{ab}}$ is independent of the choice of $\Lsub$, and is spanned by $(\epsilon^-_1,\epsilon^-_2)\cdot p^{n-2}$, with $(\epsilon^-_1,\epsilon^-_2)$ the image of $(e^-_2,e^-_3,e^-_4)\begin{pmatrix}-\iota_p(\bar{\zeta}_0)\\ &1\end{pmatrix}\begin{psm}0&1\\c^{-1}_{31}c_{11}&c^{-1}_{31}c_{21}\\1 &0\end{psm}$ under the quotient map $\cG^\natural_{\bar{\eta}}\ra\cA_{\bar{\eta}}$ which is a basis of $\cA_{\bar{\eta}}[p^n]^-$. (Here we identify $(\cG_{\bar{\eta}}[p^n])^\rmf$ with $\cG^\natural_{\bar{\eta}}[p^n]$ by the canonical isomorphism between them.) Thus we get
\begin{equation}  \label{eq:Ap2-}
   \cA_{\bar{\eta}'}=\cA_{\bar{\eta}}/\cA_{\bar{\eta}}[p^2]^-.
\end{equation}
Denote by $\left<p\right>^-$ the operator on $M_{(0,0)}(K'_{f,\clabel};\bZ/p^m\bZ)$ induced by taking the quotient by $\cA[p^2]^-$. Combining \eqref{eq:g+p} and \eqref{eq:Ap2-}, we see that the trace  cancels the normalizer in \eqref{eq:open} due to the independence of the choice of $\Lsub$, and 
\[
    \Phi_\clabel(U^+_{p,3}f) =(\left<p\right>^-)^2 \Phi_{\clabel'}(f).
\]
With $K^p_f$ fixed, there exists $N_1\geq 1$, such that $\begin{psm} 1&&&\ast\\&1&&\ast\\&&1&\ast\\&&&p^{-N_1}\end{psm}_p\clabel$ and $\clabel$ represent the same element in the double coset $C(K^p_fK^1_{p,n})^\flat_\ord$, and $(\left<p\right>^-)^{2N_1}=1$ on $M_{(0,0)}(K'_{f,\clabel};\bZ/p^m\bZ)$. For such $N_1$, we deduce that 
\[
   \Phi_\clabel\left((U^+_{p,3})^{N_1}f\right) = \Phi_{g}(f).
\]
for all $f\in \pV^\flat$ and $g\in C(K^p_fK^1_{p,n})^\flat_\ord$. A similar calculation shows that that there exists $N_2\geq 1$ such that
\[
\Phi_\clabel\left((U^-_{p,1})^{N_2}f\right) = \Phi_{\clabel}(f).
\]
for all $f\in \pV^\flat$ and $\clabel\in C(K^p_fK^1_{p,n})^\flat_\ord$. 

It remains to consider  $U^+_{p,2}$. 
Write $\clabel^+_p=\begin{psm}a_{11}&a_{12}&a_{13}&\ast\\a_{21}&a_{22}&a_{32}&\ast\\ a_{31}&a_{32}&a_{33}&\ast\\ &&&\ast\end{psm}$ and $\clabel^{\prime+}_p=\clabel^+_p \cdot \gamma^+_{C,p}= \begin{psm}a'_{11}&a'_{12}&a'_{13}&a'_{14}\\ a'_{21}&a'_{22}&a'_{23}&a'_{24}\\ a'_{31}&a'_{32}&a'_{33}&a'_{34}\\ &&&a'_{44}\end{psm}$ with $\gamma_{C,p}^+$ as in (\ref{eq:gamma+2}). Then we have \eqref{eq:g'} and \eqref{eq:a'}, and
$\clabel'\in C(K^p_fK^1_{p,n})^\flat_\ord$ if and only if
\begin{equation}\label{eq:u1u2}
   \begin{pmatrix}u_1\\u_2\end{pmatrix}\equiv -\begin{pmatrix}a_{21}&a_{22}\\a_{31}&a_{32}\end{pmatrix}^{-1}\begin{pmatrix}a_{23}\\a_{33}\end{pmatrix}\mod p.
\end{equation}
Now suppose $\clabel'\in C(K^p_fK^1_{p,n})^\flat_\ord$, {\it i.e.} the entries $u_1,u_2$ in $\gamma^\pm_{C,p}$ (see (\ref{eq:gamma+2})) satisfy the congruence \eqref{eq:u1u2}. Then we have
$
    \clabel^{\prime+}_p\in\begin{psm}p^{-1}&\ast&\ast&\ast\\&1&&\ast\\&&1&\ast\\&&&p^{-1}\end{psm}\clabel^+_p K^1_{p,n},
$
and as an element in the double coset $C(K^p_fK^1_{p,n})^\flat_\ord$, and $\clabel'$ represents the same element as $\begin{psm}p^{-1}&&&\\&1&&\\&&1&\\&&&p^{-1}\end{psm}\clabel$ in $C(K^p_fK^1_{p,n})^\flat_\ord$. The congruence \eqref{eq:u1u2} implies that the $c_{ij}$'s in \eqref{eq:cij} satisfies
\begin{equation}\label{eq:u1u2'}
\begin{aligned}
   u_1&\equiv -c^{-1}_{31}c_{11} \mod p, &u_2&\equiv -c^{-1}_{31}c_{21} \mod p.
\end{aligned}
\end{equation}
From the definition ~\eqref{eq:Cs} and the relation \eqref{eq:u1u2'}, by using the basis \eqref{eq:Gmu-basis} of $(\cG_{\ol{\eta}}[p^n])^\rmf$ and  the basis of $\Lsub$ in \eqref{eq:gamma+2}, we see that  $\Lsub^\rmf$ is spanned by
\[
   c^{-1}_{31}c_{11} e^+_1+c^{-1}_{31}c_{21}e^+_2+e^+_3,\quad e^-_1, \quad
   (e^-_2,e^-_3,e^-_4)\cdot p^{n-2}\begin{pmatrix}-\iota_p(\bar{\zeta}_0)^{-1}\\ &1\end{pmatrix}\begin{pmatrix}0&1\\c^{-1}_{31}c_{11}&c^{-1}_{31}c_{21}\\1 &0\end{pmatrix}.
\]
Like in the case of $U^+_{j,3}$, we can deduce that  $C^{\mr{ab}}=\cA_{\bar{\eta}}[p^n]^-$, and 
\[
\cA_{\bar{\eta}'}=\cA_{\bar{\eta}}/\cA_{\bar{\eta}}[p^2]^-.
\]
The independence of $\Lsub^{\mr{ab}}$ on the $\ast$'s in $\gamma^+_{\Lsub,p}$ cancels the normalizer in \eqref{eq:open}, and we get
\[
    \Phi_\clabel(U^+_{p,2}f) =(\left<p\right>^-)^2 \Phi_{\clabel'}(f).
\]
For fixed $K^p_f$, take $N_3$ such that
$
\begin{psm}p^{-1}&&&\\&1&&\\&&1&\\&&&p^{-1}\end{psm}^{N_3}_p\clabel
$ 
and $\clabel$ represent the same element in the double coset $C(K^p_fK^1_{p,n})^\flat_\ord$ and $(\left<p\right>^-)^{2N_3}=1$ on $M_{(0,0)}(K'_{f,\clabel};\bZ/p^m\bZ)$. Then 
\[
   \Phi_\clabel\left((U^+_{p,3})^{N_3}f\right) = \Phi_{\clabel}(f).
\]
for all $f\in \pV^\flat$ and $\clabel\in C(K^p_fK^1_{p,n})^\flat_\ord$. By taking $N=N_1N_2N_3$, we conclude the proof.
\end{proof}

\begin{prop}\label{prop:Uptoflat}
$(U^+_{p,2})^mf\in V^\flat_{n,m}$ for all $f\in V_{n,m}$.
\end{prop}

\begin{proof}
We shall show that for all $\clabel\in C(K^p_fK^1_{p,n})_\ord- C(K^p_fK^1_{p,n})^\flat_\ord$ and $f\in V_{n,m}$, 
\begin{equation}\label{eq:p-div}
\Phi_\clabel\left(U^+_{p,2}f\right)\in pM_{(0,0)}(K'_{f,\clabel};\bZ/p^m\bZ);
\end{equation}
it is easy to see that the proposition follows from this. For a given $\clabel\in C(K^p_fK^1_{p,n})_\ord- C(K^p_fK^1_{p,n})^\flat_\ord$, we use the setting described at the beginning of this subsection. As in the proof of Proposition~\ref{prop:Up-equiv}, we consider $\cA_{\bar{\eta}'}$, the abelian quotient of the Raynaud extension associated to $\cG_{\bar{\eta}'}=\cG_{\bar{\eta}}/C$. We have 
\[
    \cA_{\bar{\eta}'}=\cA_{\bar{\eta}}/C^{\mr{ab}},
\]
where $\cA_{\bar{\eta}}$ is the abelian quotient of the Raynaud extension associated to $\cG_{\bar{\eta}}$, and $\Lsub^{\mr{ab}}$ is defined in \eqref{eq:Cs}. Since the $\bU_p$-operator considered here is $U^+_{p,2}$, the Lagrangian subgroup $\Lsub\subset\cG_{\bar{\eta}}[p^2]$ is spanned by $(e^\pm_1,e^\pm_2,e^\pm_3,e^\pm_4)\cdot p^n \gamma^\pm_{\Lsub,p}$ with $\gamma_{\Lsub,p}^\pm$  as in \eqref{eq:gamma+2}. We use the description \eqref{eq:ordC-2} and \eqref{eq:flatC-2} for $C(K^p_fK^1_{p,n})_\ord$ and $C(K^p_fK^1_{p,n})^\flat_\ord$.  Write $(\clabel^{-1}_p)^+=\begin{psm}c_{11}&c_{12}&c_{13}&\ast\\c_{21}&c_{22}&c_{23}&\ast\\ c_{31}&c_{32}&c_{33}&\ast\\&&&\ast\end{psm}$, $(\clabel^{-1}_p)^-=\begin{psm}d_{11}&\ast&\ast&\ast\\ &\ast&\ast&\ast\\&\ast&\ast&\ast\\&\ast&\ast&\ast\end{psm}$ (which are elements in $\GL_4(\bZ_p)$). The condition $\clabel\notin C(K^p_fK^1_{p,n})^\flat_\ord$ implies that $c_{31}\in p\bZ_p$ and $c_{11}$ or $c_{21}$ belongs to $\bZ^\times_p$. 

Suppose $c_{11}\in \bZ^\times_p$. The group $(\cG_{\bar{\eta}}[p^n])^\mu$ is spanned by 
\begin{equation}\label{eq:Gmu'}
  c_{11} e^+_1+c_{21}e^+_2+c_{31}e^+_3,\quad e^-_1,
\end{equation}
and $(\cG_{\ol{\eta}}[p^n])^\rmf$ is spanned by
\begin{equation}\label{eq:Gmu-basis'}
   e^+_1,\quad e^+_2,\quad e^+_3,\quad e^-_1,\quad 
   (e^-_2,e^-_3,e^-_4)\begin{pmatrix}-\iota^c_p(\zeta_0)^{-1}\\ &1\end{pmatrix}\begin{pmatrix}c_{11}&0\\0&c_{11}\\-c_{21}&c_{31} \end{pmatrix}.
\end{equation}
By the definition of $\Lsub^\rmf$ in \eqref{eq:Cs}, a direct computation shows that $\Lsub^\rmf$ is spanned by $(\cG_{\bar{\eta}}[p^n])^\mu$ plus
\begin{equation} \label{eq:Cabv}
   \left((c_{11}u_2-c_{21}u_1)e^+_2+c_{11} e^+_3\right)\cdot p^{n-1},
   \quad  (e^-_2,e^-_3,e^-_4)\begin{pmatrix}-\iota^c_p(\zeta_0)^{-1}\\ &1\end{pmatrix}\begin{pmatrix}c_{11}&0\\0&c_{11}\\-c_{21}&c_{31} \end{pmatrix} \begin{pmatrix}c_{11}&0\\c_{11}u_2-c_{21}u_1&p\end{pmatrix}.
\end{equation}
Hence, $\Lsub^{\mr{ab}}$ is spanned by the image of \eqref{eq:Cabv} modulo $(\cG_{\bar{\eta}}[p^n])^\mu$.
Let $u'_2=c_{11}u_2-c_{21}u_1$.  The trace in \eqref{eq:open} corresponds to a sum with $u_1,u'_2$ and the two $\ast$'s in $\gamma_{\Lsub,p}$ varying in $\bZ/p\bZ$. Since \eqref{eq:Cabv} depends only on $u'_2$, we see  that $C^{\mr{ab}}$, and hence $\cA_{\bar{\eta}'}$ depends only on $u'_2$. Its independence of the two $\ast$'s implies that the sum over the two $\ast$'s cancels the normalization factor in \eqref{eq:open}. Its independence of $u_1$ implies that the sum over $u_1$ contributes a factor $p$. Therefore, the evaluation of $\Phi_\clabel\left(U^+_{p,2}f\right)$ at $\bar{\eta}'$ is divisible by $p$. This shows the inclusion (\ref{eq:p-div}), hence the result in this case.  
The case $c_{21}\in\bZ^\times_p$ can be similarly treated.
\end{proof}

\subsection{The ordinary projection on $\pV$ and the fundamental exact sequence}
In this section, we use the exact sequence in Proposition~\ref{prop:exV} and the three propositions proved in \S\ref{sec:three-prop} to show that $\lim\limits_{r\ra\infty}(U_p)^{r!}f$ converges for all $f\in\pV$,  that $\pV^{\ast}_\so=\Hom_{\bZ_p}\left(e_\so\pV,\bQ_p/\bZ_p\right)$ is a free $\Lambda_\so$-module of finite rank, and deduce the fundamental exact sequence in part (4) of Theorem~\ref{prop:main}.

\begin{thm}[Vertical Control Theorem]
Let $U_p=U^+_{p,2}U^+_{p,3}U^-_{p,1}$. Then for each $f\in\pV$, $\lim\limits_{r\to\infty}(U_p)^{r!}f$ converges, and we can define the semi-ordinary projector as $e_\so=\lim\limits_{r\to\infty}(U_p)^{r!}$. The $\bZ_p\llb T_{\so}(\bZ_p)\rrb$-module 
$
\pV^{\ast}_\so=\Hom_{\bZ_p}\left(e_\so\pV,\bQ_p/\bZ_p\right)
$
is free of finite rank over $\Lambda_\so$.
\end{thm}
\begin{proof}
Given $f\in V_{n,m}$, we define the following finiteness 
property for $f$:
\begin{equation}
   \parbox{\dimexpr\linewidth-4.5em}{%
    \strut
    The submodule generated by $\left(U_p\right)^{r}f$, $r\geq 0$, is finitely generated over $\bZ/p^m$.
    \strut
  }\tag{F}
\end{equation}
It is easy to see that the convergence of $\lim\limits_{r\ra\infty}(U_p)^{r!}f$ follows from the property (F) for $f$. By Proposition~\ref{prop:Uptoflat}, in order to show that all $f\in V_{n,m}$ satisfy (F), it suffices to show that all $f\in V^\flat_{n,m}$ satisfy (F). Given $f\in V^\flat_{n,m}$, it follows from Proposition~\ref{prop:Up-equiv} that $f'=U^N_pf-f$ belongs to $V^0_{n,m}$. Proposition~\ref{prop:V0-converge} implies that there exists $M\geq 0$ such that 
\[
   U^{M+1}_p f'\in \mr{Span}_{\bZ}\left\{f',U_pf',U^2_pf',\dots,U^M_pf'\right\}.
\]
Then
\[
   U^{M+N+1}_p f\in \mr{Span}_{\bZ}\left\{f,U_pf,U^2_pf,\dots,U^{M+N}_pf\right\},
\]
from which it follows that for all $r\geq 0$, 
\[
   U^r_p f\in \mr{Span}_{\bZ}\left\{f,U_pf,U^2_pf,\dots,U^{M+N}_pf\right\}.
\]
Hence (F) holds for $f$. Therefore, we have proved that $\lim\limits_{r\ra\infty}(U_p)^{r!}f$ converges for all $f\in\pV$, and we can define the semi-ordinary projector $e_\so$ on $\pV$ as 
\begin{equation}\label{eq:eso-noncusp}
   e_\so=\lim\limits_{r\ra\infty}(U_p)^{r!}.
\end{equation}
By Proposition~\ref{prop:Uptoflat}, we have
\begin{equation}\label{eq:eVflat}
   e_\so\pV=e_\so\pV^\flat.
\end{equation}

Applying $e_\so$ to the Pontryagin dual of the exact sequence in Proposition~\ref{prop:exV} and using \eqref{eq:eVflat}, we obtain the exact sequence
\begin{equation}\label{eq:exactV}
   0\lra M_{(0,0)}(K'_{f,\clabel};\bZ_p)\otimes\bZ_p\llb T_\so(\bZ_p)/U_{\cK,K^p_f}\rrb \lra \pV^*_\so\lra \pV^{0,*}_\so\lra 0.
\end{equation}
We know that $M_{(0,0)}(K'_{f,\clabel};\bZ_p)$ is a free $\bZ_p$-module of finite rank, so the leftmost term in \eqref{eq:exactV} is a free $\Lambda_\so$-module of finite rank. The rightmost term  $\pV^{0,*}_\so$ is also a free $\Lambda_\so$-module of finite rank by Theorem~\ref{thm:cuspidal}. Since $\mr{Ext}^1_{\Lambda_\so}(M,N)$ vanishes for free $\Lambda_\so$-modules $M$, we deduce that $\pV^*_\so$, as a $\Lambda_\so$-module, is isomorphic to the direct sum of the terms on the two ends of \eqref{eq:exactV}, and therefore is a free $\Lambda_\so$-module of finite rank. 
\end{proof}

Part (2) of Theorem~\ref{prop:main} for non-cuspidal semi-ordianry families can be proved in the same way as the cuspidal  semi-ordianry families. The freeness of $\pV^{0,*}_\so$ implies that applying $\Hom_{\Lambda_\so}(\,\cdot\,,\Lambda_\so)$ preserves the exactness of \eqref{eq:exactV}, and part (4) of Theorem~\ref{prop:main} follows.

\subsection{The Fourier--Jacobi expansion}\label{sec:FJ-def}
We introduce the Fourier--Jacobi expansions of $p$-adic forms on $\GU(3,1)$, which will be used in \S\ref{sec:nondegFJ} for analyzing the Klingen Eisenstein family on $\GU(3,1)$ constructed in  \S\ref{sec:construct}. 

In \S\ref{sec:bd-strata}, for an ordinary cusp label $\clabel\in C(K^p_f K^1_{p,n})_\ord$, we described $\fX_{\clabel,n}$, the formal completion of $\sT^\tor_n$ along the boundary stratum $\sZ_{\clabel,n}$. From the description there, we see that 
\[
   H^0\left(\fX_{\clabel,n},\cO_{\fX_{\clabel,n}}\right)=\left(\prod_{\beta} H^0\left(\sC^\ord_{\clabel,n},\cL(\beta)\right)\right)^{\Gamma_\clabel},
\]
where $\beta$ runs over ${\bf S}^\vee_{\Phi^{(\clabel)}_{K^p_fK^1_{p,n}}}\cap {\bf P}_{\Phi^{(\clabel)}_{K^p_fK^1_{p,n}}}$, which can be identified with a subset of $\Her_1(\cK)_{\geq 0}$, and $\cL(\beta)$ is the invertible sheaf over $\sC^\ord_{\clabel,n}$ of $\beta$-homogeneous functions on $\Xi^\ord_{\clabel,n}$. Therefore, given $\clabel\in  C(K^p_f K^1_{p,n})_\ord$ and $\beta\in \Her_1(\cK)_{\geq 0}$, the restriction to $\fX_{\clabel,n}$ induces a map
\begin{equation}\label{eq:FJnm}
   V_{n,m}\lra H^0\left(\sC^\ord_{\clabel,n},\cL(\beta)\otimes\bZ/p^m\bZ\right).
\end{equation}

We consider  $\clabel=\bid_4$ and $\beta\in\Her_1(\cK)_{\geq 0}$. In this case, the $p$-component of level group for the Shimura variety $\sS_{G',K'_{f,\bid_4,n}}$  in the diagram \eqref{eq:Xi} is
\[
   K'_{p,\bid_4,n}=\left\{g\in \GU(2)(\bZ_p):g^+\equiv \begin{pmatrix}\ast&\ast\\0&1\end{pmatrix}\mod p^n\right\},
\]
(depending on $n$).
Let 
\begin{equation}\label{eq:VJ}
    \gls{VJbeta}=\varprojlim_m\varinjlim_n H^0\left(\sC^\ord_{\bid_4,n},\cL(\beta)\otimes\bZ/p^m\bZ\right).
\end{equation}
The map  gives the map of taking the $\beta$-th Fourier--Jacobi coefficient of $p$-adic forms on $\GU(3,1)$ along the boundary stratum labeled by $\bid_4$:
\begin{equation}\label{eq:FJ-V}
   \glsuseri{FJ}:V_{\GU(3,1)}\lra V^{J,\beta}_{\GU(2)}.
\end{equation}


\section{The construction of the Klingen family}\label{sec:construct}

\subsection{Some notation}\label{sec:construct-notation}
Let $\gls{Kinfty}$ be the maximal abelian pro-$p$ extension of $\cK$ unramified outside $p$ and $\gls{GammaK}=\Gal(\cK_\infty/\cK)$. Then $\Gamma_\cK\cong \bZ^2_p$. Denote by $\glsuseri{Omega}\in\bC^\times$ (resp. $\glsuserii{Omega}\in \hat{\bZ}^{\ur,\times}_p$) the complex CM period ($p$-adic CM period) with respect to the embeddings in \eqref{eq:Kembd} (cf. \cite[Section 2.8]{HsiehMC}). We also fix an isomorphism $\ol{\bQ}_p\cong \bC$ compatible with the embeddings in \eqref{eq:Kembd}. Let $\gls{L}\subset\ol{\bQ}_p$ be a sufficiently large finite extension of $\bQ_p$, and denote by $\gls{OLur}$ the ring of integers of the completion of the maximal unramified extension of $L$. 

Four our later use of theta correspondence for unitary groups, we also fix a Hecke character
\begin{align*}
   \gls{lambda}:\cK^\times\backslash\bA^\times_\cK&\lra \bC^\times, 
\end{align*} 
such that
\begin{align*}
   \lambda|_{\bA^\times_\bQ}&=\eta_{\cK/\bQ},
   &\lambda_{\infty}(z)&=\frac{z}{|z\bar{z}|^{1/2}}.
\end{align*} 

\subsection{Our setup}\label{sec:setup}

We assume the following conditions on $\gls{pi}$, an irreducible cuspidal automorphic representation of $\GL_2(\bA_\bQ)$ generated by a newform $\glsuseri{f}$ of weight $2$:
\begin{itemize}[leftmargin=2.5em]
\item for all finite places $v$ of $\bQ$, $\pi_v$ is either unramified or Steinberg or  Steinberg twisted by an unramified quadratic character of $\bQ^\times_v$,
\item $\pi_p$ is unramified,
\item there exists a prime $q$ not split in $\cK$ such that $\pi$ is ramified at $q$, and if $2$ does not split in $\cK$, then $\pi$ is ramified at $2$, 
\item $\bar{\rho}_\pi|_{\Gal(\ol{\bQ}/\cK)}$ is irreducible, (which is automatically true if $\pi$ is not ordinary at $p$ because in this case $\bar{\rho}_\pi|_{G_\cK,\fp}\cong \bar{\rho}_\pi|_{G_{\bQ,p}}$ is irreducible by \cite{Edixhoven}), where $\bar{\rho}_\pi$ denotes the residual representation of the Galois representation $\rho_\pi$.
\end{itemize}

We also fix an algebraic Hecke character $\glsuserii{xi}:\cK^\times\backslash\bA^\times_\cK\ra\bC^\times$ of $\infty$-type $\left(0,k_0\right)$ with $\gls{k0}$ an even integer.

\subsection{The weight space}\label{sec:wt}

For an algebraic Hecke character $\chi:\cK^\times\backslash\bA^\times_\cK\ra\bC^\times$ of $\infty$-type $(k_1,k_2)$, we denote by $\chi_0$ its associated unitary Hecke character, {\it i.e.}
\[
   \chi_0=\chi|\cdot|^{-\frac{k_1+k_2}{2}}_{\bA_\cK},
\]
and denote by $\chi_{p\adic}$ its associated $p$-adic character, sometimes called the \emph{$p$-adic avatar} of $\chi$, as
\begin{equation}\label{eq:chip}
\begin{aligned}
   &\chi_{p\adic}:\cK^\times\backslash \bA^\times_{\cK,f}\lra\ol{\bQ}^\times_p,
   &\chi_{p\adic}(x)=\chi(x)x^{k_1}_\fp x^{k_2}_{\bar{\fp}}.
\end{aligned}
\end{equation}
Applying this convention to $\xi$, we get $\glsuseri{xi}$ and $\glsuseriii{xi}$.

The weight space we will use for constructing the semi-ordinary Klingen Eisenstein family on $\GU(3,1)$ is $\Hom_\cont\big(\Gamma_\cK,\ol{\bQ}^\times\big)$. By identifying $\Gamma_\cK$ with a quotient of $\cK^\times\backslash\bA^\times_{\cK,f}$, the arithmetic points in the weight space are $p$-adic avatars of the algebraic  Hecke characters of $\cK^\times\backslash\bA^\times_\cK$ whose $p$-adic avatars factor through $\Gamma_\cK$. We will use $\glsuseri{tau}$ to denote such a Hecke character, and  $\glsuseriii{tau}$ to denote its $p$-adic avatar. By the above convention, $\glsuserii{tau}$ denotes the associated unitary character of $\tau$.  The interpolation points we will use are the $\tau$'s satisfying that $\xi\tau$ has $\infty$-type $(0,k)$ with $k$ an even integer $\geq 6$. (One can consider more general $\infty$-types by using Masss--Shimura differential operators. For our purpose here, only considering the case of $\infty$-type $(0,k)$ suffices.)

\subsection{The groups}\label{sec:gps}

By the assumption on $\pi$ in \S\ref{sec:setup}, we can fix a prime $\gls{q}$ such that 
\[
    \left\{\begin{array}{ll}
    \text{$q$ does not split in $\cK$ and $\pi_q$ is ramified}, &\text{if $2$ splits in $\cK$},\\
    q=2, &\text{if $2$ does not split in $\cK$}. 
    \end{array}\right.
\] 
Let $\gls{D}$ be the quaternion algebra over $\bQ$ ramified exactly at $q$ and $\infty$, and let $\pi^D$ be the Jacquet--Langlands transfer of $\pi$ to $D^\times(\bA_\bQ)$.

\vspace{.5em}

We can take a square free positive integer $\gls{fs}$, coprime to $pD_{\cK/\bQ}$ and all the primes inert in $\cK$ where $D$ splits,  such that $\mr{inv}_v(D)=(-\fs,D_{\cK/\bQ})_v$, where $(\bdot,\bdot)_v$ is the Hilbert symbol. Define $\glsuseri{zeta}=\delta\begin{pmatrix}\fs\\&1\end{pmatrix}$ with $\delta$ the totally imaginary element in $\cK$ with $\delta^2$ a $p$-adic unit as fixed in {\bf Notation}. This $\zeta_0$ is a skew--Hermitian matrix. Define the unitary group $\U(2)$ (resp. similitude unitary group $\GU(2)$) over $\bZ$ as: for all $\bZ$-algebra $R$,
\begin{align}
   \label{eq:U(2)}\U(2)(R)&=\{g\in \GL_2(R\otimes_{\bZ}\cO_\cK):g\zeta_0\ltrans{\ol{g}}=\zeta_0\}\\
   \GU(2)(R)&=\{g\in \GL_2(R\otimes_{\bZ}\cO_\cK):g\zeta_0\ltrans{\ol{g}}=\nu(g)\zeta_0,\,\nu(g)\in R^\times\}.
\end{align}

We have 
\[
    D=\{g\in M_2(\cK):g\zeta_0 \ltrans{\bar{g}}=\det(g)\zeta_0\}=\left\{\begin{pmatrix}a&-\fs\bar{b}\\b&\ol{a}\end{pmatrix}:a,b\in\cK\right\},
\]
and we view both $D^\times$ and $\GU(2)$ as subgroups of $\mr{Res}_{\cK/\bQ}\GL_2(\cK)$. Then the homomorphism 
\begin{align*}
   \cK^\times\times D^\times&\lra \GU(2)(\bQ), &(a,g)&\lra ag
\end{align*} 
induces an isomorphism
\begin{equation}\label{eq:DGU}
    \GU(2)\simeq (\mr{Res}_{\cK/\bQ}\bG_m\times D^\times)/\{(a,a^{-1}\bid_2):a\in\bG_m\}.
\end{equation}

Given an automorphic form on $\phi$ on $D^\times$ with central character $\chi_1$, by picking an extension $\chi$ of $\chi_1$ to $\cK^\times\backslash\bA^\times_\cK$, one obtains an automorphic form $\phi^{\GU(2)}$ on $\GU(2)$ by \eqref{eq:DGU}.  

We denote by $\gls{piD}$ the Jacquet--Langlands transfer of $\pi$ to $D^\times$, and by $\glsuserii{f}$ the unique automorphic form on $D^\times$ which is a newform for the action of $D^\times(\bQ_v)\simeq\GL_2(\bQ_v)$ for all $v\neq q,\infty$ and takes value $1$ at $\bid_2$. By our assumption, $\pi^D$ has trivial central character. We can extend $f^D$ to a form $\glsuseriii{f}$ via \eqref{eq:DGU} and the trivial character. We denote this form by $f^{\GU(2)}$, and denote by $\gls{piU}$ the automorphic representation of $\GU(2)(\bA_\bQ)$ generated by $f^{\GU(2)}$.

\vspace{.5em}

Define the unitary group $\U(3,1)$ (resp. similitude unitary group $\GU(3,1)$) over $\bZ$ as
\begin{align*}
    \U(3,1)(R)&=\left\{g\in \GL_4(R\otimes_{\bZ}\cO_\cK):g\begin{pmatrix}&&1\\&\zeta_0\\-1\end{pmatrix}\ltrans{\ol{g}}=\begin{pmatrix}&&1\\&\zeta_0\\-1\end{pmatrix}\right\},\\
    \GU(3,1)(R)&=\left\{g\in \GL_4(R\otimes_{\bZ}\cO_\cK):g\begin{pmatrix}&&1\\&\zeta_0\\-1\end{pmatrix}\ltrans{\ol{g}}=\nu(g)\begin{pmatrix}&&1\\&\zeta_0\\-1\end{pmatrix},\,\nu(g)\in R^\times\right\},
\end{align*}
Let 
\begin{equation}\label{eq:Kling-P}
   \gls{PU}=\left\{\begin{pmatrix}x&\ast&\ast\\ &g&\ast\\&&\nu(g)\bar{x}^{-1}\end{pmatrix}\in\GU(3,1): g\in \GU(2),\,\,x\in \mr{Res}_{\cK/\bQ}\bG_m\right\},
\end{equation}
the (standard) Klingen parabolic subgroup of $\GU(3,1)$. Its Levi subgroup is 
\[
   M_{P_{\GU{3,1)}}}\simeq \mr{Res}_{\cK/\bQ}\bG_m\times \GU(2).
\] 
We consider the Klingen Eisenstein series on $\GU(3,1)$ inducing $\pi^{\GU(2)}\boxtimes\xi_0\tau_0|\cdot|^s_{\bA^\times_\cK}$ from $P_{\GU(3,1)}$.

\vspace{.5em}
Define the unitary group $\U(1)$ over as
\[
   \U(1)(R)=\left\{a\in (R\times_\bZ\cO_\cK)^\times:a\bar{a}=1\right\}.
\]

The projection $\varrho_\fp:\cK_p=\cK\otimes\bQ_p\ra\bQ_p$ induces maps
\begin{align*}
   \U(2)(\bQ_p)&\lra \GL_2(\bQ_p), 
   &\U(3,1)(\bQ_p)&\lra \GL_4(\bQ_p),
   &\U(1)(\bQ_p)&\lra \bQ^\times_p
\end{align*}
and they are all isomorphisms. We denote the inverse maps by $\gls{varrho-1}$:
\begin{equation}\label{eq:varrho-1}
\begin{aligned}
   \varrho^{-1}_\fp:\GL_2(\bQ_p)&\overset{\cong}{\lra}\U(2)(\bQ_p),
   &\varrho^{-1}_\fp:\GL_4(\bQ_p)&\overset{\cong}{\lra}\U(3,1)(\bQ_p),
   &\varrho^{-1}_\fp:\bQ_p^\times&\overset{\cong}{\lra}\U(1)(\bQ_p).
\end{aligned}
\end{equation}

\subsection{The Klingen Eisenstein series and the doubling method}

We briefly recall the definition of Klingen Eisenstein series and Garrett's (generalized) doubling method formula which expresses a Klingen Eisenstein series on $\GU(3,1)$ as an integral involving a Siegel Eisenstein series on $\GU(3,3)$.

\subsubsection{The Klingen Eisenstein series on $\GU(3,1)$}
Let $U_{P_{\GU(3,1)}}$ be the unipotent subgroup of the Klingen parabolic subgroup $P_{\GU(3,1)}$ in \eqref{eq:Kling-P}.  For a unitary character $\xi_0\tau_0:\cK^\times\backslash\bA^\times_\cK\ra\bC^\times$ and a complex number $s$, define $I_{P_{\GU(3,1)}}(s,\xi_0\tau_0)$ as the space of smooth $K$-finite functions (where $K$ is a maximal compact subgroup of $\GU(3,1)(\bA_\bQ)$)
\[
   F(s,\xi_0\tau_0):U_{P_{\GU(3,1)}}(\bA_\bQ)M_{P_{\GU{3,1)}}}(\bQ)\backslash \GU(3,1)(\bA_\bQ)\lra \bC
\]
satisfying
\begin{enumerate}[(i)]
\item $F(s,\xi_0\tau_0)\left(\begin{pmatrix}x\\&\bid_2\\&&\ol{x}^{-1}\end{pmatrix}g\right)=\xi_0\tau_0(x)|x\bar{x}|^{s+\frac{3}{2}}_{\bA_\bQ}$ for all $x\in \bA^\times_\cK$, $g\in\GU(3,1)(\bA_\bQ)$,
\item for all $g\in\GU(3,1)(\bA_\bQ)$, the function $g_1\mapsto F(s,\xi_0\tau_0)\left(\begin{pmatrix}1\\&g_1\\&&\nu(g_1)\end{pmatrix}g\right)$ is a cuspidal automorphic form on $\GU(2)(\bA_\bQ)$.
\end{enumerate}
The Klingen Eisenstein series on $\GU(3,1)$ attached to $F(s,\xi_0\tau_0)\in I_{P_{\GU(3,1)}}(s,\xi_0\tau_0)$ is defined as
\begin{align*}
   \glsuseri{EklingF}=\sum_{\gamma\in P_{\GU(3,1)}(\bQ)\backslash \GU(3,1)(\bQ)} F(s,\xi_0\tau_0)(\gamma g).
\end{align*}

If the cuspidal automorphic form in (ii) belongs to $\pi^{\GU(2)}$, the Galois representation attached to $E^\Kling(\bdot\,;F(s,\xi_0\tau_0))|_{s=\frac{k-3}{2}}$ (where we assume that $\xi\tau$ has $\infty$-type $(0,k)$) is
\[
    \xi\tau\cdot\epsilon^{-2}_{\mr{cyc}}\circ \Nm \,\oplus\, (\xi\tau)^{-c}\cdot\epsilon_{\mr{cyc}}\circ \Nm\,\oplus\, \rho_\pi|_{\Gal(\ol{\bQ}/\cK)},
\]
where the Hecke characters of $\cK^\times\backslash\bA^\times_\cK$ are viewed as characters of $\Gal(\ol{\bQ}/\cK)$. The congruences between this Klingen Eisenstein series and cuspidal forms on $\GU(3,1)$ can be used to construct elements in the Selmer group for $\rho_\pi(\epsilon^2_{\mr{cyc}})|_{\Gal(\ol{\bQ}/\cK)}\otimes\xi^{-1}\tau^{-1}$.

\subsubsection{The Siegel Eisenstein series on $\GU(3,3)$}
Let $\GU(3,3)$ be the similitude unitary group over $\bZ$ defined by
\[
    \GU(3,3)(R)=\left\{g\in\GL_6(R\otimes_{\bZ}\cO_\cK):g\begin{pmatrix}&\bid_3\\-\bid_3\end{pmatrix}\ltrans{\bar{g}}=\nu(g)\begin{pmatrix}&\bid_3\\-\bid_3\end{pmatrix},\,\nu(g)\in R^\times\right\}.
\]
The group $\GU(3,1)\times_{\bG_m} \GU(2)=\{(g_1,g_2)\in\GU(3,1)\times\GU(2):\nu(g_1)=\nu(g_2)\}$ embeds into $H$ by
\begin{equation}\label{eq:embedding}
\begin{aligned}
   \gls{imath}:\GU(3,1)\times_{\bG_m} \GU(2)&\lra \GU(3,3)\\ 
   (g_1,g_2)&\lra \cS^{-1}\begin{pmatrix}g_1\\&g_2\end{pmatrix}\cS,
\end{aligned}
\end{equation}
where $\gls{cS}=\begin{pmatrix}1&\\&\bid_2&&-\frac{\zeta_0}{2}\\&&1\\&-\bid_2&&-\frac{\zeta_0}{2}\end{pmatrix}$.

\vspace{.5em}
Let $\gls{QU}$ be the (standard) Siegel parabolic subgroup of $\GU(3,3)$. For a $\bZ$-algebra $R$,
\[
    Q_{\GU(3,3)}(R)=\left\{\begin{pmatrix}A&B\\0&D\end{pmatrix}\in \GL_6(R\otimes_{\bZ}\cO_\cK):D=\nu\ltrans{\bar{A}}^{-1},\nu\in R^\times,\,A^{-1}B\in\Her_3(R\otimes_{\bZ}\cO_\cK)\right\}.
\]
For a place $v$ of $\bQ$, a unitary character $\xi_{0,v}\tau_{0,v}$ of $\cK^\times_v$ and a complex number $s$, the degenerate principal series $I_{Q_{\GU(3,3)},v}(s,\xi_{0,v}\tau_{0,v})$ is defined as the space consisting of smooth $\GU(3,3)(\bZ_v)$-finite functions $
   f_v(s,\xi_{0,v}\tau_{0,v}):\GU(3,3)(\bQ_v)\lra \bC$ 
satisfying
\begin{align*}
    f_v(s,\xi_{0,v}\tau_{0,v})\left(\begin{pmatrix}A&B\\0&D\end{pmatrix}g\right)=\xi_{0,v} \tau_{0,v}(\det A)|\det A D^{-1}|^{s+\frac{3}{2}}_vf(g)
\end{align*}
for all $g\in \GU(3,3)(\bQ_v),\,\begin{pmatrix}A&B\\0&D\end{pmatrix}\in Q_{\GU(3,3)}(\bQ_v)$.

\vspace{.5em}

Given a unitary character $\xi_0\tau_0:\cK^\times\backslash\bA^\times_\cK\ra\bC^\times$ and a section $f(s,\xi_0\tau_{0})=\bigotimes'_v f_v(s,\xi_{0,v}\tau_{0,v})$ inside $\bigotimes'_v I_{Q_{\GU(3,3)},v}(s,\xi_{0,v}\tau_{0,v})$, the Siegel Eisenstein series on $\GU(3,3)$ attached to $f(s,\xi_{0}\tau_0)$ is defined as
\begin{align*}
    \glsuseri{ESiegf}&=\sum_{\gamma\in Q_{\GU(3,3)}(\bQ)\backslash \GU(3,3)(\bQ)} f(s,\xi_0\tau_0)(\gamma g).
\end{align*}

\subsubsection{The (generalized) doubling method formula}

\begin{thm}[\cite{GaKl}]\label{thm:doubling}
Let $\xi_0\tau_0:\cK^\times\backslash\bA^\times_\cK\ra\bC^\times$ be a unitary character. For a section $f(s,\xi_0\tau_0)\in I_{Q_{\GU(3,3)}}(s,\xi_0\tau_0)$ and  a cuspidal automorphic form $\varphi$ on $\GU(2)(\bA_\bQ)$,
\begin{align*}
   \int_{\U(2)(\bQ)\backslash\U(2)(\bA_\bQ)}E^\Sieg\big(\imath\left(g,g_1t\right);f(s,\xi_0\tau_0)\big)\varphi\left(g_1t\right)(\xi_0\tau_0)^{-1}(\det g_1t)\,dg_1
   =E^\Kling\big(g;F(f(s,\xi_0\tau_0),\varphi)\big),
\end{align*}
where $t$ is an element of $\GU(3,1)(\bA_\bQ)$ with $\nu(t)=\nu(g)$ (and it is easy to see that the left hand side does not depend on the choice of $t$), and $F(f(s,\xi_0\tau_0),\varphi)$ is the section in $I_{P_{\GU(3,1)}}(s,\xi_0\tau_0)$ defined by
\begin{align*}
   \gls{Fvarphi}(g)&=\int_{\U(2)(\bA_\bQ)}f(s,\tau_0)(\imath(g,g_1t))\,\varphi(g_1t)(\xi_0\tau_0)^{-1}(\det g_1t)\,dg_1.
\end{align*}
\end{thm}

\subsection{The auxiliary data for the Klingen family}\label{sec:aux}
Let $\glsuseri{Sigma}$ be the set of finite places of $\bQ$ containing the prime $2$, the prime $q$, the finite places $v\neq p$ where $\pi$ or $\xi$ or $\cK/\bQ$ is ramified and the primes dividing $\fs$. We denote by $\glsuserii{Sigma}$ the set of primes in $\Sigma$ split in $\cK$ and $\glsuseriii{Sigma}$ the set of primes in $\Sigma$ non-split in $\cK$. 

The Klingen Eisenstein family we will construct is not of the optimal level. It depends on some auxiliary data which are chosen such that we can prove the desired properties of its non-degenerate Fourier--Jacobi coefficients. In this subsection, we fix these auxiliary data. We first fixes places $\ell,\ell'$ and positive integers $c_v$ for $v\in\Sigma_\ns\cup\{\ell'\}$. Then we choose Hecke characters $\chi_\theta,\chi_h$ satisfying a list of conditions. Then we fix  positive integers $c_v$ for $v\in\Sigma_\rms\cup\{\ell\}$. The Siegel Eisenstein family on $\GU(3,3)$ we use for constructing the Klingen Eisenstein family depends on the choice of $\ell,\ell,c_v$, $v\in\Sigma\cup\{\ell,\ell'\}$, the automorphic form on $\GU(2)$ to pair with the restriction to $\GU(3,1)\times_{\bG_m}\GU(2)$ of the Seigel Eisenstein family is chosen from the space \eqref{eq:M-varphi} defined below which depends on $\chi_{h,v}$, $v\in\Sigma_\rms\cup\{\ell\}$. The characters $\chi_\theta,\chi_h$ will show up in our analysis of non-degenerate Fourier coefficients.

\vspace{.5em}
We fix:
\begin{enumerate}
\item[--] primes $\ell,\ell'\neq 2,p$ such that $\ell$ splits in $\cK/\bQ$, $\ell'$ is inert in $\cK/\bQ$, and $\pi_\ell,\pi_{\ell'}$ are unramified,
\item[--] for each place $v\in\Sigma_{\mr{ns}}\cup\{\ell'\}$, a positive integer \[
   \gls{cv} > \max\{\ord_v(\cond(\lambda_v),\ord_v(\cond(\pi_v)),3)\}+1.
\]
\end{enumerate}

Before choosing the auxiliary Hecke characters $\chi_\theta,\chi_h$,  we need to first introduce a Proposition on certain Schwartz functions in the Schr{\"o}dinger model of Weil representations of $\U(2,1)$. Let $V,V^-$ (resp. $W_\beta$) be the two dimensional skew-Hermitian spaces (resp. one dimensional Hermitian space) over $\cK$, and $e_1,e_2\in V\otimes W_\beta$, $e^-_1,e^-_2\in V^-\otimes W_\beta$ as in \S\ref{sec:Schro}. We have $\U(V\oplus V^-)\cong \U(2,2)$ and $\U(W_\beta)\cong \U(1)$.

\begin{prop}\label{prop:nschi}
Let $v$ be a finite place of $\bQ$ non-split in $\cK/\bQ$. Denote by $\omega_{\beta,v}(\bdot\,,\bdot)$ the Weil representation of $\U(2,2)(\bQ_v)\times \U(1)(\bQ_v)$ on the Schr{\"o}dinger model $\pzS\left(\cK_v(e_1+e^-_1)\oplus \cK_v(e_2+e^-_2)\right)$ (with respect to $\lambda_{W_\beta}=\lambda$). Let
\[
    \sT:\pzS\big(\bQ_v e_1\oplus\bQ_v e_2\oplus\bQ_v e^-_1\oplus\bQ_v e^-_2\big)\lra  \pzS\big(\cK_v(e_1+e^-_1)\oplus \cK_v(e_2+e^-_2)\big)
\] be the intertwining map between Schr{\"o}dinger models.

For $\beta=1$, there exists a character $\chi_v:\U(1)(\bQ^\times_v)\ra\bC^\times$ with 
\[
    \ord_v(\cond(\chi_v))>\max\{\ord_v(\cond(\chi_v|_{\bQ^\times_v},\ord_v(\cond(\lambda_v)),\ord_v(\cond(\pi_v)),3\},
\] 
a Schwartz function $\phi_{1,v}$ on $\bQ_v e_1\oplus\bQ_v e_2$, and $u_1,\dots,u_t\in\U(1)(\bQ_v)$, $b_1,\dots,b_t\in \bC$ such that the function
\[
   y\longmapsto \sum_{j=1}^t  b_j\int_{\bQ_v e_1\oplus\bQ_v e_2}  \sT^{-1}\left(\omega_{\beta,v}\left(\begin{pmatrix}\bid_2&q^{-c_v}\cdot\bid_2\\0&\bid_2\end{pmatrix} ,u_j\right) \Phi_{0,v}\right)(x,y)\cdot \phi_{1,v}(x)\,dx
\] 
is nonzero and belongs to the $\chi_v\lambda^2_v|_{\U(1)(\bQ_v)}$-eignespace for the action of $\U(1)(\bQ_v)$, where $\Phi_{0,v}$ is the characteristic function of $\cO_{K,v}\frac{e_1+e^-_1}{2}\oplus \cO_{\cK,v}\frac{e_2+e^-_2}{2}$.
\end{prop}
\begin{proof}
\cite[Lemma 6.26]{WanU31}.
\end{proof}

Now we fix auxiliary Hecke characters 
\begin{enumerate}
\item[--]  $\gls{chi}:\cK^\times\backslash\bA^\times_\cK\ra\bC^\times$ of $\infty$-type $(0,0)$ with $\chi_h\chi^c_\theta|_{\bA^\times_\bQ}=\triv$ satisfying the following properties:
\end{enumerate}
\begin{enumerate}[(i)]
\item $\chi_\theta,\chi_h$ are unramified away from $\Sigma\cup\{p,\ell,\ell'\}$ and $\chi_h\chi^c_\theta$ is unramified at $q$.

\item $\chi_{\theta,\fp},\chi_{h,\bar{\fp}}$ are unramified, and $\chi_{h,\fp}|_{\bZ^\times_p}=\chi^{-1}_{\theta,\bar{\fp}}|_{\bZ^\times_p}=\xi_{0,\fp}|_{\bZ^\times_p}$.

\item For $v\in\Sigma_{\ns}\cup\{\ell'\}$,
\begin{align*}
    &\chi_{\theta,v}|_{\U(1)(\bQ_v)}=\text{a $\chi_v$ as in Proposition~\ref{prop:nschi}}, \\
    &\ord_v(\cond(\chi_{h,v}\chi^c_{\theta,v})),\,\ord_v(\cond((\lambda^2_v\chi_{h,v}\chi_{\theta,v}))> \ord_v(\cond(\pi_v)),\text{ if $v\neq q$},
\end{align*}

\item For $v=\fv\ol{\fv}\in\Sigma_\rms\cup\{\ell\}$, $\chi_{\theta,\fv}$ is unramified, and 
\[
   \ord_{\bar{\fv}}(\cond (\chi_{\theta,\bar{\fv}}))\geq 2\ord_v(\cond( \pi_v))+2, 
\]
and
\begin{align*}
   \ord_{\fv}\left(\cond((\chi_{h}\chi^c_{\theta})_\fv\right)
   &=\left\{\begin{array}{ll}\ord_{\fv}\left(\cond(\chi_{\theta,{\bar{\fv}} })\right)-\ord_v(\cond(\pi_v)), &\pi_v\text{ ramified,}\\
    \ord_{\fv}\left(\cond(\chi_{\theta,{\bar{\fv}} })\right)-1, &\pi_v\text{ unramified}.\end{array}\right.
\end{align*}

\item If $q=\fq^2$ is ramified in $\cK/\bQ$, 
\[
    (\chi_{h}\chi^c_{\theta})_q(\varpi_\fq)=\chi_q(q),
\]
where $\chi_q$ is the unramified quadratic character of $\cK^\times_q$ such that $\pi_q\cong \text{Steinberg}\otimes\chi_q$, and $\varpi_\fq$ is a uniformizer of $\cK_\fq$.

\item Furthermore, the value
\begin{equation}\label{eq:L-1}
   \left(\frac{\Omega_p}{\Omega_\infty}\right)^4\pi^{-3}\cdot
   \gamma_p\left(\frac{1}{2},\pi_{p}\times \left(\lambda^2\chi_h\chi_\theta\right)_{\bar{\fp}}\right)^{-1}
   L\left(\frac{1}{2},\mr{BC}(\pi)\times\lambda^2\chi_h\chi_\theta\right)
\end{equation}
is a $p$-adic unit, and
\begin{equation}\label{eq:L2}
    L^q\left(\frac{1}{2},\mr{BC}(\pi)\times\chi_h\chi^c_\theta\right)\neq 0.
\end{equation}
For some algebraic Hecke character $\tau:\cK^\times\backslash\bA^\times_\cK\ra\bC^\times$ such that $\tau_{p\adic}$ factors through $\Gamma_\cK$ and $\xi\tau$ has $\infty$-type $(0,k)$, $k\geq 2$, the values
\begin{align}
   \label{eq:L-5}&\left(\frac{\Omega_p}{\Omega_\infty}\right)^k(2\pi i)^{1-k}\Gamma(k-1)\cdot\gamma_{\bar{\fp}}\left(\frac{k-2}{2},\chi_h\chi_\theta^c\xi_0\tau_0\right)^{-1} L^{p\infty} \left(\frac{k-2}{2},\chi_h\chi_\theta^c\xi_0\tau_0\right),\\
    \label{eq:L-6}&\left(\frac{\Omega_p}{\Omega_\infty}\right)^{k-2}(2\pi i)^{2-k}\Gamma(k-2) \cdot L_\fp\left(\frac{k-2}{2},\lambda^2\chi_h\chi_\theta\xi^c_0\tau^c_0\right) L^{p\infty} \left(\frac{k-2}{2},\lambda^2\chi_h\chi_\theta\xi^c_0\tau^c_0\right),\\
   \label{eq:L-5-q}& 1-(\chi_h\chi_\theta^c\xi_0\tau_0)_q(q)q^{-\frac{k-2}{2}},
\end{align} 
are all $p$-adic units.
\end{enumerate}

The existence of the characters $\chi_\theta,\chi_h$ satisfying the above assumptions follows from the mod $p$ nonvanishing results. We can first choose $\chi_{\theta,1},\chi_{h,1}$ satisfying the conditions (1)-(5). (The inert prime $\ell'$ is introduced to ensure the existence.) Then we can apply mod $p$ nonvanishing results \cite{HsiehRankin} (for the $L$-values in \eqref{eq:L-1}) and \cite{Hsiehnonvan} (for the $L$-values in \eqref{eq:L-5}\eqref{eq:L-6}) to choose a character $\nu$ of $\ell$-power conductor such that $\chi_\theta=\chi_{\theta,1}\nu$, $\chi_h=\chi_{h,1}\nu$ satisfy (1)-(6). (The conditions in (3) on the conductors at non-split primes and  the condition (5) implies that the local root numbers are $+1$ as required for applying \cite{HsiehRankin}. Our assumption that $\xi$ has $\infty$-type $(0,k_0)$ with $k_0$ even implies that the $L$-values in \eqref{eq:L-5}\eqref{eq:L-6}) fall into the non-residually self-dual case for which \cite[Theorem B]{Hsiehnonvan} can be applied.) (The strategy explained in \S\ref{sec:strategy} explains why we need to make such a choice the auxiliary $\chi_\theta,\chi_h$.)

\vspace{.5em}

Besides the positive integers $c_v$ for $v\in\Sigma_\ns\cup\{\ell'\}$ fixed earlier in this subsection, we also fix 
\begin{enumerate}
\item[--] for each place $v\in\Sigma_\rms\cup\{\ell\}$, a positive integer
\end{enumerate}  
\begin{align}
  \label{eq:c_v} \gls{cv}>
   \max\Big\{\ord_{\bar{\fv}}(\cond(\chi_{\theta,\bar{\fv}})),\ord_{\fv}(\cond(\xi_{0,\fv})),\ord_{\bar{\fv}}(\cond(\xi_{0,\bar{\fv}})),\ord_{\bar{\fv}}(\cond(\lambda_{\bar{\fv}}))\Big\}.
\end{align} 

\vspace{1em}
With our fixed $c_v$'s for all $v\in\Sigma\cup\{\ell,\ell'\}$, define
\begin{equation}\label{eq:K_v}
    \gls{Kv}=\Big\{g\in\GU(2)(\bZ_v):g\equiv\bid_2\mod q^{2c_v+\ord_v(\delta\fs)}_v \Big\},
\end{equation}
and the tame level group $K^p_f\subset\GU(2)(\hat{\bZ}^p)$ as
\begin{equation}\label{eq:k^p_f}
    \gls{Kpf}=\bigotimes_{v\notin\Sigma\cup\{\infty,\ell,\ell'\}}\GU(2)(\bZ_v)\bigotimes_{v\in\Sigma\cup\{\ell,\ell'\}}K_v.
\end{equation}
Given an open subgroup $K_p\subset\GU(2)(\bZ_p)$ and a ring $R$, let\begin{equation}\label{eq:MGU2}
   \gls{MU2}=\left\{\text{functions on  $\GU(2)(\bQ)\backslash\GU(2)(\bA_{\bQ})/K^p_fK_{p}\U(2)(\bR)$ valued $R$}\right\},
\end{equation}
the space of $R$-valued automorphic forms on $\GU(2)$ of weight $(0,0)$ and level $K^p_fK_p$. 

\vspace{1em}
We also define the following twist of the form $f^{\GU(2)}$:
\begin{equation}\label{eq:f-twist}
   \gls{ftwisted}=\prod_{\substack{v=\fv\bar{\fv}\in\Sigma_\rms\cup\{\ell\}}} \sum_{n\in(\bZ_v/q^{t_v}_v\bZ_v)^\times}\chi_{h,\bar{\fv}}(-n)\,R\left(\begin{pmatrix}1\\n&1\end{pmatrix}_\fv\begin{pmatrix}q^{-t_v}_v\\&1\end{pmatrix}_\fv\right)f^{\GU(2)},
\end{equation}
where $t_v=\ord_{\bar{\fv}}(\cond(\chi_{h,\bar{\fv}}))$. 
Then $f^{\GU(2)}_{\chi_h}\in M_{\GU(2)}\left(K^p_f\GU(2)(\bZ_p);\cO_L\right)$. Let 
\begin{equation}\label{eq:M-varphi'}
   L\big[\GU(2)\big(\bQ_{\Sigma_\ns\cup\{\ell'\}}\big)\big]\cdot f^{\GU(2)}_{\chi_h}
\end{equation}
denote the space generated by $f^{\GU(2)}_{\chi_h}$ and the action of $\GU(\bQ_v)$, $v\in\Sigma_\ns\cup\{\ell'\}$. 

\begin{rmk}
We will use $\varphi$ inside the space
\begin{equation}\label{eq:M-varphi}
    M_{\GU(2)}\left(K^p_f\GU(2)(\bZ_p);\cO_L\right)\cap L\big[\GU(2)\big(\bQ_{\Sigma_\ns\cup\{\ell'\}}\big)\big]\cdot f^{\GU(2)}_{\chi_h}
\end{equation} 
to pair with the Siegel Eisenstein family on $\GU(3,3)$ to get the Klingen Eisenstein family on $\GU(3,1)$. The twist at the split primes is to make the nebentypus match those of the auxiliary CM families $\bmtheta$ and $\bfh$ constructed from our chosen auxiliary characters $\chi_\theta$ and $\chi_h$. The flexibility at the non-split primes is to ensure the non-vanishing of the local triple product integrals at the non-split primes.
\end{rmk}

\subsection{The choice of the local sections for the Siegel Eisenstein series on $\GU(3,3)$}\label{sec:sec-choice}

With $c_v$'s chosen in \S\ref{sec:aux}, given an algebraic Hecke character $\tau$ such that $\xi\tau$ has $\infty$-type $(0,k)$ with $k$ an even integer $\geq 6$, we make the following choices of $f(s,\xi_0\tau_0)\in I_{Q_{\GU(3,3)}}(s,\xi_0\tau_0)$, so that after a suitable normalization, $E^\Sieg(\bdot\,;f(s,\xi_0\tau_0))\big|_{s=\frac{k-3}{2}}$ can be interpolated by a two-variable $p$-adic family when $\tau$ varies.

\subsubsection{The archimedean place} 
We choose $f_\infty(s,\xi_0\tau_0)\in I_{Q_{\GU(3,3)},\infty}(s,\tau_0)$ as
\begin{align*}
   f_{\infty}(s,\xi_0\tau_0)\left(g=\begin{pmatrix}A&B\\C&D\end{pmatrix}\right)=&\,(\det g)^k|\det(g)|_{\bC}^{\frac{s}{2}-\frac{3}{4}-\frac{k}{4}}\\
   &\times\det\left(C\begin{psm}i\\&\frac{\zeta_0}{2}\end{psm}+D\right)^{-k}\left|\det\left(C\begin{psm}i\\&\frac{\zeta_0}{2}\end{psm}+D\right)\right|_{\bC}^{-s+\frac{3}{2}+\frac{k}{2}}.
\end{align*}

\subsubsection{Unramified places.}

For  $v\notin\Sigma\cup\{p,\ell,\ell'\}$, we choose $f_v(s,\xi_0\tau_0)\in I_{Q_{\GU(3,3)},v}(s,\tau_0)$ to be the standard spherical section, {\it i.e.} the section that is invariant under the right translation of $\GU(3,3)(\bZ_v)$ and takes value $1$ at $\bid_6$.

\subsubsection{Places $v\in\Sigma\cup\{\ell,\ell'\}$}
We choose $f_v(s,\xi_0\tau_0)\in I_{Q_{\GU(3,3)},v}(s,\tau_0)$ as
\begin{equation}\label{eq:fv-ram}
\begin{aligned}
   f_v(s,\xi_0\tau_0)\left(g=\begin{pmatrix}A&B\\C&D\end{pmatrix}\right)=&\,|\nu(g)|^{s+\frac{3}{2}}_v |\det C\ltrans{\bar{C}}|^{-s-\frac{3}{2}}_v 
   (\xi_0\tau_0)_v\left(\nu(g)(\det \bar{C})^{-1}\right)\\
   &\times \mathds{1}_{\Her(3,\cO_{\cK,v})}\left(C^{-1}D+\begin{pmatrix}0&0\\0&q_v^{-c_v}\cdot \bid_2\end{pmatrix}\right),
\end{aligned}
\end{equation}
where $c_v$ is the fixed positive integer in \S\ref{sec:aux}.

\subsubsection{The place $p$}

We have the isomorphisms
\begin{align*}
   \U(3,3)(\bQ_p)&\lra \GL_6(\bQ_p), & g=\big(a_{ij}\big)&\longmapsto \big(\varrho_\fp(a_{ij})\big),\\ 
   \Her_3(\cK_p)&\lra M_{3,3}(\bQ_p), &x=\big(x_{ij}\big)&\longmapsto \varrho_\fp(x)=\big(\varrho_\fp(x_{ij})\big). 
\end{align*}
and we will often use them to identify $\U(3,3)(\bQ_p)$ with $\GL_6(\bQ_p)$ and $\Her_3(\cK_p)$ with $M_{3,3}(\bQ_p)$. 

\vspace{.5em}

For $x\in \Her_3(\cK_p)$, we write $\varrho_{\fp}(x)$ as $\begin{pmatrix}x_{11}&x_{12}&x_{13}\\x_{21}&x_{22}&x_{23}\\x_{31}&x_{32}&x_{33}\end{pmatrix}$. Define the Schwartz function $\alpha_{\xi\tau,p}$ on $\Her_3(\cK_p)$ as
\[
   \alpha_{\xi\tau,p}\left(x\right)=\mathds{1}_{\Her_3(\cO_{\cK,p})}(x)\cdot \mathds{1}_{\bZ^\times_p}(x_{21})\,\mathds{1}_{\GL_2(\bZ_p)}\begin{pmatrix}x_{21}&x_{22}\\x_{31}&x_{32}\end{pmatrix}(\xi_0\tau_0)^{-1}_{\fp}\left(\det\begin{pmatrix}x_{21}&x_{22}\\x_{31}&x_{32}\end{pmatrix}\right).
\]
and let $\cF^{-1}\alpha_{\xi\tau,p}$ be the inverse Fourier transform of $\alpha_{\xi\tau,p}$, {\it i.e.}
\[
   \cF^{-1}\alpha_{\xi\tau,p}(x)=\int_{\Her_3(\cK_p)} \alpha_{\xi\tau,p}(y)\,\be_p(\Tr\, xy)\,dy,
\]
where $\be_p$ is the additive character in \eqref{eq:bev}. (A simple computation shows that  
\begin{equation}\label{eq:inv-F}
\begin{aligned}
   \cF^{-1}\alpha_{\xi\tau,p}(x)=&\,p^{-3t}\fg\left((\xi_0\tau_0)_{\fp}^{-1}\right)^2\cdot \mathds{1}_{\bZ^5_p}(x_{11},x_{21},x_{31},x_{32},x_{33})\cdot \mathds{1}_{\bZ^2_p}(x^{-1}_{23}x_{13},x^{-1}_{23}x_{22})\\
   &\times\mathds{1}_{p^{-t}\bZ^\times_p}(x_{23})\mathds{1}_{p^{-t}\bZ^\times_p}\left(x_{12}-x_{13}x^{-1}_{23}x_{22}\right)(\xi_0\tau_0)_{\fp}\left(p^{2t}\det\begin{pmatrix}x_{12}&x_{13}\\x_{22}&x_{23}\end{pmatrix}\right).
\end{aligned}
\end{equation}
if $(\xi_0\tau_0)_\fp$ has conductor $p^t$, $t\geq 1$.) In order to define the $f_p(s,\xi_0\tau_0)$ for our purpose, we first define the section
\begin{align*}
   f^{\text{big-cell}}_p\left(-s,(\xi_0\tau_0)^{-c}\right)\left(g=\begin{pmatrix}A&B\\C&D\end{pmatrix}\right)
   &=|\nu(g)|^{-s+\frac{3}{2}}_p |\det C\ltrans{\bar{C}}|^{s-\frac{3}{2}}_p 
   (\xi_0\tau_0)_{p}\left(\nu(g)(\det C)\right)\\
   &\quad\times \cF^{-1}\alpha_{\xi\tau,p}(C^{-1}D),
\end{align*}
which belongs to $I_{Q_{\GU(3,3)},p}(-s,(\xi_0\tau_0)^{-c})$ and is supported on the big cell 
\[
Q_{\GU(3,3)(\bQ_p)}\begin{pmatrix}0&-\bid_3\\\bid_3&0\end{pmatrix}Q_{\GU(3,3)(\bQ_p)}.
\]

We choose $f_p(s,\xi_0\tau_0)\in I_{Q_{\GU(3,3)},p}(s,\xi_0\tau_0)$ as
\begin{equation}\label{eq:fp}
\begin{aligned}
   f_p(s,\xi_0\tau_0)(g)=&\,\gamma_p\left(-2s,(\xi^\bQ_0\tau^\bQ_{0})^{-1}\right)\gamma_p\left(-2s-1,(\xi^\bQ_0\tau^{\bQ}_{0})^{-1}\eta_{\cK/\bQ}\right)\gamma_p\left(-2s-2,(\xi^\bQ_0\tau^{\bQ}_{0})^{-1}\right)\\
   &\times M_p\left(-s,(\xi_0\tau_0)^{-c}\right)f^{\text{big-cell}}_p\left(-s,(\xi_0\tau_0)^{-c}\right)(g\Upsilon_p),
\end{aligned}
\end{equation}
where $\gls{xiQ}$ (resp. $\gls{tauQ}$) denotes the restriction of $\xi_0$ (resp. $\tau_0$) to $\bQ^\times\backslash\bA^\times_\bQ$, and $M_p\left(-s,(\xi_0\tau_0)^{-c}\right)$ is the intertwining operator, {\it i.e.}
\begin{equation}\label{eq:intw}
\begin{aligned}
   &M_p\left(-s,(\xi_0\tau_0)^{-c}\right)f^{\text{big-cell}}_p(-s,(\xi_0\tau_0)^{-c})(g\Upsilon_p)\\
   =&\int_{\Her_3(\cK_p)} f^{\text{big-cell}}_p(-s,(\xi\tau_0)^{-c})\left(\begin{pmatrix}&-\bid_3\\\bid_3\end{pmatrix}\begin{pmatrix}\bid_3&y\\&\bid_3\end{pmatrix}g\Upsilon\right)\,dy,
\end{aligned}
\end{equation}
and $\gls{Upsilonp}$ is the element in $\U(3,3)(\bQ_p)$ such that
\[
   \varrho_\fp(\Upsilon_p)=
   \begin{pmatrix}
   1&\\&\bid_2&&-\frac{\zeta_0}{2}\\&&1\\&-\bid_2&&-\frac{\zeta_0}{2}
   \end{pmatrix}^{-1}.
\]

\begin{rmk}
When $(\xi^\bQ_0\tau^\bQ_0)_p$ is ramified with conductor $p^t$,
\begin{align*}
   &\gamma_p\left(-2s,(\xi^\bQ_0\tau^\bQ_{0})^{-1}\right)\gamma_p\left(-2s-1,(\xi^\bQ_0\tau^{\bQ}_{0})^{-1}\eta_{\cK/\bQ}\right)\gamma_p\left(-2s-2,(\xi^\bQ_0\tau^{\bQ}_{0})^{-1}\right)\\
   =&\,\fg\left((\xi^\bQ_0\tau^{\bQ}_0)_p^{-1}\right)^{-3}(\xi^\bQ_0\tau^\bQ_0)_p(p)^{-3t}p^{(6s+6)t}.
\end{align*}
Combining this with \eqref{eq:inv-F}, one can see that the formula for the section $f_p(s,\xi_0\tau_0)$ in \eqref{eq:fp} agrees with the formula in \cite[(6-5)]{WanU31}.) 
\end{rmk}

\begin{rmk}
For the unitary group attached to $\begin{pmatrix}&\bid_n\\-\bid_n\end{pmatrix}$ and $\beta\in\Her_n(\cK)$, the factor for the functional equation of the $\beta$-th local Fourier coefficients of the degenerate principal series inducing $(\xi_0\tau_{0})^{-c}_p|\cdot|_{\cK_p}^{-s}$  is 
\begin{equation}\label{eq:cns}
   c_{n,p}(-s,(\xi_0\tau_0)^{-c},\beta)=c_{\beta,p}\cdot (\xi^\bQ_0\tau^\bQ_0)_{p}(\det\beta)|\det\beta|^{2s}_{\bQ_p}\cdot\prod_{j=1}^n\gamma_p\left(-2s+1-j,(\xi^\bQ_0\tau^\bQ_{0})^{-1}\eta_{\cK/\bQ}^{n+j}\right)^{-1},
\end{equation}
where $c_{\beta,p}$ is a constant independent of $s$ and $\xi_0\tau_0$. This factor also shows up in the functional equation of the doubling zeta integrals. Note that in our choice of $f_p(s,\xi_0\tau_0)$ in \eqref{eq:fp}, the product of gamma factors is exactly the product of gamma factors in $c_{3,p}(-s,(\xi_0\tau_0)^{-c},\beta)$.
\end{rmk}


Next, we construct a $p$-adic family interpolating these $E^\Sieg(-;f(s,\xi_0\tau_0))\big|_{s=\frac{k-3}{2}}$ with $f(s,\xi_0\tau_0)$ chosen as above. We will construct $p$-adic families of automorphic forms as measures on $\Gamma_\cK$ valued in the space of $p$-adic forms. The next subsection introduces some notions about $p$-adic measures.

\subsection{$p$-adic measures and $p$-adic families of automorphic forms}\label{sec:meas}
 
For a $p$-adically complete $\cO_L$-module $M$ and a compact abelian group $Y$ with totally disconnected topology, denote by $\Meas(Y,M)$ the space of $M$-valued $p$-adic measures on $Y$. Given $\mu\in \Meas(Y,M)$ and a continuous function $\eta:Y\ra R$ with $R$ a $p$-adically complete $\cO_L$ algebra, we write
\[
    \mu(\eta)=\int_Y \eta(y)\,d\mu(y).
\]
to denote the value of $\mu$ at $\eta$, which is an element in $R\wh{\otimes}_{\cO_L}M$. 

For each $y\in Y$, we have the delta measure $\delta_y\in \Meas(Y,M)$ defined by
\begin{equation}\label{eq:delta-meas}
   \delta_y(\eta)=\eta(y).
\end{equation}
Given $M_1,M_2$ and $\mu_i\in \Meas(Y,M_i)$, $i=1,2$, the convolution $\mu_1\ast\mu_2\in \Meas(Y,M_1\wh{\otimes}M_2)$ is defined by
\begin{equation}\label{eq:clv-meas1}
   (\mu_1\ast\mu_2)(\eta)=\int_Y\int_Y \eta(yz)\,d\mu_1(y)\,d\mu_2(z),
\end{equation}
If $\chi$ is a continuous character of $Y$ valued in $\cO^\times_{\bar{L}}$, then
\begin{equation}\label{eq:clv-meas2}
   (\mu_1\ast\mu_2)(\chi)=\mu_1(\chi)\otimes\mu_2(\chi).
\end{equation}

Let $A$ be a group with a homomorphism $A\ra Y$ and an action on $M$. Then the group $A$ acts on $\Meas(Y,M)$ in two ways: 
\begin{enumerate}
\item[--] $A$ acts on $\Meas(Y,M)$ via its action on $M$,
\item[--] $A$ acts on $\Meas(Y,M)$ via the homomorphism $A\ra Y$ and the translation of $Y$ on itself.
\end{enumerate}
Define
\begin{equation}\label{eq:eqiv-meas}
   \gls{MeasYM}=\left\{\begin{array}{ll}\text{$p$-adic measures in $\Meas(Y,M)$ on which}\\ \text{the two actions of $A$ are compatible}\end{array}\right\}.
\end{equation}

In our applications, $Y$ will be taken to be $\Gamma_\cK$ or $U_{\cK,p}=1+p\cO_{\cK,p}$, $A$ will be taken to be $\bZ^\times_p\times\bZ^\times_p$, which we identify with $\cO^\times_{\cK,p}$ via $(\varrho_\fp,\varrho_{\bar{\fp}})$, and  $M$ will be taken to be either $\hat{\cO}^{\ur}_L$ or spaces of $p$-adic forms on unitary groups.  Next, we introduce the spaces of $p$-adic families of automorphic forms we will use.

\subsubsection{$p$-adic families on $\GU(3,1)$} 
Let $\gls{VGU31}=\varprojlim\limits_m\varinjlim\limits_n V_{n,m}$ (defined in \S\ref{subsec:p-adicforms}). In \eqref{eq:eqiv-meas}, put $Y=\Gamma_\cK$, $A=\bZ^\times_p\times\bZ^\times_p$ with the homomorphism 
\begin{equation}\label{eq:ft}
\begin{tikzcd}[column sep=huge]
   A\simeq \cO^\times_{\cK,p}\arrow[r,"\text{natural}","\text{embedding}"']& \cK^\times\backslash\bA^\times_\cK\arrow[r,"\text{natural}","\text{projection}"']&\Gamma_\cK=Y,
\end{tikzcd}
\end{equation}
and $M=V_{\GU(3,1),\,\xi}$, the component of $V_{\GU(3,1)}$ on which the kernel of \eqref{eq:ft}, acts through the character $\xi_{p\adic}$.  We make $(a_1,a_2)\in \bZ^\times_p\times\bZ^\times_p$ act on $V_{\GU(3,1),\,\xi}$  by the usual action of
$
    \varrho^{-1}_\fp\begin{psm}1\\&1\\&&a_1\\&&&a_2\end{psm}\in T_\so(\bZ_p)\subset \U(3,1)(\bZ_p)
$
on $V_{\GU(3,1),\,\xi}$ multiplied by the scalar $\xi^{-1}_{p\adic,p}(a_1,a_2)$. (Here $\varrho^{-1}_\fp$ is the isomorphism in \eqref{eq:varrho-1}, and this action factors through the quotient of $\bZ^\times_p\times\bZ^\times_p$ by the kernel of \eqref{eq:ft}.) Then we get
\begin{equation}\label{eq:MU31}
    \Meas\left(\Gamma_\cK,V_{\GU(3,1),\,\xi}\right)^\natural,
\end{equation}
the $\xi_{p\adic}$-component of the space of $p$-adic families on $\GU(3,1)$. The compatibility of the two actions of $\bZ^\times_p\times\bZ^\times_p$ implies that the value at $\tau_{p\adic}$ of an element in \eqref{eq:MU31} is a $p$-adic form with $p$-adic nebentypus 
\[
   \left(\triv,\triv,(\xi\tau)_{p\adic,\fp},(\xi\tau)^{-1}_{p\adic,\bar{\fp}}\right).
\] 
(By saying $p$-adic nebentypus, we refer to the usual action of  $\varrho^{-1}_\fp\begin{psm}a_1&\ast&\ast&\ast\\&a_2&\ast&\ast\\&&a_3&\ast\\&&&a_4\end{psm}\in\U(3,1)(\bZ_p)$ on $V_{\GU(3,1)}$.)

The space \eqref{eq:MU31} contains the subspace of semi-ordinary families:
\begin{equation}\label{eq:MU31-so}
    \Meas\left(\Gamma_\cK,e_\so V_{\GU(3,1),\,\xi}\wh{\otimes}\hat{\cO}^\ur_L\right)^\natural,
\end{equation}
which is naturally equipped with an $\hat{\cO}^\ur_L\llb\Gamma_\cK\rrb$-module structure. The Klingen Eisenstein family we will construct belongs to this space. (It is tautological that an element in  \eqref{eq:MU31-so} gives rise to an element in $\xi_{p\adic}$-component of $\cM_\so\wh{\otimes}_{\bZ_p\llb T_\so(\bZ_p)\rrb}\hat{\cO}^\ur_L\llb\Gamma_\cK\rrb=\Hom_{\hat{\cO}^\ur_L\llb\Gamma_\cK\rrb}\left(\pV^*_\so,\hat{\cO}^\ur_L\llb\Gamma_\cK\rrb\right)$, see \S\ref{sec:Kl-ideal}.)

\subsubsection{$p$-adic families on $\GU(3,3)$}

Let $\gls{VGU33}$ be the space of $p$-adic forms on $\GU(3,3)$ defined by considering global sections of the structure sheaf on the Igusa towers for the Shimura variety of $\GU(3,3)$. We will construct a Siegel Eisenstein family in 
\begin{equation*}
    \Meas\left(\Gamma_\cK,V_{\GU(3,3)}\right).
\end{equation*}

\subsubsection{$p$-adic families on $\U(2)$}\label{sec:U2-family}
For constructing the Klingen Eisenstein family, we use an automorphic form $\varphi$ in the space \eqref{eq:M-varphi}, which is spherical at $p$ and does not vary over the weight space. It is when analyzing the non-degenerate Fourier--Jacobi coefficients of our Klingen Eisenstein family, we need auxiliary $p$-adic families on $\U(2)$.

Let
\begin{align*}
    N_p(\bZ_p)&=\left\{g\in \U(2)(\bZ_p):\varrho_\fp(g)=\begin{pmatrix}\ast&\ast\\0&1\end{pmatrix}\right\},\\
    N'_p(\bZ_p)&=\left\{g\in \U(2)(\bZ_p):\varrho_\fp(g)=\begin{pmatrix}1&\ast\\0&\ast\end{pmatrix}\right\}.
\end{align*}
We define the following two spaces of $p$-adic forms on $\U(2)$ of tame level $K^p_f\cap\U(2)(\wh{\bZ}_p)$:
\begin{align*}
   \glsuseri{VU2}&=\left\{\text{continuous functions }\U(2)(\bQ)\backslash \U(2)(\bA_{\bQ,f})/\big(K^p_f\cap\U(2)(\wh{\bZ}_p)\big) N_p(\bZ_p)\lra \cO_L \right\},\\
  \glsuserii{VU2}&=\left\{\text{continuous functions }\U(2)(\bQ)\backslash \U(2)(\bA_{\bQ,f})/\big(K^p_f\cap\U(2)(\wh{\bZ}_p)\big) N'_p(\bZ_p)\lra \cO_L \right\},
\end{align*}
and let $\glsuseriii{VU2}$ (resp. $\glsuseriv{VU2}$) denote the subspace of $V_{\U(2)}$ (resp. $V'_{\U(2)}$) on which the kernel of \eqref{eq:ft} acts through the character $(\triv,\xi_{p\adic,\fp})$ (resp. $(\xi^{-1}_{p\adic,\fp},\triv)$).  We have the operator $U_p$ acting on ${\tt F}\in V_{\U(2),\,\xi}$ (resp. ${\tt F}\in V'_{\U(2),\,\xi^{-1}}$) as
\[
   (U_p{\tt F})(g)=\sum_{x\in\bZ/p\bZ}{\tt F}\left(g\varrho^{-1}_\fp\begin{pmatrix}p&x\\&1\end{pmatrix}\right)
\]
with $\varrho^{-1}_\fp$ the isomorphism in \eqref{eq:varrho-1}. The ordinary projection is defined as $\gls{eord}=\lim\limits_{n\ra\infty}U^{n!}_p$.

\vspace{.5em}

In \eqref{eq:eqiv-meas}, put $Y=\Gamma_\cK$, $A=\bZ^\times_p\times\bZ^\times_p$ with the homomorphism $A\ra Y$ as in \eqref{eq:ft}, and $M=V_{\U(2),\,\xi}$ on which  we make $(a_1,a_2)\in \bZ^\times_p\times\bZ^\times_p$ act by the right translation of $\varrho^{-1}_\fp\begin{pmatrix}1\\&a_1\end{pmatrix}$ multiplied by the scalar $\xi^{-1}_{p\adic,\fp}(a_1)$. (This action factors through the quotient of $\bZ^\times_p\times\bZ^\times_p$ by the kernel of \eqref{eq:ft}.) Then we get 
\begin{equation}\label{eq:MU2}
   \Meas\left(\Gamma_\cK,V_{\U(2),\,\xi}\right)^\natural.
\end{equation}
The evaluation at $\tau_{p\adic}$ of an element in \eqref{eq:MU2} is a $p$-adic form on $\U(2)$ with $p$-adic nebentypus 
\[
   \left(\triv,(\xi\tau)_{p\adic,\fp}\right).
\] 
The non-degenerate Fourier--Jacobi coefficients of our Klingen Eisenstein family are $p$-adic measures on $\Gamma_\cK$ valued in $p$-adic Jacobi forms on $\U(2)$. Pairing them with a fixed Jacobi form on   $\U(2)$ gives an element in \eqref{eq:MU2}. (See \S\ref{sec:theta_1J}.)

\vspace{1em}
We will also need to consider some auxiliary $p$-adic families which are $p$-adic measures on $U_{\cK,p}=1+p\cO_{\cK,p}$ valued in $V_{\U(2),\,\xi}$ and $V'_{\U(2),\,\xi^{-1}}$. We will fix $\wtp$, a power of $p$, such that raising to the $\wtp$-th power maps $\Gamma_\cK$ into $U_{\cK,p}$. (See \S\ref{sec:UKp} for an explanation why we need to consider $U_{\cK,p}$ and fix such an $\wtp$.)

\vspace{.5em}

Put $Y=U_{\cK,p}=1+p\cO_{\cK,p}$, $A=\bZ^\times_p\times\bZ^\times_p$ with the homomorphism
\[
\begin{tikzcd}[column sep=huge]
   A=\bZ^\times_p\times\bZ^\times_p\arrow[r,"\eqref{eq:ft}"]& \Gamma_\cK \arrow[r,"\text{$\wtp$-th power}"]& U_{\cK,p}=Y,
\end{tikzcd}
\]  
and $M=e_\ord V_{\U(2),\,\xi}$ (resp.  $e_\ord V'_{\U(2),\,\xi^{-1}}$) in \eqref{eq:eqiv-meas}, on which we make $(a_1,a_2)\in A$ act through the right translation by $\varrho^{-1}_\fp\begin{pmatrix}1\\&a_1\end{pmatrix}$ (resp.  $\varrho^{-1}_\fp\begin{pmatrix}a^{-1}_1\\&1\end{pmatrix}$) multiplied by the scalar $\xi^{-1}_{p\adic,\fp}(a_1)$ (resp. $\xi_{p\adic,\fp}(a_1)$). Then we get 
\begin{equation}\label{eq:MU2'}
   \Meas\left(U_{\cK,p},e_\ord V_{\U(2),\,\xi}\right)^\natural \quad \left(\text{resp. } \Meas\left(U_{\cK,p},e_\ord V'_{\U(2),\,\xi^{-1}}\right)^\natural\right).
\end{equation}
For $\tau_{p\adic}\in\Hom_\cont\left(\Gamma_\cK,\ol{\bQ}^\times_p\right)$, the evaluation at $\tau_{p\adic}|_{U_{\cK,p}}$ of an element in \eqref{eq:MU2} is a $p$-adic form on $\U(2)$ with nebentypus 
\[
   \left(\triv,(\xi\tau^\wtp)_{p\adic,\fp}\right)
   \quad \left(\text{resp. } \left(\xi\tau^{\wtp})^{-1}_{p\adic,\fp},\triv\right)\right).
\]

We will construct auxiliary CM families 
\[
   \glsuseri{theta},\glsuserii{h}\in \Meas\left(U_{\cK,p},e_\ord V_{\U(2),\,\xi}\right)^\natural,
   \quad \glsuserii{theta},\glsuseri{h}\in \Meas\left(U_{\cK,p},e_\ord V'_{\U(2),\,\xi^{-1}}\right)^{\natural}, 
\]
in \S\ref{sec:CM-families} and use them to study the non-degenerate Fourier--Jacobi coefficients of the Klingen Eisenstein family (which are elements in \eqref{eq:MU2} and gives rise to measures on $U_{\cK,p}$ by the map \eqref{eq:PN*}.)

\subsection{The $p$-adic family of Siegel Eisenstein series on $\GU(3,3)$}\label{sec:Sieg}
We normalize the Siegel Eisenstein series $E^\Sieg(\bdot\,;f(s,\xi_0\tau_0)\big|_{s=\frac{k-3}{2}}$ (with $f(s,\xi_0\tau_0)$ chosen as in \S\ref{sec:sec-choice}) as
\begin{equation}\label{eq:normalization}
   \gls{Enorl}=\left(\frac{2^{-6}(-2\pi i)^{3k}}{\pi^3\prod_{j=0}^2\Gamma(k-j)}\right)^{-1}d^{\,\Sigma\cup\{\infty,p,\ell,\ell'\}}_3(s,\tau_0)\cdot E^\Sieg(\bdot\,;f(s,\xi_0\tau_0)\big|_{s=\frac{k-3}{2}},
\end{equation}
where
\begin{equation}\label{eq:d3}
\begin{aligned}
   d^{\,\Sigma\cup\{\infty,p,\ell,\ell'\}}_3(s,\xi_0\tau_0)&=\prod_{v\notin \Sigma
\cup\{p,\ell,\ell'\}} d_{3,v}(s,\xi_0\tau_0), \\
   \gls{d3v}&=\prod_{j=1}^3 L_v\left(2s+j,\xi^\bQ_0\tau^\bQ_{0}\eta^{3-j}_{\cK/\bQ}\right).
\end{aligned}
\end{equation}

\begin{thm}\label{thm:Sieg-F}
There exists a $p$-adic measure $\gls{bESieg}\in\Meas\left(\Gamma_\cK,V_{\GU(3,3)}\right)$ such that
\[
   \bfE^\Sieg(\tau_{p\adic})=E^\Sieg_{\xi\tau}
\]
for all algebraic Hecke character $\tau:\cK^\times\backslash\bA^\times_\cK\ra\bC^\times$ such that its $p$-adic avatar $\tau_{p\adic}$ factors through $\Gamma_\cK$ and $\xi\tau$ has $\infty$-type $\left(0,k\right)$ with $k\geq 6$ even.
\end{thm}

\begin{proof}
The measure is constructed by interpolating the $q$-expansions of the $E^\Sieg_{\xi\tau}$'s. First, we record the formulas for the Fourier coefficients in \cite[\S\S6E-6H]{WanU31}. Given $\beta\in\Her_3(\cK)$ and $g\in\GU(3,3)(\bA_\bQ)$, let 
\[
   W_{\beta,v}\big(g_v,f_v(s,\xi_0\tau_0)\big)=\int_{\Her_3(\cK_v)}f_v(s,\xi_0\tau_0)\left(\begin{pmatrix}&-\bid_3\\\bid_3\end{pmatrix}\begin{pmatrix}\bid_3&\varsigma\\0&\bid_3\end{pmatrix}g_v\right)\be_v(-\Tr\,\beta\varsigma)\,d\varsigma
\]
(where $\be_v$ is the additive character in \eqref{eq:bev}), and
\[
   E^\Sieg_{\beta}\big(g;f(s,\xi_0\tau_0)\big)
   =\int_{\Her_3(\cK\backslash\bA_\cK)}E^\Sieg\left(\begin{pmatrix}\bid_3&\varsigma\\0&\bid_3\end{pmatrix}g;f(s,\xi_0\tau_0)\right)\be_{\bA_\bQ}(-\Tr\,\beta\varsigma)\,d\varsigma.
\]
Because the sections at $v\in\Sigma\cup\{\ell,\ell'\}$ are chosen to be supported on the big cell, for $g\in Q_{\GU(3,3)}(\bA_\bQ)$, we have
\[
   E^\Sieg_{\beta}\big(g;f(s,\xi_0\tau_0)\big)=\prod_v W_{\beta,v}\big(g_v,f_v(s,\xi_0\tau_0)\big).
\]
For our choice of $f_v(s,\xi_0\tau_0)$ in \S\ref{sec:sec-choice}, the formulas for $W_{\beta,v}\big(g_v,f_v(s,\xi_0\tau_0)\big)$ are as follows.

\vspace{.5em}

\item[--]\emph{The archimedean place.}

\begin{align*}
   &\left.W_{\beta,\infty}\left(\begin{pmatrix}y&x\ltrans{\bar{y}}^{-1}\\0&\ltrans{\bar{y}}^{-1}\end{pmatrix};f_\infty(s,\xi_0\tau_0)\right)\right|_{s=\frac{k-3}{2}}\\
   =&\left\{\def\arraystretch{2}\begin{array}{ll}
   \frac{2^{-6}(-2\pi i)^{3k}}{\pi^3\prod_{j=0}^2\Gamma(k-j)}(\det\bar{y})^k \be_{\bR}\left(\Tr\,\beta(x+i y\ltrans{\bar{y}})\right)\cdot (\det\beta)^{k-3}, &\beta\in\Her_3(\cK)_{>0}\\
   0, &\text{otherwise}.\end{array}\right.
\end{align*}

\item[--] $v\notin\Sigma\cup\{p,\ell,\ell'\}$. For $\det\beta\neq 0$, 
\begin{align*}
   \left.W_{\beta,v}\left(\diag(A,D);f_v(s,\xi_0\tau_0)\right)\right|_{s=\frac{k}{2}}
   &=\xi_{0,v}\tau_{0,v}(\det D)|\det AD^{-1}|^{-\frac{k}{2}+3}_{\bQ_v}\cdot  d_{3,v}(s,\xi_0\tau_{0})^{-1}\\
   &\quad\times \mathds{1}_{\Her_3(\cO_{\cK,v})^*}(D^{-1}\beta A)\cdot h_{v,D^{-1}\beta A}\left(\xi^\bQ_{0,v}\tau^\bQ_{0,v}(q_v)q_v^{-2s-3}\right),
\end{align*}
where $h_{v,D^{-1}\beta A}\in\bZ[X]$ is a monic polynomial depending only on $v$ and $D^{-1}\beta A$, and is the constant polynomial when
 when $D^{-1}\beta A$ belongs to $\GL_3(\cO_{\cK,v})$.

\vspace{.5em}

\item[--] $v\in\Sigma\cup\{\ell,\ell'\}$.

\begin{align*}
   \left.W_{\beta,v}\left(\diag(A,D);f_v(s,\xi_0\tau_0)\right)\right|_{s=\frac{k-3}{2}}
   &=|D_{\cK/\bQ}|^{3/2}_v \cdot \xi_{0,v}\tau_{0,v}(\det D)|\det AD^{-1}|^{-\frac{k}{2}+3}_{\bQ_v}\\
   &\quad\times\mathds{1}_{\Her_3(\cO_{\cK,v})^*}(D^{-1}\beta A)\,\be_v\left(\frac{(D^{-1}\beta A)_{22}+(D^{-1}\beta A)_{33}}{q^{c_v}_v}\right).
\end{align*}

\item[--] \emph{The place $p$.} An easy computation shows that 
\begin{align*}
   \left.W_{\beta,p}\left(\bid_6;f^{\text{big-cell}}_p(-s,\xi^{-c}_0\tau^{-c}_0)\right)\right|_{s=\frac{k-3}{2}}
   =\cF(\cF^{-1}\alpha_{\xi\tau,p})(\beta)=\alpha_{\xi\tau,p}(\beta).
\end{align*}
The functional equation for $W_{\beta,p}$ in \cite[(14)]{LapidRallis} implies that
\[
   W_{\beta,p}\left(g;M_p\left(-s,\xi^{-c}_0\tau^{-c}_{0}\right)f^{\text{big-cell}}_p(-s,\xi^{-1}_0\tau^{-1}_0)\right)=c_{3,p}(-s,\xi^{-c}_0\tau^{-c}_{0},\beta)\cdot W_{\beta,p}\left(g;f^{\text{big-cell}}_p(-s,\xi^{-c}_0\tau^{-c}_0)\right)
\]
with $c_{3,p}(-s,\xi^{-c}_0\tau^{-c}_{0},\beta)$ defined as in \eqref{eq:cns}. It follows that
\begin{align*}
    \left.W_{\beta,p}\big(\bid;f_p(s,\xi_0\tau_0)\big)\right|_{s=\frac{k-3}{2}}=c_{\beta,p}\cdot(\xi^\bQ_0\tau^\bQ_0)_p(\det\beta)|\det\beta|^{2s}_{\bQ_p}\cdot \alpha_{\xi\tau,p}(\beta).
\end{align*}

\vspace{.5em}

Denote by $E^{\Sieg}_{\xi\tau,\beta}\big(\diag(A,D)\big)$ the $\beta$-th coefficient in the ($p$-adic) $q$-expansion of $E^\Sieg_{\xi\tau}$ at the ($0$-dimensional) cusp indexed by  $\diag(A,D)\in\GU(3,3)(\bA^p_{\cK,f})$. From the above formulas for the local Fourier coefficients at the archimedean place, we see that $E^{\Sieg}_{\xi\tau,\beta}\big(\diag(A,D)\big)$ is nonzero only for $\beta$ inside
\[
   \fS_3=\left\{\beta\in\Her_3(\cO_\cK)_{>0}:\varrho_{\fp}(\beta_{21})\in\bZ^\times_p,\,\begin{pmatrix}\varrho_\fp(\beta_{21})&\varrho_\fp(\beta_{22})\\\varrho_\fp(\beta_{31})&\varrho_\fp(\beta_{32})\end{pmatrix}\in\GL_2(\bZ_p)\right\},
\]
and for $\beta\in\fS_3$ we have
\begin{equation}\label{eq:Ebeta}
\begin{aligned}
  E^{\Sieg}_{\xi\tau,\beta}\big(\diag(A,D)\big)&=
  \text{a constant independent of $\tau$}\\
   &\quad\times (\xi\tau)_{p\adic}(\det D)
   \prod_{v\notin\Sigma\cup\{p,\ell,\ell'\}} h_{v,\ltrans{\bar{A}}\beta A}\left((\xi\tau)^\bQ_{p\adic,v}(q_v)\right),\quad \\
   &\quad\times (\xi\tau)_{p\adic,\fp}\left(\det\begin{pmatrix}\beta_{21}&\beta_{22}\\\beta_{31}&\beta_{32}\end{pmatrix}^{-1}\right)
    (\xi\tau)^{\bQ}_{p\adic,p}(\det\beta)
\end{aligned}
\end{equation}
(Note that thanks to the condition that $\xi\tau$ has $\infty$-type $(0,k)$, we have $(\xi\tau)_{p\adic,\fp}=(\xi\tau)_\fp$ and $(\xi\tau)_{p\adic}|_{\cO^\times_{\cK,\fp}}=\xi_0\tau_0|_{\cO^\times_{K,\fp}}$.)  

The factors on the right hand side of \eqref{eq:Ebeta} are interpolated by $\delta$-measures in $\Meas\left(\Gamma_\cK,\cO_L\right)$. The convolution of those $\delta$-measures gives an element in $\Meas\left(\Gamma_\cK,\cO_L\right)$ which interpolates the Fourier coefficient $E^{\Sieg}_{\xi\tau,\beta}\big(\diag(A,D)\big)$. (See \eqref{eq:delta-meas}\eqref{eq:clv-meas1}\eqref{eq:clv-meas2} for the definitions of $\delta$-measures and the convolution of $p$-adic measures.) We see that there exists a $p$-adic measure 
\[
   \bfE^\Sieg\in \Meas\Big(\Gamma_\cK,\bigoplus\limits_{\text{$p$-adic cusps}}\cO_L\llb\fS_3\rrb \Big)
\] 
that interpolates the $p$-adic $q$-expansions of $E^\Sieg_{\xi\tau}$. By the $q$-expansion principle and the fact that the $E^\Sieg_{\xi\tau}$'s are automorphic forms on $\GU(3,3)$, we deduce that $\bfE^\Sieg$ belongs to $\Meas\left(\Gamma_\cK,V_{\GU(3,3)}\right)$. (Here we view the space of $p$-adic measures valued in $V_{\GU(3,3)}$ as a subspace of the space of $p$-adic measures valued in copies of $\cO_L\llb\fS\rrb$ via the embedding of $V_{\GU(3,3)}$  into the $\cO_L$-module of $q$-expansions.)
\end{proof}

\subsection{The semi-ordinary family of Klingen Eisenstein series}\label{sec:constructKF}
Let $E^\Sieg_{\xi\tau}\big|_{\GU(3,1)\times\GU(2)}$ be the automorphic form obtained by the composition of the restriction from $\GU(3,3)$ to $\GU(3,1)\times_{\bG_m}\GU(2)$ via the embedding \eqref{eq:embedding} and the extension by zero from $\GU(3,1)\times_{\bG_m}\GU(2)$ to $\GU(3,1)\times\GU(2)$. Define the twsist $E^\Sieg_{\xi\tau}\big|^\star_{\GU(3,1)\times\GU(2)}$ as 
\[
   E^\Sieg_{\xi\tau}\big|^\star_{\GU(3,1)\times\GU(2)}(g_1,g_2)=E^\Sieg_{\xi\tau}\big|_{\GU(3,1)\times\GU(2)}(g_1,g_2)\cdot\xi^{-1}_{p\adic}\tau^{-1}_{p\adic}(\det g_2).
\]
Then by our choice of the section $f_p(s,\xi_0\tau_0)$, we see that
\[
   E^\Sieg_{\xi\tau}\big|^\star_{\GU(3,1)\times\GU(2)}\in V_{\GU(3,1),\,\xi}\otimes M_{\GU(2)}\left(K^p_fK_{p,0};\hat{\cO}^\ur_L\right)
\]
with $\gls{Kp0}=\left\{g\in\GU(2)(\bZ_p):\varrho_\fp(g)=\begin{pmatrix}\ast&\ast\\p\ast&\ast\end{pmatrix}\right\}$. (Note that $K_{p,0}$ equals the $K'_{p,g}$ in \eqref{eq:K''}. See \eqref{eq:MGU2}  for the definition of $M_{\GU(2)}\left(K^p_fK_{p,0};\hat{\cO}^\ur_L\right)$.)

Therefore, by applying a twist of $(\xi\tau)^{-1}_{p\adic}\circ\,\det$  to the family $\bfE^\Sieg\big|_{\GU(3,1)\times\GU(2)}$ on the second factor, we obtain the $p$-adic family
\begin{equation}\label{eq:resE}
    \bfE^\Sieg\big|^\star_{\GU(3,1)\times\GU(2)}\in \Meas\left(\Gamma_\cK, V_{\GU(3,1),\,\xi}\right)^\natural \wh{\otimes} \,M_{\GU(2)}\left(K^p_fK_{p,0};\hat{\cO}^\ur_L\right)
\end{equation}
interpolating $E^\Sieg_{\xi\tau}\big|^\star_{\GU(3,1)\times\GU(2)}$ when $\tau$ varies. (To see that it belongs to the $^\natural$-subspace, one just needs to check the nebentypus.)

Attached to a $\varphi$ in the space \eqref{eq:M-varphi} is a linear functional
\begin{align*}
   \left<\bdot,\varphi\right>: M_{\GU(2)}\left(K^p_fK_{p,0};\hat{\cO}^{\ur}_L\right)&\lra \hat{\cO}^\ur_L\\
   \varphi'&\longmapsto \left<\varphi',\varphi\right> =\int_{\GU(2)(\bQ)\backslash \GU(2)(\bA_{\bQ,f})} \varphi'(g)\varphi(g)\,dg,
\end{align*}
where the integral is understood as a sum over the finite set $\GU(2)(\bQ)\backslash \GU(2)(\bA_{\bQ,f})/K^p_fK_{p,0}$. Applying the linear functional $\left<\bdot,\varphi\right>$ to \eqref{eq:resE} gives a family
\begin{equation}\label{eq:KF}
   \gls{bEKling}= \left<  \bfE^{\Sieg}|^\star_{\GU(3,1)\times\GU(2)},\,\varphi\right>\in \Meas\left(\Gamma_\cK, V_{\GU(3,1),\,\xi}\wh{\otimes}\,\hat{\cO}^\ur_L\right)^\natural.
\end{equation}

\begin{prop}
For  $\tau_{p\adic}$ satisfying the conditions in Theorem~\ref{thm:Sieg-F}, the evaluation of the family $\bfE^\Kling_\varphi$ at $\tau_{p\adic}$ is a Klingen Eisenstein series on $\GU(3,1)$ inducing $\xi_0\tau_0|\cdot|^{\frac{k-3}{2}}\boxtimes\pi^{\GU(2)}$, more precisely, the automorphic form 
\begin{equation}\label{eq:Kling-int}
    g\longmapsto \left.\int_{\U(2)(\bQ)\backslash \U(2)(\bA_{\bQ})/\U(2)(\bR)} E^\Sieg\big(\imath(g,g_2);f(s,\xi_0\tau_0)\big)\varphi(g_2)(\xi_0\tau_0)^{-1}(\det g_2)\,dg_2\,\right|_{s=\frac{k-3}{2}}
\end{equation}
normalized by the factor in \eqref{eq:normalization}.

Moreover, $\bfE^\Kling_\varphi\in \Meas\left(\Gamma_\cK,e_\so V_{\GU(3,1),\,\xi}\wh{\otimes}\,\hat{\cO}^\ur_L\right)^\natural$,   {\it i.e.} is a semi-ordinary family.
\end{prop}

\begin{proof}
From the construction of the family $\bfE^\Kling_\varphi$, it is easy to see that its evaluation $\tau_{p\adic}$ satisfying the conditions in Theorem~\ref{thm:Sieg-F} equals the form given in \eqref{eq:Kling-int} normalized by the factor in \eqref{eq:normalization}.  By the doubling method formula (Theorem~\ref{thm:doubling}), we know that \eqref{eq:Kling-int} is a Klingen Eisenstein series inducing $\xi_0\tau_0|\cdot|^{\frac{k-3}{2}}\boxtimes\pi^{\GU(2)}$.

In order to show that  $\bfE^\Kling_\varphi$ is a semi-ordinary family, for all $\tau$ satisfying the conditions in Theorem~\ref{thm:Sieg-F} and $(\xi\tau)_\fp$ ramified, we look at the action of $\bU_p$-operators on the form \eqref{eq:Kling-int}, which has weight $(0,0,0;k)$. Given integers $m,m',m''\geq 0$, by the formula for adelic $\bU_p$-operators in \eqref{eq:adeleUp} and our choice of the local section at $p$ in \S\ref{sec:sec-choice}, the action of $(U^+_{p,2})^{m}(U^+_{p,3})^{m'}(U^-_{p,1})^{m''}$ on the form \eqref{eq:Kling-int} can be computed by considering
\begin{equation}\label{eq:Eso-1}
   p^{2m+km''}\int_{N(\bZ_p)} f^{\text{big-cell}}_p\left(\frac{3-k}{2},(\xi_0\tau_0)^{-c}\right)\left(gu\begin{psm}p^{m+m'}\\&p^{m+m'}\\&&p^{m'}\\&&&p^{-m''}\\&&&&1\\&&&&&1\end{psm}\right)\,du,
\end{equation}
where we identify $\U(3,3)(\bQ_p)$ with $\GL_6(\bQ_p)$ via $\varrho_\fp:\cK\otimes_\bQ\bQ_p\ra\bQ_p$. Writing $g=\begin{pmatrix}A&B\\C&D\end{pmatrix}$ in $3\times 3$ blocks, we have
\begin{equation}\label{eq:Eso-2}
\begin{aligned}
  \eqref{eq:Eso-1}=&\, p^{2m+km''}\cdot \left((\xi\tau)_\fp(p)p^{3}\right)^{2m+3m'}\left((\xi\tau)_{\bar{\fp}}(p)p^{3}\right)^{m''}\\
  &\times p^{-4m-3m'-3m''}\sum_{u_1,u_2\in \bZ/p^m\bZ}\sum_{\substack{{v_1,v_2\in\bZ/p^{m+m'+m''}}\\ v_3\in\bZ/p^{m'+m''}\bZ}}\\
  &\quad(\cF^{-1}\alpha_{\xi\tau,p})\left(\begin{psm}p^{-m-m'}\\&p^{-m-m'}\\&&p^{-m'}\end{psm}\begin{psm}1&&-u_1\\&1&-u_2\\&&1\end{psm}\left(C^{-1}D+\begin{psm}v_1&0&0\\v_2&0&0\\v_3&0&0\end{psm}\right)\begin{psm}p^{-m''}\\&1\\&&1\end{psm}\right).
\end{aligned}
\end{equation}
A direct computation by using the formula \eqref{eq:inv-F} shows that when $(\xi\tau)_\fp$ is ramified
\begin{equation}\label{eq:Eso-3}
\begin{aligned}
   &\sum_{u_1,u_2}\sum_{v_1,v_2,v_3}(\cF^{-1}\alpha_{\xi\tau,p})\left(\begin{psm}p^{-m-m'}\\&p^{-m-m'}\\&&p^{-m'}\end{psm}\begin{psm}1&&-u_1\\&1&-u_2\\&&1\end{psm}\left(C^{-1}D+\begin{psm}v_1&0&0\\v_2&0&0\\v_3&0&0\end{psm}\right)\begin{psm}p^{-m''}\\&1\\&&1\end{psm}\right)\\
   =&\,p^{2m}(\cF^{-1}\alpha_{\xi\tau,p})\left(C^{-1}D\begin{psm}1\\&p^{-m-m'}\\&&p^{-m-m'}\end{psm}\right).
\end{aligned}
\end{equation}
Combining \eqref{eq:Eso-2} and \eqref{eq:Eso-3}, we obtain

\begin{align*}
   \eqref{eq:Eso-1}
   =&\,p^{6m+6m'}(\xi\tau)_\fp(p)^{2m+3m'}\left((\xi\tau)_{\bar{\fp}}(p)p^k\right)^{m''}\cdot \left(\xi\tau)_{\bar{\fp}}(p)p^3\right)^{-2m-2m'}\\
    &\times f^{\text{big-cell}}_p\left(\frac{3-k}{2},(\xi_0\tau_0)^{-c}\right)\left(g\begin{psm}\bid_4\\&p^{-m-m'}\cdot\bid_2\end{psm}\right),\\
    =&\,(\xi\tau)_\fp(p)^{m'}\left((\xi\tau)_{\bar{\fp}}(p)p^k\right)^{m''}\cdot (\xi\tau)_{\bar{\fp}}(\xi\tau)^{-1}_{{\fp}}(p)^{-2m-2m'}\\
    &\times f^{\text{big-cell}}_p\left(\frac{3-k}{2},(\xi_0\tau_0)^{-c}\right)\left(g\begin{psm}\bid_4\\&p^{-m-m'}\cdot\bid_2\end{psm}\right)
\end{align*}
and the corresponding Siegel Eisenstein series paired with $\varphi\cdot (\xi_0\tau_0)^{-1}\circ\det$ over $\U(2)$ gives 
\begin{align*}
    (\xi\tau)_\fp(p)^{m'}\left((\xi\tau)_{\bar{\fp}}(p)p^k\right)^{m''}\cdot \text{the automorphic form given in \eqref{eq:Kling-int}}.
\end{align*}
Here we use the condition that $\varphi$ has trivial central character. Both $(\xi\tau)_\fp(p)$ and $(\xi\tau)_{\bar{\fp}}(p)p^k$ are $p$-adic units, so the form \eqref{eq:Kling-int} is an eigenvector for the action of  $(U^+_{p,2})^{m}(U^+_{p,3})^{m'}(U^-_{p,1})^{m''}$ with eigenvalue a $p$-adic unit. This shows that the evaluations of the Klingen Eisenstein family $\bfE^\Kling_\varphi$ at all $\tau$ satisfying the conditions in Theorem~\ref{thm:Sieg-F} with $(\xi\tau)_\fp$ ramified are semi-ordinary. Hence,  $\bfE^\Kling_\varphi$ is semi-ordinary. 
\end{proof}

\section{The degenerate Fourier--Jacobi coefficients of the Klingen family}\label{sec:degFJ}
This section is devoted to proving Theorem~\ref{thm:const}, which states that the degenerate Fourier--Jacobi coefficients of $\bfE^\Kling_\varphi$ are divisible by the $p$-adic $L$-function attached to $\mr{BC}(\pi)$.

\subsection{The divisibility of the degenerate Fourier--Jacobi coefficients by $p$-adic $L$-functions}

Given a cusp label $\clabel\in C(K^p_f)$, we have the map 
\[
   \Phi_g: \Meas\left(\Gamma_\cK,e_\so V_{\GU(3,1),\,\xi}\wh{\otimes}\,\hat{\cO}^\ur_L\right)^\natural\lra \Meas\left(\Gamma_\cK,M_{\GU(2)}\left(K^p_{f,g}K_{p,0};\hat{\cO}^\ur_L\right)\right),
\]
(which corresponds to the map $\Phi_\clabel$ in the fundamental exact sequence in part (4) of Theorem~\ref{prop:main}). Here $K^p_{f,g}$ is an open subgroup of $\GU(2)(\hat{\bZ}^p)$ depending on the cusp label $g\in C(K^p_f)$. Applying it to our Klingen Eisenstein family $\bfE^\Kling_\varphi$, we get
\[
   \Phi_{g}\left(\bfE^\Kling_\varphi\right)\in \Meas\left(\Gamma_\cK,M_{\GU(2)}\left(K^p_{f,g}K_{p,0};\hat{\cO}^\ur_L\right)\right)
\]
Evaluating it at $g'\in \GU(2)(\bA_{\bQ,f})$ gives
\[
   \Phi_{g}\left(\bfE^\Kling_\varphi\right)(g')\in\Meas\left(\Gamma_\cK,\hat{\cO}^\ur_L\right)\simeq \hat{\cO}^\ur_L\llb \Gamma_\cK\rrb.
\]

\vspace{.5em}

By the work of Kubota--Leopoldt on  $p$-adic $L$-functions for Dirichlet characters, there exists a (unique) $p$-adic $L$-function $\gls{Lxi}\in \hat{\cO}^\ur_L\llb \Gamma_\cK\rrb$ satisfying the interpolation property:
\begin{align}
   \label{eq:pL-xi} \cL^{\Sigma\,\cup\{\ell,\ell'\}}_{\xi,\bQ}(\tau_{p\adic})
   =&\,\frac{\Gamma(k-2)}{(2\pi i)^{k-2}}
   \cdot \gamma_p\left(3-k,(\xi^\bQ_0\tau^\bQ_{0})^{-1}\right) \cdot L^{\Sigma\,\cup\{\infty,p,\ell,\ell'\}}\left(k-2,\xi^\bQ_0\tau^\bQ_0\right)
\end{align}
for all algebraic Hecke character $\tau:\cK^\times\backslash\bA^\times_\cK\ra\bC^\times$ such that $\tau_{p\adic}$ factors through $\Gamma_\cK$ and $\tau^{\bQ}_\infty=\sgn^k|\cdot|^k_{\bR}$. 

By \cite{EisWan} (and the archimedean computation in \cite{AZIU}), there exists a (unique) $p$-adic $L$-function   
$\gls{Lpipartial}\in \hat{\cO}^\ur_L\llb \Gamma_\cK\rrb$ satisfying the interpolation property: for all algebraic Hecke character $\tau:\cK^\times\backslash\bA^\times_\cK\ra\bC^\times$ such that $\tau_{p\adic}$ factors through $\Gamma_\cK$ and $\xi\tau$ has $\infty$-type $\left(k_1,k_2\right)$ with $k_1,k_2\in\bZ$, $k_1\leq 0$, $k_2\geq 2-k_1$, 
\begin{equation}\label{eq:pL-pi} 
\begin{aligned}
   &\cL^{\Sigma\,\cup\{\ell,\ell'\}}_{\pi,\cK,\xi}(\tau_{p\adic})\\
   =&\,C^{\Sigma\,\cup\{\ell,\ell'\}}_{\pi,\xi}\cdot  \left(\frac{\Omega_p}{\Omega_\infty}\right)^{2(k_2-k_1)}\frac{\Gamma(k_2)\Gamma(k_2-1)}{(2\pi i)^{2k_2-1}}
   \cdot \gamma_p\left(\frac{3-(k_1+k_2)}{2},\pi^\vee_{p}\times (\xi_0\tau_0)^{-1}_{\bar{\fp}}\right)\\
    &\times L^{\Sigma\,\cup\{\infty,p,\ell,\ell'\}}\left(\frac{k_1+k_2-1}{2},\mr{BC}(\pi)\times\xi_0\tau_0\right),
\end{aligned}
\end{equation}
with $C^{\Sigma\,\cup\{\ell,\ell'\}}_{\pi,\xi}$ a nonzero number independent of $\tau_{p\adic}$.

\begin{thm}\label{thm:const}
Suppose that $\varphi$ belongs to the space \eqref{eq:M-varphi}. Then for all cusp labels $g\in C(K^p_f)$ and $g'\in \GU(2)(\bA_{\bQ,f})$,  we have
\[
    \left(\Phi_{g}\left(\bfE^\Kling_\varphi\right)(g')\right)\subset \left(\cL^{\Sigma\,\cup\{\ell,\ell'\}}_{\xi,\bQ}\cL^{\Sigma\,\cup\{\ell,\ell'\}}_{\pi,\cK,\xi}\right)
\]
in $\hat{\cO}^\ur_L\llb \Gamma_\cK\rrb$.
\end{thm}

\begin{proof}[Proof of Theorem~\ref{thm:const}.]
For $\tau$ as in Theorem~\ref{thm:Sieg-F}, let $E^\Kling_{\varphi,\xi\tau}$ be the classical Klingen Eisenstein series whose corresponding $p$-adic form is $\bfE^\Kling_\varphi(\tau_{p\adic})$. The theorem is proved by computing $\Phi_{g}\left(\bfE^\Kling_\varphi\right)(\tau_{p\adic})$ with 
\begin{align*}
   g&=k_{\Sigma,\ell,\ell'}=\bigotimes\limits_{v\in\Sigma\,\cup\{\ell,\ell'\}}k_v, &k_v\in \U(3,1)(\bZ_v),
\end{align*}
which is essentially the $0$-th Fourier--Jacobi coefficient of the $E^\Kling_{\varphi,\xi\tau}$ at 
\[
   x=\begin{psm}1\\&g'\\&&\nu(g')\end{psm}k_{\Sigma_f,\ell,\ell'}\,w_{3,p},
\] 
with $w_{3,p}\in\U(3,1)(\bQ_p)$ defined by
\[
   \varrho_{\fp}(w_{3,p})=\begin{pmatrix}&&1\\1\\&1\\&&&1\end{pmatrix}.
\] 
Note that $w_{3,p}\in\U(3,1)(\bQ_p)$ is also our fixed element in $P^\flat_{\tt D}(\bZ_p)$ for obtaining the isomorphism \eqref{eq:w-isom}).

By the description of $\bfE^\Kling_\varphi(\tau_{p\adic})$ in \eqref{eq:Kling-int} and the doubling method formula (Theorem~\ref{thm:doubling}), we have
\[
   E^\Kling_{\varphi,\xi\tau}(x)=\begin{array}{ll}\text{normalization factor}\\\text{on the RHS of \eqref{eq:normalization}}\end{array}\cdot E^\Kling\big(x;F(f(s,\xi_0\tau_0),\varphi)\big)\,\Big|_{s=\frac{k-3}{2}}
\]
with $F(f(s,\xi_0\tau_0),\varphi)$ the section in $I_{P_{\GU(3,1)}}(s,\xi_0\tau_0)$ given by
\begin{align*}
   F(f(s,\tau_0),\varphi)(x)&=\int_{\U(2)(\bA_\bQ)}f(s,\xi_0\tau_0)(\imath(x,g_1t))\varphi(g_1t)(\xi_0\tau_0)^{-1}(\det g_1t)\,dg_1\\
   &=\int_{\U(2)(\bA_\bQ)}f(s,\xi_0\tau_0)\left(\imath\left(k_{\Sigma_f,\ell,\ell'}\,w_{3,p},g_1\right)\right)\varphi(g'g_1)(\xi_0\tau_0)^{-1}(\det g_1)\,dg_1,
\end{align*}
where $t$ is an element in $\GU(2)(\bA_\bQ)$ with $\nu(t)=\nu(x)$. What we need to compute is
\begin{equation}\label{eq:const1}
  \int_{U_{P_{\GU(3,1)}}(\bQ)\big\backslash U_{P_{\GU(3,1)}}(\bA_\bQ)} E^\Kling\big(ux;F(f(s,\xi_0\tau_0),\varphi)\big)\,du.
\end{equation}
It follows from \cite[II.1.7]{MWEis} that 
\begin{align*}
   \eqref{eq:const1}=F(f(s,\xi_0\tau_0),\varphi)(x)+M_{P_{\GU(3,1)}}(s,\xi_0\tau_0)F(f(s,\xi_0\tau_0),\varphi)(x).
\end{align*}
The archimedean computation in \cite[Corollary 5.11]{WanL} shows that for our choice of $f_\infty(s,\xi_0\tau_0)$ in \S\ref{sec:sec-choice} and $k\geq 6$,
\begin{equation}\label{eq:van-M}
   \left.M_{P_{\GU(3,1)}}(s,\xi_0\tau_0)F(f(s,\xi_0\tau_0),\varphi)\right|_{s=\frac{k-3}{2}}=0.
\end{equation}

We reduce to compute $F(f(s,\xi_0\tau_0),\varphi)(x)$. Take arbitrary $\varphi'\in \bar{\pi}^{\GU(2)}$ spherical away from $\Sigma\,\cup\{\ell,\ell'\}$ and factorizable with respect to an isomorphism $\bar{\pi}^{\GU(2)}\simeq \bigotimes_v \bar{\pi}^{\GU(2)}_{v}$. Put
\begin{align*}
   &\left<F(f(s,\xi_0\tau_0),\varphi)(\bdot\, k_{\Sigma_f,\ell,\ell'}\,w_{3,p}),\varphi'\right>\\
   =&\, \int_{\GU(2)(\bQ)Z_{\GU(2)}(\bA_\bQ)\backslash\GU(2)(\bA_\bQ)}F(f(s,\xi_0\tau_0),\varphi)(g'\,k_{\Sigma,\ell,\ell'}\,w_{3,p})\,\varphi'(g')\,dg'.
\end{align*}
Then we have
\begin{align*}
   \left<F(f(s,\xi_0\tau_0),\varphi)(\bdot\,k_{\Sigma,\ell,\ell'}\,w_{3,p}),\,\varphi'\right>=&\,\prod_{v\not\in\Sigma\,\cup\{p,\ell,\ell'\}}  Z_v\left(f_v(s,\xi_0\tau_0),\bid_4,\varphi_{v},\varphi'_v\right)\\
   &\times \prod_{v\in\Sigma\,\cup\{\ell,\ell'\}} 
   Z_v\left(f_v(s,\xi_0\tau_0),k_v,\varphi_{v},\varphi'_v\right) 
   \cdot Z_p\left(f_p(s,\xi_0\tau_0),w_{3,p},\varphi_{p},\varphi'_p\right),
\end{align*}
where for $g_v\in\GU(3,1)(\bQ_v)$,
\begin{equation}\label{eq:Zv}
\begin{aligned}
   &Z_v\left(f_v(s,\xi_0\tau_0),g_v,\varphi_{v},\varphi'_v\right)\\
  =&\,\int_{\U(2)(\bQ_v)}f_v(s,\xi_0\tau_0)(\imath(g_v,g_1))(\xi_0\tau_0)^{-1}_{v}(\det g_1)\left<\pi^{\GU(2)}_v(g_1)\varphi_{v},\varphi'_v\right>\,dg_1.
\end{aligned}
\end{equation}

The restriction of $f_v(s,\xi_0\tau_0)(\bdot\,\imath(g_v,\bid_2))$ to $\U(2,2)(\bQ_v)$ is a section in the degenerate principal series of $\U(2,2)(\bQ_v)$ inducing $\xi_0\tau_0|\cdot|^{s+\frac{1}{2}}_{\bA_\cK}$. Hence the integral \eqref{eq:Zv} is essentially a doubling zeta integral. We have the following formulas.

\vspace{.5em}

\item[--] \emph{The archimedean place.}
\begin{align*}
   Z_\infty\left(f_\infty(s,\xi_0\tau_0),\bid_4,\varphi_\infty,\varphi'_\infty\right)
   =\left<\varphi_{\infty},\varphi'_\infty\right>,
\end{align*}
because $f_\infty(s,\xi_0\tau_0)(\imath(\bid_4,\bdot)$, $\varphi_{\infty}$, $\varphi'_\infty$ are all invariant under $\U(2)(\bR)$.

\vspace{.5em}

\item[--] $v\notin \Sigma\,\cup\{\ell,\ell',p\}$.
\begin{align*}
   Z_v\left(f_v(s,\xi_0\tau_0),\bid_4,\varphi_{v},\varphi'_v\right)
   =d_{2,v}\left(s+\frac{1}{2},\xi_0\tau_0\right)^{-1}\cdot L_v\left(s+1,\mr{BC}(\pi)\times\xi_0\tau_{0}\right)
\end{align*}
with
\[
   \gls{d2v}=\prod_{j=1}^2 L_v(2s+1+j,\xi^\bQ_0\tau^\bQ_{0}\eta^{2-j}_{\cK/\bQ}).
\]
This follows from the standard formula for unramified local doubling zeta integrals \cite[Proposition 3, Remark 3]{LapidRallis}.

\vspace{.5em}

\item[--] $v\in\Sigma\,\cup\{\ell,\ell'\}$. By \cite[Lemma 6.20]{WanU31} or an easy direct computation,  one can see that for $\varphi$ in the space \eqref{eq:M-varphi}, 
\[
   Z_v\left(f_v(s,\xi_0\tau_0),k_v,\varphi_v,\varphi'_v\right)=0, \quad\text{if }k_v\notin P_{\U(3,1)}(\bZ_v)\begin{pmatrix}&&1\\&\bid_2\\-1\end{pmatrix} P_{\U(3,1)}(\bZ_v),
\]
and for $k_v$ such that $Z_v\left(f_v(s,\xi_0\tau_0),k_v,\varphi_v,\varphi'_v\right)\neq 0$, 
\begin{align*}
   Z_v\left(f_v(s,\xi_0\tau_0),k_v,\varphi_{v},\varphi'_v\right)
   =\vol(\fY_{c_v})\cdot\left((\xi^\bQ_0\tau^\bQ_0)_v(q_v)|q_v|^{s+\frac{3}{2}}_v\right)^{-c_v}\left<\varphi_{v},\varphi'_v\right>,
\end{align*}
where $\fY_{c_v}$ is an open subgroup of $\{k'_v\in\U(2)(\bZ_v):k'_v\equiv \bid_2\mod q^{c_v}_v\}$, and $c_v$ is the positive integer in the definition of $f_v(s,\xi_0\tau_0)$ in \eqref{eq:fv-ram}.

\vspace{.5em}
\item[--] \emph{The place $p$.} 
\begin{equation}\label{eq:Zp-formula}
\begin{aligned}
    &Z_p\left(f_p(s,\xi_0\tau_0),w_{3,p},\varphi_p,\varphi'_p\right)\\
    =&\,\gamma_p\left(-2s,(\xi^\bQ_0\tau^\bQ_{0})^{-1}\right) 
   \gamma_p\left(-s,\pi_p\times(\xi_0\tau_0)^{-1}_{\bar{\fp}}\right)\left<\varphi_{p},\varphi'_p\right>
\end{aligned}
\end{equation}
See \S\ref{sec:Z_p} for the computation. Note that we only need the value at $w_{3,p}\in\GU(3,1)(\bQ_p)$ rather than a general $k_p\in \GU(3,1)(\bZ_p)$. This significantly simplifies the computation. The value at a general $k_p\in \GU(3,1)(\bZ_p)$ can be difficult to compute since intertwining operators are not easy to compute completely. 

\vspace{1em}

Combing the local formulas, since $\varphi'$ is arbitrary, we see that for $\tau$ as in Theorem~\ref{thm:Sieg-F},  $F(f(s,\xi_0\tau_0),\varphi)(\bdot\, k_{\Sigma,\ell,\ell'}\,w_{3,p})\big|_{s=\frac{k-3}{2}}$ is either $0$ or equals 
\begin{align*}
   &d^{\,\Sigma\,\cup\{\infty,p,\ell,\ell'\}}_3(s,\xi_0\tau_0)^{-1}
    \cdot C_{\Sigma\cup\{\ell,\ell'\}}\cdot  \gamma_p\left(-2s,(\xi^\bQ_0\tau^\bQ_{0})^{-1}\right) 
    L^{\Sigma\,\cup\{\infty,p,\ell,\ell'\}}\left(2s+1,\xi^\bQ_0\tau^\bQ_0\right)\\
    &\times \gamma_p\left(-s,\pi_p\times(\xi_0\tau_0)^{-1}_{\bar{\fp}}\right)
     L^{\Sigma\,\cup\{\infty,\ell,\ell',p\}}\left(s+1,\mr{BC}(\pi)\times\xi_0\tau_0\right)\Big|_{s=\frac{k-3}{2}}\cdot \varphi
\end{align*}
for some $C_{\Sigma\cup\{\ell,\ell'\}}\in \cO_L$ independent of $\tau$. Combining this formula with the normalization ~\eqref{eq:normalization}, we see that 
$\Phi_{g}\left(\bfE^\Kling_\varphi\right)(g')$, for all $g'\in\GU(2)(\bA_{\bQ,f})$, is divisible by an element in $\hat{\cO}^\ur_L\llb\Gamma_\cK\rrb$ whose value equals that of $\cL^{\Sigma\,\cup\{\ell,\ell'\}}_{\pi,\xi,\cK}\cL^{\Sigma\,\cup\{\ell,\ell'\}}_{\xi,\bQ}$ at  all $\tau_{p\adic}$ as in Theorem~\ref{thm:Sieg-F}. Evaluations at those $\tau_{p\adic}$'s  already uniquely determine elements in  $\hat{\cO}^\ur_L\llb\Gamma_\cK\rrb$. Therefore, it follows that $\Phi_{g}\left(\bfE^\Kling_\varphi\right)(g')$ is divisible by $\cL^{\Sigma\,\cup\{\ell,\ell'\}}_{\xi,\bQ}\cL^{\Sigma\,\cup\{\ell,\ell'\}}_{\pi,\cK,\xi}$.

\end{proof}

\subsection{The computation of local zeta integrals at $p$}\label{sec:Z_p}

The ramification conditions in \cite[Definition 6.30]{WanU31} on $\xi_{0,p}\tau_{0,p}$ and the characters from which $\pi_p$ is induced are not satisfied in our case here because our $\pi_p$ is unramified. Those conditions are used  {\it loc.cit} to simplify the computation involving the intertwining operators. Here, we use the functional equation for doubling zeta integrals to handle the intertwining operator and compute $Z_p\left(f_p(s,\xi_0\tau_0),w_{3,p},\varphi_p,\varphi'_p\right)$.

\begin{proof}[Proof of \eqref{eq:Zp-formula}.]
Inside the degenerated principal series on $\GU(2,2)(\bQ_p)$ inducing the character $(\xi_0\tau_0)^{-c}_{p}|\cdot|^{-s}$, we define the section $f_p^{\U(2,2)}(-s,(\xi_0\tau_0)^{-c})$  as
\begin{align*}
   f^{\U(2,2)}_{p}(-s,(\xi_0\tau_0)^{-c})\left(g_2=\begin{pmatrix}A&B\\C&D\end{pmatrix}\right)
   = |\det C\ltrans{\bar{C}}|^{s-1}_p(\xi_0\tau_0)_p\left(\det C\right)\cdot \cF^{-1}\alpha_{\xi\tau,p}\begin{pmatrix}0&C^{-1}D\\0&0\end{pmatrix},
\end{align*}
where the Schwartz function $\alpha_{\xi\tau,p}$ is defined in \eqref{eq:inv-F}. Then we have
\begin{equation}\label{eq:fU22}
\begin{aligned}
   &\int_{y_1\in\bQ_p,y_2\in\cK^2_p}f^{\text{big-cell}}_p(-s,(\xi_0\tau_0)^{-c})\left( 
   \begin{psm}&&-1\\&\bid_2\\1\\&&&1\end{psm}
   \begin{psm}1&-y_2&y_1&\\&\bid_2&\\&&\ltrans{\bar{y}}_2&1\end{psm}
    \jmath(g_2)\Upsilon^{-1}_p\imath(w_{3,p},\bid_2)\Upsilon_p
   \right)\,dy_1dy_2\\
   =&\,f^{\U(2,2)}_p\left(-s-\frac{1}{2},(\xi_0\tau_0)^{-c}\right)(g_2),
\end{aligned}
\end{equation}
where $\jmath:\U(2,2)\ra\U(3,3)$ is defined as 
\[
   \jmath\begin{pmatrix}A&B\\C&D\end{pmatrix}=\begin{pmatrix}1\\&A&&B\\&&1\\&C&&D\end{pmatrix}.
\]
To see \eqref{eq:fU22}, writing $g_2=\begin{pmatrix}A&B\\C&D\end{pmatrix}$, we have
\begin{align*}
    &\varrho_\fp\left(\begin{psm}&&-1\\&\bid_2\\1\\&&&1\end{psm}
   \begin{psm}1&-y_2&y_1&\\&\bid_2&\\&&\ltrans{\bar{y}}_2&1\end{psm}
    \jmath(g_2)\Upsilon^{-1}_p\imath(w_{3,p},\bid_2)\Upsilon_p\right)\\
    =&\,\varrho_\fp\begin{pmatrix}0&0&-1&0\\0&A&0&B\\1&-y_2A&y_1&-y_2B\\0&C&\ltrans{\bar{y}}_2&D\end{pmatrix}\begin{pmatrix}&1\\\bid_2\\&&1\\&&&\bid_2\end{pmatrix}
    =\varrho_\fp\begin{pmatrix}0&0&-1&0\\A&0&0&B\\-y_2A&1&y_1&-y_2B\\C&0&\ltrans{\bar{y}}_2&D\end{pmatrix}
\end{align*}
and  
\begin{align*}
   \varrho_\fp\left(\begin{pmatrix}-y_2A&1\\C&0\end{pmatrix}\begin{pmatrix}^{-1}y_1&-y_2B\\\ltrans{\bar{y}}_2&D\end{pmatrix}\right)
   =\varrho_\fp\begin{pmatrix}C^{-1}\ltrans{\bar{y}}_2&C^{-1}D\\y_1+y_2AC^{-1}\ltrans{\bar{y}}_2&y_2\ltrans{\bar{C}}^{-1}\end{pmatrix}.
\end{align*} 
It follows that the integrand of the left hand side of \eqref{eq:fU22} equals
\begin{align*}
   f^{\U(2,2)}_p\left(-s+\frac{1}{2},(\xi_0\tau_0)^{-c}\right)(g_2)\cdot\mathds{1}_{\bZ_p}(y_1)\cdot \mathds{1}_{\cO^2_{\cK,p}}(y_2\ltrans{\bar{C}}^{-1}),
\end{align*}
so 
\[  
   \text{LHS of \eqref{eq:fU22}}=f^{\U(2,2)}_p\left(-s+\frac{1}{2},(\xi_0\tau_0)^{-c}\right)(g_2)\cdot\det|C\ltrans{\bar{C}}|_p=f^{\U(2,2)}_p\left(-s-\frac{1}{2},(\xi_0\tau_0)^{-c}\right)(g_2).
\]
Put $\cS_0=\begin{pmatrix}1&-\frac{\zeta_0}{2}\\-1&-\frac{\zeta_0}{2}\end{pmatrix}$ and $\Upsilon_{0,p}$ to be the element in $\U(2,2)(\bQ_p)$ such that $\varrho_\fp(\Upsilon_{0,p})=\cS^{-1}_0$.  Set 
\begin{align}
   \label{eq:imath0}\imath_0:\U(2)\times\U(2)&\lra \U(2,2), &\imath_0(g_1,g'_1)=\cS^{-1}_0\begin{pmatrix}g_1\\&g'_1\end{pmatrix}.
\end{align} 
By the definition of the intertwining operator $M_p(-s,\xi\tau^{-1}_0(\xi_0\tau_0)^{-c})$ (in \eqref{eq:intw}) and \eqref{eq:fU22}, 
\begin{align*}
   &M_p(-s,(\xi_0\tau_0)^{-c})f^{\text{big-cell}}_p(-s,(\xi_0\tau_0)^{-c})\left(\imath(w_{3,p},g_1)\right)\\
   =&\int_{\begin{psm}y_1&y_2\\ \ltrans{\bar{y}}_2&y_4\end{psm}\in\Her_3(\cK_p)} 
   f^{\text{big-cell}}_p(-s,(\xi_0\tau_0)^{-c})\left(
   \begin{psm}&&-1\\&\bid_2\\1\\&&&1\end{psm}
   \begin{psm}1&-y_2&y_1&\\&\bid_2&\\&&\ltrans{\bar{y}}_2&1\end{psm}\right.\\
   &\hspace{12em}\left.\begin{psm}1\\&&&-\bid_2\\&&1\\&\bid_2\end{psm}\begin{psm}1\\&\bid_2&&y_4\\&&1\\&&&\bid_2\end{psm}
   \imath(w_{3,p},g_1)\Upsilon_p
   \right)\,dudy_4\\
   =&\,M^{\U(2,2)}_p\left(-s-\frac{1}{2},(\xi_0\tau_0)^{-c}\right) f^{\U(2,2)}_p(\left(-s-\frac{1}{2},(\xi_0\tau_0)^{-c}\right)\left(\imath_0(1,g_1)\Upsilon_{0,p}\right),
\end{align*}
and
\begin{align*}
   f_p(s,\xi_0\tau_0)\left(\imath(w_{3,p},g_1)\right)=&\,\gamma_p\left(-2s,(\xi^\bQ_0\tau^\bQ_{0})^{-1}\right)\gamma_p\left(-2s-1,(\xi^\bQ_0\tau^{\bQ}_{0})^{-1}\eta_{\cK/\bQ}\right)\gamma_p\left(-2s-2,(\xi^\bQ_0\tau^{\bQ}_{0})^{-1}\right)\\
   &\times M^{\U(2,2)}_p\left(-s-\frac{1}{2},(\xi_0\tau_0)^{-c}\right) f^{\U(2,2)}_p(\left(-s-\frac{1}{2},(\xi_0\tau_0)^{-c}\right)\left(\imath_0(1,g_1)\Upsilon_{0,p}\right).
\end{align*}
By the functional equation of the local doubling zeta integral \cite[(19)(25)]{LapidRallis}, 
\begin{align*}
   &\int_{\U(2)(\bQ_p)}M^{\U(2,2)}_p\left(-s-\frac{1}{2},(\xi_0\tau_0)^{-c}\right) f^{\U(2,2)}_p\left(-s-\frac{1}{2},(\xi_0\tau_0)^{-c}\right)\left(\imath_0(1,g_1)\Upsilon_{0,p}\right)\\
   &\hspace{15em}\times(\xi_0\tau_0)^c_{p}(\det g_1)\left<\pi^{\GU(2)}_p(g_1)\varphi_{p},\varphi'_p\right>\,dg_1\\
   =&\,c_{2,p}\left(-s-\frac{1}{2},(\xi_0\tau_0)^{-c},\begin{psm}&\bid_2\\\bid_2\end{psm}\right)\gamma_p\left(-s,\mr{BC}(\pi)\times (\xi_0\tau_0)^{-c}\right)\\
   &\times\int_{\U(2)(\bQ_p)} f^{\U(2,2)}_p\left(-s-\frac{1}{2},(\xi_0\tau_0)^{-c}\right)\left(\imath_0(1,g_1)\Upsilon_{0,p}\right)(\xi_0\tau_0)^{c}_{p}(\det g_1)\left<\pi^{\GU(2)}_p(g_1)\varphi_{p},\varphi'_p\right>\,dg_1,
\end{align*}
with the factor
\[
   c_{2,p}\left(-s-\frac{-1}{2},(\xi_0\tau_0)^{-c},\begin{psm}&\bid_2\\\bid_2\end{psm}\right)=\gamma_p\left(-2s-1,(\xi^\bQ_0\tau^{\bQ}_0)^{-1}\eta_{\cK/\bQ}\right)^{-1}\gamma_p\left(-2s-2,(\xi^\bQ_0\tau^{\bQ}_0)^{-1}\right)^{-1}
\]
as defined in \eqref{eq:cns}. Therefore,
\begin{equation}\label{eq:intU22}
\begin{aligned}
   Z_p\left(f_p(s,\xi_0\tau_0),w_{3,p},\varphi_p,\varphi'_p\right)
   =&\,\gamma_p\left(-2s,(\xi^\bQ_0\tau^\bQ_{0})^{-1}\right)\gamma_p\left(-s,\mr{BC}(\pi)\times (\xi_0\tau_0)^{-c}\right)\\
   &\hspace{-11em}\times\int_{\U(2)(\bQ_p)} f^{\U(2,2)}_p\left(-s-\frac{1}{2},(\xi_0\tau_0)^{-c}\right)\left(\imath_0(1,g_1)\Upsilon_{0,p}\right)(\xi_0\tau_0)^{c}_{p}(\det g_1)\left<\pi^{\GU(2)}_p(g_1)\varphi_{p},\varphi'_p\right>\,dg_1.
\end{aligned}
\end{equation}
By noting that
\[
  f^{\U(2,2)}_p\left(-s-\frac{1}{2},(\xi_0\tau_0)^{-c}\right)\left(\imath_0(1,g_1)\Upsilon_{0,p}\right)=\cF^{-1}\alpha_{\xi\tau,p}\begin{pmatrix}0&g_1\\0&0\end{pmatrix}
\]
and using the formula \eqref{eq:inv-F} for $\cF^{-1}\alpha_{\xi\tau,p}$ and that $\varphi_p$ is spherical, an easy computation shows that 
\begin{equation}\label{eq:Zp1}
   \text{the integral on the RHS of \eqref{eq:intU22}}
   =p^{-2ts-2t}\fg\left((\xi_0\tau_0)^{-1}_{\fp}\right)^2(\xi_0\tau_0)_\fp(p)^{2t}\cdot \left<\varphi_{p},\varphi'_p\right>,
\end{equation}   
where $p^t$, $t\geq 1$, is the conductor of $(\xi_0\tau_0)_\fp$. Since $\pi_p$ is unramified, we have
\begin{equation}\label{eq:Zp2}
   p^{-2ts-2t}\fg\left((\xi_0\tau_0)^{-1}_{\fp}\right)^2(\xi_0\tau_0)_\fp(p)^{2t}=\gamma_p\left(s+1,\pi_p\times(\xi_0\tau_0)_\fp\right)
\end{equation}
Combining \eqref{eq:intU22}, \eqref{eq:Zp1} and \eqref{eq:Zp2}, we get
\begin{align*}
   \eqref{eq:intU22}&=\gamma_p\left(-2s,(\xi^\bQ_0\tau^\bQ_{0})^{-1}\right)\gamma_p\left(-s,\mr{BC}(\pi)\times (\xi_0\tau_0)^{-c}\right)
   \gamma_p\left(s+1,\pi_p\times(\xi_0\tau_0)_\fp\right)\cdot \left<\varphi_{p},\varphi'_p\right>\\
   &=\gamma_p\left(-2s,(\xi^\bQ_0\tau^\bQ_{0})^{-1}\right) 
   \gamma_p\left(-s,\pi_p\times(\xi_0\tau_0)^{-1}_{\bar{\fp}}\right)\cdot \left<\varphi_{p},\varphi'_p\right>.
\end{align*}
\end{proof}


\section{The non-degenerate Fourier--Jacobi coefficients of the Klingen family}\label{sec:nondegFJ}

The map \eqref{eq:FJ-V} of taking the $\beta$-th Fourier--jacobi coefficient along the boundary stratum indexed by the cusp label $\bid_4\in C(K^p_fK^1_{p,n})_\ord$ induces
\[
   \FJ_{\beta}:\Meas\left(\Gamma_\cK, V_{\GU(3,1)}\right)\lra \Meas\left(\Gamma_\cK,V^{J,\beta}_{\GU(2)}\right).
\]
For $u\in \bigotimes_{v\in\Sigma_{\ns}}\U(1)(\bQ_v)$, let 
\begin{equation}\label{eq:E-betau}
   \gls{Ebeta}=\FJ_{\beta} \left(\begin{psm}u\\&\bid_2\\&&u\end{psm}\bfE^\Kling_\varphi\right)\in \Meas\left(\Gamma_\cK,V^{J,\beta}_{\GU(2)}\wh{\otimes}\,\hat{\cO}^\ur_L\right)
\end{equation}
(The group $\GU(3,1)(\bA^p_f)$ acts on the Igusa tower and acts on $V_{\GU(3,1)}$, and $\begin{psm}u\\&\bid_2\\&&u\end{psm}\bfE^\Kling_\varphi$ denotes the action of $\begin{psm}u\\&\bid_2\\&&u\end{psm}$ on $\bfE^\Kling_\varphi$.) The goal of this section is to prove Proposition~\ref{prop:E-Kling-nv} on the nonvanishing properties of the $\bfE^\Kling_{\varphi,\beta,u}$'s.

\vspace{.5em}
In \S\ref{sec:Schro}-\S\ref{sec:Jforms}, we briefly recall some basics we will use about Weil representations, theta series, Jacobi forms and $p$-adic Petersson inner products. In \S\ref{sec:unfolding} and \S\ref{sec:strategy}, we unfold the Siegel Eisenstein series on $\GU(3,3)$ to compute the non-degenerate Fourier--Jacobi coefficients of the Klingen Eisenstein series on $\GU(3,1)$, and sketch our strategy for relating them to certain $L$-values, for which we can apply mod $p$ nonvanishing results. Following that strategy, we need to choose an auxiliary Jacobi form on $\U(2)$ and two CM families on $\U(2)$. \S\ref{sec:theta_1J}-\S\ref{sec:ext-GU2} are about constructing the auxiliary Jacobi form and CM families and some other technical preparations. The desired nonvanishing property is proved in \S\ref{nv-EKling}.

\subsection{The Schr{\"o}dinger model of Weil representation and the intertwining maps for different polarizations}\label{sec:Schro}

The non-degenerate terms in the Fourier--Jacobi expansion of an automorphic form on $\GU(3,1)(\bA)$ are Jacobi forms on the Jacobi group associated $\GU(2)$. 

Let $V$ (resp. $V^-$) be the two dimensional skew-Hermitian space $(\cK^2,\zeta_0)$ (resp. $(\cK^2,-\zeta_0)$).  We write elements in $V,V^-$ as row vectors and fix the basis $v_1=(1,0),v_2=(0,1)$ for $V$ (resp. $v^-_1=(1,0),v^-_2=(0,1)$ for $V^-$). With respect to the fixed basis, the unitary groups $\U(V)$ and $\U(V^-)$ are naturally identified with our $\U(2)$ (defined in \eqref{eq:U(2)}).  

We fix the basis $v_1,v_2,v^-_1,v^-_2$ for $V\oplus V^-$ and identify $\U(V\oplus V^-)$ with the unitary group whose $R$-points are
\[
    \left\{g\in\GL_4(R\otimes_{\bZ}\cO_\cK):\ltrans{\ol{g}}\begin{pmatrix}\zeta_0\\&-\zeta_0\end{pmatrix}g=\begin{pmatrix}\zeta_0\\&-\zeta_0\end{pmatrix}\right\}.
\]
Define the unitary group $\U(2,2)$ as
\begin{align*}
   \U(2,2)(R)&=\left\{g\in\GL_4(R\otimes_{\bZ}\cO_\cK):\ltrans{\ol{g}}\begin{pmatrix}&\bid_2\\-\bid_2\end{pmatrix}g=\begin{pmatrix}&\bid_2\\-\bid_2\end{pmatrix}\right\}.
\end{align*}
We fix the isomorphism 
\begin{align*}
    \U(V\oplus V^-)&\lra \U(2,2), &g&\longmapsto \begin{pmatrix}\bid_2&-\frac{\zeta_0}{2}\\-\bid_2&-\frac{\zeta_0}{2}\end{pmatrix}^{-1}g \begin{pmatrix}\bid_2&-\frac{\zeta_0}{2}\\-\bid_2&-\frac{\zeta_0}{2}\end{pmatrix},
\end{align*}
which induces the embedding 
\begin{align*}
   \U(2)\times\U(2)&\lra\U(2,2), &(g_1,g_2)&\longmapsto  \begin{pmatrix}\bid_2&-\frac{\zeta_0}{2}\\-\bid_2&-\frac{\zeta_0}{2}\end{pmatrix}^{-1}\begin{pmatrix}g_1\\&g_2\end{pmatrix} \begin{pmatrix}\bid_2&-\frac{\zeta_0}{2}\\-\bid_2&-\frac{\zeta_0}{2}\end{pmatrix}.
\end{align*}
For a nonzero $\beta\in\Her_1(\cK)$, denote by $W_\beta$ the one dimensional Hermitian space $(\cK,\beta)$. Fix a basis $w_\beta=1$ for $W_\beta$, and we can identify $\U(W_\beta)$ with the group $\U(1)$ defined by
\[
    \U(1)(R)=\{g\in\GL_1(R\otimes_\bZ\cO_\cK):\bar{g}g=1\}.
\]

We consider the Weil representations for the dual pairs $\U(2)\times\U(1)$ and $\U(2,2)\times\U(1)$. The space $V\otimes_{\bQ} W_\beta$ is a four dimensional $\bQ$-vector space equipped with a non-degenerate symplectic pairing induced from the skew-Hermitian form on $V$ and the Hermitian form on $W_\beta$. The group $\U(V)\times \U(W_\beta)$ embeds into $\Sp(V\otimes_\bQ W_\beta)$ and forms a reductive dual pair. Let $\wt{\Sp}(V\otimes_\bQ W_\beta)$ denote the metaplectic group. Given a pair of Hecke characters $\lambda_V,\lambda_{W_\beta}:\cK^\times\backslash\bA^\times_\cK\ra\bC^\times$ such that $\lambda_V|_{\bA^\times_\bQ}=\eta_{\cK/\bQ}^{\dim_\cK V}$ and $\lambda_{W_\beta}|_{\bA^\times_\bQ}=\eta^{\dim_\cK W_\beta}_{\cK/\bQ}$, there is a splitting 
\begin{equation}\label{eq:splitting}
   \U(V_v)\times \U(W_{\beta,v})\lhra \wt{\Sp}(V_v\otimes_\bQ W_{\beta,v})
\end{equation}
for all places $v$ of $\bQ$. The Weil representation of $\wt{\Sp}(V_v\otimes_\bQ W_{\beta,v})$ (with respect to our fixed additive character $\be_{\bA_\bQ}:\bQ\backslash\bA_\bQ\ra\bC^\times$ in \eqref{eq:bev}) induces the Weil representation of $\U(V_v)\times \U(W_{\beta,v})$. Similarly, with respect to $\lambda_{V^-},\lambda_{V\oplus V^-}:\cK^\times\backslash\bA^\times_\cK\ra\bC^\times$ satisfying $\lambda_{V^-}|_{\bA^\times_\bQ}=\eta_{\cK/\bQ}^{\dim_\cK V^-}$, $\lambda_{V\oplus V^-} |_{\bA^\times_\bQ}=\eta_{\cK/\bQ}^{\dim_\cK V+\dim_\cK V^-}$, we have the Weil representations of $\U(V^-_v)\times U(W_{\beta,v})$ and  $\U(V\oplus V^-_v)\times U(W_{\beta,v})$.

\vspace{.5em}
We choose our $\lambda_V,\lambda_{V^-},\lambda_{W_\beta}$ as
\begin{equation}\label{eq:split-char}
\begin{aligned}
   \lambda_V=\lambda_{V^-}&=\lambda^2,
   &\lambda_{W_\beta}&=\lambda,
\end{aligned}
\end{equation}
where $\lambda$ is the Hecke character fixed in \S\ref{sec:construct-notation}.

\begin{rmk}
In some literature, instead of $\lambda_V=\lambda_{V^-}=\lambda^2$, $\lambda_V=\lambda_{V^-}=\triv$ is used.
\end{rmk}

Put
\begin{align*}
   e_1&=v_1\otimes w_\beta, &e_2&=v_2\otimes w_\beta,
   &e^-_1&=v^-_1\otimes w_\beta, &e^-_2&=v^-_2\otimes w_\beta.
\end{align*}
Then $\bQ e_1\oplus\bQ e_2$ is a maximal isotropic subspace of the symplectic space $V\otimes_\bQ W_\beta$. 
We have the polarizations
\begin{align}
   \label{eq:V-pol}V\otimes W_\beta&=\big(\bQ e_1\oplus\bQ e_2\big)\oplus \big(\bQ \delta e_1\oplus\bQ \delta e_2\big),
\end{align}
(Where $\delta$ is our fixed totally imaginary element in $\cK$.) The Schr{\"o}dinger  model of the Weil representation of $\U(V_v)\times\U(W_{\beta,v})$ is an action of $\U(V_v)\times\U(W_{\beta,v})$ on $\pzS(\bQ_v e_1\oplus\bQ_v e_2)$, the space of Schwartz functions on $\bQ_v e_1\oplus\bQ_v e_2$. We write this action as
\begin{align*}
   \omega_{\beta,v}(g,u): \pzS(\bQ_v e_1\oplus\bQ_v e_2)&\lra \pzS(\bQ_v e_1\oplus\bQ_v e_2),
   &g\in\U(V_v),\,u\in \U(W_{\beta,v}).
\end{align*}
Similarly, we have the polarization
\begin{align}
   \label{eq:V--pol}V^-\otimes W_\beta&=\big(\bQ e^-_1\oplus\bQ e^-_2\big)\oplus \big(\bQ \delta e^-_1\oplus\bQ \delta e^-_2\big).
\end{align}
and the Schr{\"o}dinger model
\begin{align*}
   \omega_{\beta,v}(g,u): \pzS(\bQ_v e^-_1\oplus\bQ_v e^-_2)&\lra \pzS(\bQ_v e^-_1\oplus\bQ_v e^-_2),
   &g\in\U(V^-_v),\,u\in \U(W_{\beta,v}).
\end{align*}
For a place $v=\fv\bar{\fv}$ split in $\cK$, we have $\cK_v\simeq \cK_\fv\times\cK_{\bar{\fv}}$ with $\cK_\fv\simeq\cK_{\bar{\fv}}\simeq \bQ_v$. In addition to the above polarizations, we also consider the polarizations
\begin{align*}
   V_v\otimes W_{\beta,v}&=\big(\cK_\fv e_1\oplus\cK_\fv e_2\big)\oplus \big(\cK_{\bar{\fv}} e_1\oplus\cK_{\bar{\fv}} e_2\big),\\
   V^-_v\otimes W_{\beta,v}&=\big(\cK_\fv e^-_1\oplus\cK_\fv e^-_2\big)\oplus \big(\cK_{\bar{\fv}} e^-_1\oplus\cK_{\bar{\fv}} e^-_2\big),
\end{align*} 
and the intertwining maps
\begin{align}
   \pzS\big(\cK_\fv e_1\oplus\cK_\fv e_2\big)&\lra \pzS\big(\bQ_v e_1+\bQ_ve_2\big),\\
  \label{eq:int-S} \pzS\big(\cK_\fv e^-_1\oplus\cK_\fv e^-_2\big)&\lra \pzS\big(\bQ_v e^-_1+\bQ_ve^-_2\big).
\end{align}

For $V\oplus V^-$, we have the polarization:
\[
   (V\oplus V^-)\otimes W_\beta=\big(\cK(e_1+e^-_1)\oplus \cK(e_2+e^-_2)\big)\oplus \big(\cK(e_1-e^-_1)\oplus \cK(e_2-e^-_2)\big).
\]
Writing an element in $\cK_v(e_1+e^-_1)\oplus \cK_v(e_2+e^-_2)$ as $X\in M_{1,2}(\cK_v)$, for the Schr{\"o}dinger model
\begin{align*}
   \omega_{\beta,v}(g,u): \pzS\big(\cK(e_1+e^-_1)\oplus \cK(e_2+e^-_2)\big)&\lra \pzS\big(\cK(e_1+e^-_1)\oplus \cK(e_2+e^-_2)\big),
   &g\in\U(2,2),\,u\in \U(W_{\beta,v}),
\end{align*}
we have the following formulas:
\begin{equation}\label{eq:Weil-fmla}
\begin{aligned}
   \omega_{\beta,v}(\bid_4,u)\,\Phi(X)&=\Phi(u^{-1}X),\\
    \omega_{\beta,v}\left(\begin{pmatrix}A\\&\ltrans{\bar{A}}\end{pmatrix},1\right)\Phi(X)&=\lambda_v(\det A)|\det A|^{1/2}_{\cK_v}\cdot \Phi(XA),\\
    \omega_{\beta,v}\left(\begin{pmatrix}\bid_2&B\\&\bid_2\end{pmatrix},1\right)\Phi(X)&=\be_v\left(\beta\cdot XB\ltrans{\bar{X}}\right)\cdot \Phi(X),\\
    \omega_{\beta,v}\left(\begin{pmatrix}&-\bid_2\\\bid_2\end{pmatrix},1\right)\Phi(X)&= |\beta|_{\bQ_v} \int \Phi(Y)\cdot \be_v\left(\Tr_{\cK_v/\bQ_v}(\beta \cdot Y\ltrans{\bar{X}})\right)\,dY.
\end{aligned}
\end{equation}

\subsection{The Heisenberg group and Jacobi forms}\label{sec:Jforms}
Jacobi forms on $\U(2)$ and $\U(2,2)$ show up in our computation of the non-degenerate Fourier--Jacobi coefficients of the evaluations at classical points of the Klingen family $\bfE^\Kling$.

\vspace{.5em}

First, we introduce Jacobi forms on $\U(2)$. Recall that $V$ is the skew Hermitian space $(\cK^2,\zeta_0)$. Denote by $\gls{HV}$ be the Heisenberg group associated to $V$. For a $\bQ$-algebra $R$
\[
   H(V)(R)=(V\otimes_{\bQ}R)\ltimes R,
\] 
and the multiplication is
\[
  (x_1,\sigma_1)(x_2,\sigma_2)=\left(x_1+x_2,\sigma_1+\sigma_2+\frac{x_1\zeta_0 \ltrans{\bar{x}_2}+\ol{x_1\zeta_0 \ltrans{\bar{x}_2}}}{2}\right).
\]
It is easy to see that $H(V)$ is isomorphic to $U_{P_{\GU(3,1)}}$ via
\begin{equation}\label{eq:31HU}
   (x,\sigma)\longmapsto u(x,\sigma)=\begin{pmatrix}
   1&x&\sigma+\frac{1}{2}x\zeta_0\ltrans{\bar{x}}\\&\bid_2&\zeta_0\ltrans{\bar{x}}\\&&1
   \end{pmatrix}.
\end{equation}

The Jacobi group associated to $V$ is the semi-direct product $H(V)\rtimes \U(V)$, which we identify with 
a subgroup of $P_{\GU(3,1)}$ by
\begin{equation}\label{eq:H-to-U31}
\begin{aligned}
   \big((x,\sigma),g_1\big)&\longmapsto u(x,\sigma)\,m(g_1),
   &&(x,\sigma)\in H(W_0),\, g_1\in\U(V),
\end{aligned}
\end{equation}
where for $g_1\in \U(V)$,
\[
   m(g_1)=\begin{pmatrix}1\\&g_1\\&&1\end{pmatrix}.
\]

A Jacobi form on $H(V)\rtimes\U(V)$ or a Jacobi form on $\U(V)$ of index $\beta$ is a smooth function on $H(V)(\bQ)\rtimes\U(V)(\bQ)\backslash H(V)(\bA_\bQ)\rtimes\U(V)(\bA_\bQ)$ such that
\begin{enumerate}
\item[--] for all $\sigma\in\bA_\bQ$, the left translation by $\big((0,\sigma),\bid_2\big)$ equals the multiplication by $\be_{\bA_\bQ}(\beta\sigma)$,
\item[--] the right translation of a maximal compact subgroup and the action of the center of the universal enveloping Lie algebra at $\infty$ satisfy finiteness conditions.
\end{enumerate}
It is easy to see that a Jacobi form of index $\beta$ holomorphic at $\infty$ of weight $(0,0)$ corresponds a section in $H^0\left(\sC,\cL(\beta)\right)$ with $\sC$ a torsor over the Shimura variety of $\GU(2)$ of  an abelian scheme isogeneous to the universal abelian scheme, and $\cL(\beta)$ the invertible sheaf over $\sC$ of $\beta$-homogeneous functions. 

\vspace{.5em}

We will also need Jacobi forms on $\U(2,2)$. Denote by $\bV$ the skew Hermitian space $\cK^4$ equipped with the skew Hermitian form $\begin{pmatrix}&\bid_2\\-\bid_2\end{pmatrix}$, and by $\gls{HVV}$ its associated Heisenberg group whose $R$-points for a $\bQ$-algebra $R$ are 
\[
   H(\bV)(R)=(\bV\otimes_\bQ R)\ltimes R.
\]
We write an element in $H(\bV)(R)$ as $(x,y,\sigma)$, $x,y\in R\otimes_\bQ\cK^2$, $\sigma\in R$. The multiplication is given by
\[
   (x_1,y_1,\sigma_1)(x_2,y_2,\sigma_2)=\left(x_1+x_2,y_1+y_2,,\sigma_1+\sigma_2+\frac{x_1\ltrans{\bar{y}_2}+\bar{x}_1\ltrans{y}_2-\ltrans{y_1\bar{x}_2}-\bar{y}_1\ltrans{x_2}}{2}\right).
\]
The Jacobi group associated to $\bV$ is the semi-direct product $H(\bV)\rtimes\U(\bV)$.

Denote by $P_{\GU(3,3)}$ the parabolic subgroup of $\GU(3,3)$ consisting of elements whose entries in the first column are all $0$ except the $(1,1)$ entry, and denote by  $U_{P_{\GU(3,3)}}$ its unipotent subgroup. We identify $H(\bV)$ with $U_{P_{\GU(3,3)}}$ by
\begin{align*}
   (x,y,\sigma)&\longmapsto u(x,y,\sigma)
   =\begin{pmatrix}1&x&\sigma+\frac{x\ltrans{\bar{y}}-y\ltrans{\bar{x}}}{2}&y\\&\bid_2&\ltrans{\bar{y}}&\\&&1\\&&-\ltrans{\bar{x}}&\bid_2\end{pmatrix},
\end{align*}
and identify $H(\bV)\rtimes\U(2,2)$ with a subgroup of $P_{\GU(3,3)}$ by
\begin{equation}\label{eq:H-to-U33}
\begin{aligned}
   \big((x,y,\sigma),g\big)&\longmapsto u(x,y,\sigma)\,m(g),
   &&(x,y,\sigma)\in H(\bW_0),\,g\in\U(2,2).
\end{aligned}
\end{equation}
where for $g=\begin{pmatrix}A&B\\C&D\end{pmatrix}\in\U(2,2)$,
\begin{equation}\label{eq:U22-U33}
    m(g)=\begin{pmatrix} 1\\&A&&B\\&&1\\&C&&D\end{pmatrix}.
\end{equation}
Similarly as above, we define Jacobi forms on $\U(\bV)$ of index $\beta$ to be smooth functions on $H(\bV)(\bQ)\rtimes\U(\bV)(\bQ)\backslash H(\bV)(\bA_\bQ)\rtimes\U(\bV)(\bA_\bQ)$ on which the left translation of $\big((0,0,\sigma),\bid_4\big)$ equal the multiplication of $\be_{\bA_\bQ}(\beta\sigma)$ plus finiteness conditions for the right translation of a maximal compact subgroup and the center of the universal enveloping Lie algebra at $\infty$. 

\vspace{1em}

We have the following embedding of skew Hermitian spaces 
\begin{align*}
   V&\lra \bV, &x&\longmapsto \left(x,-\frac{x\zeta_0}{2}\right).
\end{align*}
It induces the embedding of unitary groups
\begin{align*}
   \U(V)&\lra\U(\bV)=\U(2,2), &g_1&\lra \begin{pmatrix}\bid_2&-\frac{\zeta_0}{2}\\-\bid_2&-\frac{\zeta_0}{2}\end{pmatrix}^{-1}\begin{pmatrix}g_1\\&\bid_2\end{pmatrix}\begin{pmatrix}\bid_2&-\frac{\zeta_0}{2}\\-\bid_2&-\frac{\zeta_0}{2}\end{pmatrix}
\end{align*}
and induces an embedding of the Jacobi groups
\begin{equation}
\begin{aligned}
   H(V)\rtimes \U(V)&\lra H(\bV)\rtimes \U(\bV), \\
   \left((x,\sigma),\begin{pmatrix}a&b\\c&d\end{pmatrix}\right)&\longmapsto \left(\left(x,-\frac{x\zeta_0}{2},\sigma\right),\begin{pmatrix}\bid_2&-\frac{\zeta_0}{2}\\-\bid_2&-\frac{\zeta_0}{2}\end{pmatrix}^{-1}\begin{pmatrix}g_1\\&\bid_2\end{pmatrix}\begin{pmatrix}\bid_2&-\frac{\zeta_0}{2}\\-\bid_2&-\frac{\zeta_0}{2}\end{pmatrix}\right).
\end{aligned}
\end{equation}
From this embedding, we also get the embedding
\begin{equation}\label{eq:JGU(3,3)} 
   \begin{aligned}
   \big(H(V)\rtimes \U(V)\big)\times\U(V^-)&\lra H(\bV)\rtimes \U(\bV), \\
   \big(((x,\sigma),g_1),g_2\big)&\longmapsto \left(\left(x,-\frac{x\zeta_0}{2},\sigma\right),\begin{pmatrix}\bid_2&-\frac{\zeta_0}{2}\\-\bid_2&-\frac{\zeta_0}{2}\end{pmatrix}^{-1}\begin{pmatrix}g_1\\&g_2\end{pmatrix}\begin{pmatrix}\bid_2&-\frac{\zeta_0}{2}\\-\bid_2&-\frac{\zeta_0}{2}\end{pmatrix}\right).
\end{aligned}
\end{equation}

The embeddings \eqref{eq:H-to-U31}\eqref{eq:H-to-U33}\eqref{eq:JGU(3,3)} are compatible with the embedding $\imath$ in \eqref{eq:embedding} in the sense that the following diagram commutes:
\[\xymatrixcolsep{5pc}\xymatrix{
   \big(H(V)\rtimes\U(V)\big)\times\U(V^-) \ar@{^{(}->}[r]^-{\eqref{eq:H-to-U31}\times \id}\ar[d]^-{\eqref{eq:JGU(3,3)}} & \GU(3,1)\times_{\bG_m}\GU(2)\ar[d]^-{\imath}\\
   H(\bV)\rtimes \U(\bV) \ar@{^{(}->}[r]^-{\eqref{eq:H-to-U33}} & \GU(3,3).
}   
\]

\subsection{Theta series}

Given a Schwartz function $\phi\in \bigotimes'_v\pzS(\bQ_ve_1\oplus\bQ_ve_2)$ (with $\phi_v=\mathds{1}_{\bZ_v e_1+\bZ_v e_2}$ for all most all finite places $v$), we define the theta series $\theta(\phi,\bdot)$, which is an automorphic form  on $\U(V)(\bA_\bQ)$, as
\[
   \theta(\phi,g)=\sum_{x\in \bQ e_1\oplus\bQ e_2}\omega_{\beta}(g)\Phi(x).
\]
We can also define the Jacobi theta series $\theta^J(\phi,\bdot,\bdot)$, which is a Jacobi form on $H(V)\rtimes\U(2)$, as
\begin{equation}\label{eq:theta-J}
   \theta^J(\phi,(y,\sigma),g)=\be_{\bA_\bQ}\left(\beta(\sigma+\frac{1}{2}y\zeta_0\ltrans{\bar{x}})\right)\sum_{x\in \bQ e_1\oplus\bQ e_2}\omega_{\beta}(g)\Phi(x+y).
\end{equation}
Similarly, given $\phi\in \bigotimes'_v\pzS(\bQ_v e^-_1\oplus\bQ_v e^-_2)$, we can define the theta series $\theta_\beta(\phi,\bdot)$ on $\U(V^-)(\bA_\bQ)$.

\subsection{The unfolding}\label{sec:unfolding}

\vspace{1em}

Given a holomorphic automorphic form $\cF$ on $\GU(3,1)(\bA)$ and $g\in \GU(3,1)$ and $\beta\in\Her_1(\cK)_{>0}=\bQ_{>0}$, the $\beta$-th Fourier--Jacobi coefficient of $\cF$ at $g$ is defined as 
\begin{align*}
   a_\beta(g;\cF)=\int_{\sigma\in\Her_1(\cK)\backslash\Her_1(\bA_\cK)} \cF\left(\begin{pmatrix}1&&\sigma\\&\bid_2\\&&1\end{pmatrix} g\right)\be_{\bA}(-\beta\sigma)\,d\sigma.
\end{align*}
With $g\in\GU(3,1)$ fixed, the function
\begin{align*}
   H(V)(\bA_\bQ)\rtimes \U(V)(\bA_\bQ)&\lra \bC\\
   \big((x,\sigma),g_1\big)&\longmapsto a_\beta\big(u(x,\sigma)m(g_1);\cF\big)
\end{align*}
is a Jacobi form on $H(V)\ltimes \U(V)$ of index $\beta$.

\vspace{1em}
The Klingen Eisenstein family $\bfE^\Kling_\varphi$ interpolates a normalization of the Klingen Eisenstein series $E^\Kling\big(\bdot;F(f(s,\xi_0\tau_0),\varphi)\big)$, which can be expressed as the integral in \eqref{eq:Kling-int}. We can compute its $\beta$-th Fourier--Jacobi coefficients by first computing 
\begin{equation}\label{eq:E-Sieg-beta}
    E^{\Sieg}_\beta(g;f(s,\xi_0\tau_0))
    =\int_{\Her_1(\cK)\backslash\Her_1(\bA_\cK)}E^\Sieg\left(\begin{pmatrix}1&&\sigma\\&\bid_2\\&&1\\&&&\bid_2\end{pmatrix}g;f(s,\xi_0\tau_0)\right)\be_{\bA_\bQ}(-\beta\sigma)\,d\sigma,
\end{equation}
and then pairing its restriction to $\GU(3,1)\times_{\bG_m} \GU(2)$ with $\varphi$ on $\GU(2)$. By unfolding the Siegel Eisenstein series in the integrand on the right hand side of \eqref{eq:E-Sieg-beta}, it's not difficult to show that
\begin{equation}\label{eq:Esi-beta}
\begin{aligned}
      &E^{\Sieg}_\beta(h;f(s,\xi_0\tau_0))=\sum_{\gamma\in Q_{\U(2,2)}(\bQ)\backslash \U(2,2)(\bQ)}\sum_{x\in\cK^2}\\
      &\hspace{2em}\int_{\Her_1(\bA_\cK)}f(s,\xi_0\tau_0)\left(\begin{pmatrix}&\bid_3\\-\bid_3\end{pmatrix}\begin{pmatrix}1&&\varsigma&x\\ &\bid_2&\ltrans{\bar{x}}\\&&1\\&&&\bid_2\end{pmatrix}m\left(\begin{pmatrix}&\bid_2\\-\bid_2\end{pmatrix}\gamma\right) h\right)\be_{\bA_\bQ}(-\beta\sigma)\,d\varsigma,
\end{aligned}
\end{equation}
where $m(\bdot)$ is the embedding of $\U(2,2)$ into $\U(3,3)$ as defined in \eqref{eq:U22-U33}. (cf \cite[Section 3.3.1]{WanL})

\begin{defn}\label{def:FJ}
For $g_v\in\GU(3,3)(\bQ_v)$, $x_v\in\cK^2_v$ and $f_v(s,\xi_0\tau_0)\in I_{Q_{\GU(3,3)},v}(s,\xi_0\tau_0)$, define
\begin{align*}
   &\glsuseri{FJbetav}\\
   =&\,\int_{\Her_1(\cK_v)} f_v(s,\xi_0\tau_{0})\left(\begin{pmatrix}&\bid_3\\-\bid_3\end{pmatrix}\begin{pmatrix}1&&\varsigma&x_v\\ &\bid_2&\ltrans{\bar{x}_v}\\&&1\\&&&\bid_2\end{pmatrix}m\begin{pmatrix}&\bid_2\\ -\bid_2\end{pmatrix} g_v\right)\be_\bA(-\beta\sigma)\,d\sigma.
\end{align*}
For $g\in\GU(3,3)(\bA_\bQ)$, $x\in\bA^2_\cK$ and $f(s,\xi_0\tau_0) =\otimes_v f_v(s,\xi_0\tau_{0})$, let
\[
   \FJ_\beta\big(g,x;f(s,\xi_0\tau_0)\big)=\otimes_v \FJ_{\beta,v}\big(g_v,x_v;f_v(s,\xi_0\tau_0)\big).
\]
\end{defn}

The unfolding result \eqref{eq:Esi-beta} gives the following proposition.
\begin{prop}\label{prop:FJ}
Let $\beta\in\Her_1(\cK)_{>0}$ and $f(s,\xi_0\tau_0)\in I_{Q_{\GU(3,3)}}(s,\xi_0\tau_0)$. The $\beta$-th Fourier--Jacobi coefficient of the Siegel Eisenstein series $E^\Sieg\big(\bdot\,;f(s,\xi_0\tau_0)\big)$ at $g\in \U(3,3)(\bA_\bQ)$ equals
\[
    E^{\Sieg}_\beta(g;f(s,\xi_0\tau_0))=\sum_{\gamma\in Q_{\U(2,2)}(\bQ)\backslash \U(2,2)(\bQ)}\sum_{x\in\cK^2} \FJ_\beta\big(m(\gamma)g,x;f(s,\xi_0\tau_0)\big),
\]
with $\FJ_{\beta}$ defined as in Definition~\ref{def:FJ} and $m(\bdot)$ the embedding \eqref{eq:U22-U33}.
\end{prop}

One can easily check that $\FJ_{\beta,v}$ satisfies
\begin{align*}
   \FJ_{\beta,v}\left(u(x,y,\sigma)g,x_0;f_v(s,\xi_0\tau_{0})\right)
   &=\be_v\left(\beta(\sigma+\frac{x\ltrans{\bar{y}}-y\ltrans{\bar{x}}}{2})\right)\,\FJ_{\beta,v}\big(g,x_0+x;f_v(s,\xi_0\tau_{0})\big),\\
   \FJ_{\beta,v}\left(m\begin{pmatrix}A&B\\0&\ltrans{\bar{A}}^{-1}\end{pmatrix}g,x_0;f_v(s,\xi_0\tau_{0})\right)
   &=|\det A\ltrans{\bar{A}}|^{s+1}_v\xi_0\tau_{0}(\det A)\cdot |\det A\ltrans{\bar{A}}|^{\frac{1}{2}}_v\be_v(\beta xB\ltrans{\bar{A}}\ltrans{\bar{x}})\\
   &\quad\times\FJ_{\beta,v}\big(g,x_0 A;f_v(s,\xi_0\tau_{0})\big).
\end{align*} 
Comparing them with the formulas in \eqref{eq:Weil-fmla} for the Schr{\"o}dinger model of Weil representations, we see that for a fixed $g_v\in \U(3,3)(\bQ_v)$, $\FJ_{\beta,v}\big(u(x_v,y_v,\sigma_v)\,m(g_{1,v})g_v,x_0;f_v(s,\xi_0\tau_{0})\big)$ is essentially (a finite sum of) the product
\[
   f_{2,v}\left(s,\xi_0\tau_{0}\lambda^{-1}\right)(g_{1,v})\cdot \omega_{\beta,v}\big(u(x_v,y_v,\sigma_v)\,m(g_{1,v})\big)\Phi_v(x_0),
\]
with $f_{2,v}\left(s,\xi_0\tau_{0}\lambda^{-1}\right)$ a section in the degenerate principal series $I_{Q_{\U(2,2),v}}\left(s,\xi_0\tau_{0}\lambda^{-1}\right)$, $\Phi_v$ a Schwartz function on $\cK^2_v$ and $\omega_{\beta,v}$ the Weil representation of $\U(2,2)$ briefly recalled in \S\ref{sec:Schro}. Both $f_{2,v}\left(s,\xi_0\tau_0\lambda^{-1}\right)$ and $\Phi_v$ are determined by $f_v(s,\xi_0\tau_{0})\in I_{Q_{\GU(3,3)},v}(s,\xi_0\tau_0)$ and $g_v\in\GU(3,3)(\bQ_v)$. The computation at an unramified place $v$ gives
\begin{align*}
   &\FJ_{\beta,v}\big(u(x,y,\sigma) m(g_{1,v}),x_0;f^\sph_v(s,\xi_0\tau_{0})\big)\\
   =&\,L_v\left(2s+3,\xi^\bQ_0\tau^\bQ_0\right)^{-1}\cdot f^\sph_{2,v}(s,\xi_0\tau_0\lambda^{-1})(g_{1,v})\cdot \omega_{\beta,v}\big(u(x,y,\sigma) m(g_{1,v})\big)\mathds{1}_{\cO^2_{\cK,v}}(x_0),
\end{align*} Therefore, the Jacobi form on $H(\bV)\rtimes\U(2,2)$ given by
\begin{align*}
   \big((x,y,\sigma),g\big)&\longmapsto E^{\Sieg}_\beta(u(x,y,\sigma)m(g);f(s,\xi_0\tau_0))
\end{align*}
is essentially (a finite sum of) the product 
\[  
   L\left(2s+3,\xi^\bQ_0\tau^\bQ_0\right)^{-1}\cdot E^\Sieg_2\left(s,\xi_0\tau_{0}\lambda^{-1}\right)\cdot \Theta^J,
\] 
with $E^{\Sieg}_2\left(s,\xi_0\tau_{0}\lambda^{-1}\right)$ a Siegel Eisenstein series on $\U(2,2)$ and $\Theta^J$  a Jacobi theta function on the $H(\bV)\rtimes\U(2,2)$.  The restriction of $\Theta^J$ from $H(\bV)\rtimes\U(2,2)$ to $\big(H(V)\rtimes \U(2)\big)\times\U(2)$ via \eqref{eq:H-to-U33} is essentially (a finite sum of) $\theta^J_3\boxtimes\theta_2$, with $\theta^J_3$ (resp. $\theta_2$) a Jacobi theta function (resp. theta function) on $H(V)\rtimes \U(2)$ (resp. $\U(2)$) attached to a Schwartz function $\phi_3$ (resp. $\phi_2$) on $\bA_{\cK}$. Therefore,
\begin{equation}\label{eq:E-beta}
    \left.E^{\Sieg}_\beta(\bdot;f(s,\xi_0\tau_0))\right|_{\big(H(V)\rtimes \U(2)\big)\times\U(2)}\approx L\left(2s+3,\xi^\bQ_0\tau^\bQ_0\right)^{-1}\cdot \theta^J_3\boxtimes\theta_2 \cdot  \left. E^\Sieg_2\left(s,\xi_0\tau_{0}\right)\right|_{\U(2)\times \U(2)}.
\end{equation}
Here $\approx$ means equal up to normalizations and more precise formulas for local sections are needed to be an actual identity. We will also use the notation $\approx$ several times in the next section. All the identities with  $\approx$ are only for the purpose of illustrating the idea of relating the non-degenerate Fourier--Jacobi coefficients with $L$-values, and will not be used for our rigorous analysis of $\bfE^\Kling_{\varphi,\beta,u}$ starting from \S\ref{sec:theta_1J}.

\subsection{Our strategy of analyzing the non-degenerate Fourier--Jacobi coefficients of the Klingen Eisenstein family}\label{sec:strategy}
Before we move on to the involved computations, we give a brief explanation of how we choose the auxiliary data to study the pairing of \eqref{eq:E-beta} with $\varphi$. The description of this section should also explain why the choices of the auxiliary data in \S\ref{sec:aux} are made for constructing $\bfE^\Kling_\varphi$ and guaranteeing that its non-degenerate Fourier--Jacobi coefficients satisfy the nonvanishing properties needed for applications.

\subsubsection{Choosing the auxiliary $\theta^J_1$}
The first step to analyze \eqref{eq:E-beta} is to pick a suitable Schwartz function $\phi_1$ on $\bA_\cK$, let $\gls{thetaJ1}$ be the associated Jacobi theta function on $H(V)\rtimes\U(2)$ (defined as in \eqref{eq:theta-J}), and define the linear functional
\begin{align*}
   l_{\theta^J_1}:\{\text{Jacobi forms on $H(V)\rtimes\U(2)$}\}&\lra \{\text{automorphic forms on $\U(2)$}\}
\end{align*}
by
\begin{equation}\label{eq:FJ-fcnl}
   l_{\theta^J_1}(\varphi^J)(g_1)=\int_{H(V)(\bQ)\backslash H(V)(\bA)} \theta^J_1\big((x,\sigma),g_1\big)\varphi^J\big((x,\sigma),g_1\big)\,dxd\sigma.
\end{equation}
Applying $l_{\theta^J_1}$ to \eqref{eq:E-beta}, we obtain an automorphic form on $\U(2)\times \U(2)$ whose value at $(g_1,g_2)$ equals
\begin{equation}\label{eq:int-H}
   \int_{H(V)(\bQ)\backslash H(V)(\bA)} \theta^J_1\big((w,\sigma),g_1\big)\cdot E^{\Sieg}_\beta\big(\imath(u(x,\sigma)m(g_1),g_2);f(s,\xi_0\tau_0)\big)\,dxd\sigma.
\end{equation}
It is a standard fact that for two Schwartz functions $\phi,\phi'$ on $\bA_{\cK}$,
\[
    \int_{H(V)(\bQ)\backslash H(V)(\bA)} \theta^J_\phi\big((x,\sigma),g_1\big)\ol{\theta^J_{\phi'}\big((x,\sigma),g_1\big)}\,dxd\sigma=\left<\phi,\ol{\phi'}\right>=\int_{\cK\backslash\bA_\cK} \phi(x)\ol{\phi'(x)}\,dx,
\]
independent of $g_1\in \U(2)$. Thanks to this fact, from \eqref{eq:int-H} we obtain
\begin{equation}
\begin{aligned}
   &l_{\theta^J_1}\left( \left.E^{\Sieg}_\beta(\bdot\,;f(s,\xi_0\tau_0))\right|_{\big(H(V)\rtimes \U(2)\big)\times\U(2)}\right)\\
   \approx &\frac{\left<\phi_1,\phi_3\right>}{L\left(2s+3,\xi^\bQ_0\tau^\bQ_0\right)} \cdot (1\boxtimes\theta_2)\cdot \left.E^\Sieg_2\left(s,\xi_0\tau_{0}\right)\right|_{\U(2)\times\U(2)}.
\end{aligned}
\end{equation}
Here $1\boxtimes\theta_2$ denotes the automorphic form on $\U(2)\times\U(2)$ which is the constant function $1$ on the first factor and $\theta_2$ on the second factor.

For an automorphic form on ${\tt F}$ on $\U(2)$ and a Hecke character $\chi:\cK^\times\backslash\bA^\times_\cK\ra\bC^\times$, we use the notation ${\tt F}^\chi$ to denote the form
\begin{equation}\label{eq:U2-twist}
    {\tt F}^\chi(g)={\tt F}(g)\cdot \chi(\det g),\quad g\in\U(2)(\bA_\bQ).
\end{equation}  
Our Klingen Eisenstein series $E^\Kling\big(\bdot\,;F(f(s,\xi_0\tau_0),\varphi)\big)$ is obtained by pairing the restriction of the Siegel Eisenstein series with $\varphi^{\ol{\xi_0\tau}_0}$, so we have
\begin{equation}
\begin{aligned}
   &l_{\theta^J_1}\left(E^\Kling_\beta\big(\bdot\,\,;F(f(s,\xi_0\tau_0),\varphi)\big)\right)\\
   \approx &\, \frac{\left<\phi_1,\phi_3\right>}{L\left(2s+3,\xi_0^\bQ\tau^\bQ_0\right)} \cdot \left<\theta_2\varphi^{\ol{\xi_0\tau}_0},\left.E^\Sieg_2\left(s,\xi_0\tau_{0}\lambda^{-1}\right)\right|_{\U(2)\times\U(2)}\right>_{1\times\U(2)},
\end{aligned}
\end{equation}
where $\left<\cdot,\cdot\right>_{1\times\U(2)}$ means integration over the second copy of $\U(2)$. 
\vspace{.5em}

The standard doubling method formula \cite{LapidRallis} implies
\[
   \left.E^\Sieg_2\left(s,\xi_0\tau_{0}\lambda^{-1}\right)\right|_{\U(2)\times\U(2)}\approx \sum_h \frac{L\left(s+\frac{1}{2},\mr{BC}(\pi_{h})\times\xi_0\tau_0\lambda^{-1}\right)}{L\left(2s+1,\xi^\bQ_0\tau^\bQ_0\right)\,L\left(2s+2,\xi^\bQ_0\tau^\bQ_0\eta_{\cK/\bQ}\right)}\frac{\ol{h} \boxtimes h^{\xi_0\tau_0\ol{\lambda}}}{\left<\ol{h},h\right>},
\]
where $h$ runs over a certain orthonormal basis of automorphic forms on $\U(2)$ of certain level. By picking an $h$, we have
\begin{equation}\label{eq:lEh1}
\begin{aligned}
   &\left<l_{\theta^J_1}\left(E^\Kling_\beta\big(-;F(f(s,\xi_0\tau_0),\varphi)\big)\right),h\right>\\
   \approx &\,\left<\phi_1,\phi_3\right>\frac{L\left(s+\frac{1}{2},\mr{BC}(\pi_h)\times\xi_0\tau_0\lambda^{-1}\right)}{d_3(s,\xi_0\tau_0)}\left<\theta_2\varphi^{\ol{\xi_0\tau_0}},h^{\xi_0\tau_0\ol{\lambda}}\right>\\
   =&\,\left<\phi_1,\phi_3\right>\frac{L\left(s+\frac{1}{2},\mr{BC}(\pi_h)\times\xi_0\tau_0\lambda^{-1}\right)}{d_3(s,\xi_0\tau_0)}\left<\theta_2^{\ol{\lambda}}\varphi,h\right>,
\end{aligned}
\end{equation}
where $d_3(s,\xi_0\tau_0)$ is as in \eqref{eq:d3}. The Petersson inner product $\left<\theta_2^{\ol{\lambda}}\varphi,h\right>$ is related to the central value of a triple product $L$-function by Ichino's formula.

\subsubsection{Choosing auxiliary $h$ and $\theta$ and relating to $L$-values}\label{sec:hthetaL}
We will choose suitable CM forms $h$ and $\theta$ on $\U(2)$ such that we can apply the available mod $p$ nonvanishing results to study $L\left(s+\frac{1}{2},\mr{BC}(\pi_h)\times\xi_0\tau_0\lambda^{-1}\right)$ and $\left<\theta_2^{\ol{\lambda}}\varphi,h\right>$ in \eqref{eq:lEh1}.

Let $\chi_h$ and $\chi_\theta$ be two unitary Hecke characters of $\cK^\times\backslash\bA^\times_\cK$ such that $\chi_h\chi^c_\theta|_{\bA^\times_{\bQ}}=\triv$ (as chosen in \S\ref{sec:aux}). Let $h_0$ be the theta lift of $\chi^{-1}_h|_{\U(1)}$ to $\U(2)$, and $h$ be the automorphic form on $\U(2)$ obtained from $h_0$ by
\begin{equation}\label{eq:h-twist}
\begin{aligned}
   h(g)&=\chi_h\chi^c_\theta(a)\cdot h_0(g), &g\in\U(2),\,a\in\bA^\times_\cK,\,\det g=a\bar{a}^{-1}.
\end{aligned}
\end{equation}
Thanks to the condition $\chi_h\chi^c_\theta|_{\bA^\times_{\bQ}}=\triv$, the definition of $h(g)$ does not depend on the choice of $a\in\bA^\times_\cK$. 

The form $\theta_2$ does not generate an irreducible representation of $\U(2)$. Let $\theta$ be the projection of $\theta_2$ to its $\chi_\theta\lambda^2|_{\U(1)}$-eigenspace for the action of the center of $\U(2)$. Then $\theta$ is a theta lift of $\chi_\theta\lambda^2|_{\U(1)}$. Since our $\varphi$ is assumed to have the trivial central character and $h$ has the central character $\chi^{-1}_\theta|_{\U(1)}$, we have
\begin{equation}\label{eq:theta'}
   \left<\theta_2^{\ol{\lambda}}\varphi,h\right>=\left<\theta^{\ol{\lambda}}\varphi,h\right>.
\end{equation}

By \cite[Theorem C.5]{GI-FD}, with the choice of the splitting characters for the theta correspondence between $\U(2)$ and $\U(1)$ as in \eqref{eq:split-char}, we have
\begin{equation}\label{eq:L-h}
\begin{aligned}
    \mr{BC}\left(\pi_{h}\right)&=\mr{BC}\left(\pi_{h_0}\right)\otimes\chi_h\chi^c_\theta=\left(\mr{BC}\left(\chi^{-1}_h|_{\U(1)}\right)\lambda^{-1}\oplus \lambda\right)\otimes\chi_h\chi^c_\theta\\
    &=\left(\chi^{-1}_h\chi^c_h\lambda^{-1}\oplus\lambda\right)\otimes\chi_h\chi^c_\theta=\chi^c_h\chi^c_\theta\lambda^{-1}\oplus \chi_h\chi^c_\theta\lambda,
\end{aligned}
\end{equation}
and  
\begin{equation}\label{eq:L-theta}
    \mr{BC}\big(\pi_{\theta^{\ol{\lambda}}}\big)=\big(\mr{BC}\left(\chi_\theta\lambda^2|_{\U(1)}\big)\lambda^{-1}\oplus\lambda\right)\bar{\lambda}\bar{\lambda}^{-c}=\left(\chi_\theta\chi^{-c}_\theta\lambda^3\oplus\lambda\right)\lambda^{-2}=\chi_\theta\chi^{-c}_\theta\lambda\oplus\lambda^{-1}.
\end{equation}
Thus, we have 
\begin{equation}\label{eq:L12}
\begin{aligned}
    L\left(s+\frac{1}{2},\mr{BC}(\pi_h)\times\lambda^{-1}\xi_0\tau_0\right)
    &=L\left(s+\frac{1}{2},\lambda^{-2}\chi^c_h\chi^c_\theta\xi_0\tau_0\right)\,L\left(s+\frac{1}{2},\chi_h\chi^c_\theta\xi_0\tau_0\right)\\
    &=L\left(s+\frac{1}{2},\lambda^2\chi_h\chi_\theta\xi^c_0\tau^c_0\right)\,L\left(s+\frac{1}{2},\chi_h\chi^c_\theta\xi_0\tau_0\right).
\end{aligned}
\end{equation}
By the triple product formula
\begin{equation}\label{eq:triple-1}
\begin{aligned}
   \left<\theta^{\bar{\lambda}}\varphi,h\right>\left<\tilde{\theta}^\lambda_3\tilde{\varphi},\tilde{h}_3\right>
   &\approx \frac{\big<\varphi,\tilde{\varphi}\big>}{L(1,\pi,\mr{Ad})}\,
   \frac{\big<h,\tilde{h}_3\big>}{\zeta_\cK(1)\,L(1,\chi_h\chi^{-1}_h\lambda^2)}\,
   \frac{\big<\theta,\tilde{\theta}_3\big>}{\zeta_\cK(1)\, L(1,\chi_\theta\chi^{-c}_\theta\lambda^2)}\\
   &\quad\times L\left(\frac{1}{2},\mr{BC}(\pi)\times\chi_h\chi_\theta\lambda^2\right)L\left(\frac{1}{2},\mr{BC}(\pi)\times \chi_h\chi^c_\theta\right),
\end{aligned}
\end{equation}
where $\tilde{\varphi}$ (resp. $\tilde{h}_3$, $\tilde{\theta}_3$) are suitable forms taken from the dual representation of $\pi^{\U(2)}=\pi_\varphi$ (resp. $\pi_{h}$, $\pi_{\theta}$). (In our case here, the triple product $L$-function for $\pi^D\times\pi^D_\theta\times\pi^D_h$  factorizes as the product of the two $L$-functions in \eqref{eq:triple-1}. Also, note that Ichino's triple product formula is for automorphic forms on $D^\times$, we actually need to relate the integral on $\U(2)$ for  $\left<\theta^{\bar{\lambda}}\varphi,h\right>\left<\tilde{\theta}^\lambda_3\tilde{\varphi},\tilde{h}_3\right>$ to an integral on $D^\times$. In \S\ref{sec:ext-GU2}, we discuss extending automorphic forms on $\U(2)$ to $\GU(2)$. Thanks to \eqref{eq:DGU}, the integral over $\GU(2)$ is the same as the integral over $D^\times$.)

Combining \eqref{eq:lEh1}\eqref{eq:theta'}\eqref{eq:L12}\eqref{eq:triple-1}, we get
\begin{equation}\label{eq:FJ-L1256}
\begin{aligned}
   &d_3(s,\xi_0\tau_0)\cdot \left<l_{\theta^J_1}\left(E^\Kling_\beta\big(\bdot\,;F(f(s,\xi_0\tau_0),\varphi)\big)\right),h\right>\left<\tilde{\theta}^\lambda_3\tilde{\varphi},\tilde{h}_3\right>\\
   \approx&\,\left<\phi_1,\phi_3\right> \frac{\big<\varphi,\tilde{\varphi}\big>}{L(1,\pi,\mr{Ad})} \cdot \frac{\big<h,\tilde{h}_3\big>}{\zeta_\cK(1)\,L(1,\chi_h\chi^{-1}_h\lambda^2)}\,
   \frac{\big<\theta,\tilde{\theta}_3\big>}{\zeta_\cK(1)\, L(1,\chi_\theta\chi^{-c}_\theta\lambda^2)}\\
   &\times  L\left(s+\frac{1}{2},\lambda^2\chi_h\chi_\theta\xi^c_0\tau^c_0\right)\,L\left(s+\frac{1}{2},\chi_h\chi^c_\theta\xi_0\tau_0\right)
   \cdot  L\left(\frac{1}{2},\mr{BC}(\pi)\times\chi_h\chi_\theta\lambda^2\right)L\left(\frac{1}{2},\mr{BC}(\pi)\times \chi_h\chi^c_\theta\right).
\end{aligned}
\end{equation}

\vspace{.5em}
The relation between $E^\Kling_\beta\big(-;F(f(s,\xi_0\tau_0),\varphi)\big)$ and $L$-values illustrated in \eqref{eq:FJ-L1256} explains the nonvanishing conditions on the $L$-values in \eqref{eq:L-1}\eqref{eq:L2}\eqref{eq:L-5}\eqref{eq:L-6} in our choice of the auxiliary Hecke characters $\chi_\theta$ and $\chi_h$ in \S\ref{sec:aux}.

\vspace{.5em} 
In the following, guided by the strategy described in this section, we carry out the necessary local computations and the construction of the auxiliary objects to prove the the desired property of $\bfE^\Kling_{\varphi,\beta,u}$ with $\beta=1$. The computation is place by place, and we have the following cases:
\begin{enumerate}
\item[--] For the archimedean place, we only need to consider the cases for which there are standard choices of local sections and  computations are easy.
\item[--] For unramified places, we have the standard formulas for spherical sections.
\item[--] For the place $p$ and the places in $\Sigma_\rms\cup\{\ell\}$, we compute precise formulas.
\item[--] For the places in $\Sigma_\ns\cup\{\ell'\}$, we do not attempt to compute precise formulas. At these places, the local data actually do not change when $\tau$ varies in a $p$-adic family, and we only make sure that local integrals contribute a nonzero scalar.
\end{enumerate}

\subsection{The auxiliary Jacobi form $\theta^J_1$}\label{sec:theta_1J}
Given a Schwartz function $\phi$ on $\bA^2_\bQ$ such that $\phi_{\infty}$ is the standard Gaussian function, let $\theta^J$ be the associated Jacobi form  on $\U(2)$ (as defined in \eqref{eq:theta-J}). The map attached to this $\theta^J$ defined in \eqref{eq:FJ-fcnl} extends to $V^{J,\beta}_{\GU(2)}$ and induces 
\begin{equation}\label{eq:l-JF}
   \gls{lthetaJ}:\Meas\left(\Gamma_\cK,V^{J,\beta}_{\GU(2)}\right)\lra  \Meas\left(\Gamma_\cK,V_{\U(2)}\right).
\end{equation}
Apply it to $\bfE^\Kling_{\varphi,\beta,u}$ (defined in \eqref{eq:E-betau}), we get
\[
   l_{\theta^J}\left(\bfE^\Kling_{\varphi,\beta,u}\right)\in \Meas\left(\Gamma_\cK,V_{\U(2),\,\xi}\wh{\otimes}\,\hat{\cO}^\ur_L\right)^\natural.
\]

\begin{prop}\label{prop:phi2f2}
Let $\beta=1\in \Her_1(\cK)$. There exists a Schwartz function $\phi_1$ on $\bA^2_\bQ$ with $\phi_{1,\infty}$ the standard Gaussian function, elements $u_1,\dots,u_r\in\bigotimes\limits_{v\in\Sigma_\ns\cup\{\ell'\}}\U(1)\left(\bQ_v\right)$, and constants $b_1,\dots,b_r\in \cO_L$,  (all independent of $\tau$), such that 
\[
   l_{\theta^J_1}\left(\sum_i b_i\bfE^\Kling_{\varphi,\beta,u_i}\right)\in \Meas\left(\Gamma_\cK,V_{\U(2),\,\xi}\right)^\natural,
\]
the image of $\sum\limits_i b_i\bfE^\Kling_{\varphi,\beta,u_i}$ under the map \eqref{eq:l-JF} attached to the Jacobi form $\gls{thetaJ1}$ corresponding to $\phi_1$, satisfies the following interpolation properties: for all $\tau_{p\adic}$ as in Theorem~\ref{thm:Sieg-F},
\begin{equation}\label{eq:lE}
\begin{aligned}
   l_{\theta^J_1}\left(\sum\nolimits_i b_i\bfE^\Kling_{\varphi,\beta,u_i}\right)(\tau_{p\adic})
   =&\, d^{\Sigma\,\cup\{p,\ell,\ell'\}}_2\left(s,\xi_0\tau_0\lambda^{-1}\right)\\
  &\hspace{-7em}\times \left<\theta(\phi_{2,\tau})\left(\bdot\,\begin{pmatrix}&-1\\1\end{pmatrix}\right)\varphi^{\ol{\xi_0\tau}_0},\left.E^\Sieg_2\left(-\,;f_2(s,\xi_0\tau_{0}\lambda^{-1})\right)\right|_{\U(2)\times\U(2)}\right>_{1\times\U(2)},
\end{aligned}
\end{equation}
with $\phi_{2,\xi\tau}\in\pzS(\bA^2_{\bQ})$ (independent of $\tau$ except $v=p$) and $f_2(s,\xi_0\tau_0\lambda^{-1})\in\bigotimes_vI_{Q_{\U(2,2)},v}(s,\xi_0\tau_0\lambda^{-1})$ described as follows:
\vspace{.5em}
-- $v=\infty$. 
\begin{align*}
   \phi_{2,\xi\tau,\infty}(y_1,y_2)&=\frac{\fs^{1/4}\Nm(\delta)^{1/4}}{2} e^{-2\pi\sqrt{\Nm(\delta)}\,\begin{psm}y_1&y_2\end{psm}\zeta_0\begin{psm}y_1\\y_2\end{psm}},\\
   f_{2,\infty}(s,\xi_0\tau_0\lambda^{-1})\left(g'=\begin{pmatrix}A&B\\C&D\end{pmatrix}\right)&= \frac{(-2\pi i)^k}{\Gamma(k)}\cdot \det\left(C\frac{\zeta_0}{2}+D\right)^{-k}\left|C\frac{\zeta_0}{2}+D\right|^{-s+1+\frac{k}{2}}_\bC
\end{align*}

\vspace{1em}
-- $v\notin\Sigma\cup\{\ell,\ell',p\}$.
\begin{align*}
    \phi_{2,\xi\tau,v}&=\mathds{1}_{\bZ_v}\times\mathds{1}_{\bZ_v},\\
    f_{2,v}(s,\xi_0\tau_0\lambda^{-1})&=L_v\left(2s+3,\xi^\bQ_0\tau^\bQ_0\right)^{-1}\cdot f^\sph_{2,v}(s,\xi_0\tau_0\lambda^{-1})
\end{align*} 

\vspace{.5em}
-- $v=\fv\ol{\fv}\in\Sigma_\rms\cup\{\ell\}$. 
\begin{align*}
   \phi_{2,\xi\tau,v}= \text{Image of \eqref{eq:int-S} of } \Big((y_1,y_2)&\longmapsto \mathds{1}_{\cO_{\cK,\fv}}(y_1)\cdot \mathds{1}_{\cO^\times_{\cK,\bar{\fv}}}(y_2)\chi_{\theta,\bar{\fv}}\lambda_{\bar{\fv}}(y_2)\Big)
\end{align*}
\begin{equation}\label{eq:f2-ramv}
\begin{aligned}
   f_{2,v}(s,\xi_0\tau_0\lambda^{-1})\left(g'=\begin{pmatrix}A&B\\C&D\end{pmatrix}\right)
   =&\,|\det C\ltrans{\bar{C}}|^{-s-1}(\xi_0\tau_{0}\lambda^{-1})_v\left(\det\bar{C}^{-1}\right)\\
   &\times\mathds{1}_{\Her_2(\cO_{\cK,v})}\left(C^{-1}D+q^{-c_v}_v\cdot\bid_2\right)
\end{aligned}
\end{equation}

\vspace{.5em}  
-- $v\in\Sigma_\ns\cup\{\ell'\}$. 
\begin{align*}
  \phi_{2,\xi\tau,v}=&\,\text{a nonzero Schwartz function on $\bQ^2_v$ inside the $\chi_\theta\lambda^2|_{\U(1)(\bQ_v)}$-eigenspace}\\
  &\,\text{for the action of $\U(1)(\bQ_v)$, invariant under $K_v$, independent of $\tau$,}\\
   f_{2,v}(s,\xi_0\tau_0\lambda^{-1})=&\,\text{same as \eqref{eq:f2-ramv}}.
\end{align*}

\vspace{.5em}
-- The place $p$. 
\begin{equation}\label{eq:phi2p}
\begin{aligned}
   \phi_{2,\xi\tau,p}=&\text{Image of \eqref{eq:int-S} of }\\&\Big((y_1,y_2)\longmapsto p^{-2t}\fg\left((\xi_0\tau_0)^{-1}_{\fp}\right)\cdot \mathds{1}_{\bZ^\times_p}(p^ty_1)(\xi_0\tau_0)_{\fp}(-p^ty_1)  \cdot\mathds{1}_{\bZ_p}(p^ty_2)\Big),
\end{aligned}
\end{equation}
\begin{align*}
   f_{2,p}(s,\xi_0\tau_0\lambda^{-1})=&\,\gamma_p\left(-2s,(\xi^\bQ_0\tau^\bQ_0)^{-1}\eta_{\cK/\bQ}\right)\gamma_p\left(-2s-1,(\xi^\bQ_0\tau^\bQ_0)^{-1}\right)\\
   &\times M'_p\left(-s,(\xi_0\tau_0\lambda^{-1})^{-c}\right)f^{\textnormal{big-cell}}_{2,p}\left(-s,(\xi_0\tau_0\lambda)^{-c}\right)(\bdot\,\Upsilon'_p),
\end{align*}
\quad\quad where 
\begin{itemize}
\item $M'_p\left(-s,(\xi_0\tau_0\lambda)^{-c}\right):I_{Q_{\U(2,2)},p}\left(-s,(\xi_0\tau_0\lambda)^{-c}\right)\ra I_{Q_{\U(2,2)},p}(s,\xi_0\tau_0\lambda^{-1})$ is the intertwining operator,
\item $\Upsilon'_p$ is the element in $\U(2,2)(\bQ_p)$ with $\varrho_\fp(\Upsilon'_p)=\begin{pmatrix}\bid_1&-\frac{\zeta_0}{2}\\-\bid_2&-\frac{\zeta_0}{2}\end{pmatrix}^{-1}$,
\item \begin{align*}
   f^{\textnormal{big-cell}}_{2,p}\left(-s,(\xi_0\tau_0\lambda)^{-c}\right)\begin{pmatrix}A&B\\C&D\end{pmatrix}
   =&\,|\det C\ltrans{\bar{C}}|^{s-1}_p(\xi_0\tau_{0}\lambda^{-1})_p(\det C)\cdot \alpha'_{\xi\tau,p}(C^{-1}D).
\end{align*}
with $\alpha'_{\xi\tau,p}$ the Schwartz function on $\Her_2(\cK_p)$ defined as
\[
   \alpha'_{\xi\tau,p}(x)=\fg\left((\xi_0\tau_0)^{-1}_{\fp}\right)\cdot \mathds{1}_{\bZ^3_p}(x_{11},x_{21},x_{22})\cdot \mathds{1}_{\bZ^\times_p}(p^t x_{12})(\xi_0\tau_0)_{\fp}(p^tx_{12}),\quad \varrho_\fp(x)=\begin{pmatrix}x_{11}&x_{12}\\x_{21}&x_{22}\end{pmatrix},
\]
\end{itemize}

\end{prop}

\begin{proof}
This proposition is proved in \cite[sections 6.E-6.H]{WanU31} by computing $\FJ_{\beta,v}\big(g,x;f_v(s,\xi_0\tau_0)\big)$ place by place. The place $\infty$ is done in Lemmas 6.9 and 6.11 and the unramified places are done in Lemma 6.18. (Note that at the unramified places, we have $d_{2,v}(s,\xi_0\tau_0\lambda^{-1})=\frac{d_{3,v}(s,\xi_0\tau_0)}{L_v(2s+3,\xi^\bQ_0\tau^\bQ_0)}$.) The places in $\Sigma_\rms\cup\{\ell\}$ are done in Lemmas 6.22 and 6.24, and the places in $\Sigma_\ns\cup\{\ell'\}$ are done in Lemmas 6.22 and 6.26. Note that Lemma 6.26 {\it loc.cit} is Proposition~\ref{prop:nschi} and in \S\ref{sec:aux} we choose $\chi_\theta$ such that at $v\in\Sigma_\ns\cup\{\ell'\}$ we have nonzero $\phi_{2,\xi\tau,v}$ described in this proposition. The place $p$ is done from Lemmas 6.35 to the end of the section 6.H. The Schwartz function $\phi_{2,\tau_0,p}$ is the $\phi'_{2,p}$ given below Definition 6.38. Note that our $\FJ_{\beta,v}\big(g,x;f_v(s,\xi_0\tau_0)\big)$ defined in Definition 4.5.1. equals to $\FJ_\beta(f_{sieg,v};z,x,g\eta^{-1},1)$ in \cite{WanU31} with $z=s$ and $\eta=\begin{pmatrix}&\bid_2\\-\bid_2\end{pmatrix}$, so the results in {\it loc.cit} all contain the extra $\eta$ but the formulas in this proposition do not.

At the place $p$, in fact the pairing of $\phi_{1,p}(x)$ and  $\FJ_{\beta,p}\big(g,(x,y)_p;f_p(s,\xi_0\tau_0\big)$ along $x$ does not directly equal $\phi_{2,\tau,p}(y)\cdot f_{2,p}(s,\xi_0\tau_0\lambda^{-1})(g)$, but equals the average
\begin{equation}\label{eq:real-p}
    \int_{\bZ_p}\omega_{\beta,p}\left(\varrho_{\fp}^{-1}\begin{pmatrix}1&0\\n&1\end{pmatrix},1\right)\phi_{2,\tau,p}(y)\cdot f_{2,p}(s,\tau_0\lambda^{-1})\left(g\Upsilon'_p\begin{pmatrix}1&\\&1\\&&1&\\&&n&1\end{pmatrix}\right)\,dn.
\end{equation}
Meanwhile, $\varphi$ is spherical at $p$, in particular invariant under the right translation by $\begin{pmatrix}1&\\\bZ_p&1\end{pmatrix}$, so by replacing \eqref{eq:real-p} $\phi_{2,\tau,p}(y)\cdot f_{2,p}(s,\xi_0\tau_0\lambda^{-1})(g)$ does not change the resulting Petersson inner product with $\varphi$ over $\U(V^-)=\U(2)$. 

The Schwartz function \eqref{eq:phi2p} is the the $(v_3,v_4)$-part of $\Phi_{\cD,p}$ where $\Phi_{\cD,p}$ is the Schwartz function on $\bQ^4_p$ defined and appearing in \cite[Lemma 6.37, Definition 6.36]{WanU31}. As mentioned above, our formula  in the Proposition corresponds to replacing $g$ by $g\eta^{-1}$ in the formula in {\it loc.cit}, so there is a translation by $\begin{pmatrix}&-1\\1\end{pmatrix}$ in  the theta series in \eqref{eq:lE}.
\end{proof}

\subsection{The construction of the auxiliary CM families {\boldmath $ h,\tilde{h}_3,\theta,\tilde{\theta}_3$}}\label{sec:CM-families}
Besides $\theta^J_1$, our strategy described in \S\ref{sec:strategy} also includes auxiliary CM forms. We want to construct a CM family $\bfh$ on $\U(V)=\U(2)$ and pair it with $l_{\theta^J_{1}}\left(\sum_i b_i\bfE^\Kling_{\varphi,\beta,u_i}\right)$. The splitting characters for the theta correspondence between $\U(2)$ and $\U(1)$ are chosen as in \eqref{eq:split-char}.

\subsubsection{The auxiliary group $U_{\cK,p}$}\label{sec:UKp}

First, we note that for $\tau_{p\adic}$ as in Theorem~\ref{thm:Sieg-F}, the nebentypus of $l_{\theta^J_{1}}\left(\sum_i b_i\bfE^\Kling_{\varphi,\beta,u_i}\right)$ at $p$ is $\left(\triv,(\xi_0\tau_0)_\fp|_{\bZ^\times_p}\right)$ ({\it i.e.} $\varrho^{-1}_\fp\begin{pmatrix}\ast&\ast\\0&a_2\end{pmatrix}\in \U(2)(\bZ_p)$ acts by $(\xi_0\tau_0)_\fp(a_2)$). In order to get a CM family $\bfh$ with the correct nebentypus to pair with $l_{\theta^J_{1}}\left(\sum_i b_i\bfE^\Kling_{\varphi,\beta,u_i}\right)$, we want Hecke characters of $\cK^\times\backslash\bA^\times_\cK$ unramified away from $p\infty$ with restriction  $\left(\tau_{0,\fp}|_{\bZ^\times_p},\triv\right)$ on $\cO^\times_{\cK,p}$. However, in general there is no canonical way to extend local characters $\left(\tau_{0,\fp}|_{\bZ^\times_p},\triv\right)$ of $\cO^\times_{\cK,p}$  to global Hecke characters. Hence, we need to consider an auxiliary group $U_{\cK,p}$ and the non-standard type of $p$-adic measures defined in \eqref{eq:MU2'}.

\vspace{.5em}

Let $\glsuseri{UKp}=1+p\cO_{\cK,p}$ and $\glsuserii{UKp}=1+p\cO_{\cK,\fp}$, $\glsuseriii{UKp}=1+p\cO_{\cK,\bar{\fp}}$. Then the natural map
\[
   U_{\cK,p}\lra \cK^\times\backslash\bA^\times_{\cK,f}\lra \Gamma_\cK
\]
is an injection. We can pick $\gls{wtp}$, a non-negative power of $p$, such that raising to the $\wtp$-th power maps $\Gamma_\cK$ into $U_{\cK,p}$. Define
\[ 
\begin{tikzcd}[column sep = huge]
   \gls{PN}:\Gamma_\cK   \arrow[r,hook,"\text{$\wtp$-th power}"] &U_{\cK,p} \arrow[r,hook,"\text{natural proj.}"] & U_{\cK,\fp}\simeq 1+p\bZ_p.
\end{tikzcd}
\] 
Then given a (local) $p$-adic character $\epsilon:U_{\cK,\fp}\ra\ol{\bQ}_p^\times$, the composition $\epsilon\circ \pP_\wtp$ is a  (global) $p$-adic character of $\Gamma_\cK$ whose restriction back to $U_{\cK,p}=U_{\cK,\fp}\times U_{\cK,\bar{\fp}}$ is $\left(\epsilon^{\wtp},\triv\right)$.  

In particular, for an algebraic Hecke character $\tau:\cK^\times\backslash\bA^\times_\cK\ra\bC^\times$ of $\infty$-type $\left(0,k\right)$ whose $p$-adic avatar $\tau_{p\adic}$ factors through $\Gamma_\cK$, we can define
\[
  \gls{tauP}=\left.\tau_{p\adic}\right|_{U_{\cK,\fp}}\circ \pP_N.
\]
Then $\tau_{\fp,\pP_\wtp}$ is a $p$-adic Hecke character of $\Gamma_\cK$ with
\[
    \tau_{\fp,\pP_\wtp}|_{\cO^\times_{\cK,p}}=\left(\tau^{\wtp}_{\fp}|_{\bZ^\times_p},\triv\right)=\left(\tau^\wtp_{0,\fp}|_{\bZ^\times_p},\triv\right).
\]
Because $\tau$ has $\infty$-type $(0,k)$, the local character $\tau_{p\adic}|_{U_{\cK,\fp}}$ is of finite order, and it follows that the (global) character $\tau_{\fp,\pP_\wtp}:\Gamma_\cK\ra\ol{\bQ}^\times_p$ is also of finite order. Hence, it takes values in $\ol{\bQ}^\times$ and is also an algebraic Hecke character of $\cK^\times\backslash\bA^\times_\cK$ of $\infty$-type $\left(0,0\right)$. (From the definition in \eqref{eq:chip}, we can see that the $p$-adic avatar of a Hecke character of $\infty$-type $(0,0)$ is itself.)

\vspace{.5em}
The map $\pP_\wtp$ induces a map
\begin{align}
   \label{eq:PN*}\gls{PNast}:\Meas\left(\Gamma_\cK,V_{\U(2)}\right)^\natural
   &\lra \Meas\left(U_{\cK,p},V_{\U(2)}\right)^\natural.
\end{align}
(See \S\ref{sec:U2-family} for the definition of these spaces of $p$-adic measures.) We have
\[
   \pP_{\wtp,\ast}\left(l_{\theta^J_{1}}\left(\sum\nolimits_i b_i\bfE^\Kling_{\varphi,\beta,u_i}\right)\right)\left(\tau_{p\adic}|_{U_{\cK,_p}}\right)= l_{\theta^J_{1}}\left(\sum\nolimits_i b_i\bfE^\Kling_{\varphi,\beta,u_i}\right)(\tau^\wtp_{p\adic})
\]
for every $\tau_{p\adic}\in\Hom_\cont\left(\Gamma_\cK,\ol{\bQ}^\times_p\right)$. In order to show the nonvanishing property of $l_{\theta^J_{1}}\left(\sum_i b_i\bfE^\Kling_{\varphi,\beta,u_i}\right)$ in Proposition~\ref{prop:E-Kling-nv}, it suffices to show the nonvanishing for $\pP_{\wtp,\ast}\left(l_{\theta^J_{1}}\left(\sum_i b_i\bfE^\Kling_{\varphi,\beta,u_i}\right)\right)$.

\vspace{.5em}

Next, we construct auxiliary CM families $\bfh$, $\tilde{\bfh}_3$, $\bmtheta$, $\tilde{\bmtheta}_3$ on $\U(2)$ as $p$-adic measures on $U_{\cK,p}$.

\vspace{.5em}

\subsubsection{The auxiliary CM families $\bfh$ and $\tilde{\bfh}_3$}\label{sec:Fh}

\begin{prop}\label{prop:h}
There exist a CM family $\glsuseri{h0}\in \Meas\left(U_{\cK,p},e_\ord V'_{\U(2),\,\xi^{-1}}\,\hat{\otimes}\,\hat{\cO}^\ur_L\right)^{\natural}$ and a CM family $\glsuserii{h0}\in \Meas\left(U_{\cK,p},e_\ord V_{\U(2),\,\xi}\,\hat{\otimes}\,\hat{\cO}^\ur_L\right)^{\natural}$ such that for every $\tau_{p\adic}$ as in Theorem~\ref{thm:Sieg-F} sufficiently ramified at $p$,
\begin{align*}
    \bfh_0\left(\tau_{p\adic}|_{U_{\cK,p}}\right)(g)&=\int_{\U(1)(\bQ)\backslash\U(1)(\bA_\bQ)} \theta\left(\phi_{h,\tau}\right)(g,u)\left(\chi_h\tau_{\fp,\pP_\wtp}\right)(u)\,du,\\
    \bfh'_0\left(\tau_{p\adic}|_{U_{\cK,p}}\right)(g)&=\int_{\U(1)(\bQ)\backslash\U(1)(\bA_\bQ)} \theta\left(\phi'_{h,\tau}\right)(g,u)\left(\chi^{-1}_h\tau^{-1}_{\fp,\pP_\wtp}\right)(u)\,du,
\end{align*} 
the theta lift of $\chi^{-1}_h\tau^{-1}_{\fp,\pP_\wtp}\Big|_{\U(1)}$ (resp. $\chi_h\tau_{\fp,\pP_\wtp}\Big|_{\U(1)}$) to $\U(2)$) with respect to the Schwartz function $\phi_{h,\tau}$ (resp. $\phi'_{h,\tau}$) on $\bA^2_\bQ$ described as follows:
\begin{enumerate}
\item[--] $v\notin \Sigma\cup \{\ell,\ell',p\}$. 
\[
    \phi_{h,\tau,v}=\phi'_{h,\tau,v}=\phi_{2,\xi\tau^\wtp,v} \text{ in Proposition~\ref{prop:phi2f2}}.
\] 
\item[--] $v=\fv\bar{\fv}\in \Sigma_\rms\cup\{\ell\}$. We have $\chi_h\tau_{\fp,\pP_\wtp}|_{\cO^\times_{\cK,v}}=\chi_h|_{\cO^\times_{\cK,v}}$.
\begin{align*}
     \phi_{h,\tau,v}&=\text{Image of \eqref{eq:int-S} of }\Big((y_1,y_2)\longmapsto \mathds{1}_{\bZ_v}(y_1 )\cdot\mathds{1}_{\bZ^\times_v}(y_2)\chi_{h,\fv}\chi^{-1}_{h,\bar{\fv}}\lambda^{-1}_{\bar{\fv}}(y_2)\Big),\\
     \phi'_{h,\tau,v}&=\text{Image of \eqref{eq:int-S} of }\Big((y_1,y_2)\longmapsto \mathds{1}_{\bZ_v}(y_1 )\cdot\mathds{1}_{\bZ^\times_v}(y_2)\chi^{-1}_{h,\fv}\chi_{h,\bar{\fv}}\lambda_{\bar{\fv}}(y_2)\Big).
\end{align*}
\item[--] $v\in \Sigma_\ns\cup\{\ell'\}$. We have $\left.\chi_h\tau_{\fp,\pP_\wtp}\right|_{\U(1)(\bQ_v)}=\left.\chi_h\right|_{\U(1)(\bQ_v)}$.
\begin{align*}
    \phi_{h,\tau,v}=&\,\text{a Schwartz function on $\bQ^2_v$ invariant under $K_v$, independent of $\tau$, }\\
   &\,\text{belonging to the $\chi^{-1}_{h,v}|_{\U(1)(\bQ_v)}$-eigenspace for the action of $\U(1)(\bQ_v)$},\\
   \phi'_{h,\tau,v}=&\,\text{a Schwartz function on $\bQ^2_v$ invariant under $K_v$,, independent of $\tau$, }\\
   &\,\text{belonging to the $\chi_{h,v}|_{\U(1)(\bQ_v)}$-eigenspace for the action of $\U(1)(\bQ_v)$,}\\
   &\,\text{and }\int_{\bQ^2_v}\phi_{h,\tau,v}(y)\phi'_{h,\tau,v}(y)\,dy\neq 0.
\end{align*}

\item[--] $v=p$. We have $(\chi_h\tau_{\fp,\pP_\wtp})_\fp=(\xi_0\tau^\wtp_0)_\fp$. Let $t=\ord_p\left(\cond\left((\xi_0\tau^\wtp_0)_{\fp}\right)\right)$. 
\begin{align*}
    \phi_{h,\tau,p}&=\text{Im of \eqref{eq:int-S} of }\\
    &\quad\Big((y_1,y_2)\longmapsto p^{-t}\fg\left((\xi_0\tau^\wtp_0)^{-1}_{\fp}\right))\cdot \mathds{1}_{p^{-t}\bZ^\times_p}(y_2)(\xi_0\tau^{\wtp}_0)_{\fp}(-p^ty_2)\cdot \mathds{1}_{\bZ_p}(y_1)\Big),\\
    \phi'_{h,\tau,p}&=\text{Image of \eqref{eq:int-S} of }\Big((y_1,y_2)\longmapsto  \mathds{1}_{\bZ^\times_p}(y_1)(\xi_0\tau^{\wtp}_0)_{\fp}(y_1)\cdot \mathds{1}_{\bZ_p}(y_2)\Big).
\end{align*}

\end{enumerate}

\end{prop}

\begin{proof}
See \cite[\S8B]{WanU31}.  The idea of constructing $\bfh_0$ (resp. $\bfh'_0$) is to first use the chosen Schwartz function to construct a family of theta series on $\U(2,2)$ by interpolating the $q$-expansions, and then restrict the family to $\U(V)\times\U(V^-)$ and evaluate at a suitable point $u_{aux}\in\U(V)(\bA_{\bQ})$ (resp. $u'_{aux}\in\U(V^-)(\bA_{\bQ})$) to get the desired family on $\U(V^-)$ (resp. $\U(V)$). 
\end{proof}

\begin{defn}
The map \eqref{eq:h-twist} between automorphic forms on $\U(2)$ defined in terms of the tame character $\chi_h\chi^c_\theta$ extends to a map from $V_{\U(2)}$ to $V_{\U(2)}$. Applying this twisting map to $\bfh_0$ in Proposition~\ref{prop:h} defines the CM family $\glsuseri{h}\in \Meas\left(U_{\cK,p},e_\ord V'_{\U(2),\,\xi^{-1}}\right)^{\natural}$. 

Replacing  $\chi_h\chi^c_\theta$ by $\chi^{-1}_h\chi^{-c}_\theta$ and applying the corresponding twisting map map to $\bfh'_0$ in Proposition~\ref{prop:h} defines the CM family $\glsuserii{h}\in \Meas\left(U_{\cK,p},e_\ord V_{\U(2),\,\xi}\right)^{\natural}$. 
\end{defn}

\subsubsection{The auxiliary CM families $\bmtheta$ and $\tilde{\bmtheta}_3$}\label{sec:Ftheta}

We construct a CM family $\bmtheta$ whose specializations are closely related to the theta series $\theta(\phi_{2,\tau})$ in Proposition~\ref{prop:phi2f2}. (The theta series $\theta(\phi_{2,\tau})$ is not an eigenform and does not belong to the theta lift of one character of $\U(1)$. The specialization of  $\bmtheta$ at $\tau_{p\adic}|_{U_{\cK,p}}$ is essentially the projection of $\theta(\phi_{2,\tau^\wtp})$ to the theta lift of $\lambda^2\chi_\theta\tau^{-c}_{\fp,\pP_\wtp}|_{\U(1)}$.) We also construct a CM family $\tilde{\bmtheta}_3$ dual to $\bmtheta$. 

\begin{prop}\leavevmode

\noindent (1) There exist a CM family $\glsuseri{theta}\in\Meas\left(U_{\cK,p},e_\ord V_{\U(2),\,\xi}\,\hat{\otimes}\,\hat{\cO}^\ur_L\right)^\natural$ and also a CM family $\glsuserii{theta}\in\Meas\left(U_{\cK,p},e_\ord V'_{\U(2),\,\xi^{-1}}\,\hat{\otimes}\,\hat{\cO}^\ur_L\right)^{\natural}$ such that for every $\tau_{p\adic}$ as in Theorem~\ref{thm:Sieg-F} ramified at $p$, 
\begin{equation}\label{eq:bftheta}
\begin{aligned}
   \bmtheta\left(\tau_{p\adic}|_{U_{\cK,p}}\right)(g)&= \int_{\U(1)(\bQ)\backslash\U(1)(\bA_\bQ)}  \theta\left(\phi_{\theta,\tau}\right)\left(g,u\right)\left(\lambda^2\chi_\theta\tau^{-c}_{\fp,\pP_\wtp}\right)^{-1}(u)\,du,\\
   \tilde{\bmtheta}_3\left(\tau_{p\adic}|_{U_{\cK,p}}\right)(g)&= \int_{\U(1)(\bQ)\backslash\U(1)(\bA_\bQ)}  \theta\left(\phi'_{\theta,\tau}\right)\left(g,u\right)\,(\lambda^2\chi_\theta\tau^{-c}_{\fp,\pP_\wtp})(u)\,du,
\end{aligned}
\end{equation}
the theta lift of $\lambda^2\chi_\theta\tau^{-c}_{\fp,\pP_\wtp}\Big|_{\U(1)}$ (resp. $\lambda^{-2}\chi^{-1}_\theta\tau^{c}_{\fp,\pP_\wtp}\Big|_{\U(1)}$) to $\U(2)$ with respect to the Schwartz function $\phi_{\theta,\tau}$ (resp. $\phi'_{\theta,\tau}$) on $\bA^2_\bQ$ described as follows:
\vspace{.5em}

\begin{enumerate}
\item[--] $v\notin\Sigma\,\cup\,\{\ell,\ell',p\}$. 
\begin{align*}
   \phi_{\theta,\tau,v}=\phi'_{\theta,\tau,v}=\phi_{2,\xi\tau^\wtp,v} \text{in Proposition~\ref{prop:phi2f2}}.
\end{align*}

\item[--] $v=\fv\bar{\fv}\in\Sigma_\rms\,\cup\,\{\ell'\}$. We have $\chi_{\theta}\tau^{-c}_{\fp,\pP_\wtp}|_{\cO^\times_{\cK,v}}=\chi_{\theta}|_{\cO^\times_{\cK,v}}$ and $\chi_{\theta,\fv}$ unramified.
\begin{align*}
   \phi_{\theta,\tau,v}&=\text{Image of \eqref{eq:int-S} of } \Big((y_1,y_2)\longmapsto \mathds{1}_{\cO_{\cK,\fv}}(y_1)\cdot \mathds{1}_{\cO^\times_{\cK,\bar{\fv}}}(y_2)\chi_{\theta,\bar{\fv}}\lambda_{\bar{\fv}}(y_2)\Big)\\
   &=\phi_{2,\tau^\wtp,v} \text{in Proposition~\ref{prop:phi2f2}},\\
   \phi'_{\theta,\tau,v}&=\text{Image of \eqref{eq:int-S} of } \Big((y_1,y_2)\longmapsto \mathds{1}_{\cO_{\cK,\fv}}(y_1)\cdot \mathds{1}_{\cO^\times_{\cK,\bar{\fv}}}(y_2)(\chi_{\theta,\bar{\fv}}\lambda_{\bar{\fv}})^{-1}(y_2)\Big)
\end{align*}

\item[--] $v\in \Sigma_\ns\,\cup\,\{\ell'\}$. We have $\chi_{\theta}\tau^{-c}_{\fp,\pP_\wtp}|_{\U(1)(\bQ_v)}=\chi_{\theta}|_{\U(1)(\bQ_v)}$.
\begin{align*}
   \phi_{\theta,\tau,v}&=\phi_{2,\xi\tau^\wtp,v} \text{in Proposition~\ref{prop:phi2f2}},\\
   \phi'_{\theta,\tau,v}&=\text{a Schwartz function on $\bQ^2_v$ invariant under $K_v$, independent of $\tau$,}\\
  &\quad\text{belonging to the $\chi^{-1}_\theta\lambda^{-2}|_{\U(1)(\bQ_v)}$-eigenspace for the action of $\U(1)(\bQ_v)$, }\\
  &\quad\text{and }\int_{\bQ^2_v}\phi_{\theta,\tau,v}(y) \phi'_{\theta,\tau,v}(y)\,dy\neq 0
\end{align*}

\item[--] $v=p$. We have $(\chi_\theta\tau^{-c}_{\fp,\pP_\wtp})_{\bar{\fp}}=(\xi_0\tau^\wtp_0)^{-1}_\fp$. Let  $t=\ord_p\left(\cond\left((\xi_0\tau^\wtp_0)_\fp\right)\right)$.
\begin{align*}
    \phi_{\theta,\tau,p}&=\text{Image of \eqref{eq:int-S} of }\\
    &\quad\Big((y_1,y_2)\longmapsto p^{-t}\fg\left((\xi_0\tau^{\wtp}_0)^{-1}_{\fp}\right)\cdot \mathds{1}_{p^{-t}\bZ^\times_p}(y_1)(\xi_0\tau^{\wtp}_0)_{\fp}(-p^ty_1)\cdot \mathds{1}_{\bZ_p}(y_2)\Big),\\
    &=\sum_{\bZ/p^t\bZ}\omega_{\beta,p}\left(\begin{pmatrix}1&0\\n&1\end{pmatrix},1\right)\phi_{2,\tau^\wtp,p},\\
   \phi'_{\theta,\tau,p}&=\text{Image of \eqref{eq:int-S} of }\Big((y_1,y_2)\longmapsto  \mathds{1}_{\bZ_p}(y_1)\cdot \mathds{1}_{\bZ^\times_p}(y_2)(\xi_0\tau^\wtp_0)_{\fp}(y_2)  \cdot\Big),
\end{align*}
where $\phi_{2,\tau^\wtp,p}$ is as in Proposition~\ref{prop:phi2f2}.
\end{enumerate}

\noindent (2) The specialization $\bmtheta\left(\tau_{p\adic}|_{U_{\cK,p}}\right)$ equals the projection of $\sum\limits_{n\in\bZ/p^t\bZ}\theta(\phi_{2,\tau^\wtp})\left(g\begin{pmatrix}1&0\\n&1\end{pmatrix}\right)$ to the $\left.\lambda^2\chi_\theta\tau^{-c}_{\fp,\pP_\wtp}\right|_{\U(1)}$-eigenspace for the action of the center of $\U(2)$. (Note that the projection of a theta series to the $\left.\lambda^2\chi_\theta\tau^{-c}_{\fp,\pP_\wtp}\right|_{\U(1)}$-eigenspace for the action of the center is a theta lift of $\left.\lambda^2\chi_\theta\tau^{-c}_{\fp,\pP_\wtp}\right|_{\U(1)}$.)

\end{prop}

\begin{proof}
See \cite[\S8B]{WanU31}. The idea is the same as the construction of $\bfh_0$ and $\bfh'_0$, and (2) follows immediately from (1).
\end{proof}

\subsection{$p$-adic Petersson inner product on $\U(2)$}\label{sec:p-inner}

We need $p$-adic Petersson inner product for families on $\U(2)$. Let $V_{\U(2)}$, $e_\ord V'_{\U(2)}$ be the space of $p$-adic forms on $\U(2)$ defined in \S\ref{sec:U2-family}. 

Following \cite[p.33]{hsieh-triple}, we can define a $p$-adic Petersson inner product
\begin{equation}\label{eq:p-inner}
   \left<\bdot,\bdot\right>_{p\adic}: V'_{\U(2),\ord}\times V_{\U(2)}\lra \cO_L
\end{equation}
such that for all $\phi\in V_{\U(2)}$ (resp. $\phi'\in e_\ord V'_{\U(2)}$) of with $K_{p,0}(p^n)$-nebentypus $(\triv,\chi)$ (resp. $(\chi',\triv)$) 
\[
  \glsuseri{Petersson} =\left\{\begin{array}{ll}
     \sum\limits_{g}\, (U^{-n}_p\phi')\left(g\varrho^{-1}_\fp\begin{pmatrix}&1\\-p^n\end{pmatrix}\right)\phi\left(g\right), &\chi\chi'=\triv,\\[1.5em]
     0, &\chi\chi'\neq\triv,
   \end{array}\right.
\]
where $K_{p,0}(p^n)=\left\{g\in\U(2)(\bZ_p):\varrho_\fp(g)\equiv\begin{pmatrix}\ast&\ast\\0&\ast\end{pmatrix}\mod p^n\right\}$, and in the sum, $g$ runs over the finite set $\U(2)(\bQ)\backslash \U(2)(\bA^p_{\bQ,f})/K^p_f K_{p,0}(p^n)$.

A useful property of this $p$-adic Petersson inner product is that
\[
   \left<\phi',U_p\phi\right>_{p\adic}=\omega_{\phi',p}(\varrho^{-1}_\fp(p))^{-1}\left<U_p \phi, \phi'\right>_{p\adic}.
\]
if $\phi'$ has central character $\omega_{\phi'}$. Here $\varrho^{-1}_\fp$ is the isomorphism from $\bQ^\times_p$ to $\U(1)(\bQ_p)$.

\vspace{.5em}
The $p$-adic Petersson inner product in \eqref{eq:p-inner} induces the $p$-adic Petersson inner product for $p$-adic families  on $\U(2)$:
\begin{equation}\label{eq:Petersson-family}
   \left<\bdot,\bdot\right>_{p\adic}: \Meas\big(U_{\cK,p},e_\ord V'_{\U(2),\,\xi^{-1}}\big)^\natural\times\Meas\left(U_{\cK,p},V_{\U(2),\,\xi}\right)^\natural
    \lra \Meas\big(U_{\cK,p},\cO_L\big)\simeq \cO_L\llb U_{\cK,p}\rrb.
\end{equation}

\subsection{The Rallis inner product formulas for $\big<\boldmath \theta,\tilde{\theta}_3\big>_{p\adic}$ and $\big<\boldmath h,\tilde{h}_3\big>_{p\adic}$.}\label{sec:Rallis-inner prod}

From the description of our strategy in \S\ref{sec:strategy}, in particular \eqref{eq:FJ-L1256}, we see that we need a formula relating $\big<\bmtheta,\tilde{\bmtheta}_2\big>_{p\adic}$ and $\big<\bfh,\tilde{\bfh}_3\big>_{p\adic}$ to certain  $p$-adic $L$-functions.

\vspace{.5em}

Let $\gls{L3}\in \Meas\left(U_{\cK,p},\hat{\cO}^\ur_L\right)\otimes_\bZ\bQ\simeq\hat{\cO}^\ur_L\llb U_{\cK,p}\rrb\otimes_\bZ\bQ$ be the $L$-function interpolating the special value at $s=1$ of $L\left(s,\pi^{\GL_2}_\theta,\mr{Ad}\right)$, $L\left(s,\pi^{\GL_2}_h,\mr{Ad}\right)$, which factorize as 
\begin{align*}
   L\left(s,\pi^{\GL_2}_\theta,\mr{Ad}\right)&=L\left(s,\eta_{\cK/\bQ}\right)L\left(s,\lambda^2\chi_\theta\chi^{-c}_\theta\tau_{\fp,\pP_\wtp}\tau^{-c}_{\fp,\pP_\wtp}\right),\\
   L\left(s,\pi^{\GL_2}_h,\mr{Ad}\right)&= L\left(s,\eta_{\cK/\bQ}\right)L\left(s,\lambda^{-2}\chi^{-1}_h\chi^{c}_h\tau^{-1}_{\fp,\pP_\wtp}\tau^{c}_{\fp,\pP_\wtp}\right)\\
   &=L\left(s,\eta_{\cK/\bQ}\right)L\left(s,\lambda^2\chi_h\chi^{-c}_h\tau_{\fp,\pP_\wtp}\tau^{-c}_{\fp,\pP_\wtp}\right).
\end{align*}
More precisely, for $\tau_{p\adic}$ as in Theorem~\ref{thm:Sieg-F},
\begin{align*}
   \cL_3\left(\tau_{p\adic}|_{U_{\cK,p}}\right)
   &=\pi^{-1}\cdot L^\infty\left(1,\eta_{\cK/\bQ}\right)\\
   &\quad \times \left(\frac{\Omega_p}{\Omega_\infty}\right)^2\pi^{-2}\cdot \gamma_{\bar{\fp}}\left(0,\lambda^2\chi_\theta\chi^{-c}_\theta\tau_{\fp,\pP_\wtp}\tau^{-c}_{\fp,\pP_\wtp}\right)^{-1}\cdot L^{p\infty}\left(1,\lambda^2\chi_\theta\chi^{-c}_\theta\tau_{\fp,\pP_\wtp}\tau^{-c}_{\fp,\pP_\wtp}\right),\\
   \cL_4\left(\tau_{p\adic}|_{U_{\cK,p}}\right)
   &=\pi^{-1}\cdot L^\infty\left(1,\eta_{\cK/\bQ}\right)\\
   &\quad \times \left(\frac{\Omega_p}{\Omega_\infty}\right)^2\pi^{-2}\cdot \gamma_{\bar{\fp}}\left(0,\lambda^2\chi_h\chi^{-c}_h\tau_{\fp,\pP_\wtp}\tau^{-c}_{\fp,\pP_\wtp}\right)^{-1}\cdot L^{p\infty}\left(1,\lambda^2\chi_h\chi^{-c}_h\tau_{\fp,\pP_\wtp}\tau^{-c}_{\fp,\pP_\wtp}\right).
\end{align*}
The existence of $\cL_3,\cL_4$ follows from Katz $p$-adic $L$-functions.

\begin{prop}\label{prop:Rallis}
There exists $C_\theta,C_h\in L^\times$ such that 
\begin{align*}
   \big<\tilde{\bmtheta}_3,\bmtheta\big>_{p\adic}&=C_\theta\cdot \cL_3,
   &\big<\bfh,\tilde{\bfh}_3\big>_{p\adic}&=C_h\cdot \cL_4.
\end{align*}
(See \eqref{eq:Petersson-family} for the definition of the $p$-adic Petersson inner product for families.)
\end{prop}
\begin{proof}
The proof is the same as \cite[Proposition 8.9]{WanU31}, which computes the specializations of the left hand side by using the Rallis inner product formula. 
\end{proof}

\subsection{Extending CM forms on $\U(2)$ to $\GU(2)$}\label{sec:ext-GU2}

Let $H=\GU(2)(\bQ)Z_{\GU(2)}(\bA_\bQ)\U(2)(\bA_\bQ)$. Then $H$ is a subgroup of  $\GU(\bA_\bQ)$ of index $2$ consisting of $g\in\GU(2)(\bA)$ with $\nu(g)\in \bQ^\times\Nm(\bA^\times_\cK)$. Suppose that $\phi$ is an automorphic form on $\U(2)$ and generates an irreducible automorphic representation $\pi_\phi$ of $\U(2)(\bA_\bQ)$. Let $\omega:\cK^\times\backslash\bA^\times_\cK\ra\bC^\times$ be a Hecke character extending the central character of $\pi_\phi$. We can first extend $\phi$ to a function on $H$ by
\begin{align*}
   \gamma a g_1&\longmapsto \omega(a)\phi(g_1), &\gamma\in\GU(2)(\bQ),\,a\in Z_{\GU(2)}(\bA_\bQ),\,g_1\in\U(2)(\bA_\bQ),
\end{align*}
and then extend this function by zero from $H$ to $\GU(2)(\bA_\bQ)$. We denote this extension of $\phi$ to $\GU(2)(\bA_\bQ)$ by $\breve{\phi}$. This $\breve{\phi}$ is an automorphic form on $\GU(2)(\bA_\bQ)$.

\begin{prop}\label{prop:ext-CM}
Suppose that $\pi_\phi$ is of CM type, {\it i.e.} if $\mr{BC}(\pi_\phi)\simeq \mr{BC}(\chi_1)\lambda\oplus\mr{BC}(\chi_2)\lambda$ for some characters $\chi_1,\chi_2$ of $\U(1)(\bQ)\backslash\U(1)(\bA_\bQ)$. Then $\breve{\phi}$ generates an irreducible automorphic representation $\pi_{\breve{\phi}}$ of  $\GU(2)(\bA)$ and $\pi_{\breve{\phi}}|_{\U(2)}\simeq \pi_\phi\oplus\pi'_\phi$, where $\pi'_\phi$ is isomorphic to the conjugation of $\pi_\phi$ by an element in $\GU(2)(\bA_\bQ)$ outside $H$. Moreover, if $\phi\in\pi_\phi$ is a pure tensor in $\pi_\phi\simeq\bigotimes_v\pi_{\phi,v}$, then $\breve{\phi}$ is also a pure tensor in $\pi_{\breve{\phi}}\simeq\bigotimes_v\pi_{\breve{\phi},v}$.
\end{prop}
\begin{proof}
Let $\sigma$ be the automorphic representation of $\GU(2)(\bA_\bQ)$ generated by $\breve{\phi}$. Then either $\sigma$ is irreducible and $\sigma\simeq \sigma\otimes\eta_{\cK/\bQ}\circ \nu$ or $\sigma\simeq \sigma_1\oplus \sigma_2$ with $\sigma_1\not\simeq\sigma_2$ and $\sigma_2\simeq \sigma_1\otimes\eta_{\cK/\bQ}\circ\nu$. If $\pi_\phi$ is of CM type and $\sigma_1$ is a direct summand of $\sigma$, we know that $\sigma_1\simeq\sigma_1\otimes\eta_{\cK/\bQ}\circ\nu$. Hence, $\sigma$ must be irreducible. (See also \cite[Lemma 8.1]{WanU31}.) We have $\pi_{\breve{\phi}}\simeq \pi_\phi$. Since $\breve{\phi}$ is obtained by extension by zero, it belongs to the direct summand $\pi_\phi$ in $\pi_{\breve{\phi}}|_{\U(2)}\simeq \pi_\phi\oplus\pi'_\phi$. Hence, $\phi$ being a pure tensor in $\pi_\phi$ implies that $\breve{\phi}$ is a pure tensor in $\pi_{\breve{\phi}}$.
\end{proof}

The above extension from $\U(2)$ to $\GU(2)$ also works for $p$-adic automorphic forms. We denote by $\glsuseri{hext}$  the extension of $\bfh$  whose specialization at $\tau_{p\adic}$ is the extension by zero of $\bfh(\tau_{p\adic})$ to $\GU(2)$ with respect to the character $\chi^{-1}_\theta\tau^{-1}_{\fp,\pP_\wtp}$. Similarly, We denote by $\glsuserii{hext}$ (resp. $\glsuseri{thetaext},\glsuserii{thetaext}$) the extension of $\tilde{\bfh}_3$  whose specialization at $\tau_{p\adic}$ is the extension by zero of $\tilde{\bfh}_3(\tau_{p\adic})$ to $\GU(2)$ with respect to the character $\chi_\theta\tau_{\fp,\pP_\wtp}$ (resp. $\chi_\theta\tau_{\fp,\pP_\wtp},\chi^{-1}_\theta\tau^{-1}_{\fp,\pP_\wtp}$).

\subsection{The nonvanishing property of the degenerate Fourier--Jacobi coefficients of the Klingen Eisenstein family}\label{nv-EKling}

Let $\gls{L5}\in \Meas\left(U_{\cK,p},\hat{\cO}^\ur_L\right)\simeq\hat{\cO}^\ur_L\llb U_{\cK,p}\rrb$ be the Katz $p$-adic $L$-functions interpolating special values of $L\left(s,\chi_h\chi^c_\theta\xi_0\tau^\wtp_0\right)$, $L\left(s,\lambda^2\chi_h\chi_\theta\,\tau^{-1}_{\fp,\pP_\wtp}\tau^{c}_{\fp,\pP_\wtp}(\xi_0\tau^\wtp_0)^c\right)$. In particular, for a Hecke character $\tau:\cK^\times\backslash\bA^\times_\cK\ra\bC^\times$ unramified away from $p\infty$ of $\infty$-type $\left(0,k\right)$ with $k\geq 6$,
\begin{align*}
   \cL_5\left(\tau_{p\adic}|_{U_{\cK,p}}\right)&=\left(\frac{\Omega_p}{\Omega_\infty}\right)^{\wtp k}\frac{\Gamma(\wtp k-1)}{(2\pi i)^{\wtp k-1}}
   \cdot \gamma_{\bar{\fp}}\left(\frac{4-\wtp k}{2},(\chi_h\chi^{c}_\theta\xi_0\tau^{\wtp}_0)^{-1}\right) \\
   &\quad\times L^{p\infty}\left(\frac{\wtp k-2}{2},\chi_h\chi^c_\theta\xi_0\tau^\wtp_0\right),\\
   \cL_6\left(\tau_{p\adic}|_{U_{\cK,p}}\right)&=\left(\frac{\Omega_p}{\Omega_\infty}\right)^{\wtp k-2}\frac{\Gamma(\wtp k-2)}{(2\pi i)^{\wtp k-2}}
   \cdot  L_{\fp}\left(\frac{\wtp k-2}{2},\lambda^2\chi_h\chi_\theta\tau_{\fp,\pP_\wtp}\tau^{-c}_{\fp,\pP_\wtp}(\xi_0\tau^\wtp_0)^c\right),\\
   &\quad\times L^{p\infty}\left(\frac{\wtp k-2}{2},\lambda^2\chi_h\chi_\theta\tau_{\fp,\pP_\wtp}\tau^{-c}_{\fp,\pP_\wtp}(\xi_0\tau^\wtp_0)^c\right).
\end{align*}
The product of $\cL_5\cL_6$ interpolates the special values of $L(s,\mr{BC}(\pi_h)\times\xi_0\tau_0\lambda^{-1})$ where $\pi_h$ is the automorphic representation associated to specializations of the CM family $\bfh$. Note that the restriction of $\lambda^2\chi_h\chi_\theta\tau_{\fp,\pP_\wtp}\tau^{-c}_{\fp,\pP_\wtp}(\xi_0\tau^\wtp_0)^c$ to $\cO^\times_{\cK,p}$ is $\left(\xi_{0,\fp}\xi_{0,\bar{\fp}}\tau^\wtp_{0,\fp}\tau^\wtp_{0,\bar{\fp}},\triv\right)$, so $\cL_6$ is essentially of one-variable. In fact $\cL_6$ is the so-called ``improved'' $p$-adic $L$-function, for which the local factor at $p$ in the interpolation formula is a partial local $L$-factor instead of a partial local $\gamma$-factor.

\begin{prop}\label{prop:pair-h}
For $\sT_\ns\in \bigotimes\limits_{v\in\Sigma_\ns\cup\{\ell'\}}\cO_L\left[\GU(2)\big(\bQ_v\big)\right]$ such that $\sT_\ns\bfh$ is still invariant under the tame level group $K^p_f$ (defined \eqref{eq:k^p_f}), $\varphi$ inside the space \eqref{eq:M-varphi}, and $u_1,\dots,u_r\in \bigotimes\limits_{v\in\Sigma_\ns\cup\{\ell'\}}\U(1)(\bQ_v)$, $b_1,\dots,b_r\in\cO_L$ as in Proposition~\ref{prop:phi2f2},

\begin{equation}\label{eq:FJ-h}
   \left<\left(\sT_\ns\breve{\bfh}\right)|_{\U(2)},\pP_{N,\ast}\left(l_{\theta^J_1}\left(\sum\nolimits_i b_i\bfE^\Kling_{\varphi,\beta,u_i}\right)\right)\right>_{p\adic}
   =\cC_1\cdot\cL_5\cL_6\cdot\left<\left(\sT_\ns\breve{\bfh}\right)|_{\U(2)},\bmtheta^{\bar{\lambda}}  \varphi\right>_{p\adic},
\end{equation}
where $\bmtheta^{\bar{\lambda}}=\bmtheta\cdot(\bar{\lambda}\circ \det)$ and $\cC_1\in \left(\hat{\cO}^\ur_L\llb U_{\cK,p}\rrb\otimes_\bZ\bQ\right)^\times$.
\end{prop}

\begin{proof}
The identity is proved by comparing the evaluations of both sides at $\tau_{p\adic}$ as in Theorem~\ref{thm:Sieg-F}. For such a $\tau_{p\adic}$, put $\theta_\tau=\bmtheta(\tau_{p\adic}|_{U_{\cK,p}})$ and $h_\tau=\bfh(\tau_{p\adic}|_{U_{\cK,p}})$, which are classical ordinary CM forms. The description of $l_{\theta^J_1}\left(\sum_i b_i\bfE^\Kling_{\varphi,\beta,u_i}\right)$ in Proposition~\ref{prop:phi2f2} and the definition of the $p$-adic Petersson inner product in \S\ref{sec:p-inner} imply that 
\begin{equation}\label{eq:h-L5L6}
\begin{aligned}
   &\left(\frac{\Omega_p}{\Omega_\infty}\right)^{-2\wtp k+2}\cdot\text{LHS of \eqref{eq:FJ-h} evaluated at $\tau_{p\adic}$}\\
   =&\,d^{\Sigma\,\cup\{p,\ell,\ell'\}}_2\left(s,\xi_0\tau^\wtp_0\lambda^{-1}\right)\int_{\U(2)(\bQ)\backslash \U(2)(\bA_\bQ)}\theta_\tau\left(g_2\varrho^{-1}_\fp\begin{pmatrix}&1\\-1\end{pmatrix}\right)\varphi^{\bar{\xi}_0\bar{\tau}^\wtp_0}(g_2) \int_{\U(2)(\bQ)\backslash \U(2)(\bA_\bQ)}\\
   &\left. E^\Sieg_2\left(\imath_0(g_1,g_2);f_2(s,\xi_0\tau^\wtp_0\lambda^{-1}\right)\xi^{-1}_0\tau^{-\wtp}_0\lambda(\det g_1)\cdot(U^{-m}_p \sT_\ns \breve{h}_\tau)\left(g_1\begin{pmatrix}&1&\\-p^m\end{pmatrix}_p\right)\,dg_1dg_2\,\right|_{s=\frac{\wtp k-3}{2}}\\
    =&\,\int_{\U(2)(\bQ)\backslash \U(2)(\bA_\bQ)}\theta^{\bar{\lambda}}_\tau\left(g_2\right)\varphi(g_2)\, h'_\tau(g_2)\,dg_2
\end{aligned}
\end{equation}
with $\imath_0$ defined in \eqref{eq:imath0} and
\begin{equation*}
\begin{aligned}
    h'_\tau(g_2)&= d^{\Sigma\,\cup\{p,\ell,\ell'\}}_2\left(s,\xi_0\tau^\wtp_0\lambda^{-1}\right)^{-1}\\
    &\hspace{-3em}\times\left.\int_{\U(2)(\bA_\bQ)}f_{2,v}\left(s,\xi_0\tau^\wtp_0\lambda^{-1}\right)\left(\imath_0\left(\begin{pmatrix}&1\\-1\end{pmatrix}_pg,1\right)\right)\cdot (U^{-m}_p \sT_\ns\breve{h}_\tau)\left(g_2g\begin{pmatrix}&1&\\-p^m\end{pmatrix}_p\right)\,dg\,\right|_{s=\frac{\wtp k-3}{2}},
\end{aligned}
\end{equation*}
where $m$ is any sufficiently large integer. The integral over $g$ is essentially doubling local zeta integral. 

\vspace{.5em}

By using the same computation results in the proof of Theorem~\ref{thm:const} for all $v\neq p$, we obtain  
\begin{equation}\label{eq:h-L5L6-2} 
\begin{aligned}
   h'_\tau(g_2)&=\cC_1(\tau_{p\adic})\cdot \frac{\Gamma(\wtp k-1)\Gamma(\wtp k-2)}{(2\pi i)^{2\wtp k-3}}\cdot L^{p\infty}\left(\frac{\wtp k-2}{2},\mr{BC}(\pi_{h_\tau})\times \xi_0\tau^\wtp_0\lambda^{-1}\right)\\
   &\hspace{-2em} \times \left.\int_{\U(2)(\bQ_p)}f_{2,p}\left(s,\xi_0\tau^\wtp_0\lambda^{-1}\right)\left(\imath_0\left(\begin{pmatrix}&1\\-1\end{pmatrix}_p g,1\right)\right)(U^{-m}_p \sT_\ns\breve{h}_\tau)\left(g_2g\begin{pmatrix}&1&\\-p^m\end{pmatrix}_p\right)\,dg\,\right|_{s=\frac{\wtp k-3}{2}},
\end{aligned}
\end{equation}
where $\cC_1$ is the product of  
\begin{itemize}
\item the local doubling zeta integrals at $v\in\Sigma\,\cup\{\ell,\ell'\}$, which is a nonzero constant in $L$ by our choice, and
\item the element in $\hat{\cO}^\ur_L\llb U_{\cK,p}\rrb$ interpolating $L_q\left(\frac{\wtp k-2}{2},\chi_h\chi^c_\theta\xi_0\tau^\wtp_0\right)^{-1}$, which is a unit by our assumption.
\end{itemize}
Our choice of $\chi_h$, $\chi_\theta$ implies that $L_v\left(s,\mr{BC}(\pi_{h_\tau})\times \xi_0\tau_0\lambda^{-1}\right)=1$ for $v\in\Sigma\,\cup\{\ell,\ell'\}$. The computation at $p$ is slightly different from that in the  construction of $p$-adic $L$-functions because the section $f_{2,p}(s,\xi_0\tau_0\lambda^{-1})$ is slightly different from the one for $p$-adic $L$-functions. The Schwartz function $\alpha'_{\xi\tau,p}$ in the description of $f_{2,p}(s,\xi_0\tau_0\lambda^{-1})$ in Proposition~\ref{prop:phi2f2} only requires $x_{12}$ to be supported on $p^{-t}\bZ^\times_p$. Unlike the section for constructing $p$-adic $L$-functions, $\alpha_{\xi\tau,p}$ does not require the whole $\begin{pmatrix}x_{11}&x_{12}\\x_{21}&x_{22}\end{pmatrix}$ to be supported on $p^{-t}\GL_2(\bZ_p)$. (In fact, $f_{2,p}(s,\tau_0\lambda^{-1})$ is essentially of the same type as the section in \cite[Table 2 on p.210]{LRtz} for constructing the ``improved'' $p$-adic $L$-function.) The computation at $p$ is done in \cite[6I]{WanU31}, and plugging it into \eqref{eq:h-L5L6-2}, we get
\begin{align*}
   h'_\tau(g_2)=&\,\cC_1(\tau_{p\adic})\cdot\frac{\Gamma(\wtp k-1)\Gamma(\wtp k-2)}{(2\pi i)^{2\wtp k-3}}\cdot  L^{p\infty}\left(\frac{\wtp k-2}{2},\mr{BC}(\pi_{h_\tau})\times \xi_0\tau_0\lambda^{-1}\right)\\
    &\times\gamma_{\bar{\fp}}\left(\frac{4-\wtp k}{2},(\chi_h\chi^{c}_\theta\xi_0\tau^{\wtp}_0)^{-1}\right)
    L_{\fp}\left(\frac{\wtp k-2}{2},\lambda^2\chi_h\chi_\theta\tau_{\fp,\pP_\wtp}\tau^{-c}_{\fp,\pP_\wtp}(\xi_0\tau^\wtp_0)^c\right)\\
    &\times (U^{-m}_p\sT_\ns \breve{h}_\tau)\left(g_2\begin{pmatrix}&1&\\-p^m\end{pmatrix}_p\right).
\end{align*}
Plugging this into \eqref{eq:h-L5L6}, we see that the evaluations at $\tau_{p\adic}$ of the two sides of \eqref{eq:FJ-h} are equal.

\end{proof}

We introduce some more $p$-adic $L$-functions. Let $\glsuseri{L1}\in \Meas\left(U_{\cK,p},\hat{\cO}^\ur_L\right)\otimes_\bZ\bQ\simeq\hat{\cO}^\ur_L\llb U_{\cK,p}\rrb\otimes_\bZ\bQ$ be the $p$-adic $L$-function interpolating the central values of $L\left(s,\mr{BC}(\pi)\times \lambda^2\chi_h\chi_\theta\tau^{-1}_{\fp,\pP_\wtp}\tau^{c}_{\fp,\pP_\wtp}\right)$, {\it i.e.} for $\tau_{p\adic}$ as in Theorem~\ref{thm:Sieg-F},
\begin{align*}
   \cL_1\left(\tau_{p\adic}|_{U_{\cK,p}}\right)
   &=\left(\frac{\Omega_p}{\Omega_\infty}\right)^4\pi^{-3}  \cdot \gamma_p\left(\frac{1}{2},\pi_{f,p}\times\left(\lambda^2\chi_h\chi_\theta\tau_{\fp,\pP_\wtp}\tau^{-c}_{\fp,\pP_\wtp}\right)_{\bar{\fp}}\right)^{-1}\\
   &\quad\times L^{p\infty}\left(\frac{1}{2},\mr{BC}(\pi)\times \lambda^2\chi_h\chi_\theta\tau_{\fp,\pP_\wtp}\tau^{-c}_{\fp,\pP_\wtp}\right).
\end{align*}
The existence of such a $p$-adic $L$-function $\cL_1$ follows from \cite{HsiehRankin}. Let 
\[
   \glsuserii{L1}=\pi^{-2}\left<f,\bar{f}\right>^{-1}\cdot L^\infty\left(\frac{1}{2},\mr{BC}(\pi)\times \chi_h\chi^c_\theta\right)\in L^{\ur}.
\]
The product of $\cL_1\cL_2$ interpolates the central values of $L\left(s,\pi\times\pi^{\GL_2}_h\times\pi^{\GL_2}_{\theta^{\bar{\lambda}}}\right)$ with $h$, $\theta$ specializations of the CM families $\bfh$, $\bmtheta$.

\vspace{.5em}

Note that because $\pi$ has trivial central character and $\chi_h\chi^c_\theta|_{\bA^\times_\bQ}=\triv$, when $\cond\left(\xi_{0,p}\tau^\wtp_{0,\fp}\right)=p^t$,
\begin{equation}\label{eq:triple-gamma} 
\begin{aligned}
   &\quad\gamma_p\left(\frac{1}{2},\pi_{p}\times\left(\lambda^2\chi_h\chi_\theta\tau_{\fp,\pP_\wtp}\tau^{-c}_{\fp,\pP_\wtp}\right)_{\bar{\fp}}\right)^{-1}\\
   &=\fg\left((\xi_0\tau^\wtp_0)^{-1}_\fp\right)^2\cdot \left(\lambda^2_{\bar{\fp}}\chi_{h,\bar{\fp}}\chi_{\theta,\bar{\fp}}\tau^{-\wtp}_{0,\fp}(p)p^{1/2}\right)^{-2t}\\
   &=p^{-t}\cdot \fg\left((\xi_0\tau^\wtp_0)^{-1}_\fp\right)\left(\lambda_{\bar{\fp}}^2\chi_{\theta,\bar{\fp}}\chi^{-1}_{\theta,\fp}\tau^{-\wtp}_{0,\fp}(p)\right)^{-t}\cdot \fg\left((\xi_0\tau^\wtp_0)^{-1}_\fp\right)\left(\lambda_{\bar{\fp}}^2\chi_{h,\bar{\fp}}\chi^{-1}_{h,\fp}\tau^{-\wtp}_{0,\fp}(p)\right)^{-t}\\
   &=p^{-t}\cdot \gamma_{\bar{\fp}}\left(0,\lambda^2\chi_\theta\chi^{-c}_\theta\tau_{\fp,\pP_\wtp}\tau^{-c}_{\fp,\pP_\wtp}\right)^{-1}\cdot \gamma_{\bar{\fp}}\left(0,\lambda^2\chi_h\chi^{-c}_h\tau_{\fp,\pP_\wtp}\tau^{-c}_{\fp,\pP_\wtp}\right)^{-1}.
\end{aligned}
\end{equation}

Combining \eqref{eq:triple-gamma}, the above interpolation formulas for $\cL_1,\cL_2$, and the interpolation formulas for $\cL_3,\cL_4$ in \S\ref{sec:Rallis-inner prod}, we see that when $\cond\left(\xi_{0,p}\tau^\wtp_{0,\fp}\right)=p^t$,
\begin{equation}\label{eq:L1234}
\begin{aligned}
   \left(\frac{\cL_1\cL_2}{\cL_3\cL_4}\right)\left(\tau_{p\adic}|_{U_{\cK,p}}\right)
   =&\,\pi\cdot p^{-t}L_p\left(\frac{1}{2},\mr{BC}(\pi)\times\chi_h\chi^c_\theta\right)\\
   &\hspace{-2em}\times\frac{L^{p\infty}\left(\frac{1}{2},\mr{BC}(\pi)\times \lambda^2\chi_h\chi_\theta\tau_{\fp,\pP_\wtp}\tau^{-c}_{\fp,\pP_\wtp}\right)\cdot L^{p\infty}\left(\frac{1}{2},\mr{BC}(\pi)\times \chi_h\chi^c_\theta\right)}{L^\infty(1,\eta_{\cK,\bQ})^2 \cdot L^{p\infty}\left(1,\lambda^2\chi_\theta\chi^{-c}_\theta\tau_{\fp,\pP_\wtp}\tau^{-c}_{\fp,\pP_\wtp}\right)  \cdot L^{p\infty}\left(1,\lambda^2\chi_h\chi^{-c}_h\tau_{\fp,\pP_\wtp}\tau^{-c}_{\fp,\pP_\wtp}\right)}
\\
    =&\,\pi\frac{L^\infty\left(1,\pi,\mr{Ad}\right)}{\left<f,\bar{f}\right>}   \cdot  p^{-t}\,\frac{L_p\left(\frac{1}{2},\mr{BC}(\pi)\times\chi_h\chi^c_\theta\right)}{L_p\left(1,\eta_{\cK/\bQ}\right)^2\cdot L_p\left(1,\pi,\mr{Ad}\right)}\\
    &\times \frac{L^{p\infty}\left(\frac{1}{2},\pi\times\pi^{\GL_2}_h\times\pi^{\GL_2}_{\theta^{\bar{\lambda}}}\right)}{L^{p\infty}\left(1,\pi,\mr{Ad}\right)\cdot L^{p\infty}\left(1,\pi^{\GL_2}_\theta,\mr{Ad}\right)\cdot L^{p\infty}\left(1,\pi^{\GL_2}_h,\mr{Ad}\right)}.
\end{aligned}
\end{equation}

\begin{prop}\label{prop:triple-L}
There exist operators $\sT_\ns,\sT'_\ns \in \bigotimes\limits_{v\in\Sigma_\ns\cup\{\ell'\}}\cO_L\left[\GU(2)\big(\bQ_v\big)\right]$, and $\varphi$ inside the space \eqref{eq:M-varphi}, $\varphi'$ inside the space \eqref{eq:M-varphi'} such that $\sT_\ns\bfh$ is invariant under the tame level group $K^p_f$ (defined in \eqref{eq:k^p_f}) and
\begin{equation}\label{eq:triple-L}
\begin{aligned}
   &\left<\left(\sT_\ns\breve{\bfh}\right)|_{\U(2)},\bmtheta^{\bar{\lambda}}  \varphi\right>_{p\adic}
   \left<\Big(\sT'_\ns\breve{\tilde{\bfh}}_3\Big)|_{\U(2)},\tilde{\bmtheta}^{\lambda}_3  \varphi'\right>_{p\adic}
   =C_2\cdot \big<\bmtheta,\tilde{\bmtheta}_3\big>_{p\adic}\big<\bfh,\tilde{\bfh}_3\big>_{p\adic} \cdot \frac{\cL_1\cL_2}{\cL_3\cL_4} 
\end{aligned}
\end{equation}
with $C_2\in L^\times$. 
\end{prop}

\begin{proof}
For  $\sT_\ns,\sT'_\ns \in \bigotimes\limits_{v\in\Sigma_\ns\cup\{\ell'\}}\cO_L\left[\GU(2)\big(\bQ_v\big)\right]$ and $\varphi,\varphi'$ in \eqref{eq:M-varphi} corresponding to pure tensors in $\bigotimes_v\pi^{\GU(2)}_v$, we consider evaluations at $\tau_{p\adic}$ as in Theorem~\ref{thm:Sieg-F} with $\cond(\tau_\fp)=p^t$. Put $\theta_\tau=\bmtheta(\tau_{p\adic})$ and $h_\tau=\bfh(\tau_{p\adic})$.
\begin{equation}\label{eq:triple-eval}
\begin{aligned}
     \begin{array}{ll}\text{LHS of \eqref{eq:triple-L}}\\\text{valued at $\tau_{p\adic}$}\end{array}
   =& \int_{\U(2)(\bQ)\backslash\U(2)(\bA_\bQ)} \theta^{\bar{\lambda}}_\tau(g) \,\varphi(g) \,\left(\sT_\ns\breve{h}_\tau\right)\left(g\begin{psm}&1\\-p^t\end{psm}_{\scriptscriptstyle p}\right)\,dg\\
  &\times \int_{\U(2)(\bQ)\backslash\U(2)(\bA_\bQ)} \tilde{\theta}^{\lambda}_\tau(g) \, \varphi'(g)  \,\left(\sT'_\ns\breve{\tilde{h}}_\tau\right)\left(g\begin{psm}&1\\-p^t\end{psm}_{\scriptscriptstyle p}\right)\,dg.
\end{aligned}
\end{equation}
We can replace $\theta^{\bar{\lambda}}_\tau$ (resp. $\tilde{\theta}^\lambda_\tau$) in the integrand by $\breve{\theta}^{\bar{\lambda}}_\tau$ (resp. $\breve{\tilde{\theta}}^\lambda_\tau$), its extension by zero to $\GU(2)$ with respect the character $\chi_\theta\tau_{\fp,\tau_\wtp}$  (resp. $\chi^{-1}_\theta\tau^{-1}_{\fp,\tau_\wtp}$) (as described at the beginning of \S\ref{sec:ext-GU2}), and replace the domain of integration by $\GU(2)(\bQ)Z_{\GU(2)}(\bA_\bQ)\backslash\GU(2)(\bA_\bQ)$. Thanks to \eqref{eq:DGU}, we can further replace $\GU(2)(\bQ)Z_{\GU(2)}(\bA_\bQ)\backslash\GU(2)(\bA_\bQ)$ by $D^\times(\bQ)Z_D(\bA_\bQ)\backslash D^\times(\bA_\bQ)$. Then for $\sT_\ns,\sT'_\ns\in \bigotimes\limits_{v\in\Sigma_\ns\cup\{\ell'\}}\cO_L\left[D^\times(\bQ_v)\right]\subset \bigotimes\limits_{v\in\Sigma_\ns\cup\{\ell'\}}\cO_L\left[\GU(2)(\bQ_v)\right]$, we can rewrite \eqref{eq:triple-eval}
\begin{equation}\label{eq:tri-2}
\begin{aligned}
     \begin{array}{ll}\text{LHS of \eqref{eq:triple-L}}\\\text{valued at $\tau_{p\adic}$}\end{array}
   =&\int_{\bA^\times_\bQ\backslash D^\times(\bA_\bQ)} 
      \breve{\theta}^{\bar{\lambda}}_\tau(g) \,\varphi(g) \,(\sT_\ns\breve{h}_\tau)\left(g\begin{psm}&1\\-p^t\end{psm}_{\scriptscriptstyle p}\right)\,dg\\
      &\times\int_{\bA^\times_\bQ\backslash D^\times(\bA_\bQ)} 
      \breve{\tilde{\theta}}^{\lambda}_\tau(g) \,\varphi'(g) \,(\sT'_\ns\breve{\tilde{h}}_\tau)\left(g\begin{psm}&1\\-p^t\end{psm}_{\scriptscriptstyle p}\right)\,dg.
\end{aligned}
\end{equation}
This is the triple product integral.

Let 
\begin{align*}
   F_\tau&=\left(\breve{\theta}^{\bar{\lambda}}_\tau\otimes\varphi\otimes \sT_\ns \breve{h}_\tau\right)|_{D^\times(\bA)},
   &F'_\tau&=\left(\breve{\tilde{\theta}}^{\lambda}_{3,\tau}\otimes\varphi'\otimes \sT'_{\ns}\breve{\tilde{h}}_{3,\tau}\right)|_{D^\times(\bA)}.
\end{align*}
Then $F_\tau$ generates $\Pi_\tau=\pi^D_{\theta^{\bar{\lambda}}_\tau}\otimes\pi^D_{h_\tau}\otimes\pi^D$, where $\pi^D_{\theta^{\bar{\lambda}}_\tau}$ (resp. $\pi^D_{h_\tau}$, $\pi^D$) is the automorphic representation of $D^\times(\bA_\bQ)$ generated by the restriction of $\breve{\theta}^{\bar{\lambda}}_\tau$ (resp. $\breve{h}_\tau$, $\varphi$) to $D^\times$, and $F'_\tau$ generates $\tilde{\Pi}_\tau$, the contragredient representation of $\Pi_\tau$. 
By the isomorphism \eqref{eq:DGU} and Proposition~\ref{prop:ext-CM}, we know that $\pi^D_{\theta^{\bar{\lambda}}_\tau}$, $\pi^D_{h_\tau}$ are irreducible, and $\breve{\theta}^{\bar{\lambda}}_\tau$, $\breve{h}_\tau$ are pure tensors. By the definition of $\varphi$, we know that $\pi^D_\varphi=\pi^D$. Hence, $\Pi_\tau$ (resp. $\tilde{\Pi}_\tau$) is irreducible, and the image of $F_\tau$ (resp. $F'_\tau$) under $\Pi_\tau\simeq \bigotimes_v \Pi_{\tau,v}$  ($\tilde{\Pi}_\tau\simeq \bigotimes_v\tilde{\Pi}_{\tau,v}$) is a pure tensor $\bigotimes_v F_{\tau,v}$ (resp. $\bigotimes_v F'_{\tau,v}$). By Ichino's triple product formula \cite{Ichino-triple}, \eqref{eq:tri-2} becomes
\begin{equation}\label{eq:tri-3}
\begin{aligned}
    \begin{array}{ll}\text{LHS of \eqref{eq:triple-L}}\\\text{valued at $\tau_{p\adic}$}\end{array}
   =&\,\prod_v \int_{\bA^\times_\bQ\backslash D^\times(\bA_\bQ)} \big<\Pi_v(g_v)F_{\tau,v},F'_{\tau,v}\big>_v\,dg_v\\
   =&\,\big<f^D,\bar{f}^D\big>\big<\theta_\tau,\tilde{\theta}_{3,\tau}\big>_{p\adic}\big<h_\tau,\tilde{h}_{3,\tau}\big>_{p\adic}
   \cdot \frac{\zeta^{\Sigma\,\cup\{\infty,p,\ell,\ell'\}}(2)^2\cdot L^{\Sigma\,\cup\{\infty,p,\ell,\ell'\}}\left(\frac{1}{2},\Pi\right)}{L^{\Sigma\,\cup\{\infty,p,\ell,\ell'\}}\left(1,\Pi,\mr{Ad}\right)}\\
   &\times \prod_{v\in  \Sigma\,\cup\{\infty,\ell,\ell'\}}  \frac{\int_{\bQ^\times_v\backslash D^\times(\bQ_v)}\big<\Pi_v(g)F_{\tau,v},F'_{\tau,v}\big>_v\,dg}{\big<f^D_v,\bar{f}^D_v\big>_v\big<\theta_{\tau,v},\tilde{\theta}_{3,\tau,v}\big>_v\big<h_{\tau,v},\tilde{h}_{3,\tau,v}\big>_v}\\
   &\times  \frac{\int_{\bQ^\times_p\backslash D^\times(\bQ_p)}\big<\Pi_p(g)F_{\tau,p},F'_{\tau,p}\big>_p\,dg}{\big<f^D_p,\bar{f}^D_p\big>_p\big<\theta_{\tau,p},\begin{psm}&1\\-p^t\end{psm}\tilde{\theta}_{3,\tau,v}\big>_p\big<h_{\tau,p},\begin{psm}&1\\-p^t\end{psm}\tilde{h}_{3,\tau,p}\big>_p} 
\end{aligned}
\end{equation}
We have the following results for the local integrals at $v\in \Sigma\,\cup\{\infty,p,\ell,\ell'\}$:
\begin{enumerate}
\item[--] $v\in\Sigma_\rms\cup\{\infty,\ell\}$. 
\[
   \frac{\int_{\bQ^\times_v\backslash D^\times(\bQ_v)} \big<\Pi_v(g)\phi_v,\phi'_v\big>_v\,dg}{\big<f^D_v,\bar{f}^D_v\big>_v\big<\theta_{\tau,v},\tilde{\theta}_{3,\tau,v}\big>_v\big<h_{\tau,v},\tilde{h}_{3,\tau,v}\big>_v}=\text{a nonzero constant independent of $\tau$}
\]
For $v=\infty$, the integral is $1$ because $\Pi_\infty$ is the trivial representation. For $v=\Sigma_\rms\cup\{\ell\}$, this is \cite[Lemmas 8.13-8.17]{WanU31}. Note that the twist at $v\in\Sigma_\rms\cup\{\ell\}$ in \eqref{eq:f-twist} guarantees that the local integrals are nonzero constants. 

\vspace{.5em}
\item[--] $v\in \Sigma_\ns\cup\{\ell'\}$. There exist $\sT_v,\sT'_v\in\cO_L\left[D^\times(\bQ_v)\right]\subset\cO_L \left[\GU(2)(\bQ_v)\right]$ and $\varphi_v,\varphi'_v\in \pi^{\GU(2)}_v$ such that $\varphi_v$ and $\sT_v\bfh$ are invariant under the level group $K_v$ (defined in \eqref{eq:K_v}) and
\[
   \frac{\int_{\bQ^\times_v\backslash D^\times(\bQ_v)} \big<\Pi_v(g)F_{\tau,v},F'_{\tau,v}\big>_v\,dg}{\big<f^D_v,\bar{f}^D_v\big>_v\big<\theta_{\tau,v},\tilde{\theta}_{3,\tau,v}\big>_v\big<h_{\tau,v},\tilde{h}_{3,\tau,v}\big>_v}=\text{a nonzero constant independent of $\tau$}
\]
This is proved in the following \S\ref{sec:tri-ns}. 

\vspace{.5em}

\item[--] $v=p$. 
\begin{align*}
   &\frac{\int_{\bQ^\times_p\backslash D^\times(\bQ_p)}\big<\Pi_p(g)F_{\tau,p},F'_{\tau,p}\big>_p\,dg}{\big<f^D_p,\bar{f}^D_p\big>_p\big<\theta_{\tau,p},\begin{psm}&1\\-p^t\end{psm}\tilde{\theta}_{3,\tau,v}\big>_p\big<h_{\tau,p},\begin{psm}&1\\-p^t\end{psm}\tilde{h}_{3,\tau,p}\big>_p}
   =\, p^{-t}\,\frac{\zeta_p(2)^2\cdot L_p\left(\frac{1}{2},\mr{BC}(\pi)\times\chi_h\chi^c_\theta\right)}{L_p\left(1,\eta_{\cK/\bQ}\right)^2\cdot L_p\left(1,\pi,\mr{Ad}\right)}.
\end{align*}
This is a direct consequence of  \cite[Proposition 6.3]{hsieh-triple}. We also use the formulas for local norms of the chosen vectors in the proof of Lemma 6.4 of {\it loc.cit.}.
\end{enumerate} 
Plug these results on local integrals into \eqref{eq:tri-3} and compare with \eqref{eq:L1234}, we see that
\begin{align*}
    \begin{array}{ll}\text{LHS of \eqref{eq:triple-L} }\text{valued at $\tau_{p\adic}$}\end{array}
    =C_2\cdot \big<\theta_\tau,\tilde{\theta}_{3,\tau}\big>_{p\adic}\big<h_\tau,\tilde{h}_{3,\tau}\big>_{p\adic}\cdot  \left(\frac{\cL_1\cL_2}{\cL_3\cL_4}\right)\left(\tau_{p\adic}|_{U_{\cK,p}}\right)
\end{align*}
with $C_2\in L^\times$ the product of 
\begin{itemize}
\item the local triple product integrals divided by the norms at $v\in\Sigma\cup\{\ell,\ell'\}$, which are nonzero constants, 
\vspace{.5em}
\item $\prod\limits_{v\in \Sigma\,\cup\{\ell,\ell'\}} \Big(\frac{\zeta_v(2)^2 L_v\left(\frac{1}{2},\Pi\right)}{L_v\left(1,\Pi,\mr{Ad}\right)}\Big)^{-1}$, which, by our choice of the Hecke characters $\chi_\theta,\chi_h$, equals $L_q\left(\frac{1}{2},\mr{BC}(\pi)\times\chi_h\chi^c_\theta\right)^{-1} \prod\limits_{v\in \Sigma\,\cup\{\ell,\ell'\}} \zeta_v(2)^{-2}L_v\left(1,\eta_{\cK/\bQ}\right)^2 L_v\left(1,\pi,\mr{Ad}\right)$, a nonzero constant,
\vspace{.5em}

\item $\left(\pi\frac{L^\infty\left(1,\pi,\mr{Ad}\right)}{\left<f,\bar{f}\right>}\cdot \big<f^D,\bar{f}^D\big>\right)^{-1}\in \bQ^\times$.
\end{itemize}

\end{proof}

Given $g\in \U(2)(\bA_{\bQ,f})$, evaluating a $p$-adic form at $g$ induces a map
\[
    \Meas\left(\Gamma_\cK,V_{\U(2)}\wh{\otimes}\hat{\cO}^\ur_L\right)^\natural\lra \Meas\left(\Gamma_\cK,\hat{\cO}^\ur_L\right)\simeq \hat{\cO}^\ur_L\llb\Gamma_\cK\rrb.
\]
Denote by $l_{\theta^J_1}\left(\bfE^\Kling_{\varphi,\beta,u}\right)(g)\in \hat{\cO}^\ur_L\llb\Gamma_\cK\rrb$ the image of $l_{\theta^J_1}\left(\bfE^\Kling_{\varphi,\beta,u}\right)$ under this map.

\begin{prop}\label{prop:E-Kling-nv}
Let $\beta=1\in\Her_1(\cK)$, $\theta^J_1$ be the Fourier--Jacobi form on $\U(2)$ as in Proposition~\ref{prop:phi2f2}, and $\varphi$ be an automorphic form on $\U(2)$ as in Proposition~\ref{prop:triple-L}. Let $J\subseteq \hat{\cO}^\ur_L\llb\Gamma_\cK\rrb\otimes_\bZ\bQ$ be the ideal generated by 
\begin{align*}
   &l_{\theta^J_1}\left(\bfE^\Kling_{\varphi,\beta,u}\right)(g),
   &g\in \U(2)(\bA_{\bQ,f}),\quad u\in\hspace{-.5em}\bigotimes_{v\in\Sigma_\ns\cup\{\ell'\}}\U(1)(\bQ_v).
\end{align*}
Then $J=\hat{\cO}^\ur_L\llb\Gamma_\cK\rrb\otimes_\bZ\bQ$.
\end{prop}

\begin{proof}
Suppose that $J\neq \hat{\cO}^\ur_L\llb\Gamma_\cK\rrb\otimes_\bZ\bQ$. Then $J$ is contained in a maximal ideal of $\hat{\cO}^\ur_L\llb\Gamma_\cK\rrb\otimes_\bZ\bQ$, so is contained in the maximal ideal associated to a $p$-adic character $\kappa:\Gamma_\cK\ra\ol{\bQ}^\times_p$. Let $\kappa':\Gamma_\cK\ra\ol{\bQ}^\times_p$ be a $p$-adic character whose $\wtp$-th power is $\kappa$. Then
\[
   \pP_{\wtp,*}\left(l_{\theta^J_1}\left(\bfE^\Kling_{\varphi,\beta,u}\right)\right)(\kappa'|_{U_{\cK,p}})=l_{\theta^J_1}\left(\bfE^\Kling_{\varphi,\beta,u}\right)(\kappa)=0
\]
for all $u\in \bigotimes_{v\in\Sigma_\ns\cup\{\ell'\}}\U(1)(\bQ_v)$. Then it follows from Propositions~\ref{prop:pair-h} and \ref{prop:triple-L} that 
\[
  \left( \big<\bmtheta,\tilde{\bmtheta}_3\big>_{p\adic}\big<\bfh,\tilde{\bfh}_3\big>_{p\adic} \cdot \frac{\cL_1\cL_2\cL_5\cL_6}{\cL_3\cL_4}\right)(\kappa'|_{U_{\cK,p}})=0,
\]
which, combining with Proposition~\ref{prop:Rallis}, implies that
\[
   \left(\cL_1\cL_2\cL_5\cL_6\right)(\kappa'|_{U_{\cK,p}})=0.
\]
However, the condition (3) in our choice of $\chi_h,\chi_\theta$ in \S\ref{sec:aux} implies that $\cL_1\cL_2\cL_5\cL_6$ is a unit in $\hat{\cO}^\ur_L\llb\Gamma_\cK\rrb\otimes_\bZ\bQ$, so is nonzero at all characters in $\Hom_\cont\left(U_{\cK,p},\ol{\bQ}^\times_p\right)$.  We get a contradiction. Therefore, $J=\hat{\cO}^\ur_L\llb\Gamma_\cK\rrb\otimes_\bZ\bQ$.
\end{proof}

\subsection{The local triple product integrals at non-split places}\label{sec:tri-ns}

\begin{prop}
For $v\in\Sigma_\ns\cup\{\ell'\}$, there exist $\sT_v,\sT'_v\in\cO_L\left[D^\times(\bQ_v)\right]\subset\cO_L \left[\GU(2)(\bQ_v)\right]$ and $\varphi_v,\varphi'_v\in \pi^D_v$ such that $\varphi_v$ and $\sT_v\bfh$ are invariant under the level group $K_v$ (defined in \eqref{eq:K_v}), and 
\begin{equation}\label{eq:ns-int}
    \int_{\bQ^\times_v\backslash D^\times(\bQ_v)} \big<\Pi_v(g)F_{\tau,v},F'_{\tau,v}\big>_v\,dg=\text{a nonzero constant independent of $\tau$},
\end{equation}
where $F_{\tau,v}=\left(\breve{\theta}^{\bar{\lambda}}_{\tau,v}\otimes\varphi_v\otimes \sT_v \breve{h}_{\tau,v}\right)|_{D^\times(\bQ_v)}$, $F'_{\tau,v}=\left(\sT'_v\breve{\tilde{\theta}}^{\lambda}_{3,\tau,v}\otimes\varphi'_v\otimes \sT''_{v}\breve{\tilde{h}}_{3,\tau,v}\right)|_{D^\times(\bQ_v)}$.
\end{prop}

\begin{proof}

Let $v$ be a place in $\Sigma_\ns\cup\{\ell'\}$. Then $\U(1)(\bQ_v)\subset \cO^\times_{\cK,v}$ and $\tau_{\fp,\pP_\wtp}|_{\U(1)(\bQ_v)}$ is trivial. From the construction of $\bfh,\tilde{\bfh}_3$ in \S\ref{sec:Fh} and $\bmtheta,\tilde{\bmtheta}_3$ in \S\ref{sec:Ftheta}, we see that $\pi_{h_\tau,v}$, $\pi_{\theta_\tau,v}$ and $h_{\tau,v},\tilde{h}_{3,\tau,v},\theta_{\tau,v},\tilde{\theta}_{3,\tau,v}$  do not change when $\tau$ varies. It follows that the left hand side of \eqref{eq:ns-int} is independent of $\tau$.

\vspace{.5em}
Next we show the non-vanishing. By our choice of $\chi_h,\chi_\theta$, we have $\Hom_{D^\times(\bQ_v)}\left(\Pi_v,\bC\right)\neq 0$. By \cite[Theorem 1.4]{Prasad}, we can pick $\sT'_v\in \cO_L\left[D^\times(\bQ_v)\right]$ and $\varphi'_v\in \pi^D_v$ such that the linear functional
\begin{equation}
   \sL_{F'_{\tau,v}}:\Pi_v\lra \bC, \quad\quad \sL_{F'_{\tau,v}}(F_v)=\int_{\bQ^\times_v\backslash D^\times(\bQ_v)} \big<\Pi_v(g)F_v,F'_{\tau,v}\big>_v\,dg
\end{equation}
is nonzero and spans the one-dimensional space $\Hom_{D^\times(\bQ_v)}\left(\Pi_v,\bC\right)$. What we need to show is that there exists $\sT_v\breve{h}_{\tau,v}\in\pi^D_{h_\tau,v}$ and $\varphi_v\in\pi^D_v$ invariant under the level group $K_v$ defined in \eqref{eq:K_v} such that $\sL_{F'_{\tau,v}}\left(\breve{\theta}^{\bar{\lambda}}_{\tau,v}\otimes\varphi_v\otimes\sT_v\breve{h}_{\tau,v}\right)\neq 0$. Since $\breve{\theta}^{\bar{\lambda}}_{\tau,v}$ is invariant under $K_v$, it suffices to show that there exists $\sT_v\breve{h}_{\tau,v}\in \pi^D_{h_\tau,v}$ invariant under $K_v$ such that the linear functional
\begin{equation}\label{eq:Lpi}
   \sL_{F'_{\tau,v}}\left(\breve{\theta}^{\bar{\lambda}}_{\tau,v}\otimes\bdot\otimes\sT_v\breve{h}_{\tau,v}\right): \pi_v\lra \bC
\end{equation} 
is nonzero. There are two cases.

\vspace{.5em}

\underline{Case 1}: $v\neq q$, $D^\times(\bQ_v)\cong\GL_2(\bQ_v)$. Let $\varphi_v$ denote the prime number corresponding to $v$. Recall that at the beginning of \S\ref{sec:gps}, we have fixed the positive integer $\fs$ and totally imaginary element $\delta$ in $\cK$, and $\U(2)$ is the unitary group for the skew-Hermitian form $\zeta_0=\delta\begin{pmatrix}\fs\\&1\end{pmatrix}$. By our assumption on $\fs$, there exists $\fr\in \cO^\times_{\cK,v}$ such that $\Nm_{\cK_v/\bQ_v}(\fr)=-\fs$. We fix the following isomorphism
\begin{align*}
  \ffi:\GU(2)(\bQ_v)&\stackrel{\cong}{\lra}\GU(1,1)(\bQ_v), 
  &g&\longmapsto \begin{pmatrix}\frac{\delta}{2}&\frac{\delta\fr}{2}\\1&-\fr\end{pmatrix}g\begin{pmatrix}\frac{\delta}{2}&\frac{\delta\fr}{2}\\1&-\fr\end{pmatrix}^{-1},
\end{align*}
where $\GU(1,1)$ denotes the unitary group attached to the skew-Hermitian form $\begin{pmatrix}&1\\-1\end{pmatrix}.$ Let $K^{\GL_2}_v=\ffi(K_v)\cap \GL_2(\bZ_v)$. By our assumptions in \S\ref{sec:setup}, $v\neq 2$. One can check that
\begin{equation*}
   K^{\GL_2}_v=\begin{pmatrix}1+\varpi^{c_v}_v\bZ_v&\varpi^{c_v+\epsilon_v}_v\bZ_v\\\varpi^{c_v}_v\bZ_v&1+\varpi^{c_v}_v\bZ_v\end{pmatrix}\subset\GL_2(\bZ_v).
\end{equation*}
where $\epsilon_v=\begin{cases}0,&v\text{ inert}\\ 1,&v\text{ ramified}\end{cases}$.

Given a character $\chi_v:\bQ^\times_v\ra\bC^\times$, we denote also by $\chi_v$ the character $B(\bQ_v)\ra\bC^\times$, $\begin{pmatrix}a&b\\0&d\end{pmatrix}\mapsto \chi_v(ad^{-1})$, and $\Ind_{B}^{\GL_2}(\chi_v)$ the (normalized) induction of $\chi_v$. By our assumption on $\pi$, $\pi_v$ is either isomorphic to $\Ind_{B}^{\GL(2)}(\chi_v)$ for an unramified character $\chi_v\neq |\cdot|^{\pm 1/2}_{\bQ_v}$ or is isomorphic to the irreducible sub-representation of $\Ind_{B}^{\GL(2)}(\chi_v)$ with $\chi_v=|\cdot|^{-1/2}_{\bQ_v}$. Let $\psi_1=\be_v$, the additive character of $\bQ_v$ defined in \eqref{eq:bev}, $\psi_2=\bar{\be}_v$, and $W_{1}$ (resp. $W_{2}$) be the Kirillov model of $\pi^D_{\theta^{\bar{\lambda}}_\tau,v}$ (resp. $\pi^D_{h_\tau,v}$) with respect to $\psi_1$ (resp. $\psi_2$). We have
\begin{equation}\label{eq:indF}
\begin{aligned}
    \Hom_{\GL_2(\bQ_v)}\left(W_1\otimes W_2,\tilde{\pi}\right)&\cong \Hom_{\GL_2(\bQ_v)}\left(W_1\otimes W_2,\Ind^{\GL_2}_B (\chi^{-1}_v)\right)\\
    &\cong \Hom_{B(\bQ_v)}\left(W_1\otimes W_2,\chi^{-1}_v|\cdot|^{1/2}_{\bQ_v}\right),
\end{aligned}
\end{equation}
where the second isomorphism is deduced from the Frobenius reciprocity, and when $\chi_v=|\cdot|^{-1/2}_{\bQ_v}$, the first isomorphism is because by our choice of $\chi_\theta,\chi_h$, $W_1\not\cong \tilde{W}_2$. Our choice of $\chi_\theta,\chi_h$ also implies that $W_1,W_2$ are supercuspidal, so they consist of Schwartz functions on $\bQ^\times_v$ on which the action $B(\bQ_v)\subset\GL_2(\bQ_v)$ has the formula
\begin{align}
   \label{eq:Kirillov} \begin{pmatrix}a&b\\0&d\end{pmatrix}\cdot  w_j(x)&=\omega_{\pi_j}(d)\,\psi_j(bd^{-1}x)\,w_j(ad^{-1}x), &w_j\in W_j, 
\end{align}
where $\omega_{\pi_j}$ is the central character of $\pi_j$. By the way we extend $\bfh$ and $\bmtheta$ in \S\ref{sec:ext-GU2}, we know that $\omega_{\pi_1}\omega_{\pi_2}=\triv$. It is easy to see that the map
\begin{equation}
  \sL(w_1,w_2)= \int_{\bQ^\times_v} w_1(x)w_2(x)\cdot\chi_v(x)|x|^{-1/2}_{\bQ_v}\,d^\times x, \quad\quad w_1\in W_1,\,w_2\in W_2,
\end{equation}
is a nonzero element in \eqref{eq:indF}, which is one dimensional, so it equals the image of $\sL_{F'_{\tau,v}}\left(\breve{\theta}^{\bar{\lambda}}_{\tau,v}\otimes\bdot\otimes\bdot\right)$ in \eqref{eq:indF} up to a nonzero scalar. Hence, showing the existence of $\breve{h}_{\tau,v}\in \pi^D_{h_\tau,v}$ invariant under $K_v$ such that \eqref{eq:Lpi} is nonzero reduces to showing the existence of $\pzy_2\in W_2$ invariant under $K^{\GL_2}_v$ such that $\sL(\pzy_1,\pzy_2)\neq 0$, where $\pzy_1\in W_1$ is the vector invariant under$K^{\GL_2}_v$  corresponding to $\breve{\theta}^{\bar{\lambda}}_{\tau,v}\in \pi^D_{\theta^{\bar{\lambda}}_\tau,v}$. 

For a character $\vartheta:\bZ^\times_v\ra\bC^\times$, put
\[
    \pzy_{1,\vartheta}=\int_{\bZ^\times_v}\vartheta^{-1}(a)\pzy_1(ax)\,d x,
\]
the projection of $\pzy_1$ to the $\vartheta$-eigenspace for the action of $\diag(\bZ^\times_v,1)$. Let $\vartheta^*$ be the extension of $\vartheta$ to $\bQ^\times_v$ taking value $1$ at $\varpi_v$. For each integer $m$, put define $w^m_{j,\vartheta}\in W_j$, $j=1,2$, as
\[
    w^m_{j,\vartheta}(x)=\mathds{1}_{\varpi^m_v\bZ^\times_v}(x)\,\vartheta^*(x).
\]
Then we have $\pzy_{1,\vartheta}=\sum_m c_m w^m_{1,\vartheta}$ with $c_m\in\bC$ finitely many nonzero, and $\sL(\pzy_1,w^m_{2,\vartheta^{-1}})=\sL(\pzy_{1,\vartheta},w^m_{2,\vartheta^{-1}})$. We can take a $\vartheta$ of conductor $\leq \varpi^{c_v}_v$ such that $\pzy_{1,\vartheta}\neq 0$. This $\pzy_{1,\vartheta}$ is still invariant under $K^{\GL_2}_v$. By the formula \eqref{eq:Kirillov} and (1) of Lemma~\ref{lem:level} proved below, we deduce that 
\[
    \pzy_{1,\vartheta}=\sum_{m=-(c_v+\epsilon_v)}^{c_v-n(w^0_{1,\vartheta})}c_m w^m_{1,\vartheta}, \quad\quad c_m\in\bC,
\]
with $n(w^0_{1,\vartheta}$ defined as in \eqref{eq:nw}. Take $-(c_v+\epsilon_v)\leq m_0\leq c_v-n(w^0_{1,\vartheta})$ such that $c_{m_0}\neq 0$. Then 
\[
   \sL(\pzy_1,w^{m_0}_{2,\vartheta^{-1}})=c_{m_0}\sL(w^{m_0}_{1,\vartheta},w^{m_0}_{2,\vartheta^{-1}})\neq 0,
\]
and by the formula \eqref{eq:Kirillov}, we know that $w^{m_0}_{2,\vartheta^{-1}}$ is invariant under $\begin{pmatrix}1+\varpi^{c_v}_v\bZ_v&\varpi^{c_v+\epsilon_v}_v\bZ_v\\0&1+\varpi^{c_v}_v\end{pmatrix}$, and by (2) of Lemma~\ref{lem:level}, we know that 
\[
   n(w^{m_0}_{2,\vartheta^{-1}})=m_0+c(\pi_2\otimes\vartheta^{*-1})\leq c_v-n(w^0_{1,\vartheta})+c(\pi_2\otimes\vartheta^{*-1})=c_v-c(\pi_1\otimes\vartheta^*)+c(\pi_2\otimes\vartheta^{*-1}).
\]
By our choice of $\chi_\theta,\chi_h$ and the formula \eqref{eq:cond-dih}, we have $c(\pi_1\otimes\vartheta^*)\geq c(\pi_2\otimes\vartheta^{*-1})$. Hence, $ n(w^{m_0}_{2,\vartheta^{-1}})\geq c_v$, and $w^{m_0}_{2,\vartheta^{-1}}$ is invariant under $\begin{pmatrix}1&0\\ \varpi^{c_v}_v\bZ_v&1\end{pmatrix}$. Therefore, $\pzy_2=w^{m_0}_{2,\vartheta^{-1}}$ is the desired vector in $W_2$ invariant under $K^{\GL_2}_v$ such that $\sL(\pzy_1,\pzy_2)\neq 0$.

\vspace{.5em}
\underline{Case 2}: $v=q$, $D^\times(\bQ_q)$ modulo its center is compact and $\pi^D_v$ is one-dimensional isomorphic to $\chi_q\circ \Nm_D$, where $\chi_q$ is the unramified quadratic character in condition (5) in the choice of $\chi_\theta,\chi_h$ at the beginning of \S\ref{sec:aux}, and $\Nm_D$ denotes the norm of $D$. By our choice of $\chi_\theta,\chi_h$, we have
\[
   \pi^D_{\theta^{\bar{\lambda}}_\tau,q}\otimes \chi_q\circ\Nm_D\cong \tilde{\pi}^D_{h_\tau,q}.
\] 
It follows that there exists $\sT_q\in\cO_L\left[D^\times(\bQ_v)\right]$ such that \eqref{eq:Lpi} is nonzero. Since $K_v$ is a normal subgroup of $D^\times(\bQ_q)$, $\sT_q\breve{h}_{\tau,q}$ is still invariant under $K_v$.
\end{proof}

\begin{lemma}\label{lem:level}
Let $v$ be a place of $\bQ$ non-split in $\cK$ and $\sigma:\cK^\times_v\ra \bQ^\times_v$ be a character nontrivial on the $\ker \Nm_{\cK_v/\bQ_v}$. Let $\pi_\sigma$ be the dihedral supercuspidal representation of $\GL_2(\bQ_v)$ obtained by the theta lift of $\sigma$. (See \cite[\S1.2]{Sch-gl2new} for the precise definition.) Then the conductor of $\pi_\sigma$ equals
\begin{equation}\label{eq:cond-dih}
   c(\pi_\sigma)=\left\{\begin{array}{ll}
   2\cond(\sigma), &v\text{ inert,}\\
   \cond(\sigma)+1, &v\text{ ramified}.
   \end{array}\right.
\end{equation}
by \cite[Theorem in \S2.3.2]{Sch-gl2new}. Let $W$ be the Kirillov model of $\pi_\sigma$ with respect to the additive character $\be_v^{\pm 1}$ of $\bQ_v$. For $w\in W$, define
\begin{equation}\label{eq:nw}
   n(w)=\min\left\{n\in\bZ: \text{$w$ is fixed by $\begin{pmatrix}1&0\\ \varpi^{n}_v\bZ_v&1\end{pmatrix}$}\right\}
\end{equation}
Let $\vartheta$ be a character of $\bZ^\times_v$ and $\vartheta^*$ be its extension to $\bQ^\times_v$ taking value $1$ at $\varpi_v$. For each integer $m$, define $w^m_\vartheta\in W$ as
\[
   w^m_{\vartheta}(x)=\mathds{1}_{\varpi^m_v\bZ^\times_v}(x)\,\vartheta^*(x).
\]
Suppose that $\ord_v(\cond(\sigma))>\ord_v(\cond(\sigma|_{\bQ^\times_v}))$.
\begin{enumerate}
\item Let $w=\sum\limits_{m=l_1}^{l_2} c_m w^m_\vartheta$ with $c_{l_2}\neq 0$. Then $n(w)=n(w^{l_2}_\vartheta)=n(w^0_\vartheta)+l_2$.
\item We have
\begin{align*}
    n(w^m_\vartheta)&=m+c(\pi_\sigma\otimes\vartheta^{*-1})
    =\left\{\begin{array}{ll}
    m+2\cond(\sigma\cdot\vartheta^{*-1}\circ \Nm_{\cK_v/\bQ_v}), &\text{$v$ inert},\\
    m+\cond(\sigma\cdot\vartheta^{*-1}\circ\Nm_{\cK_v/\bQ_v})+1,  &\text{$v$ ramified}.
    \end{array}\right.
\end{align*}
\end{enumerate}
\end{lemma}
\begin{proof}
(1) It is easy to see that $\begin{pmatrix}\varpi^l_v\\&1\end{pmatrix}\cdot w^m_\vartheta =w^{m-l}_\vartheta$. Thus, $n(w^m_\vartheta)=n(w^0_\vartheta)+m$. It follows that $n(w^m_\vartheta)\leq n(w^{l_2}_\vartheta)$ for all $m\leq l_2$, so $n(w)\leq n(w^{l_2}_\vartheta)$. On the other hand, if $n(w)\leq n(w^{l_2}_\vartheta)-1$, then $n(w^{l_2}_\vartheta)=n\Big(c^{-1}_{l_2}(w-\sum\limits_{l_1}^{l_2-1}c_m w^m_\vartheta)\Big)\leq n(w^{l_2}_\vartheta)-1$. This is a contradiction. Hence,  $n(w)= n(w^{l_2}_\vartheta)$.

\vspace{.5em}
(2) Let $W'$ the representation of $\GL_2(\bQ_v)$ whose underlying $\bC$-vector space equals that of $W$ with the action of $\GL_2(\bQ_v)$ defined as the action of $\GL_2(\bQ_v)$ on $W$ twisted by $\vartheta\circ\det$, {\it i.e.} for $g\in \GL_2(\bQ_v)$ and $w\in W'$, the action of $g$ on $w$ gives $\vartheta^{-1}(\det g)(g\cdot w)$, where $g\cdot w$ denotes the action of $g$ on $w$ viewed as an element in $W$. Then $W'\cong \pi_\sigma\otimes\vartheta^{*-1}$. Suppose that $w'\in W'$ is the new vector, then $\begin{pmatrix}\bZ^\times_v&\bZ_v\\0&d\end{pmatrix}$, $d\in\bZ^\times_v$, acts on it by the scalar $\sigma\vartheta^{-2}(d)$, and we have $w'=\sum\limits_{m=l_1}^{l_2}c_m w^m_\vartheta$ with $c_{l_1}\neq 0$ for some $l_1\geq 0$.  Then
\begin{equation}\label{eq:Kcond}
   n(w^0_\vartheta)\geq c(\pi_\sigma\otimes\vartheta^{*-1})=n(w')=n(w^0_\vartheta)+l_2,
\end{equation}
where the first inequality and the middle equality follows from the definition of new vectors and conductors of $\GL_2(\bQ_v)$-representations, and the right equality follows from (1). \eqref{eq:Kcond} implies that $l_1=l_2=0$, $w^0_\vartheta$ is the new vector in $W'$ and $ n(w^0_\vartheta)=c(\pi_\sigma\otimes\vartheta^{*-1})$. Since $n(w^m_\vartheta)=n(w^0_\vartheta)+m$, we see that $n(w^m_\vartheta)=m+c(\pi_\sigma\otimes\vartheta^{*-1})$.
\end{proof}

\section{Proof of Greenberg--Iwasawa Main Conjecture}\label{sec:mcpf}

\subsection{The Klingen Eisenstein ideal}\label{sec:Kl-ideal}

Let $\cM_{\so}=\Hom_{\Lambda_\so}\left(\pV^*_\so,\Lambda_\so\right)$ be the $\cO_L\llb T_\so(\bZ_p)\rrb$-module of semi-ordinary families on $\GU(3,1)$ of tame level $K^p_f$ (defined in \eqref{eq:k^p_f}) as in Theorem~\ref{prop:main}. Identifying $T_\so(\bZ_p)$ with $\bZ^\times_p\times\bZ^\times_p\simeq \cO^\times_{\cK,p}$, we have the map
\begin{equation}\label{eq:TGamma}
   T_\so(\bZ_p)\lra \cO^\times_{\cK,p}\lra \Gamma_\cK,
\end{equation}
with kernel $\mu_p\times\mu_p$. Denote by $\glsuseri{pVxi}$ (resp. $\glsuseri{Mxi}$) the sub-$\cO_L\llb T_\so(\bZ_p)\rrb$-module of $e_\so \pV$ (resp. $\cM_\so$) on which $\mu_p\times\mu_p$ acts through $\xi_{p\adic}$. We make $T_\so(\bZ_p)$ act on  $\pV_{\so,\,\xi}$ and $\cM_{\so,\,\xi}$ through its usual action multiplied by the character $\xi^{-1}_{p\adic}$. Then this action factors through the image of $T_\so(\bZ_p)$ in $\Gamma_\cK$ under the map \eqref{eq:TGamma}. Let $\glsuserii{pVxi}\subset \pV_{\so,\,\xi}$, $\glsuserii{Mxi}\subset\cM_{\so,\,\xi}$ be the cuspidal part. We consider the $\cO_L\llb\Gamma_\cK\rrb$-modules 
\begin{align*}
   \glsuseri{MxiGamma}&=\cM_{\so,\,\xi}\otimes_{\cO_L\llb T_\so(\bZ_p)\rrb} \hat{\cO}^\ur_L\llb\Gamma_\cK\rrb,
   &\glsuserii{MxiGamma}&=\cM^0_{\so,\,\xi}\otimes_{\cO_L\llb T_\so(\bZ_p)\rrb} \hat{\cO}^\ur_L\llb\Gamma_\cK\rrb.
\end{align*}

\vspace{.5em}

The Klingen Eisenstein family $\bfE^\Kling_\varphi$ constructed in \S\ref{sec:constructKF} belongs to $\Meas\left(\Gamma_\cK,e_\so V_{\GU(3,1),\,\xi}\wh{\otimes}\hat{\cO}^\ur_L\right)$. By the natural pairing between $V_{\GU(3,1),\so}$ and the space $\pV^*_\so$ in Theorem~\ref{prop:main}, we see that there is a natural map
\begin{align*}
   \Meas\left(\Gamma_\cK,e_\so V_{\GU(3,1),\,\xi}\wh{\otimes}\hat{\cO}^\ur_L\right)\times  (\pV_{\so,\,\xi})^*\lra \Meas\left(\Gamma_\cK,\hat{\cO}^\ur_L\right)\simeq \hat{\cO}^\ur_L\llb\Gamma_\cK\rrb,
\end{align*}
and it induces a $\hat{\cO}^\ur_L\llb\Gamma_\cK\rrb$-linear Hecke-equivariant map
\begin{align*}
    \Meas\left(\Gamma_\cK,e_\so V_{\GU(3,1),\,\xi}\wh{\otimes}\hat{\cO}^\ur_L\right)\lra \Hom_{\cO_L\llb T_\so(\bZ_p)\rrb}\left((\pV_{\so,\,\xi}^*,\hat{\cO}^\ur_L\llb\Gamma_\cK\rrb\right)=\cM_{\so,\,\xi,\Gamma_\cK}.
\end{align*}
Therefore, we can view $\bfE^\Kling_\varphi$ as an element in $\cM_{\so,\,\xi,\Gamma_\cK}$. Let $\bT$ be the abstract algebra generated by the unramified Hecke algebras away from $\Sigma\cup\{p,\ell,\ell'\}$ and $\bU_p$-operators at $p$. Then $\bfE^\Kling_\varphi$ is an eigen-family for the action of $\bT$, and we denote by $\lambda_{\mr{Eis},\pi,\xi}:\bT\ra \hat{\cO}^\ur_L\llb\Gamma_\cK\rrb$ the corresponding eigen-system.

\vspace{.5em}
Let $\gls{bT}$ be the reduced Hecke algebra generated by  $\bT$ acting on $\cM^0_{\so,\,\xi,\Gamma_\cK}$. Let $\gls{IEis}\subset \bT^0_{\so,\,\xi,\Gamma_\cK}$ be the ideal generated by the images of $T-\lambda_{\mr{Eis},\pi,\xi}(T)$ in $\bT^0_{\so,\,\xi,\Gamma_\cK}$, $T\in\bT$. We define the Klingen Eisenstein ideal as
\[
\begin{tikzcd}[column sep=large]
   \gls{cEpi}=\ker\Big(\hat{\cO}^\ur_L\llb\Gamma_\cK\rrb\arrow[r,"\text{structure}","\text{map}"']& \bT^0_{\so,\,\xi,\Gamma_\cK}/\cI_{\mr{Eis},\pi,\xi}\Big).
\end{tikzcd}
\]
It measures the congruences between the Hecke eigen-systems in $\bT^0_{\so,\,\xi,\Gamma_\cK}$ and the Hecke eigen-system $\lambda_{\mr{Eis},\pi,\xi}$ attached to $\bfE^\Kling_\varphi$. 

\begin{thm}\label{thm:Eis-id}
Let $\varphi$ be as in Proposition~\ref{prop:triple-L}. Let $P$ be a height one prime ideal of $\hat{\cO}^\ur_L\llb\Gamma_\cK\rrb$ such that $P\cap \hat{\cO}^\ur_L=(0)$. Then
\[
    \ord_P\left(\cL^{\Sigma\,\cup\{\ell,\ell'\}}_{\pi,\cK,\xi}\cL^{\Sigma\,\cup\{\ell,\ell'\}}_{\xi,\bQ}\right)\leq \ord_P(\cE_{\pi,\xi}),
\]
(where $\cL^{\Sigma\,\cup\{\ell,\ell'\}}_{\xi,\bQ},\cL^{\Sigma\,\cup\{\ell,\ell'\}}_{\pi,\cK,\xi}\in\hat{\cO}^\ur_L\llb\Gamma_\cK\rrb$ are the $p$-adic $L$-function satisfying the interpolation property in \eqref{eq:pL-xi}\eqref{eq:pL-pi}). 
\end{thm}

\begin{proof}

By Theorem~\ref{prop:main}, we have the fundamental exact sequence
\[
   0\lra\cM^0_{\so,\,\xi,\Gamma_\cK}\lra \cM_{\so,\,\xi,\Gamma_\cK}\lra \bigoplus_{g\in 
   C(K^p_f)} M_{\GU(2)}\left(K^p_{f,g}K_{p,0},\hat{\cO}^\ur_L\right)\otimes_{\bZ_p}\bZ_p\llb \Gamma_\cK\rrb \lra 0.
\]
Combining it with Theorem~\ref{thm:const} on the degenerate Fourier--Jacobi coefficients of $\bfE^\Kling_\varphi$, we deduce that there exists $\bfE'\in \cM_{\so,\,\xi,\Gamma_\cK}$ such that 
\[
   \bm{F}=\bfE^\Kling_\varphi-\cL^{\Sigma\,\cup\{\ell,\ell'\}}_{\xi,\bQ}\cL^{\Sigma\,\cup\{\ell,\ell'\}}_{\pi,\cK,\xi}\cdot \bfE'\in \cM^0_{\so,\,\xi,\Gamma_\cK}.
\]

Let $P$  be a height one prime ideal of $\hat{\cO}^\ur_L\llb\Gamma_\cK\rrb$ such that 
\[ 
   P\cap \hat{\cO}^\ur_L=(0), \quad \ord_P\left(\cL^{\Sigma\,\cup\{\ell,\ell'\}}_{\xi,\bQ}\cL^{\Sigma\,\cup\{\ell,\ell'\}}_{\pi,\cK,\xi}\right)=m\geq 1.
\] By Proposition~\ref{prop:E-Kling-nv}, for $\beta=1$, there exists $g\in\U(2)(\bA_\cK)$ and $u\in \bigotimes_{v\in\Sigma_\ns\cup\{\ell'\}}\U(1)(\bQ_v)$ such that $l_{\theta^J_1}\left(\bfE^\Kling_{\varphi,\beta,u}\right)(g)\notin P$. Then $l_{\theta^J_1}\left(\bm{F}_{\beta,u}\right)(g)\notin P$, and we have the $\hat{\cO}^\ur_\cK\llb\Gamma_\cK\rrb_P$-linear map
\begin{align}
   \label{eq:Ipi1}\left(\bT^0_{\so,\,\xi,\Gamma_\cK}\right)_P&\lra \hat{\cO}^\ur_\cK\llb\Gamma_\cK\rrb_P/P^m\hat{\cO}^\ur_\cK\llb\Gamma_\cK\rrb_P,
\end{align}
which sends the image of $T\in\bT$ in $\left(\bT^0_{\so,\,\xi,\Gamma_\cK}\right)_P$ to
\begin{equation*}
   \frac{l_{\theta^J_1}\left((T\cdot \bm{F})_{\beta,u}\right)(g)}{l_{\theta^J_1}\left(\bm{F}_{\beta,u}\right)(g)}\equiv \frac{l_{\theta^J_1}\left((T\cdot \bm{E}^\Kling_\varphi)_{\beta,u}\right)(g)}{l_{\theta^J_1}\left(\bm{E}^\Kling_{\varphi,\beta,u}\right)(g)}\equiv \lambda_{\mr{Eis},\pi,\xi}(T) \mod P^m.
\end{equation*}
We see that \eqref{eq:Ipi1} factors through $\left(\bT^0_{\so,\,\xi,\Gamma_\cK}\right)_P/\cI_{\mr{Eis},\pi,\xi} \left(\bT^0_{\so,\,\xi,\Gamma_\cK}\right)_P$. It follows that the composition
\begin{equation}\label{eq:Ipi3}
\begin{tikzcd}[column sep=large]
   \hat{\cO}^\ur_\cK\llb\Gamma_\cK\rrb_P\arrow[r,"\text{structure}", "\text{map}"']& \left(\bT^0_{\so,\,\xi,\Gamma_\cK}\right)_P\arrow[r,"\eqref{eq:Ipi1}"]& \hat{\cO}^\ur_\cK\llb\Gamma_\cK\rrb_P/P^m\hat{\cO}^\ur_\cK\llb\Gamma_\cK\rrb_P.
\end{tikzcd}
\end{equation} 
factors through $\hat{\cO}^\ur_\cK\llb\Gamma_\cK\rrb_P/\cE_{\pi,\xi}\hat{\cO}^\ur_\cK\llb\Gamma_\cK\rrb_P$. By the definition of \eqref{eq:Ipi1}, we also know that the composition \eqref{eq:Ipi3} is the natural projection. Hence, we deduce that $\cE_{\pi,\xi} \hat{\cO}^\ur_\cK\llb\Gamma_\cK\rrb_P\subset P^m \hat{\cO}^\ur_\cK\llb\Gamma_\cK\rrb_P$, and
\[
    \ord_P\left(\cE_{\pi,\xi}\right)\leq m= \ord_P\left(\cL^{\Sigma\,\cup\{\ell,\ell'\}}_{\xi,\bQ}\cL^{\Sigma\,\cup\{\ell,\ell'\}}_{\pi,\cK,\xi}\right).
\]

\end{proof}

\subsection{The main theorem}

We prove the following partial results towards Conjecture~\ref{conj}.

\begin{thm}\label{thm:main}
Let $\cK$ be an imaginary quadratic extension of $\bQ$ and $p\geq 3$ be a prime split in $\cK$. Let $\pi$ and $\xi$ be as in \S\ref{sec:setup}.
\begin{enumerate}
\item Let $\Gamma^+_\cK\subset\Gamma_\cK$ be the rank one $\bZ_p$-module on which the complex conjugation acts by $+1$. As fractional ideals of $\pzR=\cO^\ur_L\llb\Gamma_\cK\rrb \otimes_{\cO_L\llb \Gamma^+_{\cK}\rrb}\mr{Frac}\left(\cO_L\llb \Gamma^+_\cK\rrb\right)$,
\begin{equation}\label{eq:ideal-incl} 
    \big(\cL_{\pi,\cK,\xi}\big)\supset \mr{char}_\pzR \big(X_{\pi,\cK,\xi}\otimes _{\cO_L\llb \Gamma^+_{\cK}\rrb}\pzR\big).
\end{equation}
See \S\ref{sec:intro} for the definition of the $p$-adic $L$-function $\cL_{\pi,\cK,\xi}\in \mr{Frac}\left(\cO^\ur_L\llb \Gamma_\cK\rrb\right)$, the $\cO_L\llb\Gamma_\cK\rrb$-module $X_{\pi,\cK,\xi}$, and the characteristic ideals.

\item If $\xi$ satisfies $ \xi_{\fp}\not\equiv\xi_{\bar{\fp}}$ modulo the maximal ideal of $\cO_L$, then
\[
    \cL_{\pi,\cK,\xi}\in \cO^\ur_L\llb \Gamma_\cK\rrb
\] 
and if it further satisfies:
\begin{enumerate}
\item[--] $\xi|_{\bA^\times_\bQ}=\omega^2$ with $\omega:\bQ^\times\backslash\bA^\times_\bQ\ra\bC^\times$ the Teichm{\"u}ller character, and $k_0\equiv 0 \mod 2(p-1)$,
\item[--] for every finite place of $\bQ$ non-split in $\cK$, $\epsilon_v\left(\frac{1}{2},\mr{BC}(\pi)\times\xi_0\right)=1$,
\item[--] the conductor of $\xi$ is only divisible by primes split in $\cK/\bQ$,
\end{enumerate}
then
\[
    \big(\cL_{\pi,\cK,\xi}\big)\supset \mr{char}_{\cO^\ur_L\llb\Gamma_\cK\rrb \otimes_{\bZ} \bQ}\big(X_{\pi,\cK,\xi}\otimes_{\bZ}\bQ\big)
\]
\end{enumerate}
\end{thm}

\begin{proof}
(1) First, relaxing the conditions at $v\in\Sigma\cup\{\ell,\ell'\}$ in the definition of $\Sel_{\pi,\cK,\xi}$ in \eqref{eq:Sel} and $X_{\pi,\cK,\xi}$ in \eqref{eq:Sel}, we define the $\Sigma\cup\{\ell,\ell'\}$-primitive Selmer  group:
\begin{align*}
   \gls{Selpartial}=\ker\left\{H^1\big(\cK,T_{\pi,\cK,\xi}\otimes_{\cO_L}\cO_L\llb\Gamma_\cK\rrb\big)
   \lra \prod_{\substack{\fv \neq\fp\\ \fv\,\nmid \text{ places in $\Sigma\cup\{\ell,\ell'\}$}}}H^1\big(I_\fv,T_{\pi,\cK,\xi}\otimes_{\cO_L}\cO_L\llb\Gamma_\cK\rrb^*\big)\right\}.
\end{align*}
and
\[
    \gls{Xpartial}:=\Hom_{\bZ_p}\left(\Sel^{\Sigma\,\cup\{\ell,\ell'\}}_{\pi,\cK,\xi},\,\bQ_p/\bZ_p\right).
\]

We can assume that the module $X_{\pi,\cK,\xi}$ is $\cO_L\llb \Gamma_{\cK}\rrb$-torsion (since otherwise its characteristic ideal is defined to be $(0)$ and the inclusion \eqref{eq:ideal-incl} automatically holds). By the same argument as in the proof of \cite[Theorem 6.1.6]{JSW}, which uses the fact that the sizes of the unramified cohomology groups at primes
outside $p$ are controlled by the local Euler factors of the $p$-adic $L$-functions (\cite[Proposition 2.4]{GrVa}), it suffices to prove the inclusion
\begin{equation*}
    \left(\cL^{\Sigma\,\cup\{\ell,\ell'\}}_{\pi,\cK,\xi}\right)\supset \mr{char}_\pzR \left(X^{\Sigma\,\cup\{\ell,\ell'\}}_{\pi,\cK,\xi}\otimes _{\cO_L\llb \Gamma^+_{\cK}\rrb}\pzR\right)
\end{equation*}
as fractional ideals of $\pzR=\cO^\ur_L\llb\Gamma_\cK\rrb \otimes_{\cO_L\llb \Gamma^+_{\cK}\rrb}\mr{Frac}\left(\cO_L\llb \Gamma^+_\cK\rrb\right)$.

We reduce to show that given a height one prime ideal $P\subset \cO^\ur_L\llb \Gamma_\cK\rrb$ such that $P\cap \cO^\ur_L\llb\Gamma^+_\cK\rrb=(0)$, we have the inequality
\begin{equation}\label{eq:incl}
   \ord_P\left(\cL^{\Sigma\cup\{\ell,\ell'\}}_{\pi,\cK,\xi}\right)\leq \mr{length}_{\cO^\ur_L\llb\Gamma_\cK\rrb_P}\left(\left(X^{\Sigma\cup\{\ell,\ell'\}}_{\pi,\cK,\xi}\right)_P\right),
\end{equation}
which, by Theorem~\ref{thm:Eis-id}, is implied by
\begin{equation}\label{eq:cE-incl}
   \ord_P\left(\cE_{\pi,\xi}\right)\leq \mr{length}_{\cO^\ur_L\llb\Gamma_\cK\rrb_P}\left(\left(X^{\Sigma\cup\{\ell,\ell'\}}_{\pi,\cK,\xi}\right)_P\right).
\end{equation}

We have the following setup. 
\begin{itemize}
\item $A_0=\cO^{\ur}_L\llb\Gamma_\cK\rrb$ and $A=\hat{A}_{0,P}$.
\item $R_0=\bT^0_{\so,\xi,\Gamma_\cK}$, a prime ideal $Q\subset R$ such that $Q\cap A_0=P$, and $R=\hat{R}_{0,Q}$.
\item $G=\Gal(\ol{\bQ}/\bQ)$ and $H=\Gal(\ol{\bQ}/\cK)$. 
\end{itemize}
The following data gives us the setup (1)-(5) on \cite[page 478]{WanU31}.
\begin{enumerate}
\item Let $\nu:H\ra A^\times_0$ be the trivial character.
\item $\chi=\xi\Psi_\cK:G\ra A^\times_0$, with $\chi\not\equiv\chi^{-c}$ modulo the maximal ideal of $A_0$ by the choice of $\xi$, 
\item $\rho:G\ra \mr{Aut}_A(V)$, a Galois representation obtained from the two-dimensional Galois representation $\rho_\pi:H\ra\GL(V_\pi)$ with $V=T_\pi\otimes_{\cO_L} A$ and $T_\pi\subset V_\pi$ a $G$-stable lattice.
\item $\sigma:G\ra \mr{Aut}_{R\otimes_A F}(M)$, a Galois representation on $M=(R\otimes_A \mr{Frac}(A))^4$ obtained as the pseudo-representation associated to $R_0$ as in \cite[Proposition 7.2.1]{SU}
\item $I=\cE_{\pi,\xi}A_P$ and $J=\ker(R\ra A/I)$ where the map $R\ra A$ is defined as the map \eqref{eq:Ipi1} and is surjective, and $A/I=R/J$.
\end{enumerate}
Moreover, the data satisfies the properties (6)-(9) on \cite[page 478]{WanU31}. Properties (6)-(8) are checked in the same way as {\it loc.cit}. 

Property (9) is about the irreducibility of $\sigma$. It requires that for each $\mr{Frac}(A)$-algebra homomorphism $\lambda:R\otimes_A\mr{Frac(A)}\ra K$, $K$ a finite extension of $\mr{Frac}(A)$, the representation $\sigma_\lambda:G\ra \GL_4(K)$ obtained from $\sigma$ via $\lambda$ is either absolutely irreducible or contains an absolutely irreducible two-dimensional sub-representation whose trace is congruent to $\chi+\chi^{-c}$ mod $I$. It can be checked as follows. Let $\pzT:G\ra A$ be the pseudo representation giving rise to $\sigma$. Suppose that $\sigma$ does not satisfy (9). Then $\pzT=\delta_1+\delta_2+\delta_3$ with $\delta_1,\delta_2$ one-dimensional pseudo-character and $\delta_3$ a two-dimensional pseudo-representation with irreducible residual representation is $\Tr\bar{\rho}_\pi$.  Take an arithmetic point $x:\mathbb{T}^0_{\so,\xi,\Gamma_\cK}\ra\ol{\bQ}_p$ corresponding to an automorphic representation $\Pi$ of $\U(3,1)(\bA_\bQ)$ generated by a classical semi-ordinary cuspidal automorphic form fixed by $K^0_{p,1}$ of weight $(0,0,t^+;t^-)$, $0\gg t^+\gg -t^-$. We can associate a $2$-dimensional (irreducible) Galois representation $\rho_{3,x}:G\ra\GL_2(\ol{\bQ}_p)$ to the specialization of $\delta_3$ at $x$. As in \cite[Theorem 7.5]{SU}, a twist of $\rho_{3,x}$ descends to $\Gal(\ol{\bQ}/\bQ)$, which we denote by $\rho'_{3,x}:\Gal(\ol{\bQ}/\bQ)\ra\GL_2(\ol{\bQ}_p)$. 

If $\pi$ is ordinary at $p$, then $\Pi$ is ordinary at $p$. By the argument in \cite[Theorem 7.5]{SU}, we know that $\rho'_{3,x}$ is modular and $\Pi$ is CAP, but the condition on the weights $(0,0,t^+;t^-)$ excludes this possibility (\cite[Theorem 2.5.6]{HaEC}).  

If $\pi$ is not ordinary at $p$, then $\bar{\rho}_\pi|_{G_{\cK,\fp}}\cong\bar{\rho}_\pi|_{G_{\bQ,p}}$ is irreducible by \cite{Edixhoven}. Since $\bar{\rho}_\Pi\cong \bar{\rho}_\pi|_G\oplus \text{two characters}$, we deduce that the semi-simplification of $\rho_\Pi|_{G_{\cK,\fp}}$ is not a direct sum of four characters. Identify $\U(3,1)(\bQ_p)$ with $\GL_4(\bQ_p)$ and let $Q\subset\GL_4$ be a parabolic subgroup with its Levi subgroup isomorphic to $\GL_2\times\GL_1\times\GL_1$. As in the proof of Proposition~\ref{prop:clas-bound}, it follows from the theory of Jacquet modules that $\Pi_p$ is isomorphic to a subquotient of $\Ind^{\GL_4(\bQ_p)}_{Q(\bQ_p)}\sigma\boxtimes\chi\boxtimes\chi'$ with  $\chi,\chi'$ characters of $\bQ^\times_p$  and $\sigma$ an irreducible admissible representation of $\GL_2(\bQ_p)$ containing a nonzero vector fixed by the Iwahori subgroup. Since the semi-simplification of $\rho_\Pi|_{G_{\cK,p}}$ is not a direct sum of four characters, $\sigma$ is not Steinberg or a twist of Steinberg, and must be unramified. Hence, $\Pi_p$ is isomorphic to a subquotient of $\Ind^{\GL_4(\bQ_p)}_{B(\bQ_p)}\chi_1\boxtimes\chi_2\boxtimes\chi_3\boxtimes\chi_4$ with $\chi_1,\chi_2,\chi_3,\chi_4$ unramified characters of $\bQ^\times_p$. Let $\alpha_j=\chi_j(p)$, $1\leq j\leq 4$. We can assume that $\val_p(\alpha_1)\leq \val_p(\alpha_2)\leq \val_p(\alpha_3)\leq \val_p(\alpha_4)$. The semi-ordinarity of $\Pi$ at $p$ implies that
\begin{equation}\label{eq:alpha34}
\begin{aligned}
   \val_p(\alpha_1)+\val_p(\alpha_2)&=0,
   &\val_p(\alpha_3)&=-t^++\frac{3}{2},
   &\val_p(\alpha_4)&=t^--\frac{3}{2},
\end{aligned}
\end{equation}
and the integrality of the operator $U^+_{p,1}$ (defined in \eqref{eq:Up+1}) implies that
\begin{equation}\label{eq:alpha1}
    \val_p(\alpha_1),\val_p(\alpha_2)\geq -\frac{1}{2}.
\end{equation}
Thanks to the condition $0\gg t^+\gg-t^-$, \eqref{eq:alpha34} and \eqref{eq:alpha1} imply that $\alpha_i\alpha^{-1}_j\neq p^{\pm 1}$ unless $\{i,j\}=\{1,2\}$ and $\val_p(\alpha_1)=-\frac{1}{2}$, $\val_p(\alpha_2)=\frac{1}{2}$. However, $\val_p(\alpha_1)=-\frac{1}{2}$, $\val_p(\alpha_2)=\frac{1}{2}$ imply that $\Pi$ is ordinary at $p$ and this contradicts that $\rho_\Pi|_{G_{\cK,\fp}}$ is not a direct sum of four characters. Therefore, we have  $\alpha_i\alpha^{-1}_j\neq p^{\pm 1}$ for all $i,j$, so  $\Ind^{\GL_4(\bQ_p)}_{Q(\bQ_p)}\sigma\boxtimes\chi\boxtimes\chi'$ is irreducible and $\Pi_p$ is unramified. It follows that $\rho_\Pi|_{G_{\cK,\fp}}$ is crystalline, so $\rho'_{3,x}|_{G_{\bQ,p}}$ is crystalline. Also, $\bar{\rho}'_{3,x}\equiv\bar{\rho}_\pi$ is not induced from a Galois character for $\bQ\left(\sqrt{(-1)^{(p-1)/2}p}\right)$, because its restriction to $I_p$ has semi-simplification $\omega^i_2\oplus\omega^{pi}_2$ with $\omega_2$ the fundamental character of level $2$ and $i\equiv 1\mod (p-1)$ (since $\rho_\pi|_{G_{\bQ,p}}$ is crystalline of weight $(0,1)$), and if $\bar{\pi}$ was induced from a Galois character for $\bQ\left(\sqrt{(-1)^{(p-1)/2}p}\right)$, $i$ must be a multiple of $\frac{p+1}{2}$, which is impossible when $p$ is odd. Thus, we can apply \cite{Kisin}, we deduce that $\rho'_{3,x}$ is modular. Then $\Pi$ is CAP, {\it i.e.} it has the same system of Hecke eigenvalues as a Klingen-type
Eisenstein series). This is impossible thanks to the condition on the weights $(0,0,t^+;t^-)$ by \cite[Theorem 2.5.6]{HaEC}. Property (9) is verified.

Let 
\begin{align*}
   X^{\Sigma\,\cup\{\ell,\ell'\}}_{\xi\circ\Nm}=\Hom_{\bZ_p}\left(
   \ker\left\{H^1\left(\bQ,\cO^{\ur}_L\llb\Gamma_\cK\rrb,\xi\Psi_\cK\circ \Nm\right)\ra H^1\left(G_{\cK,\fp},\xi\Psi_\cK\circ \Nm\right)\right\},\,\bQ_p/\bZ_p\right).
\end{align*}
With the input (1)-(5) satisfying the properties (6)-(9), applying the lattice construction \cite[Proposition 9.2]{WanU31}, one proves that if $P\subset\cO^\ur_L\llb\Gamma_\cK\rrb$ is a height one prime such that $\ord_P\left(\cE_{\pi,\xi}\right)>0$ and $\left(X_{\xi\circ\Nm}\otimes_\bZ\bQ)\right)_P=0$, then the inequality \eqref{eq:cE-incl} holds. It follows that for all height one prime ideals $P\subset \cO^\ur_L\llb\Gamma_\cK\rrb$ with $P\cap \cO^\ur_L\llb\Gamma^+_\cK\rrb=(0)$, the desired inequality \eqref{eq:incl} holds.

(The congruence ideal $\cE_{\pi,\xi}$ is bounded by the product  $\cL^{\Sigma\,\cup\{\ell,\ell'\}}_{\xi,\bQ}\cL^{\Sigma\,\cup\{\ell,\ell'\}}_{\pi,\cK,\xi}$, and is expected to give a bound for the product of the characteristic ideals of $X^{\Sigma\,\cup\{\ell,\ell'\}}_{\xi\circ\Nm}$ and $X^{\Sigma\,\cup\{\ell,\ell'\}}_{\pi,\cK,\xi}$. Although, the characteristic ideals of $X^{\Sigma\,\cup\{\ell,\ell'\}}_{\xi\circ\Nm}$ is known to be generated by $\cL^{\Sigma\,\cup\{\ell,\ell'\}}_{\xi,\bQ}$ by the Iwasawa Main Conjecture proved in \cite{MWmainconj}, the lattice construction argument in \cite{SU,WanU31} only proves the desired bound for $\ord_P\left(\mr{char}_{\cO^\ur_L\llb\Gamma_\cK\rrb}\left(X^{\Sigma\,\cup\{\ell,\ell'\}}_{\pi,\cK,\xi}\right)\right)$ when $\left(X_{\xi\circ\Nm}\otimes_\bZ\bQ)\right)_P=0$. In \cite{WanU31}, height one primes in $\cO^\ur_L\llb\Gamma^+_\cK\rrb\otimes_\bZ\bQ$ do not need to be excluded thanks to Lemma 9.3 {\it loc.cit} and the fact that $\cL_{\pi,\cK,\xi}$ is not contained in any height one prime ideal in $\cO^\ur_L\llb\Gamma_\cK\rrb$ by \cite{HsiehRankin}. However, in our case with $\pi$ fixed instead of moving in a family, we cannot prove an analogue of Lemma 9.3  {\it loc.\,cit}.)

\vspace{1em}

(2) Let $\sigma_\xi$ be the irreducible cuspidal automorphic representation of $\GL_2(\bA_\bQ)$ generated by a CM form associated to $\xi$, and $\rho_{\sigma_\xi}:\Gal(\ol{\bQ}/\bQ)\ra\GL_2(\ol{\bQ}_p)$ be the Galois representation associated to $\sigma_\xi$. Then $\rho_{\sigma_\xi}\cong \Ind^\bQ_\cK \rho_{\cK,\xi}$. The condition $\xi_\fp\not\equiv\xi_{\bar{\fp}}$ modulo the maximal ideal of $\cO_L$ implies that
\begin{table}[H]  
\begin{tabular}{cp{30em}}
   (\textbf{irred}) &  $\rho_{\sigma_\xi}$ is residually absolutely irreducible,\\
   (\textbf{dist})  &  the characters of $\rho_{\sigma_\xi}|_{G_{\bQ,p}}$ on the diagonal are distinct.
\end{tabular}
\end{table}
\noindent Let $\bm{g}$ be the ordinary $\cO^\ur_L\llb\Gamma_\cK\rrb$-adic CM family passing through the CM form associated to $\xi$. As in \cite[\S7E]{WanU31}, one can construct a $p$-adic $L$-function $\gls{LpiHida}$ by taking the product of Hida's Rankin--Selberg $p$-adic $L$-functions for the family $\bm{g}$ and the newform $f\in \pi$ and  Katz's $p$-adic $L$-function restricted to the line of interpolating special values at $1$. Note that although Hida's construction assumes that $\bm{g}$ and $\pi$ are both ordinary, only the ordinarity of $\bm{g}$ is used. By using \cite[Theorem 7.1, (8.8b)]{HiTi93} and \cite[Lemma 5.3(vi)]{HiFRankin} (which relates the Petersson norm with the adjoint $L$-value at $1$), one can check that $\cL^{\mr{Hida}}_{\pi,\cK,\xi}\in \cO^\ur_L\llb\Gamma_\cK\rrb$ and the specializations of $\cL^{\mr{Hida}}_{\pi,\cK,\xi}$ and our $\cL_{\pi,\cK,\xi}$ agree at a Zariski dense subsets of the weight space. Hence, $\cL^{\mr{Hida}}_{\pi,\cK,\xi}=\cL_{\pi,\cK,\xi}$. Then under the conditions ({\bf dist}) and ({\bf dist}), by using the Iwasawa Main conjecture proved in \cite{HiTi93,HiTi94,Hida06}, one can check that $\cL^{\mr{Hida}}_{\pi,\cK,\xi}\in \cO^\ur_L\llb\Gamma_\cK\rrb$. Hence,  $\cL_{\pi,\cK,\xi}\in \cO^\ur_L\llb\Gamma_\cK\rrb$.

With further conditions that $\xi|_{\bA^\times_\bQ}=\omega^2$ and $k_0\equiv 0\mod 2(p-1)$ and conditions on the epsilon factors and the conductors of $\xi$, one can use the vanishing of the anticyclotomic $\mu$-invariant proved in \cite{HsiehRankin} to deduce that the $\cL_{\pi,\cK,\xi}$ is not divisible by any height one prime ideal in $\cO^\ur_L\llb\Gamma^+_\cK\rrb$. Thus, the inclusion in (1) implies the inclusion in (2).
\end{proof}

\section*{Index}
\renewcommand{\glossarysection}[2][]{}

\printglossaries

\bibliographystyle{alpha}
\bibliography{references}

\end{document}